\documentclass[10pt, reqno]{amsart}

\title[MF on GLSM]{Fundamental Factorization of a GLSM\\ Part I: Construction}

\author[Ciocan-Fontanine]{Ionut Ciocan-Fontanine}
\address{
  \begin{tabular}{l}
   Ionut Ciocan-Fontanine \\
   \hspace{.1in} University of Minnesota \\
   \hspace{.1in} School of Mathematics, 127 Vincent Hall \\
     \hspace{.1in}  206 Church St. SE, Minneapolis, MN, USA 55455 \\
   \hspace{.1in} Korea Institute for Advanced Study \\
   \hspace{.1in} 85 Hoegiro, Dongdaemun-gu, Seoul, Republic of Korea 02455 \\
   \hspace{.1in} Email: {\bf ciocan@math.umn.edu} \\
  \end{tabular}
}

\author[Favero]{David Favero}
\address{
  \begin{tabular}{l}
   David Favero \\
   \hspace{.1in} University of Alberta \\
   \hspace{.1in} Department of Mathematical and Statistical Sciences \\
   \hspace{.1in} Central Academic Building 632, Edmonton, AB, Canada T6G 2C7 \\
   \hspace{.1in} Korea Institute for Advanced Study \\
   \hspace{.1in} 85 Hoegiro, Dongdaemun-gu, Seoul, Republic of Korea 02455 \\
   \hspace{.1in} Email: {\bf favero@ualberta.ca} \\
  \end{tabular}
}

\author[Gu\'er\'e]{J\'er\'emy Gu\'er\'e}
\address{
  \begin{tabular}{l}
   J\'er\'emy Gu\'er\'e \\
   \hspace{.1in} Univ. Grenoble Alpes\\
   \hspace{.1in} CNRS, Institut Fourier, F-38000 Grenoble, France \\
    \hspace{.1in} Email: {\bf jeremy.guere@gmail.com} \\
  \end{tabular}
}

\author[Kim]{Bumsig Kim}
\address{
  \begin{tabular}{l}
   Bumsig Kim \\
   \hspace{.1in} Korea Institute for Advanced Study \\
   \hspace{.1in} 85 Hoegiro, Dongdaemun-gu, Seoul, Republic of Korea 02455 \\
   \hspace{.1in} Email: {\bf bumsig@kias.re.kr} \\
  \end{tabular}
}

\author[Shoemaker]{Mark Shoemaker}
\address{
  \begin{tabular}{l}
   Mark Shoemaker \\
   \hspace{.1in} Colorado State University \\
      \hspace{.1in} Department of Mathematics \\
   \hspace{.1in} 1874 Campus Delivery, Fort Collins, CO, USA, 80523-1874\\
   \hspace{.1in} Email: {\bf mark.shoemaker@colostate.edu} \\
  \end{tabular}
}

\usepackage[alphabetic]{amsrefs}
\usepackage{amsfonts}
\usepackage{amsmath}
\usepackage{amssymb}
\usepackage{amscd}
\usepackage{array}
\usepackage{subfigure}
\usepackage{tikz}
\usetikzlibrary{arrows}
\usepackage{tikz-cd}
\usepackage{hyperref}
\usepackage{mathrsfs}
\usepackage{dsfont}
\usepackage{ulem}

\usepackage{amsthm}
\usepackage{thmtools}
\declaretheoremstyle[notefont=\bfseries,notebraces={}{},
    headpunct={},postheadspace=1em]{mystyle}
\declaretheorem[style=mystyle,numbered=no,name=Condition]{condition-hand}

\usepackage[all]{xy}

\newcommand{\cc}[1]{\mathcal{#1}}  
\newcommand{\f}[1]{\mathfrak{#1}}  

\newcommand{\newterm}{\textsf}
\newcommand{\op}[1]{\operatorname{#1}}

\allowdisplaybreaks[3]
\newcommand{\sslash}{\mathbin{/\mkern-6mu/}}

\newcommand{\cE}{\mathcal{E}}

\newcommand{\PP}{\mathbb{P}}
\newcommand{\CC}{\mathbb{C}}
\newcommand{\ZZ}{\mathbb{Z}}
\newcommand{\NN}{\mathbb{N}}
\newcommand{\RR}{\mathbb{R}}
\newcommand{\QQ}{\mathbb{Q}}

\newcommand{\cO}{\mathcal{O}}
\newcommand{\cA}{\mathcal{A}}
\newcommand{\ev}{\mathrm{ev}}

\newcommand{\cB}{\mathcal{B}}
\newcommand{\cC}{\mathfrak{C}}
\newcommand{\cD}{\mathcal{D}}
\newcommand{\cF}{\mathcal{F}}

\newcommand{\bg}{\boldsymbol{(\underline g)}}

\newcommand{\cL}{\mathcal{L}}
\newcommand{\cM}{\mathcal{M}}

\newcommand{\cP}{\mathcal{P}}

\newcommand{\cV}{\mathcal{V}}
\newcommand{\cW}{\mathcal{W}}
\newcommand{\cX}{\mathcal{X}}
\newcommand{\cY}{\mathcal{Y}}
\newcommand{\cT}{\mathcal{T}}
\newcommand{\cZ}{\mathcal{Z}}

\newcommand{\Spec}{\textrm{Spec}}

\DeclareMathOperator{\GJ}{\langle J \rangle}

\DeclareMathOperator{\ch}{ch}

\newcommand{\rank}{\operatorname{rank}}

\newcommand{\logc}{ \mathring\omega ^{\mathrm{log}}_{\mathfrak{C}} }

\def\tand{\text{ and } }
\def\log{ \omega_{\cC}^{\op{log}} }
\def\nolog{ \omega_{\cC}}

\newcommand{\Xrig}{\cX^{\op{rig}}}

\theoremstyle{plain}
\newtheorem{thm}{Theorem}[subsection]
\newtheorem{theorem}[thm]{Theorem}

\newtheorem{dfn}[thm]{Definition}
\newtheorem{definition}[thm]{Definition}
\newtheorem*{dfn*}{Definition}

\newtheorem{proposition}[thm]{Proposition}

\newtheorem*{lem*}{Lemma}
\newtheorem{lemma}[thm]{Lemma}

\newtheorem{corollary}[thm]{Corollary}

\newtheorem{condition}{Condition}

\theoremstyle{definition}

\newtheorem{nt}[thm]{Notation}

\newtheorem{reduction}[thm]{Reduction}

\newtheorem{rem}[thm]{Remark}
\newtheorem*{rem*}{Remark}
\newtheorem*{rems*}{Remarks}

\newtheorem{exa}[thm]{Example}
\newtheorem*{exa*}{Example}
\newtheorem*{exait*}{\rm \em Example}
\newtheorem*{exadefit*}{\rm \em Example/Definition}

\newtheorem*{cla*}{\rm \em Claim}

\newtheorem{remark}[thm]{Remark}

\usepackage{amsxtra} 
\newcommand{\C}{\mathbb{C}}
\newcommand{\R}{\mathbb{R}} 
\newcommand{\Z}{\mathbb{Z}}

\newcommand{\Hom}{\operatorname{Hom}}

\def\Z{\op{\mathbb{Z}}}
\def\Q{\op{\mathbb{Q}}}
\def\presuper#1#2%
  {\mathop{}%
   \mathopen{\vphantom{#2}}^{#1}%
   \kern-\scriptspace%
   #2}
\def\chp{\presuper{\Z_2}{\op{ch}}}

\newcommand{\orb}{\text{\rm orb}}
\newcommand{\tw}{\text{\rm tw}}

\def\<{\left\langle}
\def\>{\right\rangle}

\def\A{\op{\mathbb{A}}}
\def\O{\op{\mathcal{O}}}
\newcommand{\E}{\mathcal{E}}
\newcommand{\F}{\mathcal{F}}
\newcommand{\T}{\mathcal{T}}

\newcommand{\dabsfact}[1]{\operatorname{D}(#1)}
\newcommand{\dbcoh}[1]{\operatorname{D}(#1)}

\newcommand{\ti }{\times}
\newcommand{\ot }{\otimes}
\newcommand{\ra }{\rightarrow}

\newcommand{\ord}{{\mathrm{ord}}}

\newcommand{\Sym}{{\mathrm{Sym}}}

\newcommand{\cK}{{\mathcal{K}}}
\newcommand{\ke }{{\varepsilon }}
\newcommand{\kb }{{\beta}}
\newcommand{\ka }{{\alpha}}

\let\wt\widetilde

\newcommand{\fM}{\mathfrak{M}}

\newcommand{\fX}{\mathfrak{X}}

\newcommand{\lan}{\langle}
\newcommand{\ran}{\rangle}

\newcommand{\cN}{\mathcal{N}}

\newcommand{\stab}{{\mathrm{st}}}

\newcommand{\bdeg}{\mathbf{deg}}

\newcommand{\LL}{\mathbb{L}}
\newcommand{\ufC}{\underline{\fC}}

\newcommand{\sL}{\mathscr{L}}

\newcommand{\td}{\mathrm{td}}

\newcommand{\sMbar}{\overline{\mathcal{M}}}
\newcommand{\sKbar}{\overline{\mathcal{K}}}
\newcommand{\Mfrak}{\mathfrak{M}}

\newcommand{\VmodtG}{V/\!\!/_\theta G}

\newcommand{\ringlog}{\mathring\omega ^{\mathrm{log}}_{\mathfrak{C}}}
\newcommand{\sG}{\mathscr{G}}
\newcommand{\VG}{[V_1/\!\!/_{\theta}G]}
\newcommand{\VGJ}{[V_1/\!\!/_{\theta} G_1]}
\newcommand{\fC}{\mathfrak{C}}
\newcommand{\pii}{ \pi_* {\Sigma_i}_*}
\newcommand{\Hoch}{\mathrm{Hoch}}

\DeclareMathOperator{\tot}{tot}
\DeclareMathOperator{\taut}{taut}

\numberwithin{equation}{section}

\begin{document}

\begin{abstract}
We define enumerative invariants associated to a hybrid Gauged Linear Sigma Model.  We prove that in the relevant special cases these invariants recover both the Gromov--Witten type invariants defined by Chang--Li and Fan--Jarvis--Ruan using cosection localization as well as the FJRW type invariants constructed by Polishchuk--Vaintrob.
The invariants are defined by constructing a ``fundamental factorization'' supported on the moduli space of Landau--Ginzburg maps to a convex hybrid model.  This gives the kernel of a 
 Fourier--Mukai transform; the associated map on Hochschild homology defines our theory.
\end{abstract}

\maketitle

\setcounter{tocdepth}{1}
\tableofcontents

\setcounter{section}{-1}
\section{Introduction}

\subsection{Background}
In the last quarter century, enumerative geometry has witnessed a dramatic transformation, due to a tremendous influx of ideas coming from string theory.  The development of Gromov--Witten theory for smooth algebraic varieties is one such example.  These results have revealed two important facts about the field, which were not previously apparent but are now crucial to its study.  First, that certain enumerative invariants of a smooth variety have a rich recursive structure when organized correctly.  For curve counting theories such as Gromov--Witten theory, the invariants form what is known as a \newterm{Cohomological Field Theory} (CohFT).  Second, the enumerative invariants for \textit{different spaces} are often related to each other in dramatic and surprising ways.  These correspondences take many forms, with most originating from physical considerations.
One of the first and most striking instances of both of the phenomena described above was the celebrated mirror theorem, first conjectured by Candelas et. al. \cite{CdlOGP} and subsequently proven by Givental \cite{G1}.
In this theorem, the rational curve counts of a quintic hypersurface $Q$ in $\PP^4$ were shown to be dictated by the variation of Hodge structure of a different space, the so-called ``mirror manifold'' to the quintic.

In attempting to realize the full scope of mirror symmetry and other related correspondences, one is led to consider more exotic geometries generalizing the notion of a smooth manifold or variety.  For instance the ``mirror'' to a smooth variety may not itself be a smooth variety but rather an \textit{orbifold}, i.e. a smooth Deligne--Mumford stack.  The Gromov--Witten theory of orbifolds has been an active area of study in the last decade \cite{ChenR1, ChenR2, AGV}.  In another direction, again inspired by mirror symmetry considerations \cite{Wi2}, one can generalize the idea of a smooth variety by adding the data of a function, or \textit{superpotential}.  A space $X$ together with a superpotential $w: X \to \mathbb{A}^1$ is known as a \newterm{Landau--Ginzburg model} (LG model).  

For example, the space $T = \tot( \cc O_{\PP^4}(-5))$ together with the superpotential \[w = p\sum_{i=1}^5 x_i^5\] is the LG model realization of the quintic hypersurface $Q = \{\sum_{i=1}^5 x_i^5 = 0\} \subset \PP^4$ (See Example~\ref{e:quintic} for details).  Note that the degeneracy locus of $w$ in $T$ is exactly the hypersurface $Q$.  A suitably defined curve-counting theory for the LG model $(T, w)$ should agree with the Gromov--Witten theory for $Q$ (see \cite{CL} for such a result).  In certain special cases, LG models have also been studied by mathematicians in more familiar contexts.  For instance given a polynomial $p: \mathbb{A}^N \to \mathbb{A}^1$ with an isolated singularity at the origin, many of the invariants of the singularity defined by $p$ can be equally considered as classical invariants of the LG model $(\mathbb{A}^N, p)$.   
Thanks to pioneering work of Witten \cite{Witten1} and a rigorous mathematical construction of Fan, Jarvis, and Ruan  \cite{FJR13,FJR07} and Polishchuk and Vaintrob \cite{PV},  an enumerative theory known as FJRW theory has been defined for these \textit{affine} LG models, in analogy with Gromov--Witten theory.

A particularly useful class of Landau--Ginzburg models are known as \newterm{Gauged Linear Sigma Models} (GLSMs).  A GLSM is defined in part by a vector space $V$ together with a linear action of a reductive group $G$ on $V$, and a $G$-invariant function $w: V \to \mathbb{A}^1$ (see Definition \ref{d:input0} for details).  The space $X$ is then given as a GIT quotient of $V$ by $G$, and the function $w$ descends to a superpotential $\bar w: X \to \mathbb{A}^1$.  In fact both the quintic hypersurface $Q$ ``='' $(T, w)$ and the affine model $(\mathbb{A}^N, p)$ are particular examples of GLSMs. 
Hence both Gromov--Witten theory (of hypersurfaces) and FJRW theory should be special cases of a more general but currently elusive curve counting theory for GLSMs.  For a given GLSM, this theory should have the structure of a cohomological field theory.

In a major recent advance, Fan, Jarvis, and Ruan  \cite{FJR15}  have begun to construct such a theory.
More precisely, they construct a collection of moduli spaces associated to a GLSM, the so-called moduli of \textit{Landau--Ginzburg stable maps} (Definition~\ref{stableLG}), and prove these spaces are Deligne--Mumford stacks.  
In addition, in the case of ``compact type'' insertions (see \cite[Definition 4.1.4]{FJR15}), they are able to endow these spaces with a virtual cycle defined on a proper substack.  Enumerative invariants can then be defined as integrals over this cycle.  This breakthrough is the first appearance of a theory of general GLSM invariants in mathematics.  Unfortunately, certain technical constraints place restrictions on when these methods can be applied, and the technique yields only a partially defined theory.

\subsection{Construction}
The primary goal of the current project is to construct an enumerative curve counting theory for GLSMs 
which both is purely algebraic and defines a full cohomological field theory.
In the current paper, we restrict our attention to a class of GLSMs called \textit{convex hybrid models}.  The theory is constructed in terms of the derived category of factorizations.  This work forms the bulk of the paper, and is completed in \S\ref{s:GLSMtheory} (see Definition~\ref{d:invariants}).

Our methods are inspired by Polishchuk and Vaintrob's innovative use of matrix factorizations to give an algebraic construction of the FJRW theory of affine LG models  \cite{PV}.  The current work extends their methods to the general setting of hybrid model GLSMs.
As with Polishchuk--Vaintrob, the construction takes place at the level of derived categories; for LG models these are categories of factorizations (also known as matrix factorizations). 

Given a hybrid model GLSM, we construct a family of smooth moduli spaces which contain the moduli of LG stable maps as a closed substack.  On these moduli spaces we define 
a \textit{fundamental factorization}, an object in the derived category of factorizations with support contained in the space of LG maps described above.

As our virtual cycle is lifted to a categorical object, so too are our generalized enumerative invariants.  Namely, the fundamental factorization serves as the kernel of a Fourier--Mukai transform - a functor from the derived category of factorizations of a hybrid model to the derived category of the moduli space of curves $\sMbar_{g,r}$.
This functor can be viewed as a categorical lift of the usual description of Gromov--Witten-type invariants in terms of cohomological integral transforms.  Indeed, the passage to 
Hochschild homology (see Section~\ref{s:hochom}) recovers the more familiar description, after application of a suitable Hochschild--Kostant--Rosenberg isomorphism identifying Hochschild homology with cohomology.

\subsection{Comparisons}
As mentioned above, both the Gromov--Witten theory of a complete intersection (in a convex GIT quotient of the form $[\VmodtG]$), as well as the FJRW theory of a singularity $w: [\mathbb{A}^N/G] \to \mathbb{A}^1$ are special cases of the hybrid model GLSMs we consider.  Indeed the term \textit{hybrid model} refers to a ``hybrid'' between Gromov--Witten and FJRW theory.  As such, a key test of the validity of our construction is to show that in these two opposite limiting cases our invariants agree with the previously defined Gromov--Witten or FJRW invariants.  These comparison results are proven in \S\ref{s:comwitothcon}, Theorems~\ref{t:mainGWcomp} and \ref{p:PVcomp1} respectively.

There are many examples of different GLSMs which define the same LG model.  
More precisely, one may easily construct a pair of GLSMs involving different vector spaces, groups and functions, but such that the corresponding GIT quotients together with their induced superpotentials are isomorphic.  In this case it is natural to ask if the corresponding GLSM invariants can be identified with one another.  In the final section we consider so-called \newterm{affine phase} GLSMs, i.e. abelian GLSMs such that the corresponding GIT quotient is isomorphic to a singularity $w: [\mathbb{A}^N/G] \to \mathbb{A}^1$. We prove (Theorem~\ref{t:equivGLSM}) that the invariants of any such affine phase agree with the FJRW invariants of the singularity as defined in \cite{PV}.

\subsection{Future work}

There are two technical steps in the construction presented herein which impose restrictions on
the types of GLSMs that we currently treat.
They are:
\begin{itemize}
	\item to embed Fan--Jarvis--Ruan's moduli space of LG maps (Definition~\ref{stableLG}) into a smooth  Deligne--Mumford stack, denoted $\square$,  which has \textit{quasi-projective coarse moduli};
	\item to realize this embedding as the zero locus of a section of a certain vector bundle on $\square$.
\end{itemize}
The first step forces us to restrict our attention to the class of so-called hybrid model GLSMs, due to the projectivity requirement.  
To obtain the second step, we impose the convexity condition of Definition \ref{d:convexity}.
These restrictions are in fact not strictly necessary, but we leave the more general case to a later paper for two reasons.  First, we wish to streamline the exposition here as much as possible.  Second, the construction described here is important in its own right, as it is necessary for our comparison with Gromov--Witten theory in \S~\ref{s:comwitothcon}.

Both Gromov--Witten theory and FJRW theory are known to have the structure of a cohomological field theory.  
In fact the general GLSM invariants constructed here also form a cohomological field theory.  For ``brevity'' we have chosen to save this result for the sequel \cite{CKFGS}.

Finally, we hope that our construction will shed new light on the connections between previously defined enumerative theories.  For example, our level of generality should be readily applicable to wall crossing results such as the LG/CY correspondence and should help to service the abundance of recent developments in this area. See e.g.~\cite{ChR, AS1, LPS, GR, CJR}.

\subsection{Structure of the paper}
\S\ref{s:OC} is in many ways an expanded introduction.  We define GLSMs and hybrid models, give a general overview of our construction of an enumerative theory for GLSMs,
and review the description of the moduli space constructed in \cite{FJR15}.

In \S\ref{s:fact}, we give an overview of categories of factorizations.   It summarizes the constructions of the (derived) pushforward, pullback, tensor product (see \S\ref{s:derfunc}), and Fourier--Mukai transforms, as well as 
 Hochschild homology (see \S\ref{s:hochom}). 
 Moreover, we set the stage for our comparison with Gromov--Witten theory, with a collection of results relating various cohomological and categorical integral transforms.

In \S\ref{s:PV machinery}, we construct a factorization associated to families of curves equipped with GLSM data.  We call this factorization the Polishchuk--Vaintrob (PV) factorization and determine its properties.  
In order to define it, we require three conditions (see \S\ref{Setup} and \S\ref{ss: admissible resolutions}).
We prove that Condition \ref{assume ev chainlevel} is automatic (see \S\ref{cond1}) and that Conditions \ref{assume alpha chainlevel} and \ref{compatible} are satisfied assuming that the base is a smooth Deligne-Mumford stack which is projective over an affine; this is called Assumption $(\star)$ (see \S\ref{ss:sc}).
We establish that the PV factorization is supported along the critical locus of the superpotential in \S \ref{s:supofthepvfac}.  Then finally in \S \ref{s:rigev}, we discuss a $\CC^*$-equivariant version of the PV factorization necessary for our construction.  

\S\ref{s:conFF} is dedicated to embedding the space of LG maps in a smooth Deligne--Mumford stack, and proving that the necessary projectivity assumptions hold on this ambient space.  

Finally in \S\ref{s:GLSMtheory}, we define our GLSM invariants.  We gather results from previous sections to construct a ``fundamental factorization'' supported on the moduli space of LG maps to the critical locus.  This forms the kernel of an integral transform, which is used to define our invariants.  In this section also we prove that these invariants are independent of the choices made in the construction.
 
In \S\ref{s:comwitothcon}, we compare the GLSM invariants of \S\ref{s:GLSMtheory} with other previously defined theories, such as a cosection localized version of Gromov--Witten theory of complete intersection varieties and the quantum singularity theory (or FJRW theory).

\subsection{Conventions and notations} \label{subsec: conventions}
We work in the algebraic category and over the field $\CC$ of complex numbers.
The algebraic stacks considered in this paper are finite type and semi-separated (i.e. with affine diagonal) over the ground field unless otherwise stated. 
A quotient stack is the stack quotient of an algebraic space by a linear algebraic group unless otherwise stated.
 Let $\rho : \cY \to M$ be the coarse moduli space of a separated DM stack $\cY$.
An invertible sheaf $\sL$ on  $\cY$ will be called {\it ample} 
 if the natural map $\rho^*\rho_*\sL \to \sL$ is an isomorphism and $\rho_*\sL$ is an ample invertible sheaf on the algebraic space $M$ (in the usual sense).  Our conventions regarding the relative ampleness of an invertible sheaf on a universal curve are described in \S \ref{subsubsec: cochain-level}.

\subsection{\it Acknowledgements} I. Ciocan-Fontanine was partially supported by NSF grant DMS-1601771.  D. Favero was partially funded by the Natural Sciences and Engineering Research Council of Canada and Canada Research Chair Program under NSERC RGPIN 04596 and CRC TIER2 229953.
J. Gu\'er\'e was first funded by the Einstein Foundation at the Humboldt University of Berlin and then by the University of Grenoble--Alpes.
B. Kim was partially supported by KIAS individual grant MG016404.
M. Shoemaker was partially supported by NSF grant DMS-1708104. 

B. Kim would like to thank K. Chung, T. Jarvis, T. Kim, S. Lee, A. Kresch,  M. Miura, and A. Polishchuk for helpful discussions.
M. Shoemaker would like to thank Y. Ruan and T. Jarvis for teaching him the GLSM framework and I. Shipman for many useful conversations on factorizations.   I. Ciocan-Fontanine and D. Favero would also like to thank the  Korea Institute for Advanced Study for their superb and abundant hospitality throughout the process of developing this work.   All authors are also grateful to the University of Minnesota's Department of Mathematics for the use of their facilities during several extended collaborative meetings.

\section{Overview of the construction}\label{s:OC}

\subsection{Input data} 

A \newterm{Gauged Linear Sigma Model}, or GLSM, is a collection of data describing a GIT quotient $[\VmodtG]$ of a vector space $V$ 
together with a superpotential $w \colon  [\VmodtG] \to \mathbb{A}^1$ and an ``$R$-charge". 
The ``open string $B$-model" theory of a GLSM is given by the \newterm{derived categories of factorizations}.  These categories have received a lot of attention in recent years (e.g. \cite{BFK14} \cite{PV}).
We discuss these categories in detail in \S\ref{s:fact} below. 
A ``closed string $A$-model" theory with target a GLSM was introduced by Fan, Jarvis, and Ruan \cite{FJR15}. We recall briefly the main points.

\begin{dfn}\label{d:input0}
The \newterm{GLSM input data} $(V, G, \CC^*_R, \theta, w)$ consists of:
\begin{enumerate}
\item a $\ZZ$-graded vector space, $V=\oplus_{n\in\ZZ}V_n$, with grading induced by the action (called the \newterm{R-charge}) of a one-dimensional torus 
$\CC^*_R  \subseteq GL(V)$;
\item an action by a linearly reductive group $G  \subseteq GL(V)$;
\item a choice of rational character $\theta \in \widehat G_\QQ$, where $\widehat G := \mathrm{Hom}(G,\CC^*)$, and $\widehat G_\QQ := \widehat G\otimes_\ZZ \QQ$; 
\item a $G$-invariant polynomial function $w \colon  V \to \mathbb{A}^1$, homogeneous of degree $\bdeg > 0$ with respect to the grading action $\CC^*_R$. 
\end{enumerate}
The actions of $G$ and $\CC^*_R$ are required to be \newterm{compatible}, that is, to satisfy
\begin{itemize} 
\item $G$ and $\CC^*_R$ commute: $g\lambda = \lambda g$ for all $g \in G$ and $\lambda \in \CC^*_R$;
\item $\langle J\rangle := G \cap \CC^*_R \simeq \mu_{\bdeg}$.
\end{itemize}
Here
\[
 J := (..., e^{2 \pi \sqrt{-1} c_i/\bdeg},  ... )\]
is the diagonal element in $GL(V)$ given by the weights $c_i$ of the $\CC ^*_R$-action on $V$.
Moreover, we assume that there are no strictly semistable points for the linearization of the $G$-action on $V$ given by $\theta$:
\begin{equation}\label{nostrictsemistablepts}
V^{ss}(\theta)=V^s(\theta).
\end{equation}
\end{dfn}

It follows from equation \eqref{nostrictsemistablepts} that the GIT stack quotient  
\[ [V/\!\!/_{\theta}G]:=[V^{ss}(\theta) / G]
\] 
is a smooth separated Deligne--Mumford stack, and that  
its coarse moduli space, $V/\!\!/_{\theta}G$, is projective over the affine scheme $\Spec\left( \text{Sym}^\bullet(V^\vee) \right)^G$.

The regular function $w \colon  V \to \mathbb{A}^1$ is called the \newterm{superpotential}.
The superpotential descends to a function on the GIT stack quotient, which, by abuse, we also denote $ w \colon  [ \VmodtG] \to \mathbb{A}^1$.  Define the closed substack
\[  Z (d w )  \subseteq [ V^{ss}(\theta) / G],\] as the critical locus of $w$, i.e., the zero locus of the section $d  w$.
We say the superpotential $w$ is \newterm{nondegenerate} if $Z(d{w})$ is proper over $\Spec(\CC)$.

Furthermore, we introduce the group 
\[
\Gamma := G \cdot \CC^*_R  \subseteq \mathrm{GL}(V).
\]
By compatibility, the R-charge induces a well defined character 
\[\begin{array}{ll} \chi \colon  &\Gamma \to \CC^* \\ & g \cdot \lambda \mapsto \lambda^{\bdeg} \end{array}\] 
for $g \in G$ and $\lambda \in \CC^*_R$.  We obtain the short exact sequence
\begin{equation}\label{e:gses}
1 \to G \to \Gamma \stackrel{\chi}{\to}  \CC^* \to 1.
\end{equation}

\begin{remark}
In \cite{FJR15}, the definition of the $A$-theory of a GLSM also requires a choice of a \newterm{good lift} of $\theta$ to $\widehat \Gamma$, that is, the choice of a character $\nu \in \widehat \Gamma$ whose restriction to $G$ is equal to $\theta$ and such that the $\Gamma$-semistable points for $\nu$ coincide with the $G$-semistable points for $\theta$:
\[ V^{ss}(\nu)=V^{ss}(\theta).\]
The existence of such a good lift is not automatic, but holds for almost all interesting examples; in particular, it will be automatic for the class of hybrid models that are the focus of this paper.
\end{remark}

\subsection{Landau--Ginzburg quasimaps}

Let $(V, G, \CC^*_R, \theta, \nu, w)$ be GLSM input data as above. 
Fix $g, r \geq 0$.
\begin{dfn}  \label{d:plg} 
A \newterm{prestable genus-$g$, $r$-pointed Landau--Ginzburg (LG) quasimap to $[V/\!\!/_\theta G]$} over a scheme $S$ consists of:
\begin{enumerate}
\item a prestable genus-$g$, $r$-pointed orbicurve  $(\pi: \cC \to S, \sG_1, \ldots, \sG_r )$ 
with a section $S \to  \sG _i$ of the gerbe marking $\sG_i$ for $1 \leq i \leq r$;
\item a principal $\Gamma$-bundle $\cP: \cC \to B\Gamma$ over $\cC$;
\item a section $\sigma: \cC \to \op{tot}\cV := \cP \times_\Gamma V$; and
\item an isomorphism $\varkappa: \chi_*(\cP) \to \mathring\omega ^{\mathrm{log}}_{\mathfrak{C}} $, \label{i:omlog}
\end{enumerate}
such that
\begin{itemize}
\item the morphism of stacks $\cP: \cC \to B\Gamma$ is representable;
\item for each geometric fiber $\cC_s$ of $\cC$, the subset $B:=\sigma^{-1}(\op{tot}\cV \setminus (\cP\times_\Gamma V^{ss}(\theta)))$ of $\cC_s$ is a finite set disjoint from the nodes and markings of $\cC_s$; the points of $B$ are called the base-points of the LG quasimap.
\end{itemize}
\end{dfn}

In item $\eqref{i:omlog}$ of the above definition, ${\log}$ is the logarithmic relative canonical line bundle and 
$\logc$ denotes the associated principal $\CC^*$-bundle. The notation $\op{tot}\cV$ is explained as follows: the mixed construction
$\cP\times_\Gamma V= [(\cP \times V)/\Gamma]$ defines a geometric vector bundle on $\cC$ and $\cV$ denotes its locally free sheaf of sections.

There is a natural notion of isomorphism of LG quasimaps, analogous to that of quasimaps \cite{CKM} which we present now.
Given a Landau--Ginzburg quasimap as above and a character $\eta \in \widehat \Gamma = \op{Hom}(\Gamma, \CC^*)$, let $ \CC(\eta)$ denote the corresponding representation and let $\cL_\eta$ denote the line bundle 
 \begin{equation}\cL_\eta := \cP \times_\Gamma \CC(\eta).\end{equation}
 If $S$ is connected,
the degree of the principal $\Gamma$-bundle $\cP$ is an element $d_0 \in \text{Hom}_{\Z}(\widehat \Gamma, \QQ)$ defined by
$$d_0( \eta) = \text{deg}(\cL_\eta) \in \QQ,$$ where $\text{deg}(\cL_\eta)$ is the degree of $\cL_\eta$ on the fibers of $\cC$.

Furthermore, by Equation~\eqref{e:gses} we have the sequence:
\[0 \to  \text{Hom}_{\Z}(\widehat G, \QQ) \to  \text{Hom}_{\Z}(\widehat \Gamma, \QQ)\stackrel{\chi_*}{\to}  \text{Hom}_{\Z}(\widehat \CC^*, \QQ) \to 0.
\]
Since $\chi_*$ is an isomorphism when restricted to $\text{Hom}_{\Z}(\widehat \CC^*_R, \QQ) $, this sequence has a distinguished splitting. Let 
\[
q \colon  \text{Hom}_{\Z}(\widehat \Gamma, \QQ) \to \text{Hom}_{\Z}(\widehat G, \QQ)
\]
 denote the map induced by this splitting, so that we have an isomorphism
 \[
 (q, \chi_*): \text{Hom}_{\Z}(\widehat \Gamma, \QQ)  \xrightarrow{\sim}  \text{Hom}_{\Z}(\widehat G, \QQ) \times  \text{Hom}_{\Z}(\widehat \CC^*, \QQ).
 \]

\begin{dfn}
The \newterm{degree} of an LG quasimap over a connected base is defined to be \[q(d_0) \in \op{Hom}_{\Z}(\widehat G, \QQ).\]
We say an
LG-quasimap over a scheme $S$ has degree $d$ if $q(d_0)=d$ for every connected component of $S$.
\end{dfn}

 \begin{rem}\label{degreeinG}
The reasoning for the definition above is as follows.  Given an LG quasimap of degree $d_0 \in \text{Hom}_{\Z}(\widehat \Gamma, \QQ)$, the image of $d_0$ under $\chi_*$ is determined by condition \eqref{i:omlog} of Definition~\ref{d:plg}.  Thus for an LG quasimap, the degree $d_0$ as an element of $\text{Hom}_{\Z}(\widehat \Gamma, \QQ)$ can be recovered from its image $q(d_0) \in \text{Hom}_{\Z}(\widehat G, \QQ)$.
\end{rem}
 
Let $Z$ be a closed subscheme of $V$, invariant under the action of $G$.  Let $\cZ_\theta$ denote $[Z \cap V^{ss}(\theta) / G]$.

\begin{dfn} A prestable LG quasimap to $\cZ_\theta$ over $S$ is a prestable LG quasimap to $[V/\!\!/_\theta G]$ satisfying the additional condition that the image of the associated $\Gamma$-equivariant map $[\sigma]:\cP\to V$ lies in $Z$.
\end{dfn}

\subsection{Stability conditions and moduli of stable LG maps}\label{sectLGmaps}
In analogy with the theory of quasimaps to GIT quotients \cite{CKM}, a family of stability conditions on LG quasimaps, indexed by
 a parameter $\ke\in \QQ_{>0}\cup \{ 0^+,\infty\}$ is introduced in \cite{FJR15}. 
In this paper, we work with the asymptotic stability $\ke=\infty$, which is directly analogous to the more familiar notion of stable maps.

 For a prestable $r$-pointed Landau--Ginzburg quasimap to $[V/\!\!/_\theta G]$ over $S$ as in Definition \eqref{d:plg}
 let $s$ be a geometric point of $S$ and let $\rho : \fC _s \to \ufC _s$ denote the coarse moduli space of the fiber $\fC _s$ at $s$. 
 By replacing $\nu$ with a multiple $m\nu$ if necessary, 
 we may assume that the natural map $\rho^*\rho_* \cL _{\nu}|_{\fC _s}  \to  \cL _{\nu}|_{\fC _s}$ is an isomorphism for each geometric point $s$. 
 
\begin{dfn}\label{stableLG}
A prestable $r$-pointed Landau--Ginzburg quasimap to $[V/\!\!/_\theta G]$ over $S$ is a \newterm{stable LG map to $[V/\!\!/_\theta G]$} if for each geometric fiber $C_s$ of $\cC$ the following two conditions hold:
\begin{enumerate}
\item the set $B$ of base-points is empty;
\item  the line bundle $(\log \otimes \cL_\nu^{\otimes M }) |_{\fC _s}$ is ample for $M$ sufficiently large.
\end{enumerate}
The same two conditions define stable LG maps to $\cZ = [Z \cap V^{ss} / G]$ for any closed $G$-invariant subscheme $Z \subseteq V$.
\end{dfn}
Note that stability depends on the good lift $\nu$ of $\theta$.

\begin{dfn}
Given $g$ and $r$, GLSM input data $(V, G, \CC^*_R, \theta, \nu, w)$, a choice of $Z  \subseteq V$ as above, and $d \in \op{Hom}_{\Z}(\widehat G, \QQ)$, the moduli stack
\[LG_{g, r}(\cZ, d),\]
of genus-$g$, $r$-pointed, degree $d$ Landau--Ginzburg maps to $\cZ$ is the stack parametrizing the degree $d$ {\it stable} families of Definition~\ref{stableLG}.  
\end{dfn}

\begin{theorem}[{\cite[Theorems 5.3.1 and 5.4.1]{FJR15}}]
The moduli space $LG_{g, r}(\cZ, d)$ is a separated Deligne--Mumford stack of finite type.  
When $\cZ$ is proper over $\mathrm{Spec}(\CC)$, the space $LG_{g, r}(\cZ, d)$ is proper as well. 
\end{theorem}
For convex hybrid models (see \S\ref{ss: hybrid models} and \S\ref{ss:convexity}), we will  prove further  that $LG_{g, r}(\cZ, d)$ is a global quotient stack with projective coarse moduli (Proposition~\ref{p: projmod}).

\subsection{Hybrid models}\label{ss: hybrid models}

In this paper we focus on \newterm{hybrid models}.
(In fact, an additional technical assumption, called convexity over $BG$, will be required; see \S\ref{ss:convexity}.)   
This class of GLSMs includes ``geometric phases", such as complete intersections in convex varieties, as well as the ``affine phases" of \cite{FJR15, PV}. 

\begin{dfn}\label{d:input}
The \newterm{hybrid model input data} $(V, G, \CC^*_R, \theta, w)$ is a GLSM input data as in Definiton \ref{d:input0}, satisfying the following additional requirements:
\begin{enumerate}
\item The graded vector space $V$ decomposes as  $V = V_1 \oplus V_2 \cong \CC^n \oplus \CC^m$, 
such that the $\CC^*_R$-action is trivial on $V_1$ and has positive weights $c_1,\dots , c_m>0$ on $V_2$;
\item The character $\theta \in \widehat G_\QQ$ is such that $V_1^{s}(\theta) = V_1^{ss}(\theta)$ and $V^{ss}(\theta)= V_1^{ss}(\theta) \times V_2$; \label{i:c}
\end{enumerate}
\end{dfn}

\begin{nt}
Let us denote
\begin{eqnarray*}
\cX & := & [V_1 /\!\!/_\theta G] := [ V_1^{ss}(\theta) / G], \\
\cT & := & [\VmodtG] := [ V^{ss}(\theta) / G], \\
\cc Z & := & Z(dw)  \subseteq \cT.
\end{eqnarray*}
\end{nt}

\begin{dfn} \label{dfn: hybrid model}
A \newterm{hybrid model} consists of input data $(V = V_1 \oplus V_2, G, \CC^*_R, \theta, w)$ from above such that
\begin{enumerate}
\item $\left( \Sym V_1^\vee \right)^G = \CC$ (so that the stack $\cX$ is projective);
\item $G$ and $\CC^*_R$ are compatible; and
\item $w$ is nondegenerate and vanishes at $0$.
\end{enumerate}
\end{dfn}

Note that, by requirement (a) of Definition \ref{dfn: hybrid model}, the stack $\cX$ is a smooth Deligne--Mumford stack whose coarse moduli is projective. 
Furthermore, by requirement (b) of Definition~\ref{d:input}, the stack $\cT$ is identified with the total space of a vector bundle (with fiber $V_2$) on the stack $\cX$.
Also, the superpotential $w$ is nondegenerate if and only if we have 
$\cc Z  \subseteq \cX  \subseteq \cT$, where $\cX  \subseteq \cT$ is the inclusion as the zero section.

\begin{remark}
In the hybrid model case, the choice of a good lift $\nu$ of $\theta$ always exists and is unique up to a multiple.
Namely, after first replacing $\theta$ by a multiple, we can assume that $\theta$ is trivial on 
$\langle J \rangle=\langle (1, \ldots, 1, e^{2 \pi i c_1/\deg}, \ldots, e^{2 \pi i c_m/\deg})\rangle$.
Hence, we can extend the character to $\Gamma$ by sending $\CC^*_R \subseteq \Gamma$ to 1.
Note that the semi-stable locus of $V$ remains unchanged in this process, meaning, in the language of \cite{FJR15}, that we indeed obtain a \newterm{good lift}. This is the reason for omitting the good lift from the notation for hybrid models.
\end{remark}

\begin{dfn}\label{d:phases}
\begin{enumerate}
\item A hybrid model GLSM $(V, G, \CC^*_R, \theta, w)$ is called an \newterm{affine phase} if $[V_1 \sslash_{\theta} G] \cong B G'$ for some finite group $G'$. 
\item A hybrid model GLSM $(V, G, \CC^*_R, \theta, w)$ is called a \newterm{geometric phase} if the group $\CC ^*_R$ acts on $V_2$ via the standard multiplication (so that $c_1=c_2= ... = c_m=1$),
and $w$ is a polynomial function which is linear on $V_2$, i.e.,
$$w \in (V_2^{\vee}\ot_\CC \Sym^{\ge 1} (V_1^\vee))^{G}.$$ 
This implies $\bdeg =1 $, therefore $\lan J \ran = 1$.
Consider the vector bundle $\cE$ (with fiber $V_2$) on $\cX$, whose total space is $\cT$.
 Since $w$ is linear on $V_2$, it gives rise to a section
$f\in H^0(\cX,\cE^\vee)$ of the {\it dual} vector bundle $\cE ^\vee $. 
Nondegeneracy of the hybrid model implies that $f$ is a regular section with smooth zero locus $\cZ :=Z(f)=Z(dw)$. 
Note that this includes the special case of $V_2 = 0$.
\item A GLSM is \newterm{Calabi--Yau} if $\cT$ is Calabi--Yau.
\end{enumerate}
\end{dfn}

\begin{exa} \label{exa : w =0}
Consider the case where $V = V_1$ i.e. $V_2 = 0$ where $V$ is a $G$-representation such that $\left( \Sym V^\vee \right)^G = \CC$.  We equip this with a trivial $\C^*_R$-action and set $\Gamma = G \times \C^*_R$.  View $w = 0$ as a section of $\O(\chi)$ where $\chi$ is the projection character.   Take a generic stability condition $\theta$ so that the GIT stack quotient $[V /\!\!/_\theta G]$ is a smooth proper Deligne--Mumford stack.  This can be thought of as the GLSM corresponding to $[V /\!\!/_\theta G]$ itself.
\end{exa}

The following special case gives the GLSM theory for the quintic 3-fold, it is simply a specific example of a geometric phase.  In general, geometric phases can be used to create GLSM theories for zero loci of regular sections of 
homogeneous vector bundles in GIT quotients.

\begin{exa}\label{e:quintic}
Consider the vector space 
\[
V = V_1 \oplus V_2 =  \CC^5 \times \CC = \Spec(\CC[x_1, \ldots, x_5, p]),
\]
 with an action of $G = \CC^*$ with weights $(1,1,1,1,1,-5)$.  The superpotential $w \colon  V \to \CC$ is given by 
 \[
 w = p \sum_{i=1}^5 x_i^5.
 \]
   This function is homogeneous of degree $1$ if we choose our R-charge action to have weights $(0,0,0,0,0,1)$.
If we choose $\theta$ to be the identity character, then $\cT \to \cX$ is given by $\text{tot}(\cO_{\PP^4}(-5)) \to \PP^4$ and $\cZ$ is the quintic threefold $\{ \sum_{i=1}^5 x_i^5 = 0\}  \subseteq \PP^4$.  
\end{exa}

Changing the stability condition in the previous example flips the situation from Gromov--Witten theory to FJRW theory.
\begin{exa}
Let $V, G$ and $w$ be as in the previous example, but modify the R-charge to have weights $(1,1,1,1,1,0)$ and let $\theta = -\op{Id}$.  Now $w$ is homogeneous of degree 5 with respect to the R-charge grading.  Re-ordering our decomposition of $V = V_1 \oplus V_2$ , we choose $V_1 = \Spec(\CC[p])$ and $V_2 = \Spec(\CC[x_1, \ldots, x_5])$.  Then 
$\VmodtG \to \cX$ is given by $[\CC^5/\mu_5] \to B \mu_5$, and when descended to the quotient, the superpotential becomes the function $ w  \colon  [\CC^5/\mu_5] \to \CC$ given by the Fermat polynomial $\sum_{i=1}^5 x_i^5$.  Note that the nondegeneracy condition in this case reduces to the requirement that the descended $ w$ has an isolated singularity at the origin.
In this case we prove that the GLSM invariants we construct agree with Polishchuck and Vaintrob's construction of FJRW invariants (see \S\ref{sec: compare PV} and~\ref{s:enhanced PV comparison}).
\end{exa}

The previous example can also be achieved more simply.  
\begin{exa}
Consider the case where $V = V_2$ i.e. $V_1 = 0$, the group $\C^*_R$ acts with positive weights, and $w$ is a quasi-homogeneous polynomial with respect to these weights with an isolated singularity at the origin (this is the nondegeneracy condition).   Let $G$ be any finite diagonal group of symmetries of $w$.  This is the general setup of FJRW theory.
\end{exa}

For completeness we also give an example where $V_2$ is not a sum of one-dimensional representations.
\begin{exa}
Let $W=\CC^3$, $G = \op{GL}(W)$, $V_1 = \op{Hom}(W,\CC^6)$, and
$V_2 = W \ot \det W \oplus \bigwedge\nolimits ^2 W$. Let $\theta$ be the determinant character and $\C^*_R$ act trivially on $V_1$ and by scaling on $V_2$.  
Denote by $S$ the tautological subbundle on $\op{Gr}(3,6) = V_1\sslash_\theta G$. We  take an element $w \in ((\op{Sym}  V_1^\vee )\ot V_2^\vee)^G$, general enough  so
that the corresponding section $\sigma$ of $S^*(1)\oplus \bigwedge\nolimits ^2 S^*$ defines a codimension 6, smooth zero locus $Z(\sigma )$ in $\op{Gr}(3, 6)$. This is a Calabi-Yau 3-fold. For more such examples, see e.g., Table 1 of \cite{IIM}. 
\end{exa}

\begin{remark}
Note that in the previous example, the locus $Z(\sigma )$ in $\op{Gr}(3, 6)$ coincides
with $Z(dw)$ in $[ V_1 \oplus V_2 \sslash_\theta G]$.  This is the case for any geometric phase. 
\end{remark}

\subsection{The plan}\label{s:plan} 
Even though Landau--Ginzburg maps generally land in the stack $[V^{ss}(\nu)/\Gamma]$,
there are evaluation maps at the markings 
$$\op{ev}^i:LG_{g, r}(\cc T, d)\to I \cc T$$ to the inertia stack of the GIT quotient
$\cc T := [V/\!\!/_\theta G]$.
The most important structure needed to define the $A$-model of a nondegenerate GLSM is a virtual class in the homology of the proper moduli space $LG_{g, r}(\cc Z, d)$ of stable LG maps to the critical locus $\cc Z := Z(d w)$. 
In \cite{FJR15}, such a class is constructed algebraically, via the cosection localization method of Kiem and Li \cite{KiemLi}, only under certain restrictions; essentially when the evaluation maps are required to land in a proper substack of
$I\cc T$ (see \cite[Def. 6.1.6]{FJR15}). As a consequence, these virtual classes are in general not sufficient to produce the desired outcome of the construction of 
a Cohomological Field Theory (CohFT) in the sense of Kontsevich and Manin \cite{KM}. (In the parlance of \cite {FJR13, FJR15}, splitting at a {\it broad node} cannot be handled using the cosection localized virtual classes.)

The ultimate goal of our project is to give a full algebraic construction of a CohFT  for GLSM targets by generalizing the approach of
Polishchuck and Vaintrob, \cite{PV}, from the FJRW-theory of hypersurface singularities \cite{FJR13}. 
The idea of this approach is to ``lift" the homological virtual class to an object in an appropriate derived category of factorizations. The CohFT is then obtained by performing a Fourier--Mukai transform and passing to Hochschild homology. 

Specifically, we seek to implement the following program.  Let $(V, G, \CC^*_R, \theta, \nu, w)$ be a nondegenerate GLSM target as above.
Fix $g, r \geq 0$ in the stable range, i.e., satisfying
$2g -2+ r > 0$, and choose $d \in \text{Hom}_{\Z}(\widehat G, \QQ)$.  
Consider the diagram
\begin{center}
	\begin{tikzpicture}[scale=1]
	\node (A) at (-2,-1.5) {$\left(I \cc T \right)^r$};
	\node (B) at (-3,0) {$LG_{g, r}(\cc Z, d)$};
	\node (C) at (0,0) {$LG_{g, r}(\cc T, d)$};
	\node (D) at (2,-1.5) {$\sMbar_{g,r}$};
	\draw[->,] (C) -- (A);
	\draw[right hook->] (B) -- (C);
	\draw[->,] (C) -- (D);
	\node[right] at (-1,-0.75){$\mathrm{ev}$};
	\node[left] at (1,-0.75){$\mathrm{st}$};
	\end{tikzpicture}
\end{center}
with $\op{ev}=(\op{ev}_1,\dots, \op{ev}_r)$ 
and $\op{st}$ the stabilization map which forgets the data $(\cP,\sigma,\varkappa)$ and stabilizes the domain curve after removing orbifold structures on the markings.
Then there is an appropriate extension of this diagram as follows:
\begin{center}
	\begin{tikzpicture}[scale=1]
	\node (A) at (-2,-1.5) {$\left(I \cc T \right)^r$};
	\node (B) at (-4.5,0) {$LG_{g, r}(\cc Z, d)$};
	\node (B') at (-2,0) {$LG_{g, r}(\cc T, d)$};
	\node (C) at (0,0) {$\square$};
	\node (D) at (2,-1.5) {$\sMbar_{g,r}$};
	\node (E) at (1,1) {$ \; \quad \quad \cO_\square \xrightarrow{\beta} E \xrightarrow{\alpha} \cO_\square \otimes \CC(\chi)$};
	\draw[->,] (C) -- (A);
	\draw[right hook->] (B) -- (B');
	\draw[right hook->] (B') -- (C);
	\draw[->,] (C) -- (D);
	\draw[->,] (E) to[bend left=20] (C);
	\node[right] at (-1,-0.75){$\mathrm{ev}$};
	\node[left] at (1,-0.75){$\mathrm{st}$};
	\end{tikzpicture}
\end{center}
where in \S\ref{s:conFF} we construct a stack $\square = \square_{g, r, d}$ with a $\CC^*_R$-action, 
a $\CC ^*_R$-equivariant vector bundle $E \to \square$, 
and $\CC^*_R$-equivariant sections $\alpha \in \Gamma(E^\vee\ot \CC (\chi) ) = \Hom (E, \cO_{\square}\ot \CC (\chi))$, $\beta \in \Gamma(E)= \Hom (\cO_{\square}, E)$
satisfying  the following properties:
\begin{enumerate}
	\item[$(1)$] The space $\square $ is a smooth separated Deligne--Mumford stack with a ``stabilization map" $\op{st}:\square\to\sMbar_{g,r}$
	and with a pure dimension in each twisted component given by the formula
	$$\dim H^0(\cV|_{C_s}) - \dim H^1(\cV|_{C_s})+ \dim \fM_{g,r}^{\mathrm{orb}}(B\Gamma )_{\log}
	+ \rank E$$
	(see  \S \ref{moduli stacks} for the definition of the Artin stack $\fM_{g,r}^{\mathrm{orb}}(B\Gamma)_{\log}$).

	\item[$(2)$] For $1 \leq i \leq r$ there exist \newterm{smooth evaluation maps} $\op{ev}^i \colon \square \to  I\cc T$ which are $\CC^*_R$-equivariant.
	\item[$(3)$] 
	\begin{enumerate}
		\item The vanishing locus of $\beta$ is canonically isomorphic to $LG_{g,r}(\cc T, d)$,
		\item the common vanishing locus of $\alpha$ and $\beta$ is canonically isomorphic with $LG_{g,r}(\cc Z, d)$ (see Proposition~\ref{p:Ksupp} and~\ref{l:finalsupp}) for which
		$\op{ev}^i$ maps and $\stab$ maps are compatible, and
		\item the composition $-\alpha^\vee \circ \beta$ is equal to $-(\boxplus_{i = 1}^r \op{ev}^i)^*(w)$.
	\end{enumerate}
\end{enumerate}
Note that in $(3.\text{c})$ above, we endow $I \cc T$ with a superpotential, which we also call $w$ by abuse of notation, via composition with the natural morphism $I \cc T \to \cc T$.
Given the above, we define the Koszul factorization 
\[
K_{g,r, d} := \{ - \alpha, \beta \}  \in \dabsfact{[\square/\CC^*_R], - \op{ev}^* (\boxplus_{i=1}^r w)}
\]
which is supported on $LG_{g,n}(\cc Z, d)$.  We call this the {\it fundamental factorization}. It will play the role of the virtual fundamental class for $LG_{g,n}(\cc Z, d)$.  Namely, the Koszul factorization defines a Fourier--Mukai transform $\Phi_{K_{g,r, d}}$ for categories of factorizations,
\begin{center}  
	\begin{tikzpicture}[scale=1]
	\draw (2,1.5) node(A1){$\dabsfact{[\square/\CC^*_R],  \op{ev}^* (\boxplus_{i=1}^r w)}$}
	(1+7,1.5) node(A2){$\dabsfact{[\square/\CC^*_R], 0}_{LG_{g,n}(\cc Z, d)}$}
	(0,0) node(A3){$\prod_{i=1}^r\dabsfact{I[V^{ss}(\theta)/\Gamma], w}$}
	(1+9,0) node(A4){$\op{D}(\sMbar_{g,r})$}
	(0.5+4.5, 1.8) node{$ \overset{\mathbb{L}}{\otimes} K_{g,r, d}$}
	(0.7,.9) node{$\LL \op{ev}^*$}
	(1+8.3, .9) node{$\RR \op{st}_*$};	
	\draw[->,>=stealth] (A1) to[bend right=0] (A2);
	\draw[->,>=stealth] (A3) to[bend right=0] (A1);
	\draw[->,>=stealth] (A2) to[bend right=0] (A4);
	\end{tikzpicture}
\end{center}
i.e.
\[
\Phi_{K_{g,r, d}} (\cc E) := \RR \op{st}_* (\LL \op{ev}^* \cc E  \overset{\mathbb{L}}{\otimes}_{\O_\square} K_{g,r, d} ).
\]

The induced map on Hochschild homology (after being suitably adjusted by an appropriate Todd class modification),  intermixed with an HKR morphism from Hochschild homology to cohomology
gives us a collection of invariants (see \S\ref{s:GLSMinvs} the precise formulation).
In a sequel \cite{CKFGS} we show that the set of these invariants for all $g, r, d$, form a cohomological field theory.

\begin{remark}
The construction is in fact a bit more involved: in order to correct the invariants by an appropriate Todd class as in \cite[Equation (5.15)]{PV}, we factorize the map $\mathrm{st}$ as $\square \to \widetilde{U} \to \sMbar_{g, r}$ where $\widetilde{U}$ is a proper and smooth Deligne--Mumford stack on which we multiply by the Todd correction (for the precise description, see Definition \ref{d:invariants}).
\end{remark}

\section{Factorizations}\label{s:fact}
Factorizations are natural objects attached to an LG model, in the same way complexes of coherent sheaves are natural objects attached to a variety.
In this section, we recall the definition of the derived category of factorizations and some properties that will be use in the paper.
The true heart of our work will be to construct a natural factorization, called the fundamental factorization (see Definition~\ref{d:koszulcut}), for each moduli space of LG stable maps (under the condition it is a convex Hybrid Model, as defined in \S\ref{ss:convexity}). The fundamental factorization plays a similar role to the virtual structure sheaf in Gromov--Witten theory.

\subsection{Derived categories of Landau--Ginzburg models}

In order to naturally define derived functors, we use the more sheaf-theoretic approach to factorizations introduced by Lin--Polmerleano and Positselski \cite{LP, EP}.  We review the necessary concepts below.

Let $\mathcal X$ be an algebraic stack 
equipped with a line bundle $\mathcal L$ and a section $w$ of $\mathcal L$.  A \newterm{factorization} is the data $\cc E = (\cc E_{1}, \cc E_0, \phi_{1}^\cc E, \phi_0^\cc E)$ where $\mathcal{E}_{1}$ and $\mathcal{E}_0$ are quasi-coherent sheaves on $\mathcal X$ and 
$$
\mathcal{E}_{1} \stackrel{\phi_1^{\mathcal{E}}}{\rightarrow} \mathcal{E}_0 \stackrel{\phi_{0}^{\mathcal{E}}}{\rightarrow} \mathcal{E}_{1} \otimes_{\O_{\cX}} \mathcal{L}
$$
are morphisms such that
\begin{equation*}\begin{aligned}
\phi_{0}^{\mathcal{E}} \circ \phi_1^{\mathcal{E}} &= w, \\
(\phi_1^{\mathcal{E}} \otimes \op{Id}_{\mathcal{L}}) \circ \phi_{0}^{\mathcal{E}} &= w.
\end{aligned}\end{equation*}

\begin{definition}
 The \newterm{shift}, denoted by $[1]$, sends a factorization $\mathcal E$ to the factorization, 
 \begin{displaymath}
  \mathcal E[1] := (\mathcal E_0,\mathcal E_{1}\otimes_{\O_{\cX}} \mathcal L,-\phi^{\mathcal E}_0,-\phi^{\mathcal E}_{1} \otimes_{\O_{\cX}} {\mathrm{Id}}_{\mathcal L}).
 \end{displaymath}
\end{definition}

A morphism between two factorizations of even degree $f  \colon  \cc E \to \F[2k]$ is a pair $f = (f_0, f_1)$ defined by 
$$
\mathrm{Hom}^{2k}_{\mathsf{Fact}(\mathcal X, w)} ( \cc E, \F) : = \mathrm{Hom}_{\mathrm{Qcoh}(\cX)}( \cc E_{0}, \F_{0} \otimes_{\O_{\cX}} \mathcal L{}^k) \oplus \mathrm{Hom}_{\mathrm{Qcoh}(\cX)}( \cc E_{1}, \F_{1} \otimes_{\O_{\cX}} \mathcal L{}^k)
$$
and, similarly,  a morphism of odd degree $f  \colon  \cc E \to \F[2k+1]$ is a pair $f = (f_0, f_{1})$ defined by 
$$
\mathrm{Hom}^{2k+1}_{\mathsf{Fact}(\mathcal X, w)} ( \cc E, \F) : = \mathrm{Hom}_{\mathrm{Qcoh}(\cX)}( \cc E_{0}, \F_{1} \otimes_{\O_{\cX}} \mathcal L{}^{k+1}) \oplus \mathrm{Hom}_{\mathrm{Qcoh}(\cX)}( \cc E_{1}, \F_{0} \otimes_{\O_{\cX}} \mathcal L{}^k).
$$
You can equip these Hom sets with a differential coming from the graded commutator with the morphisms defining $\E$ and $\F$.  
This yields a $\ZZ$-graded dg category $\mathsf{Fact}(\mathcal X, w)$.

We denote by $Z^0 \mathsf{Fact}(\mathcal X, w)$ the subcategory of $\mathsf{Fact}(\mathcal X, w)$ with the same objects but whose morphisms are only the closed degree zero morphisms between any two objects $\cc E$ and $\F$.  The homotopy category of the dg category $\mathsf{Fact}(\mathcal X, w)$ is denoted by $H^0 \mathsf{Fact}(\mathcal X, w)$.

\begin{remark}
The morphisms in $Z^0 \mathsf{Fact}(\mathcal X, w)$ are just the commuting morphisms of factorizations.  Morphisms in $H^0 \mathsf{Fact}(\mathcal X, w)$ are commuting morphisms of factorizations up to homotopy.
\end{remark}

The category $Z^0 \mathsf{Fact}(\mathcal X, w)$ is abelian (see e.g., \cite[Lemma 2.4]{BDFIK}).  Hence, the notion of exact sequences and complexes of objects in $Z^0 \mathsf{Fact}(\mathcal X, w)$ makes sense.

\begin{definition}
Given a complex of objects 
$$
\bigg[ \ldots \rightarrow \cc E^b \stackrel{f^b}{\rightarrow} \cc E^{b+1} \stackrel{f^{b+1}}{\rightarrow} \ldots \bigg]
$$
in $Z^0 \mathsf{Fact}(\mathcal X, w)$, we define the \newterm{totalization} of the complex to be the factorization $\T \in \mathsf{Fact}(\mathcal X, w)$ given by the data:
\begin{align*}
 \mathcal T_{1} & := \bigoplus_{l} \mathcal E_{1}^{2l} \otimes_{\O_{\cX}} \cL^{-l} \oplus \bigoplus_{l} \mathcal E_0^{2l+1} \otimes_{\O_{\cX}} \cL^{-l-1} \\
 \mathcal T_0 & := \bigoplus_{l} \mathcal E_{0}^{2l} \otimes_{\O_{\cX}} \cL^{-l} \oplus \bigoplus_{l} \mathcal E_{1}^{2l+1} \otimes_{\O_{\cX}} \cL^{-l} \\
  \phi^{\mathcal T}_{1} & := \begin{pmatrix} \ddots & 0 & 0 & 0 & 0 \\ \ddots & -\phi_{0}^{\mathcal E^{-1}} & 0 & 0 & 0 \\ 0 & f_{0}^{-1}  & \phi_{1}^{\mathcal E^{0}} & 0 & 0\\ 0 & 0 & f_1^{0} & -\phi_{0}^{\mathcal E^1} \otimes \op{Id}_{\mathcal L^{-1}}  & 0 \\ 0 & 0 & 0 & \ddots & \ddots \end{pmatrix} \\
 \phi^{\mathcal T}_0 & := \begin{pmatrix} \ddots & 0 & 0 & 0 & 0 \\ \ddots & -\phi_{1}^{\mathcal E^{-1}} \otimes \op{Id}_{\mathcal L}  & 0 & 0 & 0 \\ 0 & f_{1}^{-1} \otimes \op{Id}_{\mathcal L}  & \phi_0^{\mathcal E^{0}} & 0 & 0 \\ 0 & 0 & f_{0}^{0} & -\phi_{1}^{\mathcal E^1} & 0 \\ 0 & 0 & 0 & \ddots & \ddots \end{pmatrix} 
\end{align*}
\end{definition}

\begin{remark}
Given an exact sequence
\[
0 \to \mathcal E \xrightarrow{f} \mathcal F \xrightarrow{g} \mathcal G \to 0
\]
in  $Z^0 \mathsf{Fact}(\mathcal X, w)$, the totalization $\mathcal T$ is
\begin{align*}
 \mathcal T_{1} & :=  \mathcal E_1 \otimes \cL \oplus \mathcal F_0  \oplus \mathcal G_1 \\
 \mathcal T_0 & := \mathcal E_0 \otimes \cL \oplus \mathcal F_1 \otimes \cL \oplus \mathcal G_0  \\
  \phi^{\mathcal T}_1 & := 
 \begin{pmatrix}
 \phi_1^{\mathcal E} \ot \mathrm{Id}_\cL & 0 & 0 \\
 f_1 \ot \mathrm{Id}_\cL & -\phi_0^{\mathcal F} & 0 \\
 0 & g_0  & \phi^{\mathcal G}_1
  \end{pmatrix} \\
 \phi^{\mathcal T}_0 & := 
 \begin{pmatrix}
 \phi_0^{\mathcal E} \ot \mathrm{Id}_\cL & 0 & 0 \\
 f_0 \ot \mathrm{Id}_\cL & -\phi_1^{\mathcal F} \ot \mathrm{Id}_\cL & 0 \\
 0 & g_1  \ot \mathrm{Id}_\cL & \phi^{\mathcal G}_0
  \end{pmatrix} \\
  \end{align*}
This totalization can be viewed as the cone of the map of complexes
\[
\begin{tikzcd}
{[}  \mathcal E \ar[r, "f"] \ar[d] & \mathcal F {]} \ar[d, "g"]  \\
{[}  0 \ar[r] & \mathcal G {]}.
\end{tikzcd} 
\]
The components are the same those of $\mathcal E[2] \oplus \mathcal F[1] \oplus \mathcal G$ and the differential has been deformed using $f,g$.  Localizing by this object makes $\mathcal G$ isomorphic to the cone of the morphism $f$.  This is the idea in what follows.
\end{remark}

Denote by $\mathsf{CoAcyc}(\mathcal X, w)$ the minimal triangulated subcategory of $H^0 \mathsf{Fact}(\mathcal X, w)$ containing the totalizations of all short exact sequences in $Z^0 \mathsf{Fact}(\mathcal X, w)$ and closed under infinite direct sums.  Objects of $\mathsf{CoAcyc}(\mathcal X, w)$ are called \newterm{coacyclic}.

We have the following general definition.
\begin{definition}
The \newterm{coderived category} $\op{D}^{\text{co}}[\mathsf{Fact}(\mathcal X, w)]$ of $\mathsf{Fact}(\mathcal X,w)$ is the Verdier quotient of $H^0 \mathsf{Fact}(\mathcal X, w)$ by $\mathsf{CoAcyc}(\mathcal X, w)$.  
The \newterm{derived category} of $(\mathcal X, w)$, which we denote by $\dabsfact{\mathcal X, w}$, is the smallest full triangulated subcategory of $\op{D}^{\text{co}}[\mathsf{Fact}(\mathcal X, w)]$ which contains the coherent factorizations and is closed under taking direct summands.  \end{definition}

\begin{remark}
As the notation suggests, the category $\dabsfact{\mathcal X, w}$ can be thought of as the derived category of the gauged Landau--Ginzburg model $(\mathcal X, w)$.
\end{remark}

\subsection{Derived functors}\label{s:derfunc}
In this section, the goal is to define Fourier--Mukai functors in the derived category of factorizations.
They are of prime interest in this paper, since they enter explicitly into the definition of our GLSM theory.
We start with the definitions of derived pullbacks, pushforwards, and tensor products.

\

Let $f \colon  \mathcal Y \to \mathcal X$ be a morphism of algebraic stacks.  From the section $w$ of $\mathcal L$ on $\mathcal X$, we can pullback to get a section $f^*w$ of $f^* \mathcal L$ on $\mathcal Y$.  At the level of dg categories, we get the following functors:
\begin{definition}
 \begin{align*}
  f^* \colon  \mathsf{Fact}(\cX, w) & \to \mathsf{Fact}(\cY, f^*w) \\
  (\cc E_{1},\cc E_0,\phi^{\cc E}_{1},\phi^{\cc E}_0) & \mapsto (f^* \cc E_{1},f^* \cc E_0,f^* \phi^{\cc E}_{1},f^* \phi^{\cc E}_0)
 \end{align*}
 and 
 \begin{align*}
  f_* \colon  \mathsf{Fact}(\cY, f^*w) & \to \mathsf{Fact}(\cX, w) \\
  (\mathcal F_{1},\mathcal F_0,\phi^{\mathcal F}_{1},\phi^{\mathcal F}_0) & \mapsto (f_* \mathcal F_{1},f_* \mathcal F_0,f_* \phi^{\mathcal F}_{1},f_* \phi^{\mathcal F}_0).
 \end{align*}
 Note that by the projection formula $f_*(\mathcal F \otimes_{\O_{\cX}} f^*\mathcal L) \cong (f_*\mathcal F) \otimes_{\O_{\cX}} \mathcal L $ under which $f_*(f^*w)$ corresponds to $w$ so this is well-defined. Let $v$ be a section of $\cL$ on $\cX$.
  We define a dg-functor,
 \begin{displaymath}
  \otimes_{\mathcal O_{\mathcal X}}  \colon  \mathsf{Fact} (\mathcal X, w) \otimes_k \mathsf{Fact}(\mathcal X, v) \to \mathsf{Fact}(\mathcal X, v+w)
 \end{displaymath}
 by setting
 \begin{align*}
 \left(\cc E \otimes_{\O_{\cX}} \mathcal F \right)_0 & := \cc E_0 \otimes_{\O_{\cX}} \mathcal F_0 \oplus \cc E_{1} \otimes_{\O_{\cX}} \mathcal F_{1}\otimes_{\O_{\cX}} \mathcal L \\
  \left(\cc E \otimes_{\O_{\cX}} \mathcal F \right)_{1} & := \cc E_{0} \otimes_{\O_{\cX}} \mathcal F_1 \oplus \cc E_{1} \otimes_{\O_{\cX}} \mathcal F_{0} \\
 \phi^{\cc E \otimes_{\O_{\cX}} \mathcal F}_{0} & 
 := \begin{pmatrix} 
 \op{Id}_{\cc E_0} \otimes_{\O_{\cX}} \phi^{\mathcal F}_{0} &  \phi^{\cc E}_{1} \otimes_{\O_{\cX}} \op{Id}_{\mathcal F_{1}} \\ 
 -\phi^{\cc E}_{0} \otimes_{\O_{\cX}} \op{Id}_{\mathcal F_{0}} &  \op{Id}_{\cc E_{1}} \otimes_{\O_{\cX}} \phi^{\mathcal F}_{1}  
  \end{pmatrix} \\
 \phi^{\cc E \otimes_{\O_{\cX}} \mathcal F}_{1} & := 
 \begin{pmatrix}
\op{Id}_{\cc E_{0}} \otimes_{\O_{\cX}} \phi^{\mathcal F}_{1} &  - \phi^{\cc E}_{1} \otimes_{\O_{\cX}} \op{Id}_{\mathcal F_{0}} \\ 
\phi^{\cc E}_{0}\otimes_{\O_{\cX}} \mathcal L \otimes_{\O_{\cX}} \op{Id}_{\mathcal F_{1}} &   \op{Id}_{\cc E_{1}} \otimes_{\O_{\cX}} \phi^{\mathcal F}_{0}   
  \end{pmatrix}
 \end{align*}
\end{definition}

Let $\op{D}^{\op{co}}[\op{Fact}(\mathcal X, w)]_{f^*}$ denote the full subcategory of $ \op{D}^{\op{co}}[\mathsf{Fact}(\mathcal X,w)]$ 
consisting of factorizations with $f^*$-acyclic components.  Similarly, let $\op{D}^{\op{co}}[\op{Fact}(\mathcal X, w)]_{f_*}$ denote the full subcategory of 
$ \op{D}^{\op{co}}[\mathsf{Fact}(\mathcal X,w)]$ 
consisting of factorizations with $f_*$-acyclic components and  
$\op{D}^{\op{co}}[\op{Fact}(\mathcal X, w)]_{\otimes}$ denote the full subcategory of $ \op{D}^{\op{co}}[\mathsf{Fact}(\mathcal X,w)]$ 
consisting of factorizations with flat components.

\begin{definition}[{\cite[\S 3.1]{PV11}}] \label{d:nice}
 An algebraic stack $\cX$ is called a \newterm{nice quotient stack} if $\cX = [T/ H]$ where $T$ is a noetherian scheme and $H$ is a reductive linear algebraic group such that $T$ has an ample family of $H$-equivariant line bundles.
\end{definition}

\begin{remark}
 According to \cite[Theorem 4.4 and Remark 4.3]{KreschGeometry},  every smooth separated Deligne--Mumford stack $\cX$ with quasi-projective coarse moduli space is of form
$[Y/G]$ for some quasi-projective scheme $Y$ and for some $G=GL(n)$. Thus $\cX$ is a nice quotient stack. 
\end{remark}

From now on assume that $\mathcal X$ is a smooth nice quotient stack.  Then, by \cite[Proposition 2.18, Proposition 2.19, Proposition 2.20]{BDFIK} the inclusions induce equivalences of categories
\begin{align*}
\op{D}^{\op{co}}[\op{Fact}(\mathcal X, w)]_{f^*} &  \xrightarrow{ \iota_{f^*} } \op{D}^{\op{co}}[\mathsf{Fact}(\mathcal X,w)] \\
\op{D}^{\op{co}}[\op{Fact}(\mathcal X, w)]_{\otimes} &  \xrightarrow{ \iota_{\otimes}} \op{D}^{\op{co}}[\mathsf{Fact}(\mathcal X,w)]. \\
\end{align*}
 \begin{definition}
The \newterm{derived pullback} is the functor
 \begin{align*}
  \LL f^* \colon  \op{D}^{\op{co}}[\mathsf{Fact}(\mathcal X,w)] & \to \op{D}^{\op{co}}[\mathsf{Fact}(\mathcal Y, f^*w)] \\
  \cc E & \mapsto f^*\mathcal \iota_{f^*}^{-1} \cc E.
 \end{align*}
Let $\mathcal F \in \op{D}^{\op{co}}[\mathsf{Fact}(\mathcal X,v)]$.   The \newterm{derived tensor product} is the functor
 \begin{align*}
(-  \overset{\mathbb{L}}{\otimes}_{\O_{\cX}} \mathcal F) \colon  \op{D}^{\op{co}}[\mathsf{Fact}(\mathcal X,w)] & \to \op{D}^{\op{co}}[\mathsf{Fact}(\mathcal X, w+v)] \\
  \cc E & \mapsto \iota_{\otimes} ^{-1} \mathcal F \otimes_{\O_{\cX}} \cc E.
 \end{align*}
  
 \end{definition}

\newcommand{\mM}{\cM}

\begin{definition}
Let $\cM$  be a closed substack of $\mathcal X$.  Then, by definition, the map $\cM \to \cX$ is representable.  We say that a factorization $\cE$ has \newterm{support} on $\cM$ if for any scheme $T$ and any morphism $T \to \cX$ the restriction of $\cE$ to $(T \times_{\cX} \cX) \backslash (T \times_{\cX} \cM)$ is coacyclic.  We denote the full subcategory of $\dabsfact{\cX, w}$ consisting of factorizations with support on $\cM$ by $\dabsfact{\cX, w}_\cM$.
\end{definition}

Now we assume that $\mM \subseteq \mathcal Y$ is a closed substack such that the composition 
\[
\mM \hookrightarrow \mathcal Y  \xrightarrow{f} \mathcal X
\]
is proper.    Our goal now is to similarly define $  \mathbb{R}f_*$ and show that it restricts to a functor from $\dabsfact{\mathcal Y, f^*w}_{\cM}$ to $ \dabsfact{ \mathcal X, w }$.

This will require the following setup.  Let $\mathcal X, \mathcal Y$ be smooth separated nice quotient stacks, $w$ be a section of the line bundle on $\mathcal X$, 
and $f  \colon  \mathcal Y \to \mathcal X$ be a morphism of finite cohomological dimension.

Consider a smooth atlas
\[
p  \colon  Y \to \mathcal \cY
\] of $\cY$ such that $Y$ is an affine scheme.
The adjunction morphism $\op{Id} \to  {p}_*p^*$ gives rise to the (cosimplicial) \v Cech resolution
\[
0 \to \mathcal F \to {p}_*p^* \mathcal F \to  ({p}_*p^*)^2 \mathcal F  \to ...
\] of a quasi-coherent sheaf $\mathcal F$.
This is an $f_*$-acyclic resolution of $\mathcal F$.
Similarly, given any factorization $\mathcal E$, we get an exact sequence
\begin{equation}
0 \to \mathcal E \to {p}_*p^* \mathcal E \to  ({p}_*p^*)^2 \mathcal E  \to ...
\label{smooth cech cover}
\end{equation}
in $Z^0 \mathsf{Fact}(\mathcal Y, f^*w)$ since the  \v Cech resolution
is functorial and exactness is measured component-wise.  We denote by $\check{\mathcal C}(\mathcal E)$ the totalization of the above complex and by $\underline{\check{\mathcal C}}(\mathcal E)$ the totalization of the truncated complex
\[
\bigg [ 0 \to {p}_*p^* \mathcal E \to  ({p}_*p^*)^2 \mathcal E  \to ... \bigg ].
\]

Following \cite[Remark 1.3]{EP}, we will now see that the notion of coacyclicity of a factorization is local for the smooth topology.
\begin{proposition}
If $U$ is a smooth atlas of $\cX$ and the pullback of $\cE$ to $U$ is coacyclic, then $\cE$ is coacyclic.
\end{proposition}
\begin{proof} We may assume that $U$ is an affine scheme and denote the morphism $U\to\cX$ by $p$.
There is an exact triangle
\[
\mathcal E \to  \check{\mathcal C}(\mathcal E) \to \underline{ \check{\mathcal C}}(\mathcal E) \to \mathcal E [1]
\]
in $H^0 \mathsf{Fact}(\mathcal X, w)$.  Now notice that $\check{\mathcal C}(\mathcal E)$ is coacyclic by \cite[Lemma 2.16]{BDFIK}.  On the other hand, since $p$ is smooth and affine ${p}_*p^*$ is exact.  It follows that $({p}_*p^*)^i \cE$ is coacyclic for all $i$.  Therefore $\underline{ \check{\mathcal C}}(\mathcal E)$ is a homotopy limit of coacyclic complexes, hence coacyclic (as it is a cone of direct sums of coacyclic complexes, see e.g. the proof of ibid.). It follows from the exact triangle above that $\mathcal E$ is also coacyclic.
\end{proof}

\begin{corollary}
The support of a coherent factorization $\cE$ is local in the smooth topology.
\label{cor: local support}
\end{corollary}
\begin{proof}
This follows immediately from the above proposition.
\end{proof}

\begin{lemma}
Any object of $\dabsfact{\mathcal Y, f^*w}_{\mM}$ is a direct summand of an object represented by a factorization with coherent components supported on $\mM$.
\label{lem: compact generation}
\end{lemma}

\begin{proof}
We omit the proof which is a repetition of \cite[\S 1.10]{EP}.  See especially \cite[Corollary 1.10 (b)]{EP} which says exactly the desired statement phrased in the language of that paper.
\end{proof}

\begin{proposition}
There is a well-defined derived pushforward functor 
\begin{align*}
\mathbb R f_*  \colon  \dabsfact{\mathcal Y, f^*w}_{\mM} & \to \dabsfact{\mathcal X, w}  \\
\cc E & \mapsto f_*( \underline{\check{\mathcal C}}(\cE)).
\end{align*}
\end{proposition}

\begin{proof}
As $  \underline{\check{\mathcal C}}(\cE)$ is $f_*$-acyclic, this gives a functorial definition of a pushforward functor which naturally lands in the coderived category of quasi-coherent factorizations.  We need to show that $ f_*( \underline{\check{\mathcal C}}(\cE))$ actually lies in $ \dabsfact{\mathcal X, w}$, i.e. that it is generated by factorizations with coherent components.   

For each $i$, we let $D_i$ be the totalization of the complex of factorizations
\[
\bigg [ f_*{p}_*p^* \cE \to  f_*({p}_*p^*)^2 \cE \to ... \to \text{Ker}\Big( f_* ({p}_*p^*)^{i} \cE \to f_* ({p}_*p^*)^{i+1} \cE\Big) \bigg ].
 \]

Now we let 
\[
\overline{D} :=  0 \to f_*({p}_* p^* \cE) \to f_*(({p}_* p^*)^2 \cE) \to ...
\]
 be the infinite complex in $Z^0 \mathsf{Fact}(\mathcal X, w)$ corresponding to  $f_*(\underline{\check{\mathcal C}}(\cE))$ so that  $ \op{H}^{i}(\overline{D})$ is a factorization.
Then for each $i$ we have an exact triangle
\begin{equation} \label{eq: cech triangle}
\begin{tikzcd}
D_{i} \ar[r] & D_{i+1}  \ar[r] & \op{H}^{i}(\overline{D}) \ar[r] & D_{i}[1]
\end{tikzcd}
\end{equation}
in $\dabsfact{\mathcal X, w}$.

Let $\cc E_0, \cc E_1$ be the components of $\cc E$. 
Then $\op{H}^i(\overline{D})$ is a factorization whose components are 
\[
\op{H}^i( \underline{\check{\mathcal C}}(\cE_0)) = \mathbb R^i f_* \cc E_0
\]
 and 
 \[
\op{H}^i( \underline{\check{\mathcal C}}(\cE_1)) = \mathbb R^i f_* \cc E_1.
\] 
For now, assume that $\cc E_0, \cc E_1$ are coherent and supported on $\mM$.
Since, by assumption $\cM$ is proper over $\mathcal X$, it follows that $\op{H}^i( \overline{D})$ has coherent components.
Now, notice that $D_0 = \op{H}^{0}(\overline D)$.
Therefore by induction and \eqref{eq: cech triangle}, $D_n$  lies in $\dabsfact{\mathcal X, w}$.

Since by assumption $f$ has finite cohomological dimension,  it follows that $D_n \cong f_*( \underline{\check{\mathcal C}}(\cE) )$ for  $n \gg 0$.  In conclusion, under the assumption that $\cc E_0, \cc E_1$ are coherent and supported on $\mM$, we have shown that $\mathbb R f_* \cc E = f_*( \underline{\check{\mathcal C}}(\cE))$ lies in 
 $\dabsfact{\mathcal X, w}$.
The result now follows from Lemma~\ref{lem: compact generation} since for any object $\cE' \in \dabsfact{\mathcal Y, f^*w}_{\mM}$, the factorization $\mathbb R f_* \cE'$ is a direct summand of some $\mathbb R f_* \cE$ which we have shown lies in $ \dabsfact{\mathcal X, w}$ and this category is closed under direct summands by definition.
\end{proof}

\begin{proposition}[Projection Formula]\label{p:projform} Let $\mathcal X, \mathcal Y$ be smooth separated nice quotient stacks. Let $w, v$ be sections of a line bundle $\cL$ on $\mathcal{X}$.
Suppose that 
\[
f : \cY \to \cX
\]
is a morphism with proper $f|_{\cM}$. 
For any $\mathcal F \in \dabsfact{\mathcal X, v}$ and $\cc E \in \dabsfact{\cY, f^*w}_{\mM}$ one has an isomorphism in $\dabsfact{\mathcal X, w + v}$,
\[
\mathbb R f_*(\cc E \overset{\mathbb{L}}{\otimes}_{\O_{\cY}} {\mathbb L}f^*\mathcal F)  = \mathbb R  f_*(\cc E) \overset{\mathbb{L}}{\otimes}_{\O_{\cX}} \mathcal F.
\]
\label{projection formula}
\end{proposition}
\begin{proof}
Since we are working in derived categories on smooth separated nice quotient stacks, we may assume that $\mathcal F$ is locally-free and that $\cc E$ has been replaced by  $\cc E \otimes_{\O_{\mathcal X}} \check{\mathcal C}$.  Then, since the usual projection formula is functorial, we may apply it component-wise to our factorization and the morphisms between components.
\end{proof}

\begin{definition} \label{d:FourMuk} Let $\mathcal X$, $\mathcal Y$ be  smooth separated nice quotient stacks.
Let $\cL, \cN$ be line bundles on $\cX, \cY$ respectively and $w \in \op{H}^0(\cX, \cL), v \in \op{H}^0(\cY, \cN)$. Let 
\[
(f,g) : \cZ \to \cX \times \cY
\]
 be a smooth separated nice quotient stack lying over $\cX \times \cY$ and assume that $f^*\cL \cong g^*\cN$
 and that $(f, g)$ has finite cohomological dimension.  Let $P \in \dabsfact{\cZ, g^*v -f^*w}$ such that the support of $P$ is proper over $\cY$.  The \newterm{Fourier--Mukai transform} by $P$ is the functor
 \begin{align*}
 \Phi_P : \dabsfact{\cX, w} & \to \dabsfact{\cY, v} \\
\cE & \mapsto \RR g_*(P  \overset{\mathbb{L}}{\otimes}_{\O_{\cZ}}  \LL f^*\cE)
 \end{align*}

\end{definition}

\begin{rem}
In \S\ref{ss:convexity}, we construct a \textit{fundamental factorization} for each moduli space of convex hybrid model LG maps.
This factorization is used as a kernel for a Fourier--Mukai transform whose passage to Hochschild homology defines enumerative GLSM invariants.  Our fundamental factorization is a particular Koszul factorization, a concept which we will now define. \end{rem}

\subsection{Koszul factorizations}\label{s:kosfac}
In \cite{PV}, one of the main components of the construction is the so called Koszul factorization.  This construction, in fact, goes all the way back to the birth of factorizations \cite{EisMF}.

\begin{definition} \label{definition: Koszul factorization}
Let $\cc E$ be a locally-free sheaf of finite rank on $\mathcal X$.  Suppose we have the following commutative diagram of morphisms,
\[
\begin{tikzcd}
\O_{\cX} \ar[r, "\beta"] \ar[rr, bend right, "w"] & \cc E^\vee \ar[r, "\alpha^\vee"] & \cL \\
\end{tikzcd}
\]
i.e. $\alpha \in \op{H}^0(\cX, \cc E \otimes \cL), \beta \in \op{H}^0(\cX, \cc E^\vee)$ and $ \alpha^\vee \circ \beta = w$.
	The \newterm{Koszul Factorization} $\{\alpha, \beta\}$ associated to the data $(\cc E, \alpha, \beta)$ is defined as 
	\begin{align*}
	\{\alpha, \beta\}_{1} & := \bigoplus_{l \geq 0} (\bigwedge\nolimits^{2l+1} \cc E) \otimes_{\mathcal O_{\mathcal {\mathcal X}}} \mathcal L^l \\ 
	\{\alpha, \beta\}_0 & := \bigoplus_{l \geq 0} (\bigwedge\nolimits^{2l} \cc E)\otimes_{\mathcal O_{\mathcal X}} \mathcal L^l \\
	\phi^{ \{\alpha, \beta\}}_0, \phi^{ \{\alpha, \beta\}}_{1} & := \bullet \ \lrcorner \ \beta + \bullet \wedge \alpha,
	\end{align*}
	where $\bullet \ \lrcorner \ \beta, \bullet \wedge \alpha$ are the contraction and wedge operators respectively.
\end{definition}

\begin{remark}
	If $\beta$ is a regular section of $\cc E^\vee$, then  $\{\alpha, \beta\}$ is isomorphic in $\dabsfact{\mathcal X, w}$ to the factorization $(0, \O_{Z(\beta)}, 0, 0)$, i.e., it is a locally-free replacement of this factorization (see, e.g., \cite[Proposition 3.20]{BFK14}).
\end{remark}

\begin{proposition}\label{p:Ksupp}
 The Koszul factorization $\{\alpha, \beta\}$ is supported on $Z(\alpha) \cap Z(\beta)$.
	\label{prop: Koszul support}
\end{proposition}
\begin{proof}
	By Corollary~\ref{cor: local support}, the assertion is local in the smooth topology.  Hence, we may assume that $\cX = U$ is a scheme and $B = \O_{U}^n$ and $\cL = \O_{U}$.
	In this case, we may express $\alpha $ as $ \oplus_{i=1}^n \alpha_i$ and $\beta $ as $ \oplus_{i=1}^n \beta_i$ and
	\[
	\{\alpha, \beta\} =  \underset{i=1 ... n}{\overset{\mathbb{L}}{\bigotimes}} \{ \alpha_i, \beta_i \}.
	\]
	As the tensor product of an acyclic factorization with a locally free factorization is acyclic, we have reduced to the case $B = \O_{U}$.  Assume $\alpha$ is non-vanishing.  In this case the identity map for $\{\alpha, \beta\}$ has an explicit homotopy,
	\[
	\begin{tikzcd}
	\O_U \arrow[r, "\beta^\vee"] \arrow{d}[swap]{\op{Id}} & \O_U  \arrow[r, "\alpha"]  \arrow[d, swap, "\op{Id}"]  \arrow{dl}[swap]{0} & \O_U  \arrow[d, "\op{Id}"]  \arrow{dl}[swap]{\alpha^{-1}} \\
	\O_U \arrow[r, "\beta^\vee"] & \O_U  \arrow[r, "\alpha"] & \O_U. \\
	\end{tikzcd}
	\]
	There is a similar explicit homotopy when $\beta$ is non-vanishing.
\end{proof}

\subsection{Hochschild Homology}\label{s:hochom}
Hochschild homology takes the place of the cohomology of a variety when dealing with derived categories of factorizations, so that it will be the main ingredient for the state space of the GLSM theory.

Fix, now and forever, a field $k$ of characteristic $0$ and let $\mathsf{C}$ be a $k$-linear small dg-category.  
We denote by $\Hoch (\mathsf{C})$ the Hochschild chain complex of $\mathsf{C}$ (see e.g.\ \cite[\S 3.2]{keller1998cyclic}).
%

\begin{definition}
 The \newterm{Hochschild homology} $\op{HH}_{*}(\mathsf{C})$ of $\mathsf{C}$ is defined to be the homology of $\Hoch (\mathsf{C})$,
	\begin{displaymath}
\op{HH}_*(\mathsf{C}) := 	\op{H}_{*}(\Hoch (\mathsf{C})).
	\end{displaymath}	
	We denote by $\op{HH}_*(\cX,w)$ the Hochschild homology of the homotopy category of factorizations with injective components whose isomorphism class lies in $\dabsfact{\cX, w}$ 
	(This is a natural dg enhancement of $\dabsfact{\cX, w}$ on a nice quotient stack, see \cite[Corollary 2.23]{BDFIK}).
\end{definition}

Let $\mathsf C$ and $\mathsf D$ be $k$-linear small dg-categories and $F : \mathsf C \to \mathsf D$ be a dg functor.
The dg functor $F$ naturally induces a chain map $\Hoch (\mathsf C) \to \Hoch (\mathsf D)$
and hence a functorial linear map of $k$-vector spaces
\begin{equation} \label{eq: functorial push}
F_* : \op{HH}_*(\mathsf{C}) \to \op{HH}_*(\mathsf{D}).
\end{equation}

 For a quasi-compact separated scheme $X$, the Hochschild homology of a natural dg enhancement of $\dabsfact{X}$ is canonically isomorphic to the 
Hochschild homology $\op{HH}_*(X)$ of $X$ by \cite[Theorem 5.2]{keller1998cyclic}. 
	Furthermore, for a smooth separated scheme $X$, there is an isomorphism between the Hochschild homology of $X$ 
	and the Hodge cohomology of $X$.  It is commonly referred to as the HKR isomorphism.
	
	The HKR isomorphism was first proven in the affine case by Hochschild--Kostant--Rosenberg in \cite{HKR}. In our level of generality, it is due independently to Swan \cite[Corollary 2.6]{Swan2} and Kontsevich \cite{Kon}.  See also \cite{Yek}. 
We state these results as the following theorem.

\begin{theorem} \label{theorem: HKR}
	Let $X$ be a smooth separated scheme and $\cE$ be a perfect complex. There is a natural isomorphism,
	\[
	\phi_{\op{HKR}} \colon \bigoplus_l \op{HH}_l(X) \to \op{Hod}^*(X) = \bigoplus_l \bigoplus_{q-p=l} \op{H}^{p}(X, \Omega ^q_X).
	\]
	where $\op{Hod}^*(X)$ denotes the Hodge cohomology of $X$.
\end{theorem}



\begin{definition}\label{d:hkrmor}
	The isomorphism $ \phi_{\op{HKR}}$ is called the \newterm{Hochschild--Kostant--Rosenberg isomorphism} or  \newterm{HKR isomorphism} for short.
	Let $Y$ be a smooth projective variety, then the de Rham cohomology $\op{H}^*(Y)$ and the Hodge cohomology $\op{Hod}^*(Y)$ are canonically isomorphic via the Hodge decomposition $d: \op{Hod}^*(Y) \to \op{H}^*(Y)$.  Define the twisted Hodge decomposition by
	\begin{equation}\label{eqn: twist T} T: \gamma  \mapsto \frac{d(\gamma)}{(\sqrt{2\pi i} )^{|\gamma |}} \end{equation} where $|\gamma |$ denotes
the real homological degree of $\gamma$.
	Define  $$\bar{\phi}_{\op{HKR}}: \op{HH}_*(Y) \to \op{H}^*(Y)$$ to be the composition $T \circ \phi_{\op{HKR}}$.
	
	In the case that  $\cY$ is a smooth separated Deligne--Mumford stack with projective coarse moduli space, we may still define an \newterm{HKR morphism} (no longer an isomorphism) as follows.
	By \cite[Corollary 4.2]{EHKV}, there exists a proper, generically finite, surjective morphism $\pi:Y \to \cY$ such that $Y$ is a smooth 
	projective variety.
	Define $\bar{\phi}_{\op{HKR}}$ as the composition
	\[
	\bar{\phi}_{\op{HKR}}\colon \op{HH}_*(\cY) \xrightarrow{(\mathbb L \pi^*)_*} \op{HH}_*(Y) \xrightarrow{\phi_{\op{HKR}}} 
	\op{Hod}^*(Y) \xrightarrow{T} \op{H}^*(Y)
	\xrightarrow{\frac{1}{\deg \pi} \pi_*}  \op{H}^*(\cY)
	\]
	\label{def: Deligne--Mumford HKR map} 
\end{definition}

\begin{remark} \label{rem: HKRindependent}
	The definition of the HKR morphism is independent of the choice of smooth $Y$ since given two proper, generically finite, surjective representable morphisms from $Y,Y'$, we may pass to their fiber product.  It then suffices to check that the HKR map for $Y, Y'$ agree with the one provided by  a resolution of  $Y \times_\cY Y'$.  This follows from the fact that $\pi_*\pi^*$ is, cohomologically, multiplication by $\deg \pi$. 
\end{remark}

\begin{remark}
	This definition follows the one in \cite{PV} but is less than optimal.  In general, we anticipate an ``HKR isomorphism'' between the Hochschild homology of a smooth separated Deligne--Mumford stack and its Chen--Ruan cohomology.  However, at present, this has only been proved in the global quotient case by a finite group \cite{ACC}.
	The map defined above should agree with the projection to the untwisted component of the inertia stack.  This can be shown in the case where $\cZ$ is a global quotient stack by a finite group.
\end{remark}

\begin{remark}  \label{rem: relative HKR} In fact, in the above references the HKR isomorphism lifts to a natural quasi-isomorphism from the sheafified Hochshild homology to the Hodge complex
$$I_X:  \LL\Delta^* \Delta_* \cO _X \to  \sum _i \Omega _X^{i} [i].$$
Using this, given a closed subscheme $Z \hookrightarrow X$, one can also define a relative HKR isomorphism
$$\phi_{\op{HKR}}^{\op{rel}} := \RR\Gamma (\iota _Z)_! \iota _Z ^!  I_X \colon \op{HH}_*(X)_Z \to \op{Hod}^*(X)_Z$$
from the relative Hochschild homology to the relative Hodge cohomology.  It is an isomorphism by the 5-lemma.
\end{remark}

For the reader's convenience, we recall the definition of Fourier-Mukai transforms for Deligne--Mumford stacks (which motivated Definition~\ref{d:FourMuk}).

\begin{definition}\label{d:it0} 
	Let  
\[
(f,g) : \cZ \to \cX \times \cY
\]
be a representable morphism between smooth separated Deligne--Mumford nice quotient stacks.  
Let $\cK \in \dbcoh{\cZ}$ be any object with proper support over $\cY$.
	Then, we define the \newterm{Fourier-Mukai transform} to be the functor
 \begin{align*}
 \Phi_{\cK} : \dabsfact{\cX} & \to \dabsfact{\cY} \\
\cE & \mapsto \RR g_*(\cK  \overset{\mathbb{L}}{\otimes}_{\O_{\cZ}}  \LL f^*\cE)
 \end{align*}
\end{definition}

By \eqref{eq: functorial push},
this induces a map on Hochschild homology
\[
({\Phi_{\cK}})_* : \op{HH}_*(\cX) \to  \op{HH}_*(\cY).
\]
\begin{definition}\label{d:HKR Phi} 
Let $(f, g): \cZ \to \cX \times \cY$ be as in Definition~\ref{d:it0} 
with the added assumption that $\cX$ is a  smooth projective variety.
	Then, for  $\psi \in \op{H}^*(\cX)$,
	we define
	\[
	{\Phi_{\cK}^{\op{HKR}}}(\psi) := \bar \phi_{\op{HKR}}({\Phi_{\cK}})_* \bar \phi_{\op{HKR}}^{-1}(\psi).
	\]
\end{definition}

One may also define various cohomological integral transforms.
\begin{definition}
Let $(f, g): \cZ \to \cX \times \cY$ be as in Definition~\ref{d:it0} with the added assumption that $\cY$ is proper (and hence so is $\cZ$).
Given a cycle $ \phi \in  \op{H}^*(\cZ)$, we define the \newterm{cohomological integral transform} as
 \begin{align*}
 \Phi_{\phi}^{\op{H}} : \op{H}^*(\cX) & \to \op{H}^*(\cY) \\
\psi & \mapsto g_*(f^*\psi \cup \phi).
 \end{align*}
 \end{definition}

\subsection{Comparisons}
Given a geometric phase GLSM, we compare in \S \ref{s:comwitothcon} the GLSM invariants with those of Gromov--Witten theory.
The following general results will be useful for that purpose.
\subsubsection{Integral transforms and HKR}

The following theorem is a slight extension of results of Ramadoss.  It compares two induced Fourier--Mukai functors, one in Hochschild homology and the other in cohomology, under the HKR isomorphism.

\begin{theorem}[{\cite{ram}}]\label{t:ram}
Let $(f, g): \cZ \to \cX \times \cY$ be as in Definition~\ref{d:it0}, with the added assumptions that
$\cX$ is a  smooth projective variety, which will be denoted by $X$; 
the stack $\cY$ has projective coarse moduli space; and finally 
the stack $\cZ$ is a proper Deligne--Mumford stack. 
Then, for the integral transform $\Phi_{\mathcal K} \colon  \dbcoh{X} \to \dbcoh{\cY}$ the action of ${\Phi_{\mathcal K}}_*$ under the HKR {\it morphism} is equal to the cohomological integral transform
	associated to $\op{td}(\cZ/ \cY) \op{ch}(\mathcal K) \in  \op{H}^*(\cZ  ; \C)$, i.e. 
	\[
	{\Phi_{\mathcal K}^{\op{HKR}}} = \Phi^{\op{H}}_{\op{td}(\cZ/ \cY)  \op{ch}(\mathcal K)}. 
	\]
	\label{thm: Ramadoss}
\end{theorem}

\begin{proof}
By \cite[Proposition 5.1]{KreschGeometry}, one can find a finite flat surjective morphism $\pi:Y \to \cY$ such that $Y$ is a smooth projective scheme.  
From which we get a fiber product
\[
\begin{tikzcd}
Z \ar[r, "h"] \ar[d, swap, "\pi|_Z"]  & Y \ar[d, "\pi"] \\
\cZ \ar[r, "g"] & \cY. \\
\end{tikzcd}
\]		

Fix $\gamma \in  \op{H}^*(X)$.
  We have the following chain of equalities,
\begin{align*}
{\Phi^{\op{HKR}}_{\cK}}(\gamma) & = \frac{1}{\deg \pi} \pi_* \circ  \bar \phi_{\op{HKR}} \circ \pi^* \circ  {(\Phi_{\cK}})_* \circ   \bar \phi_{\op{HKR}}^{-1}(\gamma) \\
& = \frac{1}{\deg \pi} \pi_* \circ  \bar \phi_{\op{HKR}} \circ  (\Phi_{\LL \pi|_Z^*\cK})_* \circ   \bar \phi_{\op{HKR}}^{-1}(\gamma) \\
& = \frac{1}{\deg \pi} \pi_* \circ  \bar \phi_{\op{HKR}} \circ  (\Phi_{\RR (f,h)_* \LL \pi|_Z^*\cK})_* \circ   \bar \phi_{\op{HKR}}^{-1}(\gamma) \\
& = \frac{1}{\deg \pi} \pi_* \circ  \Phi^{\op{H}}_{\td( X \times Y / Y)  \op{ch}(\RR (f,h)_* \LL \pi|_Z^*\cK)}(\gamma) \\
& = \frac{1}{\deg \pi} \pi_* \pi^* \circ  \Phi^{\op{H}}_{\td( X)  \op{ch}(\RR (f,g)_* \cK)}(\gamma) \\
& =   \Phi^{\op{H}}_{\td( X \times \cY / \cY)  \op{ch}(\RR (f,g)_* \cK)}(\gamma) \\
& =   \Phi^{\op{H}}_{\td(\cZ / \cY)  \op{ch}\cK}(\gamma). \\
\end{align*}
The first line is by definition.
The second line follows from flat base change.
The third line follows from the projection formula.
The main justification is the fourth line which follows from Theorem 2 and Theorem 6 of \cite{ram}, after identifying cohomology and Hodge cohomology as in Definition~\ref{d:hkrmor}.
The sixth line is flat base change and the seventh line is obvious.
The eighth line is \cite[Theorem 3.5]{Edidin}.   
\end{proof}

\subsubsection{Factorization categories and the derived category}\label{s:comptenpro}

Let $X$ be a smooth variety and $G$ be an affine algebraic group.  Let $T$ denote the total space of a $G$-equivariant vector bundle $\cc E$ 
on $X$, and let $f$ be a regular equivariant section of $\cc E^\vee$ on $X$.   Let $Z := Z(f)$ be the zero locus of $f$. Fiberwise dilation gives a $\CC^*$-action on $T$.  Furthermore, the variety $T$ inherits a function $w$ which corresponds to the pairing with $f$.  We have the following maps:
\begin{equation}\label{e:ISH}
 \begin{tikzcd}
T|_Z \arrow{r}{i} \arrow[d, bend right, swap, "\pi|_Z"]& T \\
Z \arrow{ur}{j} \arrow[u, bend right, "j'"] \arrow[r, "f"] & X
 \end{tikzcd}
\end{equation}

\begin{theorem}[{\cite{Isik, Shipman, Hirano}}]
There is an equivalence of categories
\[
\phi_+ := \mathbb Ri_* \mathbb L \pi|_Z^*  \colon  \dbcoh{[Z/G]} \to \dabsfact{[T / G \times \CC^*], w}.
\]
with inverse
\[
\phi^{-1}_+ = \mathbb R(\pi|_Z)_* \mathbb L i^* \overset{\mathbb{L}}{\otimes} \op{det} \cE^\vee [-\op{rank } \cE]   \colon   \dabsfact{[T / G \times \CC^*], w} \to \dbcoh{[Z/G]}.
\]
Furthermore, if $M \subseteq [Z/G]$ is a closed proper substack, then $\phi_+, \phi^{-1}_+ $ restrict to an equivalence
\[
 \dbcoh{[Z/G]}_M \cong  \dabsfact{[T / G \times \CC^*], w}_{[M / \CC^*]}
\]
where $[M / \CC^*]$ is viewed as a substack of $[T/G \times \CC^*]$ by inclusion along the zero section.
\label{thm: Hirano}
\end{theorem}
\begin{proof}
The equivalence is the main theorem of Isik's paper \cite{Isik} where the functor was defined differently.  It was proven independently by Shipman \cite{Shipman} who defined the functor this way.  The generalization we are using is due to Hirano (see \cite[Proposition 4.8]{Hirano}).  The fact that the inverse functor is so defined, is because this functor is the adjoint (see  e.g. \cite[Theorem 3.8]{EP}, \cite[Corollary 2.5.7]{Pol16}, or \cite[Theorem 4.36]{Hirano17}).

For the final statement, notice that if $A$ is supported on $M$ then $\mathbb Ri_* \mathbb L \pi|_Z^*A$ is supported on $[T|_M / G \times \CC^*]$.  On the other  hand, every factorization is supported on the critical locus of $w$ which is contained in the zero section of $[T / G \times \CC^*]$.  Hence $\mathbb Ri_* \mathbb L \pi|_Z^*A$ is supported on $[M / \CC^*]$.  Conversely, if $B$ is supported on $[M / \CC^*]$ then $\mathbb R(\pi|_Z)_* \mathbb L i^* B \overset{\mathbb{L}}{\otimes} \op{det} \cE^\vee [-\op{rank } \cE]$ is supported on $M$.
\end{proof}
We will use $\phi_- := \mathbb R i_* \mathbb L \pi|_Z^*  \colon  \dbcoh{[Z/G]} \to \dabsfact{[T / G \times \CC^*], -w}$ to denote the analogous functor but where we now view $\mathbb R i_*$ as mapping to $\dabsfact{[T / G \times \CC^*], -w}$.

\begin{lemma}
Let $M \subseteq [Z/G]$ be a closed proper substack.  The following diagram is commutative
\begin{equation} \label{eq: MFtensor}
 \begin{tikzcd}
  \dabsfact{[T / G \times \CC^*], w}_{[M/ \CC^*]}  \arrow{rr}{(- \overset{\mathbb{L}}{\otimes} \phi_-(P))} &  & \dabsfact{[T/ G \times \CC^*], 0}_{[M/ \CC^*]} \arrow{dr}{\mathbb R \pi_*}.\\
\dbcoh{[Z/G]}_M \arrow{rr}{(- \overset{\mathbb{L}}{\otimes} P \otimes \op{det} \cc E) [\rank \mathcal{E} ]} \arrow{u}{\phi_+} & & \dbcoh{[Z/G]}_M \ar{r}{\mathbb R f_* }& \dbcoh{[X/G]}_M.
\end{tikzcd}
\end{equation}
\label{FM commutative diagram}
\end{lemma}

\begin{proof}
We have the following functorial sequence of isomorphisms
\begin{align*}
\mathbb R \pi_* ( \phi_+(Q) \overset{\mathbb{L}}{\otimes} \phi_-(P) ) & = \mathbb R \pi_* (\mathbb R i_* \pi|_Z^* Q  \overset{\mathbb{L}}{\otimes} \mathbb R i_* \pi_|Z^* P)   \\
& = \mathbb R \pi_* \mathbb R i_* (\pi|_Z^* Q  \overset{\mathbb{L}}{\otimes} \mathbb L i^* \mathbb R i_* \pi|_Z^* P)   \\
& = \mathbb R f_* \mathbb R (\pi|_Z)_* (\pi|_Z^* Q  \overset{\mathbb{L}}{\otimes} \mathbb L i^* \mathbb R i_* \pi|_Z^* P)   \\
& = \mathbb R f_* (Q  \overset{\mathbb{L}}{\otimes} \mathbb R (\pi|_Z)_*  \mathbb L i^* \mathbb R i_* \pi|_Z^* P)   \\
& = \mathbb R f_* (Q  \overset{\mathbb{L}}{\otimes} \mathbb \phi^{-1}(\phi(P)) \overset{\mathbb{L}}{\otimes} \op{det} \cc E )  [\rank \mathcal{E} ]  \\
& = \mathbb R f_* (Q  \overset{\mathbb{L}}{\otimes} P \overset{\mathbb{L}}{\otimes} \op{det} \cc E )  [\rank \mathcal{E} ]  \\
\end{align*}
where we use Proposition~\ref{projection formula} in lines 2 and 4.
\end{proof}

\begin{definition}\label{d:phiplus}
Define the equivalence $\tilde \phi_{+/-} \colon  \dbcoh{[Z/G]} \to \dabsfact{[T / G\times \CC^*_R], w}$ to be
\[ \tilde \phi_{+/-} := \op{det} (\cc E^\vee)\overset{\mathbb{L}}{\otimes} \phi_{+/-} [-\rank \mathcal{E} ].\]
\end{definition}
\begin{remark}\label{r:prr}
The previous proposition can be restated as 
\[ \mathbb R \pi_* (\tilde \phi_+(P) \overset{\mathbb{L}}{\otimes} \phi_-(Q) )= {\mathbb R} f_*(P \overset{\mathbb{L}}{\otimes} Q).\]
\end{remark}

On $[T/G]$, the pull-back vector bundle $ \pi^*\cc E$ is naturally $\CC^*_R$-equivariant and has the $\CC^*_R$-invariant tautological section 
$\taut\in H^0([T/G], \pi^* \cc E)$. We also have the section $ \pi^* f \in H^0([T/G], \pi^*(\cc E^\vee \ot_{\cO_{[T/G]}} \cO_{[T/G]}(\eta)))$ where $\eta$ is the character of $\CC^*_R$ of weight one.
The composition
$$\cO_{[T/G]}\xrightarrow{\taut} \pi^*\cc E \xrightarrow{ \pi^*(f)^\vee \ot \op{id}_{\cO_{[T/G]}(\eta)}} \cO_{[T/G]}(\eta)$$
is equal to $w$,
hence we get a Koszul factorization
\begin{equation}\label{S1}
S_1 := \{ \pi^*(f)^\vee, \taut \} \in \dabsfact{[T /G \times \CC^*_R], w}.
\end{equation}
\begin{remark}\label{r:os1}
We note for future reference that $\phi_+(\cO_Z) = S_1^\vee  = \{{\taut}^\vee, \pi^*(f)\}$ and consequently, 
\[\tilde \phi_+ (\cO_Z) = \det(\cc E^\vee) \overset{\mathbb{L}}{\otimes} S_1^\vee [-\rank \cc E] = S_1.\]
\end{remark}

Now assume that $[Z/G[$ is represented by a smooth projective variety.  In this case, we can use the HKR isomorphism to compare the Hochschild cohomology of $\dabsfact{[Z/G]}$ with the usual cohomology of $[Z/G]$.
We  introduce the following Todd class correction to the HKR isomorphism to get the identification of Hochschild homology $\op{HH}_*([T / G\times \CC^*_R], w)$ with the cohomology of $[Z/G]$ we will require.
\begin{definition}\label{d:phitodd}
Let $[Z/G]$ as above be represented by a smooth projective variety.  We define the isomorphism \[\varphi_*^{\td}: \op{H}^*([Z/G]) \to \op{HH}_*([T / G \times \CC^*_R], w)\] by
\[ \varphi_*^{\td}(\gamma) := (\tilde \phi_+)_* \circ \bar \phi_{\op{HKR}}^{-1} (\td( \cc E) \cup \gamma).\]
\end{definition}

\subsubsection{Comparing localized Chern characters}\label{s:comlocchecha}

\newcommand{\pullbackp}{{[p]}}

Consider a vector bundle $A$ on a smooth Deligne--Mumford stack $S$.  
Equip $S$ with a trivial $\CC^*$-action and suppose we have a $\CC^*$-action on $ \op{tot} A $ which is equivariant with respect to the projection
 $p  \colon  \op{tot} A \to S$. This yields a fiber square,
\begin{equation}
\begin{tikzcd}
 \op{tot} A   \arrow[d, "p"] \arrow[r, "{\pi_{A }}"] & {[ \op{tot} A  / \CC^*]}  \arrow[d, "p"]\\
S \arrow[r, "{\pi_{S}}"] &  {[S / \CC^*]} .
\end{tikzcd}
\label{forget square}
\end{equation}

Let $X$ be a closed substack of $S$.  In \cite{PVold}, Polishchuk and Vaintrob defined a 2-periodic localized Chern character for a certain class of two periodic complexes $\cE$.  This is a bivariant class
\[
\chp_{X}^{\op{tot} A} (\cE) \in A^*(X\to \op{tot}A )_{\QQ}.
\]
Since, $S$ is smooth, by Remark 5.3.2 and Lemma 5.3.4 of \cite{Chiodo}, this descends to a map
\[
\chp_{X}^{\op{tot} A}: \dabsfact{{\op{tot} A} , 0}_{X} \to A^*(X\to \op{tot}A )_{\QQ}.
\]

We will now compare this to the usual Chern character in the proposition below. 
For this we view $ \mathbb R p_*$ as a functor
\[
\mathbb R p_*  \colon  \dabsfact{[\op{tot} A / \CC^*], 0}_X \to \dbcoh{S}_X
\]
using the equivalence,
\[
\dabsfact{[S/\CC^*], 0} = \dbcoh{S}.
\]

\begin{proposition} 
There is a commutative diagram

\[
\begin{tikzcd}
\dabsfact{[\op{tot} A / \CC^*], 0}_{X} \ar[d, "\mathbb R p_*"] \ar[r, "\mathbb  L \pi_A^*"] & \dabsfact{\op{tot} A , 0}_{X}  \ar[r, "\chp_{X}^{\op{tot} A}"]   & A^*(X\to \op{tot}A )_{\QQ} \ar[d, "\pullbackp\td(A)"] \\
\dbcoh{S}_X  \ar[rr, "\ch^{S}_X"] & { }& A^*(X\to S)_{\QQ} \\
\end{tikzcd}
\]
where $[p]$ is the canonical orientation of the flat map $p$. 
\label{prop: chiodo}
\end{proposition}
\begin{proof}
Let ${}^{\Z_2}K_0(\tot A)_{X}$ (resp. ${}^{\Z_2}K_0(S)_{X}$) denote the Grothendieck group of the abelian category 
of two periodic cochain complexes of coherent sheaves on $\tot A$ (resp. $S$) which are exact on $\tot A \setminus X$ (resp. $S\setminus X$).
We can view the result as coming from the following extended diagram.
\[
\begin{tikzcd}
\dabsfact{[\op{tot} A / \CC^*], 0}_{X} \ar[d, "\mathbb R p_*" ] \ar[r, "\mathbb L  \pi_A^*"] & \dabsfact{{\op{tot} A} , 0}_{X} \ar[d, "\mathbb R p_*"] \ar[r] &  {}^{\Z_2}K_0(\op{tot}(A))_{X} \ar[d, "\mathbb R p_*"] \ar[r, "\chp_{X}^{\op{tot} A}"]   
& A^* (X \ra \op{tot} A )_{\QQ} \ar[d, "\pullbackp \td(A)"] \\
\dbcoh{S}_X \ar[r, "\mathbb L \pi_S^*"]   & \dabsfact{S , 0}_X \ar[r]&  {}^{\Z_2}K_0(S)_X \ar[r, "{}^{\Z_2}\ch ^{S}_{X}"] & A^*(X\ra S)_{\QQ} \\
\end{tikzcd}
\]
The  left commutative square comes from flat base change from \eqref{forget square} and the middle square is obvious.
Now the horizontal arrow along the bottom is  just $\ch ^{S}_{X}$ by \cite[Proposition 2.2]{PVold}.  The commutativity on the right square is a localized version of \cite[Lemma 5.3.8]{Chiodo} which we now prove.

Let 
\[
W_\bullet := \bigg[ ... \xrightarrow{d_0}  W_1 \xrightarrow{d_1} W_0 \xrightarrow{d_0} W_1  \xrightarrow{d_1}  ...\bigg]
\]
be a 2-periodic complex of coherent sheaves on $\op{tot} A$ supported on $X$.  
Let $$i: S \to \op{tot} A$$ be the inclusion of the zero section.  We note that $\RR i_*$ is surjective, as follows from the fact that any coherent sheaf on $ \op{tot} A$  set-theoretically supported on $X$ has a filtration whose associated graded pieces are scheme-theoretically supported on $X$.  Let $V_\bullet$ denote a 2-periodic complex of coherent sheaves on $S$ with support on $X$ such that 
$\RR i_*(V_\bullet) = W_\bullet$.  Since $p\circ i: S \to S $ is the identity, $\RR p_* (W_\bullet) = V_\bullet$.

By  \cite[Proposition 2.3 (vi)]{PV}, 
\begin{align*}
 \chp_{X}^{\op{tot} A}(W_\bullet) = \op{td}(A)^{-1} \chp_X^S(V_\bullet) \cdot [i].
\end{align*}
From this we immediately conclude that 
\begin{align*}
\op{td}(A) \chp_{X}^{\op{tot} A}(W_\bullet)[p] & = \chp_X^S(V_\bullet)\cdot [i]\cdot[p] \\
&= \chp_X^S(\RR p_* (W_\bullet)),
\end{align*}
as desired.
\end{proof}

\section{Admissible resolutions of GLSMs}\label{s:PV machinery}
In this section, we construct a distinguished factorization, that we call a PV factorization due to its initial use by  Polishchuk--Vaintrob in \cite{PV}.

\subsection{Setup}\label{Setup} 
We start with a (partial) GLSM data $(V,G,\CC^*_R,w)$, {\it without} the choice of a linearization 
$\theta$. 
Recall from \S\ref{s:OC} that we have the groups
\[
\Gamma :=  G \cdot \CC _{R}^*  \subseteq GL(V)
\]
and a character $\chi$ entering in the exact sequence
\[
1 \to G \to \Gamma \xrightarrow{\chi} \CC^* \to 1.
\]
\noindent
Note that we may consider $V$ either as an algebraic object, i.e., as a $\Gamma$-representation, or as a geometric object, i.e., as the affine space $\op{Spec} \op{Sym} V^\vee$ with the linear action of $\Gamma$. We will switch between these two viewpoints without further comment when convenient.

\subsubsection{Some moduli stacks}\label{moduli stacks}
\begin{definition}\label{d:twistedcurves}
Let $\fM_{g,r}^{\op{tw}}$ be the moduli stack parametrizing families of (balanced) prestable twisted curves (in the sense of \cite{AV}) with $r$ gerbe markings (see \cite{Olsson07}). 
\end{definition}
By \cite[Theorems 1.9, 1.10]{Olsson07}, the stack $\fM_{g,r}^{\op{tw}}$ is a smooth Artin stack, locally of finite type. The fibered product of the universal gerbes over 
$\fM_{g,r}^{\op{tw}}$ gives the smooth Artin stack
\begin{equation}\label{d:orbi-curves}\fM^\orb_{g,r}:= \sG_1\times_{\fM_{g,r}^{\op{tw}}} \sG_2\times_{\fM_{g,r}^{\op{tw}}} \dots 
\times_{\fM_{g,r}^{\op{tw}}} \sG_r,\end{equation}
which parametrizes families of marked twisted curves together with sections of the gerbe markings. These are precisely the 
``families of prestable orbi-curves"
appearing in Definition \ref{d:plg}.
\begin{remark}\label{r:tworb}
To avoid confusion, in what follows we will use the terms ``twisted curves''  to refer to families parametrized by Definition~\ref{d:twistedcurves} and ``orbi-curves'' to refer to families parametrized by \eqref{d:orbi-curves}.
\end{remark}

\begin{definition}
For any reductive complex algebraic group $G$, we let $\Mfrak^{\orb/\tw}_{g,r}(BG)$ denote the stack parametrizing families of prestable orbi-curves (resp. twisted curves) $\mathcal{C}\to T$, with $T$ a scheme, together 
with a principal $G$-bundle  
on $\mathcal{C}$, 
such that the induced morphism $\mathcal{C} \to BG$ is representable.
\end{definition}

\begin{definition}
Define the stack $\Mfrak^{\orb/\tw}_{g,r}(B\Gamma)_{\log}$ parametrizing families of 
prestable orbi-curves (resp. twisted curves) together with a principal $\Gamma$-bundle $P$ {\it and} an isomorphism 
$\chi_*(P) \cong \logc$, such that the induced map $[P]:\mathcal{C} \to B\Gamma$ is representable.  \end{definition}

These are again smooth Artin stacks, locally of finite type (for proofs, see, e.g., \cite[Proposition 2.1.1]{CKM}, \cite[\S 2]{CCK}, and \cite[Lemma 5.2.2]{FJR15} ).

There are obvious forgetful smooth morphisms 
$$\Mfrak^{\orb/\tw}_{g,r}(BG)\ra \fM^{\orb/\tw}_{g,r}, \;\;\;\;\;
\Mfrak^{\orb/\tw}_{g,r}(B\Gamma)_{\log} \ra \Mfrak^{\orb/\tw}_{g,r}(B\Gamma)\ra \fM^{\orb/\tw}_{g,r},$$
under which the universal curves $\cC$ (together with the gerbes with sections) pull back.

\subsubsection{Stacks mapping to $\Mfrak^\orb_{g,r}(B\Gamma)_{\log}$.}
In this section we will work with an Artin stack $S$ with a map $S\ra \Mfrak^\orb_{g,r}(B\Gamma)_{\log}$ by which we pull-back the universal structures on $ \Mfrak^\orb_{g,r}(B\Gamma)_{\log}$. To fix notation, we describe these explicitly. 
First, we have a flat and proper family
\[
\pi \colon  \mathfrak C \to S
\]
 of $r$-pointed prestable twisted curves over $S$, together with sections 
 $S\ra \sG _i $
 of the $r$ disjoint gerbe markings $ \sG_i  \subseteq \mathfrak C$. 
 We denote $\Sigma _i:\sG_i \hookrightarrow \mathfrak C$ the inclusion map.
 For any $T$-point $T\to S$, with $T$ a scheme, we have a family of $r$-pointed prestable orbi-curves over $T$ by pull-back. We require that each section $T\to {\sG_i}_T$ induces an isomorphism of $T$ with the coarse moduli space of $\sG_i$. It follows that the
composition $\pi_T\circ{\Sigma_i}_T :{\sG_i}_T\to T$ is the coarse moduli map. In particular, we have a canonical isomorphism,
$$\mathbb R (\pi\circ\Sigma_i)_* \cO_{\sG_i} \cong (\pi\circ\Sigma_i)_*\cO_{\sG_i}\cong \cO_S$$
(where the first equality uses the fact that $\mu_{r_i}$ is linearly reductive).
 
Denote by $\sG:=\coprod_i \sG_i \subseteq \cC$ the union of the gerbe markings, with inclusion map $\Sigma$.
Each of $\sG_i$ is an effective Cartier divisor in $\cC$, and $\sG$ is equal to $\sG_1 + ... + \sG_r $ as a Cartier divisor. Hence there are exact sequences
$$0\to \cO_\cC(-\sG_i)\to \cO_\cC\to {\Sigma_i}_*\cO_{\sG_i}\to 0,$$
$$0\to \cO_\cC(-\sG)\to \cO_\cC\to \Sigma_*\cO_{\sG}\to 0.$$
 
 For a quasicoherent sheaf $\cF$ on $\cC$ we denote 
 $$\cF|_{\sG_i}:={\Sigma_i}_*\Sigma_i^*\cF,\;\;\;\; \cF|_\sG:=\Sigma_*\Sigma^*\cF.$$
 If $\cF$ is locally free we have induced short exact sequences,
 $$0\to \cF(-\sG)\to \cF\to \cF|_\sG\to 0.$$
 Since $\sG _i = S \times B\mu _{m_i}$ for some positive integers $m_i$, the pullback $\Sigma_i^*\cF$ 
 can be considered as a vector bundle on $S$ with fiberwise linear action by $\mu _{m_i}$.
 Hence, for each $i$, the pushforward $\pi_*(\cF|_{\sG_i})$ is the $\mu _{m_i}$-invariant part of the vector bundle.
 
Recall that for the sheaf of relative log differentials, $\log:=\omega _{\mathfrak{C}}(\sG_1 + ... + \sG_r  )$,
there are canonical isomorphisms
 \begin{equation}\label{trivial at markings}
  \Sigma _i^* \log  \cong   \cO_{\sG _i},
 \end{equation}  
 and therefore induced canonical isomorphisms
 \[
 \pi_*(\log|_{\sG_i})\cong\cO_S,\;\;\;\; \pi_*(\log|_\sG)\cong \cO_S^r.
 \]

Further, we are given a principal $\Gamma$-bundle $P$ on $\cC$, {\it and} a fixed isomorphism of principal 
$\CC^*$-bundles,
\[
\varkappa  \colon P \times _\Gamma \CC ^*  \to \ringlog ,
\]
where $\Gamma$ acts on $\CC^*$ by the character $\chi$.
We abuse notation and will use the same letter for the induced isomorphism of line bundles, $\varkappa : \cL_{\chi}\to \log$. 

The data $(\mathfrak C\ra S,\sG,P,\varkappa)$ coming from the map $S\ra \Mfrak^\orb_{g,r}(B\Gamma)_{\log}$ couples with the $\Gamma$-representation $V$ and the $G$-invariant superpotential $w$ to give additional structures.

\subsubsection{Coupling with $V$}

Let $\cV$ denote the locally-free sheaf of sections of the geometric vector bundle $P \times_\Gamma V$.
Set, as in Definition \ref{d:plg},
\[
\op{tot} \mathcal V := \Spec(\op{Sym}\cV^\vee)=P \times_\Gamma V.
\]

The sheaf $\cV$ comes with ``evaluation maps" at the markings to the inertia stack $I[V/G]$. To see this,
fix an index $i\in\{1,\dots,r\}$ and consider the total space 
$$\tot (\pi _* \cV |_{\sG_i})\xrightarrow{p_i} S, $$
with its family of curves with gerbe markings
$$
\xymatrix{ 
\widetilde{\sG}_i \ar[r]^{r_i} \ar[d]_{\widetilde{\Sigma}_i} & \sG_i \ar[d]^{\Sigma_i} \\
\widetilde{\cC} \ar[r]^{q_i} \ar[d]_{\widetilde{\pi}} & \cC \ar[d]^{\pi} \\
\tot (\pi _* \cV |_{\sG_i}) \ar[r]^{p_i} & S
}
$$
obtained by pull-back from $S$.

Denote 
$$\pi_i:=\pi\circ \Sigma_i :\sG_i\ra S,\;\;\;\; \widetilde{\pi}_i:=\widetilde{\pi}\circ \widetilde{\Sigma_i} :\widetilde{\sG}_i\ra \tot (\pi _* \cV |_{\sG_i}), $$
the projections from the gerbes to their respective bases, and set 
$$\cV_i:= \Sigma_i^*\cV, \;\;\;\; \widetilde{\cV}_i:= \widetilde{\Sigma}_i^* q_i^*\cV= r_i^*\cV_i.$$
There is a base-change isomorphism
\begin{equation}\label{gerbe base change}
p_i^*(\pi_i)_*\cV_i \cong (\widetilde{\pi}_i)_*r_i^*\cV_i = (\widetilde{\pi}_i)_*\widetilde{\cV}_i.
\end{equation}

The vector bundle $p_i^*(\pi_i)_*\cV_i=p_i^*(\pi _* \cV |_{\sG_i})$ on $\tot (\pi _* \cV |_{\sG_i})$ has the tautological section
$$\tau_i\in H^0(\tot (\pi _* \cV |_{\sG_i}), p_i^*(\pi_i)_*\cV_i).
$$
By \eqref{gerbe base change} and the identification  
$$ H^0(\widetilde{\sG}_i, \widetilde{\cV}_i) =  H^0(\tot (\pi _* \cV |_{\sG_i}), (\widetilde{\pi}_i)_*\widetilde{\cV}_i),
$$
we may view $\tau_i$ as a global section in $H^0(\widetilde{\sG}_i, \widetilde{\cV}_i)$.

Let 
$$\widetilde{P}_i:= \widetilde{\Sigma}_i^*q_i^*P=r_i^*\Sigma_i^* P$$
be the pulled-back principal $\Gamma$-bundle on $\widetilde{\sG}_i$, so that $\widetilde{\cV}_i$ is the sheaf of sections of
the geometric vector bundle $\widetilde{P}_i\times_\Gamma V$. The pair $(\widetilde{P}_i,\tau_i)$ gives
rise  to a morphism
 \begin{equation*}
[\widetilde{P}_i,\tau_i] \colon  \widetilde{\sG} _i \to [V /\Gamma].
 \end{equation*}
It is representable, since $[\widetilde{P}_i]\colon  \widetilde{\sG} _i\to B\Gamma$ is so.
Further, being the pull-back of $\sG_i$, the gerbe $\widetilde{\sG}_i$ comes with a section. It follows that the diagram 
\begin{equation*}\label{gerbe to VmodGamma}
\xymatrix{
\widetilde{\sG}_i \ar[r]^{[\widetilde{P}_i,\tau_i]} \ar[d]_{\widetilde{\pi}_i} & [V/\Gamma] \\
\op{tot} (\pi_*\cV|_{\sG_i}) & 
}
\end{equation*}
determines a morphism
\begin{equation}\label{eval VmodGamma} 
\op{tot} (\pi_* \mathcal V|_{\sG_i}) \to I[V/\Gamma],
\end{equation}
see \cite[\S3.2]{AGV}. 
    Here $I\fX$ denotes the inertia stack $\fX \times _{\fX \times \fX} \fX$ of an algebraic stack $\fX$. We allow $I\fX$ 
    to not be quasi-compact. 

The map \eqref{eval VmodGamma} factors through the natural map $I[V/G]\ra I[V/\Gamma]$. Indeed,
since $\widetilde{P}_i \times _\Gamma \CC ^*$ is canonically isomorphic to $\widetilde{P}_i/G$, using the trivialization of 
$\widetilde{\Sigma} ^*_i\omega_{\widetilde{\cC}}^{\op {log}}$ coming from \eqref{trivial at markings},
we obtain a principal $G$-bundle $\widetilde{P'}_i$ on $\widetilde{\sG}_i$ such that 
$\widetilde{P'}_i\times _G \Gamma \cong \widetilde{P}_i$ as $\Gamma$-bundles.
It follows that there is an isomorphism of geometric vector bundles $\widetilde{P'}_i\times _G V\cong  \widetilde{P}_i\times_\Gamma V$,
and therefore $\widetilde{P'}_i\times _G V$ comes with the section $\tau_i$. In other words, we have constructed a diagram
\begin{equation*}\label{gerbe to target} 
\xymatrix{
\widetilde{\sG}_i \ar[r]^{[\widetilde{P'}_i,\tau_i]} \ar[d]_{\widetilde{\pi}_i} & [V/G] \\
\op{tot} (\pi_*\cV|_{\sG_i}) & 
}
\end{equation*}
which determines the required morphism
\begin{equation}\label{eval V}
\op{ev}^i_{\cV}: \op{tot} (\pi_* \mathcal V|_{\sG_i}) \to I[V/G],
\end{equation}
factoring \eqref{eval VmodGamma}.

\subsubsection{Coupling with $w$}\label{coupling with w}
Equip $\A^1$ with the standard dilation action of $\CC^*$ and view the superpotential as a function
\[
w \colon  V \to \A^1
\]
which is $\Gamma$-equivariant with respect to the map $\chi: \Gamma\ra\CC^*$.   
This amounts to giving a $\Gamma$-invariant element  $ w \in \CC _{\chi}\ot \op{Sym} V^\vee $.  
Pairing with this $w$ gives a morphism of $\Gamma$-representations $\overline{w}: \op{Sym} V \to \C_\chi$.  
Then, viewing it as
a morphism of vector bundles on the classifying stack $B\Gamma$ and pulling back via the map $[P]:\cC\to B\Gamma$
gives a homomorphism $ [P]^*\overline{w}:  \op{Sym} \mathcal V \to \cL_\chi$ of locally free sheaves on $\cC$.

Define 
\begin{equation}\label{kappa_w}
\varkappa_w  \colon  \op{Sym} \mathcal V \to \log 
\end{equation}
as the composition
\begin{align*}
 \op{Sym} \mathcal V  \xrightarrow{[P]^*\bar{w}} \cL_\chi 
 \xrightarrow{\varkappa} \log.
\end{align*}
\begin{remark}
In the discussion of the evaluation maps and of $\varkappa_w$ we have essentially ignored the $\CC_R^*$-action, i.e., the {\it lower} grading of $V$. This was done for the purpose of streamlining the exposition in subsections \S\ref{ss: admissible resolutions}-\S\ref{ss:sc} below. However, this action is crucial for the theory developed in the paper and we will take it into account fully in \S\ref{s:rigev}
\end{remark}

\subsection{Admissible resolutions}\label{ss: admissible resolutions}
Admissible resolutions are special resolutions of $\mathbb{R}\pi_*\cV$ on $S$ satisfying three conditions (see Definition \ref{admissible}).
We prove that the first condition is valid over any base Artin stack $S\ra\fM_{g,r}(B\Gamma)_{\log}$ for which 
the map $\pi ':\cC ' \ra S$ is projective where $\cC'$ is the universal coarse curve over $S$, so that $\pi'$ is a representable morphism (see \S \ref{cond1}).
However, to obtain the second and third conditions, we need to work over a Deligne--Mumford stack $S$ whose coarse moduli space is projective over an affine, 
as we show in \S \ref{ss:sc}.
In \S \ref{PVconst}, we explain how to construct a fundamental factorization from an admissible resolution.

\subsubsection{Morphisms in the derived category} 

We construct two morphisms in the bounded derived category $\dbcoh{S}$ of coherent sheaves on $S$.

\begin{enumerate}

\item The first morphism exists when there is at least one marking and is easy to define. Namely, for each $i=1,\dots, r$, the restriction map 
$\op{res}_i:\mathcal V  \to  \mathcal V|_{\sG_i}$
induces a morphism
\begin{align}
\mathbb{R} \pi_* \mathcal V  \xrightarrow{\mathbb{R}\pi_*\op{res}_i} & \mathbb{R} \pi_* ( \mathcal V|_{\sG_i}) \notag \\
& = \pi_* ( \mathcal V|_{\sG_i}). \label{eq: evmap}
\end{align}
Hence we also have a morphism $\mathbb{R}\pi_*\op{res}= \sum_{i=1}^r \mathbb{R}\pi_*\op{res}_i$ in $\dbcoh{S}$
\begin{equation}\label{total evmap}
\mathbb{R} \pi_* \mathcal V  \xrightarrow{\mathbb{R}\pi_*\op{res}}  \pi_* ( \mathcal V|_{\sG})\cong  \oplus_{i=1}^r\pi_* ( \mathcal V|_{\sG_i}).
\end{equation}

\item 
The other morphism is more involved.
Let
\begin{equation}\label{natural sym map}
\mathrm{natural} \colon  \op{Sym}\mathbb{R} \pi_* \mathcal V \to \mathbb{R} \pi_* \op{Sym} \mathcal V 
\end{equation}
be the map coming from the counit of the adjunction $\mathbb L \pi^* \dashv \mathbb R \pi_*$.  That is, the morphism $\mathrm{natural}$ is induced from the composition, 
\begin{equation*}
\mathbb{L} \pi^* \op{Sym}  \mathbb{R} \pi_* \mathcal V \to  \op{Sym} \mathbb{L} \pi^* \mathbb{R} \pi_* \mathcal V    
\to \op{Sym} \mathcal V.
 \label{rescaled natural map} 
\end{equation*}

Composing with the derived push-forward of the map $\varkappa_w$ (see \eqref{kappa_w}) gives a morphism in the derived category
\begin{align}
\op{Sym}\mathbb{R} \pi_* \mathcal V \notag 
&  \xrightarrow{\mathrm{natural}} \mathbb{R} \pi_* \op{Sym} \mathcal V  \notag \\
& \xrightarrow{ \mathbb{R} \pi_* \varkappa_w} \mathbb{R} \pi_* \log .
\label{eq: derivedlogmap} 
\end{align}

We may combine $\eqref{eq: derivedlogmap}$ with the restriction map to $\R\pi_* \log|_{\sG}$ (in the case of no markings this becomes the zero map).  Since $\log$ trivializes on $\sG$, we have $\R\pi_* \log|_{\sG} \cong \O_S^r$.  This gives a map
\begin{equation}
 \op{Sym} \mathbb R\pi_* \cV \to  \O_S^r.
 \label{intermediate}
\end{equation}
The map above leads to the following morphism of exact triangles
\begin{equation}
\begin{tikzcd}
E \ar[r] \ar[d, dashed] & \op{Sym} \mathbb R\pi_* \cV \ar[r] \ar[d] &  \O_S^r  \ar[d] \ar[r] & E[1] \ar[d, dashed]  \\
 \mathbb R\pi_* \nolog  \ar[r] & \mathbb R\pi_* \log \ar[r] &  \O_S^r \ar[r] &  \mathbb R\pi_* \nolog [1] \\
\end{tikzcd}
\label{eq: mapofcones}
\end{equation}
where
\[
E:=\op{Cone}(\op{Sym}\mathbb R \pi_* \cV \to \O_S^{r})[-1],
\]
is the shifted cone.  The dashed arrow is not uniquely determined in this fashion.  Therefore, we need to provide a mechanism to determine it.

Let
\[
E_d :=\op{Cone}(\op{Sym}^d \mathbb R \pi_* \cV \to \O_S^{r})[-1],
\]
and
\[
\op{Sym}^d \mathfrak C := [\overbrace{\mathfrak C \times_S ... \times_S \mathfrak C}^{d-times} / S_d]
\]
where $S_d$ is the symmetric group on $d$ letters acting in the obvious way.

We have a $S_d$-equivariant diagonal map
\[
\Delta_d : [\mathfrak C / S_d] \to \op{Sym}^d \mathfrak C.
\]
This induces a morphism of exact sequences on $\op{Sym}^d \mathfrak C$
\begin{equation}
\begin{tikzcd}
0 \ar[r] & \op{ker}  \ar[r] \ar[d, "f_d"] & \cV^{\boxtimes d} \ar[r] \ar[d] &  (\Delta_d)_*\log|_{\sG}  \ar[d] \ar[r] & 0   \\
0 \ar[r] &   (\Delta_d)_* \nolog  \ar[r] &  (\Delta_d)_* \log \ar[r] &   (\Delta_d)_* \log|_{\sG} \ar[r] &  0 \\
\end{tikzcd}
\end{equation}
Pushing forward and composing with the Grothendieck trace map we obtain our desired second map
\begin{align}
E_d & \xrightarrow{ \R\pi_*f_d} \R\pi_*\nolog \notag \\
& \xrightarrow{\text{H}^1} \R^1\pi_*\nolog[-1]  \notag \\
& \xrightarrow{\text{trace}}  \O_S[-1] \label{eq: finalalpha}
\end{align}
\end{enumerate}

In the presence of markings, the morphisms \eqref{total evmap} and \eqref{intermediate} satisfy a certain compatibility.
Namely, apply the functor $\op{Sym}$ to \eqref{total evmap} and then compose with the inclusion
\[\op{Sym}\pi_* ( \mathcal V|_{\sG_i})\xrightarrow{\op{natural}} \pi_* ( \op{Sym}\mathcal V|_{\sG_i})
\]
to get
\begin{equation}
\op{Sym}\mathbb{R} \pi_* \mathcal V  \xrightarrow{\oplus_{i=1}^r \op{natural}\circ \op{Sym}\mathbb{R}\pi_*\op{res}_i}   \oplus_{i=1}^r\pi_* ( \op{Sym}\mathcal V|_{\sG_i}).
\end{equation}

Consider the commutative diagram of quasicoherent sheaves on $\cC$
\[
\begin{CD}
 \op{Sym} \cV @> [P]^*\bar{w} >>  \cL_\chi @>\varkappa>> \log \\
@V\op{restrict}VV @V \op{restrict} VV @V \op{restrict} VV\\
{\Sigma}_*{\Sigma}^* \op{Sym} \cV @> {\Sigma}_*{\Sigma}^*[P]^* \bar{w} >>  {\Sigma}_*{\Sigma}^* \cL_\chi @> {\Sigma}_*{\Sigma}^* \varkappa >> {\Sigma}_*{\Sigma}^* \log \\
\end{CD}
\]
\noindent
Applying $\mathbb{R}\pi_*$ gives a commuting diagram in $\dbcoh{S}$
\[
\begin{CD}
 \mathbb{R}\pi_*\op{Sym} \cV @>\pi_*\varkappa_w>> \mathbb{R}\pi_*\log \\
@VVV  @VVV\\
\oplus_{i=1}^r\pi_* ((\op{Sym} \cV)|_{\sG_i})  @> \bigoplus_i\pi_*{\Sigma_i}_*{\Sigma_i}^* \varkappa_w >> \cO_S^r \\
\end{CD}
\]
where we have now omitted the middle column.

The natural map \eqref{natural sym map} followed by the clockwise composition from the top left to the bottom right gives the morphism 
\eqref{intermediate}.  On the other hand, the natural map \eqref{natural sym map} followed by composition in the counterclockwise direction gives 
$\sum_{i=1}^r \pi_*{\Sigma_i}_*{\Sigma_i}^* \varkappa_w \circ \op{Sym}\mathbb{R}\pi_*\op{res}_i$.  This gives our advertised compatibility: as morphisms in the derived category, the morphism \eqref{intermediate} is equal to:
\begin{equation}\label{derived compatibility}
\sum_{i=1}^r \pi_*{\Sigma_i}_*{\Sigma_i}^* \varkappa_w \circ \op{natural}\circ \op{Sym}\mathbb{R}\pi_*\op{res}_i  .
\end{equation}

\begin{lemma} \label{lem: the function evw}
The function $w \circ \mathrm{ev} ^i_{\mathcal V}$ as a section of
$\mathrm{Sym} (\pi_* \mathcal{V} |_{\mathscr G _i})^{\vee}$ coincides with the dual of $\pi_* \Sigma _{i*} \Sigma _i^* \varkappa _w \circ   \mathrm{natural} $.
\end{lemma}

\begin{proof} First recall that the gerbes we consider are trivial i.e.\ $\mathscr G_i= S\times B\mu _{r_i}$ for some integer positive $r_i$.
Now let $P_i'$ be the reduction of $\Sigma _i^* P$ to the group $G$ by the trivialization of $P_i/G$ viewed as a principal $G$-bundle on $S$ with a $\mu_{r_i}$-equivariant structure.  Then the vector bundle $\pi_* (\mathcal{V} |_{\mathscr{G}_i})$ is nothing more than the $\mu_{r_i}$-fixed subbundle $(P_i' \times _{G} V )^{\mu _{r_i}}$
of $P_i' \times _G V$.

Thus we have natural maps of stacks
\[ (P_i' \times _{G} V )^{\mu _{r_i}} \overset{\mathrm{inc}}{\to} P_i' \times _G V \overset{\mathrm{pr}}{\to} [V/G], \] 
where $\mathrm{pr}$ is the quotient map to the second factor. These  in turn  give homomorphisms between the sheaves of $\cO_S$-algebras
\[ (\mathrm{Sym} V^{\vee})^G\otimes_{\mathbb C} \mathcal{O}_S \overset{\mathrm{pr}^{\sharp}}{\to} \mathrm{Sym} (P_i' \times _{G} V )^{\vee} 
\overset{\mathrm{inc}^{\sharp}}{\to}  \mathrm{Sym} ((P_i' \times _{G} V )^{\mu _{r_i}})^{\vee} ) \]
such that 
\begin{equation} \label{eq: wev}
\mathrm{inc}^{\sharp} \circ \mathrm{pr}^{\sharp} (w ^{\sharp} \otimes 1_{\mathcal O_S}) = w \circ \mathrm{ev} ^i_{\mathcal V}
\end{equation}
as functions on $\tot (P'_i \times _G V)^{\mu _{r_i}} = \tot \pi_* (\mathcal{V} |_{\mathscr{G}_i})$.

On the other hand, the following diagram of sheaves of $\mathcal O _S$-modules commutes 
\[ \xymatrix{  \mathrm{Sym}  (P_i' \times _{G} V )^{\mu _{r_i}}  \ar[r]_{\mathrm{natural}} \ar[rd]_{\mathrm{inc} ' \circ \mathrm{natural}} 
 &  (\mathrm{Sym}  P_i' \times _{G} V )^{\mu _{r_i}} \ar[r]_{\pi_* \Sigma _{i*} \Sigma _i^* \varkappa _w}  \ar[d]_{\mathrm{inc} '} 
 &  \ar[d]_{=} (P_i'\times _{G} \mathbb{C}_{\chi })^{\mu _{r_i}}\\
                                 &   P_i' \times _{G} \mathrm{Sym} V  \ar[r]_{P_i' (w)}   &  P_i'\times _{G} \mathbb{C}_{\chi }  = \mathcal O _S ,} \]
where $P_i'(w)$ is the obvious homomorphism induced from $w$.
Now $\mathrm{inc} ' \circ \mathrm{natural}$ is dual to $\mathrm{inc} ^{\sharp}$ and $P_i'(w)$ is dual to $\mathrm{pr}^{\sharp} (w \ot 1_{\cO_S})$.  

In conclusion, the bottom path is dual to $\mathrm{inc}^{\sharp} \circ \mathrm{pr}^{\sharp} (w ^{\sharp} \otimes 1_{\mathcal O_S})$ which equals $w \circ \mathrm{ev} ^i_{\mathcal V}$ by \eqref{eq: wev}.  Meanwhile, the top path is just $\pi_* \Sigma _{i*} \Sigma _i^* \varkappa _w \circ   \mathrm{natural} $.
\end{proof}

\subsubsection{Cochain-level realization} \label{subsubsec: cochain-level}

We will need cochain-level realizations of the two compatible derived-level morphisms \eqref{total evmap} and \eqref{eq: finalalpha}.
  Let $\rho : \cC \to \cC'$ be the natural map to the universal coarse curve. We have $\pi' \circ \rho = \pi$.
 From now on we assume that $S$ is quasi-compact and there is an invertible sheaf $\cO (1)$ on $\fC$ such that $\rho _* \cO (1)$ is a $\pi'$-ample line bundle. 
In this case we will simply say that
{\it $\cO (1)$ is $\pi$-ample}.
For example, if we have a linearization $\theta$ of the $G$-action on $V$ with a lift $\nu\in \widehat{\Gamma}$, we would assume that 
the ampleness part of the stability condition \eqref{stableLG} holds, namely that $\cO (1):= \log \ot \cL _\nu ^{\ot M}$ is $\pi $-ample  (for some positive $M$).

Choose a two-term finitely generated resolution by vector bundles,
 \[
\mathbb{R} \pi_* \mathcal V \cong  [ A \xrightarrow{d} B ] \quad \textrm{on $S$},
 \]
 which  may be constructed with the use of a $\pi $-ample invertible sheaf $\cO (1)$ (see Corollary \ref{surjectivity downstairs}.) 

\begin{dfn}\label{admissible} The resolution $[ A \xrightarrow{d} B ]$ is called \newterm{admissible} if it satisfies 
Conditions~\ref{assume ev chainlevel}, \ref{assume alpha chainlevel}, and \ref{compatible} below.
\end{dfn}

\begin{condition}
\label{assume ev chainlevel}
The resolution $[ A \to B ]$ satisfies the following:
\begin{enumerate}
\item The map \eqref{eq: evmap} is realized at the cochain level i.e.\ as a map
\begin{equation}
 A \xrightarrow{\op{ev}_A^i}  \pi _* ( \mathcal V |_{\sG_i} )
\label{eq: evchainmap}
\end{equation}
\item The map $A \xrightarrow{\oplus_{i=1}^r \op{ev}_A^i}    \pi _* (\mathcal V |_{\sG} )$ is surjective\end{enumerate} 
\end{condition}

\begin{condition}
\label{assume alpha chainlevel} 
There are homomorphisms $Z: \op{Sym} A \ra \cO_S^r$ and  $\alpha^\vee:  (\op{Sym} A) \otimes B \ra \O _S$ making 
a commutative diagram 
\begin{equation}
\begin{CD}
\op{Sym} A @>>>  (\op{Sym} A) \otimes B \\
@V Z VV @V \alpha^\vee VV\\
 \O_{S}^r@>\op{sum}>> \O_{S}
\end{CD}
\label{eq: alphachainmaporiginal}
\end{equation}
such that the restriction 
\begin{equation}
\label{eq: alphachainmap}
\begin{CD}
\bigg [ \op{Sym}^d A @> (\delta, -Z) >> ( \op{Sym}^{d-1} A )\otimes B \oplus \O_S^r  @>>> ( \op{Sym}^{d-2} A) \otimes  \bigwedge\nolimits^2 B @>>> ... \bigg ]\\
@VVV @VV(\alpha^\vee|_{\op{Sym}^{d-1} A \otimes B}, \op{sum})V \\
\bigg [ 0@>>> \O_{S} \bigg ]
\end{CD}
\end{equation}
is a cochain level realization of the map \eqref{eq: finalalpha}.
\end{condition}

\begin{condition}
The compatibility \eqref{derived compatibility} is realized at the cochain level, i.e.~the following holds:
 \[
Z =  \sum_{i=1}^r \pii \Sigma_i^* \varkappa_w \circ\op{natural}\circ \op{Sym} \op{ev}_A^i ,
 \]
 where $\op{ev}_A^i$ is the map in \eqref{eq: evchainmap}.
 \label{compatible}
\end{condition}

\begin{remark}
In what follows, when we refer to an admissible resolution $[ A \xrightarrow{d} B ]$, we mean the resolution, \textit{together with choices of maps} $\op{ev}_A^i, Z, \alpha^\vee$ satisfying the conditions above.
\end{remark}

\subsection{Definition of the Polishchuk--Vaintrob factorization}\label{PVconst}
For an admissible resolution as above, let $\tot(A):=\Spec(\op{Sym}A^\vee)$, with projection $p: \tot (A) \ra S$. 

\subsubsection{Evaluation maps}
Recall the morphism \eqref{eval V}
$$\op{ev}^i_{\cV}: \op{tot} (\pi_* \mathcal V|_{\sG_i}) \to I[V/G].$$
\begin{definition} \label{eval on tot A}
The \newterm{evaluation map} $\op{ev}^i: \tot (A) \ra  I[V/G]$ is the composition 
\[ \tot (A) \xrightarrow{\ev^i_A}  \op{tot} (\pi_* \mathcal V|_{\sG_i}) \xrightarrow{\op{ev}^i_{\cV}} I[V/G]. \]
\end{definition}

\subsubsection{Maps $\alpha$ and $\beta$}

Let
\begin{equation}\label{explicit beta}
\beta = p^*d \circ \tau_A  \colon  \O_{\op{tot}(A)} \to p^*A \to p^* B
\end{equation}
be the section induced by the differential $A\xrightarrow{d} B$,
where $\tau_A$ is the tautological section of $p^*A$ obtained by the Casimir element in $A^\vee \ot A  \subseteq \Sym A^\vee \ot A$.
Abusing notation, we also denote the associated cosection as
\[
\alpha^\vee \colon  p^*B \to \O_{\op{tot}(A)}.
\]

We abuse notation further and denote also by $w: I[V/G] \ra \mathbb{A}^1$ the composition 
$I[V/G] \ra [V/G] \xrightarrow{w} \mathbb{A}^1$.

\begin{proposition}
Given any admissible resolution $[A \to B]$ of $\mathbb{R} \pi_* \mathcal V$  we have the following equality:
\begin{equation}
\alpha^\vee \circ \beta = \sum _i w \circ \op{ev}^i
\label{depend on w}
\end{equation}
\end{proposition}

\begin{proof}
We start by recalling that $\alpha^\vee$ is the map
\[
\alpha ^{\vee} : (\mathrm{Sym} A) \otimes B \to \mathcal{O}_S
\]
 in \eqref{eq: alphachainmaporiginal}.
This yields a map
\[
\alpha^{\vee}_1 : B \to \mathrm{Sym} A^{\vee}
\]
 (since the map itself factors through a finite summand of $\mathrm{Sym} A$). 
This $\alpha^{\vee}_1$  gives an $\mathcal{O}_S$-module homomorphism 
$$\alpha ^{\vee}_2: (\mathrm{Sym} A^{\vee}) \otimes B \to \mathrm{Sym} A^{\vee} \otimes \mathrm{Sym} A^{\vee} \underset{product}{\to} \mathrm{Sym} A^{\vee}, $$
which was also denoted by $\alpha ^{\vee}$ (abusing notation). 

Denote by $w^{\sharp}$  the element of $(\mathrm{Sym} V^{\vee})^G$ induced from $w$.
To prove \eqref{depend on w}, it is enough to 
show that for every collection of local sections $a_1$, ..., $a_t$ of $A$,
\begin{equation} \label{reduced eqn} 
\langle \alpha ^{\vee}_2 \circ \beta (1) , a_1 ... a_t \rangle =\sum _{i=1, ..., r} \langle (\mathrm{ev}^i )^{\sharp}\circ w^{\sharp}, a_1 ... a_t \rangle ,
\end{equation}
where $(\mathrm{ev}^i )^{\sharp}$ is the ring homomorphism $(\mathrm{Sym} V^{\vee})^G \ \to \mathrm{Sym} \ A^{\vee}$ induced from
the composition 
\[ \mathrm{tot} A \xrightarrow{\mathrm{ev}_A^i} \mathrm{tot} (\pi_* \mathcal{V} |_{\mathscr{G}} ) \xrightarrow{\mathrm{ev}^i_{\mathcal{V}}} I[V/G] \to [V/G] \to \mathrm{Spec} (\mathrm{Sym} V^{\vee})^G.\]
Now we have a sequence of equalities  \[ \begin{array}{rll}
\langle \alpha ^{\vee}_2 \circ \beta (1) , a_1 ... a_t \rangle & = \alpha ^{\vee} ( \sum _{j=1,..., t}  a_1 ... \hat{a_j} ... a_t \otimes da_j)  \\
& =   \mathrm{sum} \circ Z (a_1 ... a_t)  \\
&=  \sum_{i=1}^r \pii \Sigma_i^* \varkappa_w \circ\op{natural}\circ \op{Sym} \op{ev}_A^i ( a_1 ... a_t)   \\
&=  \sum _{i=1, ..., r} \langle (\mathrm{ev}^i )^{\sharp}\circ w^{\sharp}, a_1 ... a_t \rangle 
\end{array}  \] 
The first line is by definition of $\alpha_2^\vee$.  The second line is by \eqref{eq: alphachainmaporiginal}.  The third line is by Condition \ref{compatible}.  The fourth line is by Lemma \ref{lem: the function evw}.  This completes the proof.
\end{proof}

\begin{remark} \label{rem: tautological section}
In \cite[Section 4.1]{PV}, they rescale the map $\alpha^\vee$.  This is, roughly, due to the fact that they view the tautological section of $\op{Sym}^t A$ as the map $ a \mapsto a^{\otimes t}$ without the factor of $t!$.  More precisely, they assemble $\alpha^\vee$ from pieces obtained from each monomial in $w$ using tautological sections of the corresponding monomial component of $\op{Sym} A$ using the R-charge decomposition $A = A_1 \oplus ... \oplus A_n$ (see \S\ref{s:rigev}).
Our rescaling of the tautological section is consistent with viewing the exponential function on the tautological section for $A$ as a tautological section of $\op{Sym} A$ which is what we do in \S\ref{Z dw}.
\end{remark}

\begin{definition}\label{d:pvfact}
Let $[A \to B]$ be an admissible resolution of $\mathbb R\pi_* \cV$ on $S$.
The \newterm{Polishchuk--Vaintrob (PV) factorization} associated to $(\mathfrak C \to S, P, V, \Gamma , \chi, \varkappa , w)$ is the Koszul factorization $\{-\alpha, \beta\}$ on $\op{tot} A$.
\end{definition}

\subsection{On Condition \ref{assume ev chainlevel}}\label{cond1}

We show that resolutions satisfying Condition \ref{assume ev chainlevel} can be constructed over an arbitrary base $S$.
\begin{lemma}
\label{ev at level C} 
There is a $\pi$-acyclic resolution $[ \cA \to \cB ]$ of $\cV$ on $\cC$, with $\cA, \cB$ locally free of finite rank, and with a homomorphism 
$ \cA \ra \cV|_{\sG}  $  satisfying:
\begin{enumerate}
\item The diagram 
\[ \xymatrix{ \cV \ar[r] \ar[d] & \cA \ar[d]  \\
   \cV|_{\sG}   \ar[r]_{=}  &   \cV|_{\sG} } \]
is commutative. 
\item The kernel of $\cA\to \cV|_{\sG}$ is locally free and $\pi$-acyclic.
\end{enumerate}
\end{lemma}

\begin{proof}
Using the $\pi $-ample line bundle $\cO(1)$,
 we may choose a finitely generated locally free resolution of  $\cV $ 
\begin{eqnarray}\label{cA'cB'}
0 \to \cV \to \cA' \to \cB ' \to 0.
\end{eqnarray}
such that $\cA' (-\sG)$ is $\pi$-acyclic.\footnote{A construction of such a resolution goes as follows. 
 First, note that for large $n$ the natural map $\pi^* \pi _* \cO (n)  \to \cO (n)$  is a surjection with  locally free kernel.
By tensoring the dual map with $\cV (n)$, we obtain an injection $\cV  \to (\pi^* \pi _* \cO (n) )^{\vee} \ot \cV (n)$ with locally free cokernel. The rest is standard.}

Composing the restriction $\cA'\to\cA'|_{\sG}$ with the first projection, we get a map $f: \cA'\oplus  \cV|_{\sG}\to \cA'|_{\sG}$. Similarly,
we have another map $g: \cA'\oplus  \cV|_{\sG}\to \cA'|_{\sG}$, obtained by composing the inclusion
$ \cV|_{\sG}\to \cA'|_{\sG}$ with the second projection.
Define $\cA$ as the equalizer
$$ \xymatrix{
\cA\ar[r]  & \cA ' \oplus  \cV|_{\sG} \ar@<.5ex>[r]^f \ar@<-.5ex>[r]_g &  \cA'|_{\sG}.
}$$
Note that
the kernel of the projection $\cA \ra  \cV|_{\sG}$ is canonically isomorphic to $\cA' (-\sG)$.
Since $\cA' (-\sG)$ is $\pi$-acyclic, so is $\cA$. 

Now we have a natural monomorphism $\cV\ra \cA$ and we define $\cB$ as its cokernel.
Since the exact sequence \eqref{cA'cB'} is locally split, one can easily check that $\cA$ and $\cB$ are finitely generated locally free.
\end{proof}

\begin{corollary} \label{surjectivity downstairs}
For $\cA, \cB$ as in Lemma \eqref{ev at level C}, we obtain 
a two term vector bundle resolution 
$$\RR \pi _* \cV \cong  [\pi_* \cA \ra \pi _*\cB]$$
with a surjective homomorphism $\pi_* \cA \ra  \oplus _i \pi_*(\cV|_{\sG_i})$.
\end{corollary}

The following lemma will be used several times in later sections.

\begin{lemma}\label{l:surjadmi}
Suppose $[A' \to B']$ is a resolution of $\R\pi_* \mathcal V$ satisfying Condition~\ref{assume ev chainlevel} (respectively \ref{assume alpha chainlevel}, \ref{compatible}) and 
\[
\begin{tikzcd}
A\arrow[r] \arrow[d, "e_A"] & B \arrow[d, "e_B"] \\
A' \arrow[r] & B'
\end{tikzcd}
\]
is a quasi-isomorphism.  Assume that $e_B$ is surjective.  Then  $[A \to B]$ is a resolution of $\R\pi_* \mathcal V$ satisfying Condition~\ref{assume ev chainlevel} (respectively \ref{assume alpha chainlevel}, \ref{compatible}).
\end{lemma}
\begin{proof}
For Condition~\ref{assume ev chainlevel}, define $\op{ev}^i_{A}$  as the following composition
\[
\op{ev}^i_{A} : A \xrightarrow{e_A}  A ' \xrightarrow{\op{ev}^i_{A'} }  \O_U
\]
which clearly satisfies Condition~\ref{assume ev chainlevel} if $\op{ev}^i_{A'}$ does.
For Conditions \ref{assume alpha chainlevel}, \ref{compatible}, define $Z$ and $\alpha^\vee$ as the following compositions
\[
Z : \op{Sym}  A \xrightarrow{\op{Sym}(e_A)} \op{Sym}  A ' \xrightarrow{ Z' } \O_U.
\]
\[
\alpha^\vee : \op{Sym}  A \otimes  B \xrightarrow{\op{Sym}(e_A) \otimes e_B} \op{Sym}  A ' \otimes  B ' \xrightarrow{ (\alpha')^\vee} \O_U.
\]
where $(\alpha')^\vee$ and $Z'$ are the corresponding maps for the resolution $[A' \to B']$.  Condition~\ref{assume alpha chainlevel} for $Z, \alpha^\vee$ is immediate from Condition~\ref{assume alpha chainlevel} for $Z', (\alpha')^\vee$.  Condition~\ref{compatible} for $Z$ follows from  Condition~\ref{compatible} for $Z'$ using functoriality of $\op{Sym}$.
\end{proof}

\subsection{On Conditions \ref{assume alpha chainlevel} and \ref{compatible}}\label{ss:sc}

In this subsection, our goal is to investigate when admissible resolutions exist on a stack 
$S\to \Mfrak^\orb_{g,r}(B\Gamma)_{\log}$. We explain how an argument 
of Polishchuk and Vaintrob \cite{PV} can be adapted to provide an explicit construction. 
This argument requires a certain cohomology vanishing statement and unfortunately we are able to ensure it holds only after imposing some restrictions on $S$. Therefore, in this subsection we will make the following:

\medskip

${\bf Assumption}\; (\star):$ $S$ is a separated Deligne--Mumford stack of finite type over $\Spec(\CC)$, 
which is a global quotient stack by a linear algebraic group action, and whose coarse moduli space is projective over an affine noetherian scheme.  

\begin{rem} Some of the requirements in Assumption $(\star)$ are relatively harmless for the applications we envision to GLSM theory. 
Typically $S$ will be an appropriate moduli stack of LG quasimaps and the finite type and Deligne--Mumford condition will be automatic once a stability condition is imposed. 
Also, these moduli stacks are all expected to be global quotients and there are well-known techniques to prove this is indeed the case in many situations. However, projectivity over an affine of the coarse moduli is not automatic.
\end{rem}

\begin{proposition} Let $S$ be a stack satisfying Assumption $(\star)$.
Let $U$ be a stack equipped with a morphism $U\xrightarrow{u} \Mfrak^\orb_{g,r}(B\Gamma)_{\log}$. If $u$ factors through a morphism $S\to \Mfrak^\orb_{g,r}(B\Gamma)_{\log}$ (for example, if $U$ is an open substack of $S$),
then there exists an admissible resolution $[A \to B]$ of $\R(\pi_U)_* \mathcal V$ on $U$.
\label{prop: satisfy all}
\end{proposition}

Two technical lemmas are required for the proof of Proposition~\ref{prop: satisfy all}.  The first lemma is a slightly more general version of
\cite[Lemma 4.2.4]{PV}, which is in fact what their argument establishes. The second lemma is a corollary of \cite[Lemma 4.2.5]{PV},
as we explain.

\begin{lemma} \label{lem: cohomology vanishing}
Let $S$ be a  stack satisfying Assumption $(\star)$.
 Let 
$\O(1)$ denote the pull-back to $S$ of a relatively ample line bundle on the coarse moduli $\underline {S}$. 
Then,
there exists a vector bundle  $\cE$ on $S$ such that for any coherent sheaf $\cF$ on $S$ the natural map
\[
\op{H}^0(S, \cE^\vee \otimes \cF(n)) \otimes \cE(-n) \to \cF
\]
is surjective for $n>>0$ and
\[
\op{H}^i(S, \cF(n)) =0
\]
for $i > 0, n >> 0$. 
\end{lemma}

\begin{lemma}
Let $S$ be a stack satisfying Assumption $(\star)$. 
Let $\cE$ be a vector bundle on $S$ satisfying the conclusion of Lemma \ref{lem: cohomology vanishing}.

Let $[C_0 \to C_1]$ be a 2-term complex of vector bundles on $S$.  Given the above setup, a positive integer $d \in \NN$, and any vector bundle $\cD$ on $S$, there exists $m_0 > 0$ such that for any $m_1 \geq m_0$ and any surjection
\[
\bar{C_1} := \cE^\vee(-m_1)^{\oplus N} \xrightarrow{\sigma} C_1\ra 0
\] 
one has
\begin{equation}
\op{H}^i(S, (\op{Sym}^{q_1}\bar C_0)^\vee \otimes \bigwedge\nolimits^{q_2} \bar C_1^\vee \otimes \cD) =0
\label{eq: ext vanishing}
\end{equation}
for  $d\ge q_1+q_2$ and $q_2\ge 1$, where the bundle $\bar{C}_0$ is the fiber product of $C_0$ and $\bar C_1$ over $C_1$, which completes the following vertical quasi-isomorphism,
\[
\begin{tikzcd}
0 \ar[d] \ar[r] & \bar C_0 \arrow[r] \arrow[d] & \bar C_1 \arrow[d] \ar[r] & 0 \ar[d]\\
0 \ar[r] & C_0\arrow[r] & C_1\ar[r] & 0 .
\end{tikzcd}
\]
\label{lem: kill ext}
\end{lemma}

\begin{proof} Given $[C_0 \to C_1]$ and $\cD$ as above, the statement of the lemma is the same as the statement of the lemma for the complex $[C_0\oplus \cD \ra C_1]$ where, instead of the original $\cD$, we now have $\cD=\cO$.  Hence, we may assume $\cD=\cO_S$.  The vanishing \eqref{eq: ext vanishing} now follows from Equation $(4.26)$ of \cite[Lemma 4.2.5]{PV}, in view of the fact that, since we work in characteristic zero, for any vector bundle $\cW$, both the symmetric power $\op{Sym}^q\cW$ and the exterior power $\bigwedge\nolimits^q\cW$ are direct summands
in $\cW^{\ot q}$.
\end{proof} 

\begin{proof}[Proof of Proposition~\ref{prop: satisfy all}.]
First notice that if $[A\to B]$ is an admissible resolution on $S$, then its pull-back is an admissible resolution on $U$. Hence we may assume $U=S$.

Now, start with a resolution $\pi_* \cA  \ra \pi_*\cB$ satisfying Condition \ref{assume ev chainlevel}, whose existence is guaranteed by Corollary \ref{surjectivity downstairs}.
We modify it so that it also satisfies 
Condition \ref{assume alpha chainlevel}.  Specifically, by Lemma~\ref{lem: kill ext}, we may choose a component-wise surjective quasi-isomorphism
\[
\begin{tikzcd}
A \arrow[r] \arrow[d, "e_A"] & B \arrow[d, "e_B"] \\
\pi_* \mathcal A \arrow[r] & \pi_*\cB
\end{tikzcd}
\]
such that 
\begin{equation}
\op{Ext}^q(\op{Sym}^i A \otimes \bigwedge\nolimits^j B, \O_S) = 0
\label{eq: extvanishing}
\end{equation}
 for $i+ j \le d_0$, $j\ge 1$ and $q \geq 1$. Here $d_0$ is the polynomial degree of $w$ 
so that $w \in \Sym^{\bullet \le d_0} V^\vee$.  
Now, we may define
\[
Z := \sum_{i=1}^r \pii \Sigma_i^* \varkappa_w \circ \op{natural}\circ \op{Sym} \op{ev}_A^i 
\]

Then, we may take $E_d$ to be the complex
\[
\begin{tikzcd}
E_d: =  \bigg [ 0 \ar[r] & \op{Sym}^d A  \arrow[r, "{(\delta, -Z)}"] & \op{Sym}^{d-1} A  \otimes B \oplus \O_S^r  \arrow[r, "{(\delta, 0)}"] &  \op{Sym}^{d-2} A  \otimes \bigwedge\nolimits^2 B \arrow[r, "\delta"] & \dotsc \bigg ]\\
\end{tikzcd}
\]

There is a spectral sequence
\[
E_1^{p,q} = \bigoplus_{p} \op{Ext}^q(E_d^p, \cO_S) \Rightarrow \op{Hom}_{\dbcoh{S}}(E_d^p, \cO_S[p+q]).
\]
 where $E_d^p$ denotes the degree $p$ component of $E_d$. 
 It follows from \eqref{eq: extvanishing} that for $p+q = -1$  and every $0 \leq  d  \leq d_0$,  the morphism \eqref{eq: finalalpha} can be realized at the cochain level
(see e.g. the proof of  Lemma 5.7 on page 89 of \cite{GKZ}).
That is, for $d \leq d_0$ there exists $\alpha^\vee_d$  realizing \eqref{eq: finalalpha} as follows
\[
\begin{tikzcd}
\op{Sym}^d A  \arrow[r, "{(\delta, -Z)}"] \arrow[d] & \op{Sym}^{d-1} A  \otimes B \oplus \O_S^r  \arrow[r, "{(\delta, 0)}"] \arrow[d, "{(\alpha^\vee_d,\op{sum})}"] &  \op{Sym}^{d-2} A  \otimes \bigwedge\nolimits^2 B \arrow[r, "\delta"] \arrow[d] & \dotsc \\
0 \arrow[r] & \O_S \arrow[r] & 0 \arrow[r] & \dotsc \\
\end{tikzcd}
\]
  Then we can define 
\[
\alpha^\vee := \sum_{d=0}^{d_0} \alpha^\vee_d.
\]
Conditions \ref{assume alpha chainlevel} and \ref{compatible} are manifestly satisfied by this construction. 
\end{proof}

Proposition \ref{prop: satisfy all} applies in particular to the case of the hybrid models of \S\ref{ss: hybrid models}. However, for the analysis of the geometric phase in \S\ref{sec: compare GW}, we will need a refinement in which the admissible resolution satisfies an additional property. This is spelled out explicitly in the following statement.

\begin{proposition}\label{prop: satisfy all geometric} 
Assume we have a hybrid model, so that $V=V_1\oplus V_2$.
Assume further that  the superpotential is linear on $V_2$, i.e., that $w \in V_2^\vee \ot \Sym^{\ge 1} (V_1^{\vee})$. Suppose
$U\xrightarrow{u} \Mfrak^\orb_{g,r}(B\Gamma)_{\log}$ factors through a stack $S$ over $\Mfrak^\orb_{g,r}(B\Gamma)_{\log}$ 
satisfying Assumption $(\star)$.
Then we may choose an admissible resolution $[A \to B]$ of $\RR\pi_*\cV$ on $U$ with
$A = A_1 \oplus A_2' \oplus  \pi_* (\cV_2|_{\sG})$, and such that $\op{ev}^i_A$ followed by the projection onto $ \pi_*  (\cV_2|_{\sG})$ agrees with the projection of $A$ onto this summand.
\end{proposition}

\begin{proof}
We use the direct sum decomposition $\cV = \cV_1 \oplus \cV_2$ to resolve the summands individually.   

First we find a resolution $[\pi_*\cA _1\to \pi_*\cB _1]$ of $\RR \pi_* \cV_1$ as in Corollary \ref{surjectivity downstairs}.

Next, choose any $\pi$-acyclic locally free resolution of   $\cV_2$ on $\cC$,
\[
0 \to \cV_2 \to \cA'_2 \to \cB'_2 \to 0.
\]
and define $\cB_2$ to be the cokernel of the map $\cV_2 \to \cA_2' \oplus \cV_2|_{\sG}$.  The snake lemma gives the diagram
\begin{equation}
\begin{tikzcd}
  &  0 \arrow[d] & 0  \arrow[d] & 0  \arrow[d]&  \\
0 \arrow[r] & \mathcal V_2(-\sG) \arrow[r] \arrow[d] & \mathcal A'_2 \arrow[r] \arrow[d] & \mathcal B _2 \arrow[r]  \arrow[d]& 0 \\
0 \arrow[r] & \mathcal V_2 \arrow[r] \arrow[d] & \mathcal A' _2\oplus  \mathcal V_2|_{\sG} \arrow[r]  \arrow[d] & \mathcal B_2\arrow[r] \ar[d] & 0 \\
0 \ar[r] & \mathcal V_2|_{\sG} \arrow[d] \arrow[r] &   \mathcal V_2|_{\sG} \arrow[d] \ar[r] & 0  \\
& 0 & 0 &  &
\end{tikzcd}
\label{snake2}
\end{equation}
Pushing forward via $\pi$, we get a resolution
$\mathbb R \pi_* \mathcal V_2 \cong [(\pi_* \mathcal A'_2 \oplus \pi_* \cV_2|_{\sG}) \to \pi_*\cB_2].$
Set 
$$A_2 := \pi _* \cA'_2 \oplus  \pi_* \cV_2|_{\sG}, \;\;\;\; B_2:= \pi_*\cB_2.$$
Then 
$$[\pi_* \cA_1\oplus A_2\to \pi _* \cB_1\oplus B_2]\cong \RR\pi_*\cV$$ manifestly satisfies Condition \ref{assume ev chainlevel}.

The morphism
$\op{Sym}[\pi_* \cA_1\oplus A_2\to \pi _* \cB_1\oplus B_2] \to [\O_S^r\to \cO_S]$
factors through the projection
\[
\op{Sym}[\pi_* \cA_1\oplus A_2\to \pi _* \cB_1\oplus B_2]  \to [A_2\ra B_2] \otimes \op{Sym}[\pi_* \cA_1 \to \pi_* \cB_1].
\]
Now we use Lemma~\ref{lem: kill ext} to modify only the first summands, by choosing a quasi-isomorphism
\[
\begin{tikzcd}
A_1 \arrow[r] \arrow[d, "e_A"] & B_1 \arrow[d, "e_B"] \\
\pi_* \mathcal A_1 \arrow[r] & \pi_*\cB_1
\end{tikzcd}
\]
such that $\op{Ext}^q(\op{Sym}^i A_1 \otimes \bigwedge\nolimits^j B_1 \otimes A_2, \O_S) = 0$,  $\op{Ext}^q(\op{Sym}^i A_1 \otimes \bigwedge\nolimits^j B_1 \otimes B_2, \O_S) = 0$
 for $i+j \le d_0 -1$, $j\ge 1$, and $q \geq 1$. 
The resolution $[A_1\oplus A_2 \to B_1 \oplus B_2]$ is admissible and satisfies the required
property, with $A':=A_1\oplus \pi_*\cA'_2$.
\end{proof}

\subsection{Support of the PV factorization}\label{s:supofthepvfac}

\subsubsection{Base change and $\tot\pi_*(\cV)$.}
Denote by
\[
\tot (\pi _*\cV) := \op{Spec } \op{Sym} \mathbb R^1 \pi_* ( \cV^\vee \otimes \nolog),
\]
so that, for any $g \colon T\ra S$, we have
$\Hom _S (T,  \tot (\pi _*\cV)) = \Gamma (\mathfrak{C} _T, \cV _T)$ by
the canonical Serre duality identification $ \mathbb R^1 \pi_* ( \cV^\vee \otimes \nolog) ^\vee \cong \pi _*\cV$.
Here $\cV _T := h^*\cV$, from the pullback diagram
\begin{equation}
\begin{tikzcd}
\mathfrak C_T \ar[r, "h"] \arrow[d, "\pi_T"] & \mathfrak C \ar[d, "\pi"] \\
T \ar[r, "g"] & S.
\end{tikzcd}
\label{C over T}
\end{equation}
Note that since $\pi$ is flat of relative dimension one, Tor-independent base change gives an isomorphism 
$\mathbb{R}^1(\pi _T)_* h^* \cE  \cong g^* \mathbb{R}^1\pi _* \cE$ for any vector bundle on $\mathfrak{C}$.
Hence $\tot ((\pi _T)_* h^* \cV ) \cong \tot (\pi _*\cV ) \ti _S T$.

\subsubsection{The map $dw_T$}
Considering $w$ as a $\Gamma$-invariant element of $\Sym V^\vee \ot \CC _{\chi}$, we obtain
its differential $dw \in \Sym V^\vee \ot V^\vee \ot \CC _{\chi}$ which is $\Gamma$-invariant.
Recall that there is a natural pairing
\begin{align*}
\op{Sym}V \times  \op{Sym}V^\vee & \to k \\
(v_1 \otimes ... \otimes v_n, \phi_1 \otimes ... \otimes \phi_n) & \mapsto \sum_{\sigma \in S_n} \prod_{i=1}^n \phi_i(v_{\sigma(i)}).
\end{align*}
This leads to a natural isomorphism in characteristic 0 which we describe on a basis
\begin{align}
(\op{Sym}V)^\vee & \to \op{Sym}V^\vee \notag \\
(e_1 \otimes... \otimes e_n)^* & \mapsto n! e_1^* \otimes... \otimes e_n^*.
\label{eq: dual identification convention}
\end{align}
Under this convention, pairing with $dw$ gives a map of $\Gamma$-representations 
\[
\overline{dw}\colon \op{Sym}V\to V^\vee \otimes \C_{\chi}.
\] 
For any $T\xrightarrow{g} S$, with $T$ a scheme,
the map $\overline{dw}$ induces a  homomorphism $\cO_S$-modules
$$\overline{dw}_T: \pi _* (\op{Sym} \cV_T) \ra \pi _* (\cV_T^\vee  \ot \log),$$
via pull-back by $[g^*P]:\cC_T\ra B\Gamma$, followed by push-forward by $\pi=\pi_T$.
The composition
$$\op{Sym} \pi _*( \cV_T  ) \xrightarrow{\op{natural}} \pi _* (\op{Sym} \cV_T) \xrightarrow{\overline{dw}_T} \pi _* (\cV_T^\vee  \ot \log) \cong (\RR ^1\pi_* \cV_T  (-\sG )) ^\vee$$
(where the last isomorphism comes from Serre duality)
induces the homomorphism
\begin{equation}\label{dwS} 
dw_T : \op{Sym} \pi _*( \cV _T)  \ot  \RR ^1\pi_* (\cV _T (-\sG ))  \to \cO_T.
\end{equation}
This can be alternatively described as the map
\[
\op{Sym} \pi _*( \cV _T)  \ot  \RR ^1\pi_* (\cV _T (-\sG ))  \to \RR ^1 \pi_* (\nolog) \cong \cO_T.
\]
induced by $\overline{dw}$ as in Section 3.2 of \cite{CL}.

\subsubsection{The degeneracy locus $Z(dw_{\tot (\pi_*\cV)}(\exp \tau_{\pi_*\cV}))$}\label{Z dw}

Consider the map \eqref{dwS} associated to the projection $p': \tot (\pi_*\cV ) \ra S$. Consider also the tautological
global section $\tau_{\pi_*\cV}$ of $\pi _*(\cV_{\tot (\pi_*\cV )})$. Together, they determine a
cosection of the sheaf $ \RR ^1\pi_* \cV _{\tot (\pi_*\cV )} (-\sG ) $,
\begin{equation}\label{cosection dw}
dw_{\tot (\pi_*\cV) }(\exp \tau_{\pi_*\cV}) :  \RR ^1\pi_* \cV _{\tot (\pi_*\cV )} (-\sG )  \to \cO_{\tot (\pi_*\cV )}. 
\end{equation}

Let $m$ be the degree of the largest monomial in the polynomial $w$.  A priori, the exponential in the formula is an element of the completed symmetric algebra,
$\exp \tau_{\pi_*\cV}\in\widehat{\op{Sym}} H^0(\tot (\pi_*\cV ) , \pi _*( \cV_{\tot (\pi_*\cV)}) )$, but we view it as an element of 
$\op{Sym} (\pi _*( \cV_{\tot (\pi_*\cV)})$ by first considering its truncation modulo $(\tau_{\pi_*\cV})^{m+1}$ and then taking the
image under the natural map
$$\op{Sym}^{\leq m}H^0(\tot (\pi_*\cV ) , \pi _*( \cV_{\tot (\pi_*\cV)}) )\ot \cO_{\tot (\pi_*\cV )}\longrightarrow 
\op{Sym}^{\leq m} (\pi _*( \cV_{\tot (\pi_*\cV)})).$$
In this way, we view $\exp \tau_{\pi_*\cV}$ as the tautological section of a truncation of $\op{Sym} (\pi _*( \cV_{\tot (\pi_*\cV)})$, which is consistent with earlier conventions (see Remark~\ref{rem: tautological section}).

We denote by
$$Z(dw_{\tot (\pi_*\cV)}(\exp \tau_{\pi_*\cV})):=\Spec(\op{coker}  dw_{\tot (\pi_*\cV) }(\exp \tau_{\pi_*\cV})) \subseteq \tot (\pi_*\cV)$$
the degeneracy locus of the cosection \eqref{cosection dw}.

\begin{proposition}\label{PV support}
	Assume that there is an admissible resolution $[A \to B]$ of $\mathbb R \pi_* \cV$ on $S$.
	Then $Z(dw_{\tot (\pi_*\cV)}(\exp \tau_{\pi_*\cV}))$ is canonically identified with a closed substack of $\op{tot} A$ and the PV factorization $\{-\alpha, \beta\}$
	is supported on $Z(dw_{\tot (\pi_*\cV)}(\exp \tau_{\pi_*\cV}))$.
\end{proposition}

The rest of this subsection will be occupied by the proof of the above Proposition.
Since by Proposition~\ref{p:Ksupp} a Koszul factorization is always supported on the common zero locus $Z(\alpha,\beta)$, we begin
with describing $Z(\beta)$ and $Z(\alpha)$ for the PV factorization.

\subsubsection{The substack $Z(\beta)$}\label{Zbeta}

Let $[A\xrightarrow{d} B]$ be a two-term vector bundle complex with
an isomorphism $\varphi$ between $[A\ra B]$ and $\RR\pi _*\cV$ in $\op{D}(S)$.
Let $p: \tot A \ra S$ be the projection and let $\beta : \cO_{\tot A}\ra p^*B$ be the section induced by $d$, as in \eqref{explicit beta}.
Denote by $Z(\beta) \subseteq \tot A$ its zero locus. The following lemma establishes in particular the first assertion of Proposition \ref{PV support}.

\begin{lemma}\label{explicit  Zbeta}
	There is an isomorphism $Z(\beta ) \cong \tot \pi_*\cV$, induced by $\varphi$.
\end{lemma}

\begin{proof}
	First we note that $Z(\beta)$ is the closed substack of $\tot A$
	parameterizing objects $f: T \ra \tot A$ for which $f^*\kb:\cO_T\ra g^*B$ is the zero map, where $g=p\circ f :T\ra S$.
	
	On the other hand, if we start with $g: T\ra S$, to give a lift $f: T\ra \tot A$ amounts to giving a global section 
	$\sigma _f \in H^0(T, g^*A)$, with $\sigma _f =f^*\tau_A$.

	We have $f^*\kb = (g^*d) \circ \sigma _f$, so it vanishes identically if and only if $\sigma_f$ factors through the kernel of
	$g^*A\xrightarrow{g^*d} g^*B$. The composition 
	$$\cO_T\xrightarrow{\sigma_f} \ker(g^*d) \stackrel{\stackrel{h^0(\varphi)}{\sim}}{\longrightarrow}(\pi_T)_*(\cV_T)$$ 
	gives the required object of $\tot \pi_*\cV$.
\end{proof}

\subsubsection{The substack $Z(\ka )$}

We keep the set-up from \S\ref{Zbeta}, and assume further that the resolution $[A\xrightarrow{d} B]\stackrel{\varphi}{\cong}\RR\pi_*\cV$ is admissible,
so that we have the cosection $\alpha^\vee : p^*B \ra \cO _{\tot A}$.

Its degeneracy locus $Z(\ka )$ (which is the same as the zero locus of the dual section of $p^*B^\vee$) parametrizes 
the closed substack of $\tot A$ whose
objects are maps $f: T\ra \tot A$, with $T$ a scheme, for which $f^*\alpha^\vee = 0$. Note that 
$f^* \alpha^\vee$ is the $\cO_T$-module homomorphism  $B_T \ra \cO_T$ given by
$$f^*(\alpha^\vee ) (b) = \alpha^\vee_T ((\exp \sigma _f )\ot b)$$ for every
local section $b$ of the vector bundle $B_T:= g^*B$ on $T$.
Here $\alpha^\vee _T : = g^*(\alpha^\vee)$, and we interpret $\exp \sigma _f$ as the local restriction of a suitable truncation (depending only on $w$) of an element in 
$\widehat{\op{Sym}}H^0(T, A_T)$, as explained in \S\ref{Z dw}.

As $[A\xrightarrow{d} B]$ is admissible, it satisfies Condition 1 of Definition \ref{admissible}. Define the vector bundle $A'$ on $S$ as the kernel in the exact sequence 
$$0\ra A'\ra A\ra \pi_* (\cV |_\sG)\ra 0.$$
The quasi-isomorphism $\varphi$ induces a quasi-isomorphism
$$[A'\xrightarrow{d'} B] \stackrel{\varphi '}{\cong} \R\pi _* \cV (-\sG),$$ 
where $d'=d|_{A'}$. Hence there is a surjective ``connecting homomorphism"
$$\partial: B \twoheadrightarrow \op{coker} d' \stackrel{\stackrel{h^1(\varphi')}{\sim}}{\longrightarrow} \RR ^1\pi _* \cV (-\sG).$$
Moreover, for any $T\xrightarrow{g} S$, we have a corresponding
$\partial :B_T \twoheadrightarrow \RR ^1\pi _* \cV_T (-\sG)$ by the base-change property of $\RR ^1\pi _* \cV (-\sG)$.

\begin{lemma}\label{alpha and dw}
	Given $T\xrightarrow{g} S$, the composition
	$$ \op{Sym}(\pi_*\cV_T)\ot B_T \xrightarrow{\op{id}\ot\partial} \op{Sym}(\pi_*\cV_T)\ot \RR ^1\pi _* \cV_T (-\sG)\xrightarrow{dw_T} \cO_T$$
	coincides with the restriction of $\alpha^\vee_T$ to $\op{Sym}(\pi_*\cV_T)\ot B_T$.
\end{lemma}

Granting the lemma, we can easily finish the proof of Proposition \ref{PV support}. Indeed, 
when $f$ factors as
$$T\ra \tot \pi_*\cV =Z(\beta) \subseteq \tot A,$$
we have $\sigma_f=f^*(\tau_{\pi_*\cV})$. 
By Lemma \ref{alpha and dw},  
$$\alpha^\vee_T ((\exp \sigma _f )\ot b) = 0 \Longleftrightarrow  dw_T ((\exp \sigma _f ) \ot \partial b ) =0$$ for 
every local section $b$ of the vector bundle $B_T$ on $T$. 
By the surjectivity of $\partial$, we conclude that 
\begin{equation} \label{eq: support equals LG}
Z(\alpha)\cap Z( \kb) = Z(dw_{\tot (\pi_*\cV )}(\exp \tau_{\pi_*\cV})),
\end{equation}
 as required.

\subsubsection{Proof of Lemma \ref{alpha and dw}}\label{a commuting diagram for support}

Let $T$ be a scheme over $S$. Since we will work exclusively over $T$, 
we drop the subscript $T$ from the curve $\cC_T$, the bundles $\mathcal V_T$, $B_T$, etc.

Consider the morphism \eqref{intermediate} and recall that $E$ is the shifted cone in $\dbcoh{T}$
\[E:=\op{Cone}(\op{Sym}\mathbb R \pi_* \cV \to \O_T^{r})[-1]
\]
and that \eqref{eq: finalalpha} provides a morphism 
$E \ra \cO _T [-1]$ in $\dbcoh{T}$.
Denote by $H^1(E\ra \cO _T)$ the induced morphism $H^1(E) \ra \cO _T$ on cohomology sheaves.

Given the resolution $[A\ra B]$ of Lemma \ref{alpha and dw}, there is a natural map 
\begin{equation}\label{map psi}
\psi : \op{Sym}(\pi_* \cV) \otimes B \ra H^1(E)
\end{equation}
defined as follows.
Represent $E$ as the complex 
$$
[\op{Sym} A \xrightarrow{d_1^E} (\op{Sym}A \ot B) \oplus \cO_S^r \xrightarrow{d_2^E} \dots ]
$$
and note that the inclusion 
$$
\op{Sym}(\pi_* \cV) \otimes B \subseteq \op{Sym}A \ot B \subseteq (\op{Sym}A \ot B) \oplus \cO_S^r 
$$
factors through $\ker d_2^E$. 
Composing with the quotient map $\ker d_2^E\ra H^1(E)$ gives the sheaf homomorphism $\psi$, independent of the choice of representative for $E$.

One easily checks that the composition $H^1(E\ra \cO_T)\circ \psi$ is precisely the restriction of 
$\alpha^\vee$ to $\op{Sym}(\pi_* \cV) \otimes B$. We are therefore reduced to proving that the diagram
\begin{equation}\label{commutes!}
\begin{tikzcd}
\op{Sym}(\pi_* \cV) \otimes B \ar[d, swap, "\op{Id} \ot \partial"]  \ar[r, "\psi"] & H^1(E)  \ar[d, "H^1(E\ra \cO_T)"] \\
\op{Sym} (\pi_* \cV ) \otimes  \mathbb R^1 \pi_* \cV (-\sG) \ar[r, swap, "dw_T"] & \O_T \\
\end{tikzcd} 
\end{equation}
is commutative.

We prove this, by first arguing that this only needs to be done for a particular resolution.  We do so below and provide the setup for a particular resolution which we show in Lemma~\ref{commutes makes proper} satisfies \eqref{commutes!}.  This is enough to deduce the statement of Proposition \ref{PV support}.  

Indeed, the maps $\partial$ and $\psi$, can be defined in the same way for {\it any} resolution which satisfies Condition 1 (but which is not necessarily admissible).
If there exists one particular resolution $[A_0\ra B_0]$ of $ \RR^1\pi_* \cV$ which satisfies Condition 1 and for which 
the diagram \eqref{commutes!} commutes, then the corresponding diagram commutes for any 
other resolution $[A\ra B]$ satisfying Condition 1. Indeed, by Lemma~\ref{l:surjective} (proven below)
there exists a resolution $[\widetilde{A}\ra \widetilde{B} ]$, together with a {\it surjective} quasi-isomorphism 
$[\widetilde{A} \ra \widetilde{B} ]\to[A \ra B]$, and a {\it surjective} quasi-isomorphism 
$[\widetilde{A}\ra \widetilde{B}] \ra [A_0 \ra B_0]$.
This induces the following commuting diagram
\[ 
\begin{tikzcd}  &  \widetilde{B}  \ar[d, twoheadrightarrow, "\partial"] \ar[dl, twoheadrightarrow] \ar[dr, twoheadrightarrow] & \\
B_0 \ar[r, twoheadrightarrow, "\partial"]& \RR^1\pi_* \cV (-\sG) & B \ar[l, twoheadrightarrow, "\partial"]\\
\end{tikzcd}
\]
From these, the claim follows immediately.  It remains to construct one resolution for which \eqref{commutes!} commutes.

Let $0 \to \cV \to [\mathcal A \to \mathcal B] \to 0$ be a resolution of $\mathcal V$ by vector bundles such that $\mathbb R^1 \pi_*\mathcal A = \mathbb R^1 \pi_* \mathcal B  = 0$.  Assume further that the map $\cV \to \cV|_{\sG}$ extends to a map $\cA \to \cV|_{\sG}$, 
with $\pi$-acyclic kernel. By Lemma \ref{ev at level C} such a resolution always exists, and $[\pi_*\cA\ra\pi_*\cB]$ is a resolution of 
$\RR\pi_*\cV$ which satisfies Condition 1.

Let $\mathcal A'$ be the kernel of the map $\cA \to \mathcal V|_{\sG}$.  Then, the snake lemma gives the following diagram:
\begin{equation}
\begin{tikzcd}
&  0 \arrow[d] & 0  \arrow[d] & 0  \arrow[d]&  \\
0 \arrow[r] & \mathcal V(-\sG) \arrow[r] \arrow[d] & \mathcal A' \arrow[r] \arrow[d] & \mathcal B \arrow[r]  \arrow[d]& 0 \\
0 \arrow[r] & \mathcal V \arrow[r] \arrow[d] & \mathcal A \arrow[r]  \arrow[d] & \mathcal B \arrow[r] \ar[d] & 0 \\
0 \ar[r] & \mathcal V|_{\sG} \arrow[d] \arrow[r] &   \mathcal V|_{\sG} \arrow[d] \ar[r] &  0 \\
& 0 & 0 &  &
\end{tikzcd}
\label{snake}
\end{equation}

\noindent
Hence the short exact sequence
\[
0 \ra \mathcal{V}  (-\sG) \ra\mathcal A' \ra \mathcal{B} \ra 0
\]
induces the connecting homomorphism,
\[
\partial  \colon   \pi_* \cB \to  \mathbb R^1 \pi_*\cV (-\sG ).
\]
Note that $\partial$ is surjective by the vanishing $\mathbb{R}^1\pi_*\cA'=0$.

\begin{lemma}\label{commutes makes proper}
	In the situation above, the following diagram is commutative:
	\begin{equation}
	\begin{tikzcd}
	\op{Sym}(\pi_* \cV) \otimes \pi_*\cB \ar[d, swap, "{\op{Id} \ot \partial }"]  \ar[r, "\psi"] & H^1(E)  \ar[d, "H^1(E\ra \cO_T)"] \\
	\op{Sym} (\pi_* \cV ) \otimes  \mathbb R^1 \pi_* \cV (-\sG) \ar[r, swap, "dw_T"] & \O_T \\
	\end{tikzcd} 
	\label{eq: commutes makes proper}
	\end{equation} 
\end{lemma}

\begin{proof}
First, notice that by Euler's Homogeneous Function Theorem\footnote{For a homogeneous function $f$ of degree $t+1$, the equality $(t+1) f = \sum x_i \partial_i f$ holds.  To apply this here, we use the fact that $w$ has no constant term (see Definition~\ref{dfn: hybrid model}).  In addition, we use the convention described in Equation~\eqref{eq: dual identification convention} which ensures the proper scaling for the above commutativity.}, the following diagram commutes,
	\begin{equation}
	\begin{tikzcd}
	\op{Sym} V \otimes V \ar[r, "dw"] \ar[d, swap, "\op{m}"] & \C_\chi \\
	 \op{Sym} V \ar[ur, "w"] &
	\end{tikzcd}
	\label{eq: eulerthm}
	\end{equation}
	where 
	\[
	\op{m}( (a_1 \otimes ... \otimes a_t) \otimes v) :=   a_1 \otimes ... \otimes a_t \otimes v .
	\]

Recall that, working over $T$, we define $\op{Sym}^d \mathfrak C := [\overbrace{\mathfrak C \times_T ... \times_T \mathfrak C}^{d-times} / S_d ]$.
Define a map
\[
m_d: \op{Sym}^d \mathfrak C \times_T  \mathfrak C \to \op{Sym}^{d+1} \mathfrak C.
\]
In what follows, consider $\cV^{\boxtimes d+1}$ on $ \op{Sym}^d \mathfrak C \times_T  \mathfrak C$  with the $S_d$-equivariant structure given by permuting the first $d$-terms and as a $S_{d+1}$-equivariant sheaf on $\op{Sym}^{d+1} \mathfrak C$ by permuting all terms.  Similarly, the sheaf $\cV^{\otimes d}$ may be consider as a sheaf on $\mathfrak C / S_d$ where the $S_d$-equivariant structure comes from permuting all terms.  From this we get a commutative diagram of sheaves on $ \op{Sym}^{d+1} \mathfrak C$:
\[
\begin{tikzcd}
(m_d)_* \cV^{\boxtimes d} \boxtimes \cV(-\sG) \ar[r] \ar[d]  &  (m_d)_*((\Delta_d)_*\cV^{\otimes d}  \boxtimes \cV(-\sG) ) \ar[d] & \\
(m_d)_* \cV^{\boxtimes d+1} \ar[r] \ar[d]  &  (m_d)_*((\Delta_d)_*\cV^{\otimes d}  \boxtimes \cV) \ar[d] & \\
\cV^{\boxtimes d+1} \ar[r]  & (\Delta_{d+1})_* \cV^{\otimes d+1} \ar[d] & \\
& (\Delta_{d+1})_* \log \ar[r] & (\sigma_{d+1})_* \log \\
\end{tikzcd}
\]

Notice that the maps from the top two terms in the diagram to the last term are zero.  Hence, we may replace the rest of the diagram by kernels of the resulting 5 maps to $ (\sigma_{d+1})_* \log$ to get a diagram
\[
\begin{tikzcd}
(m_d)_* \cV^{\boxtimes d} \boxtimes \cV(-\sG) \ar[r] \ar[d]  &  (m_d)_*((\Delta_d)_*\cV^{\otimes d}  \boxtimes \cV(-\sG) ) \ar[d] & \\
 \op{ker}_1 \ar[r] \ar[d]  &   \op{ker}_2 \ar[d]  \\
 \op{ker}_3 \ar[r] \ar[dr, "f_d"] &  \op{ker}_4 \ar[d]  \\
& (\Delta_{d+1})_* \nolog \\
\end{tikzcd}
\]

Notice that after applying $(\pi_d)_*$ the right hand side of the diagram corresponds to a sequence of functors applied to $w \circ m = dw$ from \eqref{eq: eulerthm}.  Hence the map from top to bottom pushes forward to the $d^{\op{th}}$ component of $dw$.

Applying $\R(\pi_d)_*$, therefore leads to  the following commutative diagram:
\[
\begin{tikzcd}
\op{Sym} \pi_*\cV \otimes B[-1] \ar[d] & \\
\op{Sym} \pi_*\cV \otimes \RR\pi_* \cV(-\sG)\ar[r] \ar[d]  &  \pi_* \op{Sym} \cV \otimes \RR\pi_* \cV(-\sG) \ar[d] & \\
C(\op{Sym} \RR\pi_* \cV \otimes \RR\pi_* \cV) \ar[r] \ar[d]  &   C(\RR\pi_*( \op{Sym} \cV \otimes \cV)) \ar[d]  \\
C(\op{Sym} \RR\pi_* \cV) \ar[r]  \ar[dr, "{\sum_{d=0}^{d_0} \R(\pi_d)_* f_d}"] \ar[ddr, swap, "{(\alpha^\vee, \op{sum})}"] & C(\RR\pi_*( \op{Sym} \cV) ) \ar[d]  \\
& \RR\pi_* \nolog  \ar[d, "{\op{H}^1}"] & \\
   & \O_T[-1] & \\
\end{tikzcd}
\]
where $d_0$ is the top degree of $w$ and for each complex $D$, we denote by $C(D)$ the mapping cone of the map $D \to \O_T^r$.
	Applying $\op{H}^1$ to the above, the left path becomes the top of the diagram in the statement of the lemma.  	The right path becomes the bottom.
\end{proof}

\begin{lemma}\label{l:surjective} Let $E_\bullet$ and $F_\bullet$ be quasi-isomorphic 2-term complexes of vector bundles over a smooth Deligne--Mumford stack with quasi-projective coarse moduli space.
There exists a roof diagram
\begin{equation}\label{e:roofdiagram}
\begin{tikzcd}
& \bigg [ \bar G_{1}  
 \ar[ld, twoheadrightarrow, swap] \ar[r, "d_{\bar G}"] \ar[rd, twoheadrightarrow] & \bar G_{0} \bigg ] \ar[rd, twoheadrightarrow] \ar[ld, twoheadrightarrow, crossing over, near end, swap] 
& \\
\bigg [ E_{1} \ar[r, "d_E"] & E_{0} \bigg ]  & \bigg [ F_{1} \ar[r, "d_F"]  
  & F_{0} \bigg ]
\end{tikzcd}
\end{equation}
realizing the quasi-isomorphism at the cochain level, where $\bar G_{1}$ and $\bar G_{0}$ are vector bundles and all diagonal arrows are surjective.
\end{lemma}

\begin{proof}
First, there exists a roof diagram \[ E_\bullet \xleftarrow{}  H_\bullet \xrightarrow{} F_\bullet,\]
such that both maps of complexes of coherent sheaves are quasi-isomorphisms.  Replacing $H_\bullet$ with the canonical 2-term truncation \[
K_\bullet = \bigg [ K_1 \to K_0 \bigg ] 
\] yields a quasi-isomorphic complex which still maps to $E_\bullet$ and $F_\bullet$.  There exists a surjection onto $K_{0}$ by a locally free coherent sheaf $G_{0}$.  

Now, the mapping cone of the morphism of complexes
$[0 \to G_{0}] \to [K_{1} \to K_{0}]$ is the complex
\[
[G_{0} \oplus K_{1} \to K_{0}].
\]
Since the map from $G_{0} \oplus K_{1}$ to $K_{0}$ is surjective, this complex has a single cohomology group which we denote by $G_{1}$.  This yields a complex
\[
G_\bullet := [G_{1} \to G_{0}]
\]
and a quasi-isomorphism
\[
G_\bullet \to K_\bullet.
\]

Since the map $K_\bullet \to E_\bullet$ is also a quasi-isomorphism,
the mapping cone of the morphism $[0 \to G_{0}] \to [E_{1} \to E_{0}]$ also has a single cohomology group, namely $G_{1}$.  Hence, the vector bundle $G_{1}$ can also be realized as the kernel of the surjective morphism of vector bundles
\[
G_{0} \oplus E_{1} \to E_{0}.
\]
It follows that $G_{1}$ is also a vector bundle.

This realizes the isomorphism $E_\bullet$ with $F_\bullet$ in the derived category, by a roof diagram 
\begin{equation}\label{e:roofdiagram2} E_\bullet \xleftarrow{ e_\bullet}  G_\bullet \xrightarrow{ f_\bullet} F_\bullet
\end{equation} where $G_\bullet$ is a two-term complex of vector bundles.  It is left to show that we can modify this diagram so that the morphisms of cochain complexes consist of surjective maps from $G_i$.

Let 
\[\bar G_\bullet := G_{1} \oplus E_{1} \oplus F_{1} \xrightarrow{d_{\bar G} = (d_G, \op{id}_{E_{1}}, \op{id}_{F_{1}})} G_{0} \oplus E_{1} \oplus F_{1}.\] Then we can replace \eqref{e:roofdiagram2} with 
\[ E_\bullet \xleftarrow{\bar e_\bullet} \bar G_\bullet \xrightarrow{\bar f_\bullet} F_\bullet,\]
where $\bar e_i = e_i \circ \pi_{0}$ is the composition of $e_i$ with the projection onto $G_i$, and $\bar f_i$  is defined similarly.  Finally, if we define 
\[\hat e_{1} = (e_{1} + \op{id}_{E_{1}}) \circ \pi_{12} \text{ and } \hat e_{0} = (e_{0} + d_E)\circ \pi_{12},\]
where $\pi_{12}$ is projection onto $G_i \oplus E_{1}$, the map $\hat e_\bullet:\bar G_\bullet \to E_\bullet$ is homotopic to $\bar e_\bullet$.  
It is clear that $\hat e_{1}$ is surjective.  Since $e_\bullet$ maps the cokernel of $G_\bullet$ onto the cokernel of $E_\bullet$, every element of $E_{0}$ must be contained in $\op{im}(d_E) + \op{im}(e_{0})$, so $\hat e_{0}$ is surjective as well.

  The analogous construction of $\hat f_\bullet$ gives the desired diagram \eqref{e:roofdiagram}.
\end{proof}

With the above lemmas proven, we have concluded the proof of Proposition \ref{PV support}.

\subsection{$R$-charge equivariance}\label{s:rigev}
Recall that $V=\oplus_{\eta\in \widehat{\CC^*_R}}V_\eta$ comes with a {\it lower} grading, induced by the $R$-charge action. 
Since the actions of $\CC^*_R$ and $G$ commute by assumption, each $V_\eta$ is a $G$-invariant subspace. Setting 
$\cV_\eta:=P\times_\Gamma V_\eta$, it follows that the vector bundle $\cV$ on $\cC$ has the induced 
grading $\cV =\oplus_{\eta\in \widehat{\CC^*_R}}\cV_\eta$, and similarly
$\RR\pi_*\cV =\oplus_{\eta\in \widehat{\CC^*_R}}\RR\pi_*\cV_\eta$. 

We refine the discussion of admissible resolutions by imposing that the lower grading is respected.
In other words, we work with the derived categories 
$\dbcoh {[\cC/\CC^*_R]}$ and $\dbcoh {[S/\CC^*_R]}$, where $\cC$ and $S$ are given the trivial $\CC^*_R$-actions.
This means that our admissible resolutions $[A\xrightarrow {d}B]\cong \RR\pi_*\cV$ 
will have terms $A=\oplus_{\eta\in \widehat{\CC^*_R}}A_\eta$ and $B=\oplus_{\eta\in \widehat{\CC^*_R}}B_\eta$, the differential $d$
will be $\CC^*_R$-equivariant (i.e., of degree zero with respect to the lower grading), and each
$[A_\eta\to B_\eta]$ will be a resolution of $\RR\pi_*\cV_\eta$. Conditions \ref{assume ev chainlevel}, \ref{assume alpha chainlevel}, and \ref{compatible} will now be required to hold in $\dbcoh {[S/\CC^*_R]}$.

Hence, Condition \ref{assume ev chainlevel} will require in addition that $\op{ev}_A : A\ra \pi_*(\cV|_{\sG})$ has degree zero with respect to the lower grading. 

To formulate the equivariant Condition 
\ref{assume alpha chainlevel}, recall first that $\eta_{\bdeg}:\CC^*_R\ra \Gamma\xrightarrow{\chi} \CC^*$ is the character $t\mapsto t^{\bdeg}$ of
$\CC^*_R$. Therefore, as an object of $\dbcoh {[\cC/\CC^*_R]}$, the line bundle
$\cL_\chi$ has $\CC^*_R$-weight $\bdeg$. Then $\varkappa:\cL_\chi\ra\log$ of \eqref{kappa_w} can be viewed as an isomorphism in 
$\dbcoh {[\cC/\CC^*_R]}$ by placing $\log$ in lower degree $\eta_\bdeg$ as well, i.e., by changing the target to $\log\ot\CC({\eta_\bdeg})$.
Similarly, by considering equivariant sheaf  $\cO_S({\eta_\bdeg})$ as the target
of \eqref{eq: finalalpha} we get a morphism in 
$\dbcoh {[S/\CC^*_R]}$. Condition \ref{assume alpha chainlevel} will now require that $\alpha$ realizes this morphism at cochain level in $\dbcoh {[S/\CC^*_R]}$.

The compatibility of Condition \ref{compatible} becomes the equality of two $\CC^*_R$-equivariant maps.

Proposition \ref{prop: satisfy all} (and similarly for Proposition \ref{prop: satisfy all geometric}) gives the existence of $\CC^*_R$-equivariant resolutions under the same assumptions and with the same proof (we take the bundles $\cE$ and $\cO(1)$ in Lemma \ref{lem: cohomology vanishing} to have $\CC^*_R$-weight equal to zero). 
If $[A\xrightarrow{d} B]$ is a $\CC^*_R$-equivariant admissible resolution, then  $\op{tot} (A)$ has the natural $\CC^*_R$-action which is trivial on the base $S$ and acts fiber-wise according to the decomposition $A=\oplus_{\eta\in \widehat{\CC^*_R}}A_\eta$; the projection
$p\colon \op{tot} (A)\to S$ is equivariant. 

The stacks $[V/G]$ and $I[V/G]$ have the natural residual $\CC^*_R$-action, for which the inertia map $I[V/G]\to [V/G]$ is equivariant. 
On either $[V/G]$, or on its inertia stack, the superpotential $w$ is naturally a section of $\cO(\eta_\bdeg):=\cO \ot \CC(\eta_\bdeg)$.

In addition, the stack $\op{tot}(\pi_* (\cV|_{\sG_i}))$ has the (fiber-wise) $\CC^*_R$-action coming from the decomposition $\cV =\oplus_{\eta\in \widehat{\CC^*_R}}\cV_\eta$. By its construction, the evaluation map 
$$\op{ev}^i_\cV: \op{tot}(\pi_* (\cV|_{\sG_i})\to I[V/G]$$
of \eqref{eval V} is $\CC^*_R$-equivariant. 
It follows that the composition
$$\op{ev}^i \colon  \op{tot} (A) \xrightarrow{\op{ev}_A^i} \op{tot}(\pi_* (\cV|_{\sG_i})  \xrightarrow{\op{ev}^i_\mathcal V}  I[V/G]$$
is $\CC^*_R$-equivariant. 
We also have the induced $\CC^*_R$-equivariant total evaluation map
$$\op{ev} \colon  \op{tot} (A) \to (I[V/G])^r,$$
where the target has the diagonal action of $\CC^*_R$. By a slight abuse, from now on we will not distinguish notationally the
$\CC^*_R$-equivariant maps between spaces with action and the induced maps between the respective stack quotients by $\CC^*_R$, i.e., we will also write
\begin{equation}\label{e:evtwidle}
\op{ev} \colon  [\op{tot} (A)/\CC^*_R] \to [(I[V/G])^r/\CC^*_R].\end{equation}
When such ambiguity occurs, the context should make clear which map is meant.

The $\CC^*_R$-equivariant admissible resolution $[A\xrightarrow{d} B]$ gives a $\CC^*_R$-invariant section
$\beta \colon \cO_{\op{tot}(A)}\to p^*B$ and a $\CC^*_R$-invariant cosection
$\alpha \colon  p^*B\ot \cO_{\op{tot}(A)}(\eta^{-1}_\bdeg) \to \cO_{\op{tot}(A)}$, 
satisfying
$$\alpha^\vee \circ (\beta\ot \op{id}_{\cO_{\op{tot}(A)}(\eta^{-1}_\bdeg)})= (\sum_{i=1}^r\op{ev}^i)^* w\in H^0(\op{tot}(A),\cO_{\op{tot}(A)}(\eta_\bdeg)).$$
This yields a PV factorization 
\[
\{-\alpha,\beta\} \in \op{D}([\op{tot} (A)/\CC^*_R], -(\sum_{i=1}^r\op{ev}^i)^* w)_{[Z(dw_{\tot (\pi_*\cV)}(\exp \tau_{\pi_*\cV})) / \CC^*_R]}.
\]

\section{Construction of a projective embedding}\label{s:conFF}
In this section, we aim to construct a Deligne--Mumford stack $S$ over the Artin stack $\Mfrak^\orb_{g,r}(B\Gamma)_{\log}$ such that it is equipped with an admissible resolution and such that the PV factorization is supported on $LG_{g,r}(\cc Z, d) = Z(\alpha,\beta)$.
For the moment, we can only do it for hybrid models $$(V=V_1\oplus V_2, G, \CC^*_R, \theta, w),$$
(see \S\ref{ss: hybrid models}). We recall the notations $\cX=[V_1 /\!\!/_\theta G]$ and $\cT=[\VmodtG]$ for the GIT quotients stacks, and $\cZ=Z(dw)$ for the degeneracy locus.
We emphasize that, unless $G$ is finite,
$$\cc X \neq [V_1/G] \quad \textrm{and} \quad \cc T \neq [V/G].$$
Moreover, we recall that the projection map
$\cT\xrightarrow q \cX$ realizes $\cT$ as a total space $\cT=\op{tot}\cE$, where $\cE$ is a vector bundle on $\cX$ with fiber $V_2$.

Subject to an additional technical condition made explicit in \S\ref{ss:convexity} below,
we construct the moduli space $\square = \square_{g,r,d}$ described in \S\ref{s:plan}.
The general machinery of admissible resolutions and PV factorizations from \S\ref{s:PV machinery} will then
produce a fundamental factorization 
\[K=K_{g,r, d} \in \dabsfact{[\square/\CC^*_R],  - \op{ev}^{*} \boxplus_{i = 1}^r w},\]
which we shall use to define the CohFT of the GLSM hybrid model.

\subsection{Convexity}\label{ss:convexity}
Given a hybrid GLSM $(V = V_1 \oplus V_2, G, \CC^*_R, \theta, w)$, note that by definition $\GJ$ acts trivially on $V_1$, thus 
$$G_1:=G/\GJ = \Gamma / \CC ^*_R$$ acts on $V_1$.
Introduce the rigidification of $\cc X$:
$$\cX^{\op{rig}}:=\VGJ.$$
We assume here that $\theta$ is in the image of $\widehat{G_1}\to \widehat{G}$, which can always be arranged after replacing $\theta$ by an appropriate power.

\begin{definition}\label{d:convexity}
Consider a Deligne--Mumford stack $\cM$ over $BG_1$ and assume that $\cM \ra BG_1$ is representable and smooth.
The stack $\cM$ is called \newterm{convex over $BG_1$} if for any representable morphism $f:\cc C \to \cM$ 
from a genus zero orbi-curve $\cc C$ (with any number of markings)  over $\mathrm{Spec} (\mathbb{C})$ we have \[ H^1( \cc C, f^*T_{\cM /BG _1}) = 0 . \]
\end{definition}

\begin{remark} 
\begin{enumerate}
\item Upper semi-continuity: the dimension of the vector space $H^1( \cc C, f^*T_{\cM /BG _1})$ is an upper semi-continuous function of the point $[f \colon \cc C \to \cM]$ in the moduli space of stable maps.

\item If $\cM$ is a smooth variety: the usual definition of \newterm{convexity} (over $\Spec(\CC)$) for a smooth variety $X$ is when for any morphism $f: \cc C \to X$ from a genus zero curve $\cc C$, we have $H^1( \cc C, f^* T_X) = 0$.
If a space $\cM$ as in Definition \ref{d:convexity} is a smooth variety and if the group $G_1$ is abelian, then convexity over $BG_1$ is equivalent to usual convexity, as can be seen from the distinguished triangle of tangent complexes for the morphism $\cM\ra  BG_1$. 

\item  If $G_1$ is abelian: in that case, a Deligne--Mumford stack $\cM$ is convex over $BG_1$ if and only if it is convex over $\Spec(\CC)$.
The reason is as follows. On $\cM$, there is an exact sequence of locally free sheaves: $0 \ra \cO _{\cM}^{\rank G} \ra T_{\cM /BG} \ra T_{\cM/\Spec (\CC)} \ra 0$.
Hence, it is enough to show that $H^1(\cc C, \cO_{\cc C}) = 0$ for a genus zero orbicurve. It is obvious since
$H^1(\cc C, \cO_{\cc C}) = H^1(\underline{\cc C}, \cO_{\underline{\cc C}}) = 0 $, where $\underline{\cc C}$ denotes the coarse moduli of $\cc C$.

\item If 
$G_1$ is non-abelian: convexity over $BG_1$ is a priori a stronger condition than convexity over 
$\Spec(\CC)$, though in familiar examples of GIT quotients of the form $[V_1 /\!\!/_\theta G_1]$, such as Grassmannians, both conditions hold.

\item The Orbifold case: to prove convexity in the presence of orbifold structure on $\cM$, it is not sufficient to check the condition from Definition \ref{d:convexity} for maps from orbi-curves with at most two markings.
For instance, consider a smooth projective elliptic curve $M$ with a nontrivial $\mu_2$-action.
Then the stack quotient $\cM$ be $[M/\mu_2]$ has 4 orbifold $B\mu_2$ points. 
If $\cc C$ is an irreducible orbicurve with at most two markings then there is no nontrival representable map from $\cc C$ to $\cM$. This can be seen by observation that the induced map 
$\cc C\times _{\cM} M \ra M$ must be trivial since $\cc C\times _{\cM} M$ is $\PP^1$.
Hence, there is no genus zero stable map with at most two markings to $\cM$, and the condition in Definition \ref{d:convexity} is therefore satisfied up to two markings.
However, the stack $\cM$ is not convex since the identity map $\op{id}\colon \cc C:= \cM \ra \cM$, with the four orbifold markings, has nonvanishing $H^1(\cc C, \op{id}^* T_{\cM/ B\mu_2})$.

\item For every closed point $B\mu _m$ of a Deligne--Mumford stack $\cM$ and for every $k$ with $k|m$, suppose that there is a genus zero stable map with one marking $B\mu _k$ to $\cM$.
In that case, to prove the convexity of $\cM$ over $BG_1$, it is enough to check the condition in Definition \ref{d:convexity} up to two markings.
This can be seen by considering comb-like genus zero curves with orbifold nodal points.
\end{enumerate}
\end{remark}

\medskip

For the remainder of the paper, we assume that $\Xrig$ is convex over $BG_1$. Explicitly, this means that for any {\it genus zero stable map} $f:\cc C\to \Xrig$, with any number of markings and with associated principal $G_1$-bundle $P_1$, we have the vanishing $H^1(\cc C, P_1\times_{G_1}V_1)=0$.

\subsection{Quasi-projective embeddings}\label{ssqp}

Fix $g$ and $r$ in the stable range, i.e., satisfying $2g-2 + r > 0$.  We will denote
by $\sKbar_{g, r}(\Xrig, d)$ the moduli stack of degree $d$ stable maps from families of $r$-pointed, genus $g$ twisted curves (Definition~\ref{d:twistedcurves}) to the Deligne--Mumford stack $\Xrig$.  Following the notation in \cite{AGV}, we use $\sKbar$ to emphasize that the gerbe markings may not have sections.  The corresponding moduli stack in which the gerbe markings do have sections will be denoted $\sMbar_{g,r}(\Xrig, d)$.  Note that when $\Xrig$ is a smooth variety, these two moduli stacks are identical.  More generally $\sMbar_{g,r}(\Xrig, d)$ may be defined as the fiber product of the gerbe markings over $\sKbar_{g, r}(\Xrig, d)$ (See \cite[Section~6.1.3]{AGV}).

In this subsection, following an idea in \cite{CKGWall},
we construct a closed immersion of $\sKbar_{g, r}(\Xrig, d)$ into a smooth Deligne--Mumford stack which, in turn, is an open substack of a stack satisfying Assumption $(\star)$ of \S\ref{ss:sc}. 

It will be convenient to have the following definition.

\begin{definition}\label{def: MBG}
Recall the definitions of \S \ref{moduli stacks}. Given a reductive algebraic group $H$ and a choice of character $\theta \in \widehat H = \mathrm{Hom}(H,\QQ)$, let  $\Mfrak^\tw_{g,r}(BH, d)^\circ$
denote  the open substack of $\Mfrak^\tw_{g,r}(BH, d)$ where $\log \otimes \cL_\theta^{\otimes M}$ is ample for $M$ sufficiently large.
Similarly, define $\Mfrak^{\orb/\tw}_{g,r}(B\Gamma)_{\log}^\circ$ to be the open subset of $\Mfrak^{\orb/\tw}_{g,r}(B\Gamma)_{\log}$   where $\log \otimes \cL_\theta^{\otimes M}$ is ample for $M$ sufficiently large.
\end{definition}

\medskip

Let $\mathcal{C}$ denote the universal curve over the moduli of stable curves $\sMbar_{g,r}$. Its coarse moduli space 
$\underline{\mathcal{C}}$ is projective over $\Spec (\CC )$. Choose a closed immersion $\underline{\mathcal{C}}  \subseteq \PP ^{N-1}$.
This in turn induces a morphism 
\begin{equation}\label{map v}
v: \sMbar_{g,r} \to \sMbar_{g, r}(\PP ^{N-1}, e) 
\end{equation}
for some degree $e$.
Consider the relative Picard stack $\fM_{g,r}(B\CC ^*, e)$ parametrizing line bundles of degree $e$ on prestable curves and
let $\fM_{g,r}(B\CC ^*, e)^{\circ}$ be the open locus 
for which the universal bundle $\mathcal{N}$ on the universal curve $\cC_{B\CC^*}\xrightarrow{\pi} \fM_{g,r}(B\CC ^*, e)$
is $\pi$-acyclic. 

The forgetful morphism $\sMbar_{g, r}(\PP ^{N-1}, e) \to \fM_{g,r}(B\CC ^*, e)$ induces 
$$h \colon \sMbar_{g,r} \to \fM_{g,r}(B\CC ^*, e)$$
by precomposing with the map $v$ of \eqref{map v}. If $\hat{h}\colon \mathcal{C} \to \cC_{B\CC^*}$ is the corresponding map of universal curves,
we have $\rho^*(\cO_{\PP^{N-1}}(1)|_{\underline{\mathcal{C}}})=\hat{h}^*\cN$, with $\rho \colon  \mathcal{C}\to \underline{\mathcal{C}}$ the coarse moduli map.
We may assume the map $h$
factors through $\fM_{g,r}(B\CC ^*, e)^{\circ}$ (if necessary, compose with a sufficiently high Veronese embedding $\PP ^{N-1}\hookrightarrow \PP^{N'-1}$, replace $\PP ^{N-1}$ by $\PP^{N'-1}$, and change $e$ accordingly).
This in particular implies that the image of the morphism $v$ is contained in the open locus  $\sMbar_{g, r}(\PP ^{N-1}, e)^{\circ}$ 
parametrizing unobstructed stable maps to $\PP^{N-1}$ for which maps are closed immersions of stable curves. 

Let $\op{st} \colon \fM^\tw_{g,r} \to  \sMbar_{g,r}$  be the stabilization map  defined by stabilizing the coarse curve of a twisted curve (see Definition~\ref{d:twistedcurves} for notation). For any algebraic stack $S$ with a morphism to 
$\fM^\tw_{g,r}$ we have the diagram 
\begin{equation}\label{diagram for curly N}
	\begin{tikzpicture}[scale=1]
	\node (AA) at (5.25,1.5) {$\underline{\mathcal{C}}  \subseteq \PP ^{N-1}$};
	\node (A) at (0,1.5) {$\fC_S$};
	\node (B) at (1.5,1.5) {$\mathfrak{C}$};
	\node (C) at (3,1.5) {$\mathcal{C}$};
	\node (D) at (0,0) {$S$};
	\node (E) at (1.5,0) {$\fM^\tw_{g,r}$};
	\node (F) at (3,0) {$\sMbar_{g,r}$};
	\draw[->,] (A) -- (B);
	\draw[->,] (B) -- (C);
	\draw[->,] (C) -- (AA);
	\draw[->,] (D) -- (E);
	\draw[->,] (E) -- (F);
	\draw[->,] (A) -- (D);
	\draw[->,] (B) -- (E);
	\draw[->,] (C) -- (F);
	\node at (2.25,0.75){};
	\node[above] at (2.25,0){$\mathrm{st}$};
	\node[above] at (3.75,1.5){$\rho$};
	\end{tikzpicture}
\end{equation}
The pull-back of $\cO _{\PP ^{N-1}}(1)$ via the composition of the maps in the top row is a line bundle on the universal curve $\fC_S$, which
we denote by $\cN_S$. 

Apply this to the forgetful map $\fM^\tw_{g,r}(BG _1, d)\to \fM^\tw_{g,r}$ to get a line bundle that we denote by $\cN_{BG_1}$ instead of $\cN_{\fM^\tw_{g,r}(BG _1, d)}$ for simplification.
Let $\fM^\tw_{g,r}(BG _1, d)^{\circledcirc}$ be the open substack of $\fM^\tw_{g,r}(BG _1, d)^{\circ}$ for which
\begin{eqnarray}\label{AC for M}  \RR^1\pi _* (\cV_1 (-\sG  )\ot \mathcal{N} _{BG_1} ) = 0, \end{eqnarray}
where  $\cV_1 : = \cc P_1 \times _{G_1} V_1$ is defined by the universal principal $G_1$-bundle $\cc P_1$. \footnote{In fact the constructions of this section work just as well over $\fM^\tw_{g,r}(BG _1, d)^{\circ}$ as over $\fM^\tw_{g,r}(BG _1, d)^{\circledcirc}$.  
However the vanishing condition \eqref{AC for M} may have useful future applications which is why we choose to work over $\fM^\tw_{g,r}(BG _1, d)^{\circledcirc}$.}

Similarly, the composition 
$\sKbar_{g,r} (\Xrig, d)\to \fM^\tw_{g,r}(BG _1, d)\to \fM^\tw_{g,r}$ gives a line bundle
$\cN_{\Xrig}$ on the universal curve $\cC_{\Xrig}$ of $\sKbar_{g,r} (\Xrig, d)$, compatible with $\cN_{BG_1}$. 
Possibly after replacing $\cN_{\Xrig}$ by $\cN_{\Xrig} ^{\ot m}$
for some large enough $m$ depending on $g$, $d$, $r$ (using the Veronese embedding described above),
we may assume that the forgetful map
\begin{equation}\label{e:usage of convexity} \xymatrix{ \sKbar_{g,r}(\Xrig, d) \ar[r] \ar@{-->}[rd] &  \fM^\tw_{g,r}(BG _1, d) \\
                                                   & \fM^\tw_{g,r}(BG _1, d)^{\circledcirc} \ar@{^{(}->}[u] }
\end{equation}
factors through $\fM^\tw_{g,r}(BG _1, d)^{\circledcirc}$.

\begin{remark}
It is important to notice that for the factorization of the diagram \eqref{e:usage of convexity} to hold, the convexity over $BG_1$
of $\Xrig$ is needed, as twisting by $\cN_{\Xrig}$ cannot be used to control the required $H^1$-vanishing on rational tails and rational bridges of the domain curves.
\end{remark}

Let $\CC^*$ act on $\CC^N$ diagonally, where we recall that the integer $N$ is fixed by the closed immersion $\underline{\mathcal{C}}  \subseteq \PP ^{N-1}$.
Consider $V_1\ot \CC^N$ with the $G_1$-action on the first factor and with the $\CC ^*$-action on the second second factor.
Define a quotient stack $\cc Y$ as follows:
\begin{align}\label{Definition curlyY}
\cc Y :&= \left[ \left( (\CC^N\setminus \{0\}) \times (V_1 \otimes \CC^N)^{ss}(\theta) \right) / (\CC^*\times {G_1}) \right]\\
\nonumber & = \left(\CC^N\setminus \{0\}\right) \times_{\CC^*} [(V_1 \otimes \CC^N)^{ss}(\theta)/G_1].
\end{align}

By the Hilbert-Mumford criterion, the point $(v_1, ..., v_N) \in V_1\ot \CC^N$ is $\theta$-semistable if and only if some $v_i$ is in $V_1^{ss}(\theta)$.
It follows there are no strictly semistable points in $V_1\ot\CC^N$ and hence $\cc Y$ is a separated Deligne--Mumford stack.
It is clear from its construction that $\cc Y$ is a fiber bundle over the projective space $\PP^{N-1}$, with fiber 
$[(V_1 \otimes \CC^N)^{ss}(\theta)/G_1]$. This fibration is not trivial, as it is twisted by the transition functions of $\cO_{\PP^{N-1}}(1)$.
Moreover, the coarse space map $\underline{\cc Y}\ra \PP ^{N-1} $
is a projective morphism, since 
$\left( \Sym (V_1^\vee) \right)^{G_1} = \CC$ by the definition of hybrid GLSM.
In particular, the coarse space $\underline{\cc Y}$ is projective over $\Spec(\CC)$.

Considering $d':=(d, e)$ as an element in $\widehat{G _1 \times \CC ^*}$, we have the moduli stack $\sKbar_{g,r}(\cY ,d' )$.
The projection $\cY \to \PP ^{N-1}$ induces a map
$$\op{pr} \colon \sKbar_{g,r}(\cY ,d' ) \to \sMbar_{g,r} (\mathbb{P} ^{N-1}, e).$$
On the other hand via the forgetful map $\op{fgt}\colon  \sKbar_{g,r}(\cY ,d' )\to  \fM^\tw_{g,r}$, we also have 
$$v\circ \op{st} \circ \op{fgt}\colon  \sKbar_{g,r}(\cY ,d' ) \to \sMbar_{g,r} (\mathbb{P} ^{N-1}, e).$$
 
Define $\sKbar_{g,r}^{eq} (\cc Y, d' )$ as 
the fiber product 
 \begin{equation}\label{def of M eq}
 \xymatrix{  \sKbar_{g,r}^{eq} (\cc Y, d') \ar[d]_{\op{fgt}}   \ar@{^{(}->}[r] & \sKbar_{g,r} (\cc Y, d')  \ar[d]^{ (\op{fgt} , \op{pr} )} \\
 \fM^\tw_{g,r}  \ar@{^{(}->}[r]_-{(\op{id},{v\circ \op{st}) }}  &    \fM^\tw_{g,r}  \times \sMbar_{g,r} (\mathbb{P} ^{N-1},e). } 
 \end{equation}
Since $\sMbar_{g,r} (\mathbb{P} ^{N-1})^{\circ}$ consists of closed immersions, the image of $v$ is contained in the separated {\it scheme} $\sMbar_{g,r} (\mathbb{P} ^{N-1})^{\circ}$, therefore
the ``graph of $v\circ \stab$" morphism $(\op{id}, v\circ \stab)$ is a representable closed immersion and thus 
$\sKbar_{g,r}^{eq} (\cc Y, d')$ is a closed substack
of $\sKbar_{g,r} (\cc Y, d') $.
More informally, the stack $\sKbar_{g,r}^{eq} (\cc Y, d')$ is the closed substack of $\sKbar_{g,r} (\cc Y, d')$ on which the equality $v\circ \op{st} \circ \op{fgt}=\op{pr}$ holds.

\begin{remark}\label{moduli description}
It is straightforward to check that for a scheme $S$, the $S$-points of $\sKbar_{g,r}^{eq} (\cc Y, d')$ are given by families
\begin{equation}\label{points of M eq}
 ( (\cC_S \ra S, \sG_1,\dots , \sG_r, P_1, \{u_1,\dots u_N\} ),
 \end{equation}
where $(\cC_S \ra S, \sG_1,\dots , \sG_r)$ is a twisted $r$-pointed curve of genus $g$ over $S$, $P_1$ is a principal $G_1$ bundle of degree $d$ on $\cC_S$, and $u_i$ are sections of $\cV_1\ot \cN_S$,
satisfying an appropriate stability condition. 
Here $\cV_1=P_1\times_{G_1} V_1$, and $\cN_S$ is the line bundle defined by 
the map $S \to \fM^\tw_{g,r}$ via the diagram \eqref{diagram for curly N}. The stability condition requires first that the map to $V_1\ot \CC^N$ given by the sections lands in $(V_1\ot \CC^N)^{ss}(\theta)$, and second that $\log\ot \cL_{\theta}^M$ is relatively ample for large enough $M$.

In fact, the stack $\sKbar_{g,r}^{eq} (\cc Y, d')$ could have been {\it defined} directly, without reference to $\cY$, as the moduli stack parametrizing the families 
\eqref{points of M eq}.
It is clear that this moduli definition gives an algebraic stack. In fact, it gives an open substack in the cone 
$\op{tot}\left(\pi_*(\cc V_1 \otimes \cN _{BG_1}^{\oplus N})\right)$
over $\fM^\tw_{g,r}(BG_1, d)^\circ$. Stability implies it is a Deligne--Mumford stack. 
However, its concrete realization given by \eqref{def of M eq} also implies immediately the following additional properties: 
the Deligne--Mumford stack $\sKbar_{g,r}^{eq} (\cc Y, d')$ is a global quotient by a linear algebraic group, and it has projective coarse moduli.
Indeed, the ambient stack $\sKbar_{g,r} (\cc Y, d')$ has these properties by \cite[Corollary 1.0.3 \& \S 3.2]{AGOT}, \cite[Theorem 1.4.1]{AV} 
and  \cite[Theorem 2.14]{EHKV}.
\end{remark}
Define 
${U_{\cc Y}}$ by the fiber square
\begin{equation}
\xymatrix{U_{\cc Y}\ar[r]\ar[d] & \sKbar^{eq}_{g,r}(\cc Y, d')\ar[d]\\
\fM^\tw_{g,r}(BG_1, d)^{\circledcirc}\ar[r] & \fM^\tw_{g,r}(BG_1, d)^\circ,  
}
\end{equation}
so that ${U_{\cc Y}}$ is an open substack of $\sKbar^{eq}_{g,r}(\cc Y, d')$.

Let $\cN^{eq}$ be the line bundle on the universal curve over $\sKbar_{g,r}^{eq} (\cc Y, d' )$ induced by the map $\op{fgt}$ via 
\eqref{diagram for curly N}, and note that it coincides with the pull-back of $\cN_{BG_1}$.
By its moduli description in Remark \ref{moduli description}, the stack $\sKbar^{eq}_{g,r}(\cc Y, d')$ has a perfect obstruction theory relative to $ \fM^\tw _{g,r}(BG_1,d)$, given by
$$\RR\pi_*((\cV_1\ot \cN^{eq})^{\oplus N}).$$
Hence the obstruction sheaf is isomorphic to $\RR^1\pi_*((\cV_1\ot \cN^{eq})^{\oplus N})$,
which is a quotient of $\RR^1\pi_*((\cV_1(-\sG)\ot \cN^{eq})^{\oplus N})$. By \eqref{AC for M}, the obstruction sheaf vanishes when restricted to ${U_{\cc Y}}$. This implies that ${U_{\cc Y}}$ is smooth over $\Spec (\CC )$.

\begin{lemma}\label{UY} Fix $(g,r)$ with $2g-2+r>0$.

$(a)$ There is a two-term locally free resolution 
$$ [A_1 \xrightarrow{d_1} B_1]\cong \RR\pi _* \cV_1 $$
on $\fM^\tw_{g,r}(BG _1, d)^{\circledcirc}$,
and a natural open immersion $U_\cY  \subseteq \tot A_1$, compatible with the maps to  $\fM^\tw_{g,r}(BG _1, d)^{\circledcirc}$.

$(b)$ Let $p_1\colon \tot A_1 \ra  \fM^\tw_{g,r} (BG_1, d)^{\circledcirc}$ denote the projection and let $\beta _1$ be the section of 
$p_1^* B_1|_{U_\cY}$ induced from $d_1$, with zero locus $Z(\beta_1) \subseteq U_\cY$.
Assume that $\Xrig$ is convex.  Then there is a canonical identification
$\sKbar_{g,r}(\Xrig, d) \cong Z(\beta _1)$, so that
we have a closed immersion making a commuting diagram
 \[\xymatrix{  \sKbar_{g,r}(\Xrig, d) \ar[dr]  \ar@{^{(}->}[r] &  U_{\cc Y} \ar[d]   \\
&    \fM^\tw_{g,r} (BG_1, d)^{\circledcirc}       . } \]
\end{lemma}

\begin{proof}
$(a)$ Consider the Euler sequence 
$$0\ra \cO_{\PP^{N-1}} \ra \cO_{\PP^{N-1}}(1)^{\oplus N}\ra T_{\PP^{N-1}}\ra 0$$
on $\PP^{N-1}$.
Pulling it back via the composition $ \cC_{BG_1} \to  \underline{\mathcal{C}} \to \PP^{N-1}$ and tensoring with $\cV_1$, we obtain a short exact sequence
\begin{equation}\label{e:tautV} 0 \to \cc V_1  \to \cc V_1 \otimes \cN _{BG_1}^{\oplus N} \xrightarrow{q} \cc Q \to 0\end{equation}
where $\cc Q$ is the cokernel.  Pushing forward via $\pi_*$ and using the vanishing \eqref{AC for M} gives the long exact sequence
\begin{equation}\label{e:tautVlong}
0\ra \pi_*(\cc V_1) \ra \pi_*(\cc V_1 \otimes \cN _{BG_1} ^{\oplus N})\xrightarrow{\pi_*q}  \pi_*(\cc Q)\ra \RR^1\pi_*(\cc V_1)\ra 0,
\end{equation} 
on $\fM^\tw_{g,r}(BG _1, d)^{\circledcirc}$, with the middle two terms locally free. Hence
$$[A_1 \to B_1] := [\pi_*(\cc V_1 \otimes \cN _{BG_1}^{\oplus N}) \to \pi_*(\cc Q)]$$
gives the required complex of vector bundles on the stack $\fM^\tw_{g,r}(BG _1, d)^{\circledcirc}$.
Clearly $U_\cY$ is the open locus in $\tot A_1$ obtained by imposing the condition that the sections in 
$A_1=\pi_*(\cc V_1 \otimes \cN _{BG_1}^{\oplus N})$ give a map landing in $(V_1\ot \CC^N)^{ss}(\theta)$.

$(b)$ This follows from \eqref{e:tautVlong} and \eqref{e:usage of convexity}. 
\end{proof}

\subsection{Properties of moduli of LG maps}\label{genhybmod}
In this subsection we show the moduli stacks of Landau--Ginzburg maps to
the ``base" $\cX=[V_1 /\!\!/_\theta G]$ of a hybrid model $(V = V_1 \oplus V_2, G, \CC^*_R, \theta, w)$ 
are global quotient stacks with projective coarse moduli. These properties will follow from the corresponding ones for
the related spaces of stable maps to $\Xrig$ discussed in the previous subsection.

Recall from \S\ref{s:OC} that the group $\Gamma$ surjects onto $G_1 \times \CC^*$ via the map $g \cdot \lambda \mapsto ([g], \lambda^{\bdeg})$, where $G_1 := G/\langle J \rangle$.
It yields the exact sequence 
\[ 1\to \langle J \rangle \to \Gamma \to G_1 \times \CC^* \to 1.\]
 Borrowing an idea from \cite[\S1.5]{AJ}, we consider the fiber diagram:
\begin{equation}\label{diagram}\begin{tikzcd}
\cC_{\Xrig,\Gamma} \ar{r} \ar{d}  & B\Gamma \ar{d} \\
\cC_{\Xrig} \arrow[r, "{\left[\cP\right]}  \times \log"] 
\ar{d}
 & BG_1  \times B\CC^*\\
\sKbar_{g,r}(\Xrig, d) & 
\end{tikzcd}
\end{equation}
where $\cC_{\Xrig}$ is the universal curve, we denote by $\left[\cc P\right]$ the universal principal $G_1$-bundle, and $\cC_{\Xrig,\Gamma}$ is defined as the fiber product.
The map $B\Gamma \to BG _1 \times B\CC^*$ is an \'etale gerbe, therefore the map $\cC_{\Xrig,\Gamma} \to\cC_{\Xrig}$ is also an \'etale gerbe.

For each connected component of the space $\sKbar_{g,r}(\Xrig, d),$
let $F$ be the class of a fiber  of $\cC_{\Xrig} \to \sKbar_{g,r}(\Xrig, d)$ in the homology group of $\cC_{\Xrig}$.
Let us consider the moduli space
$$\sMbar_{g,r} \left(\cC_{\Xrig,\Gamma} / \sKbar_{g,r}(\Xrig, d), F\right)$$
of balanced stable maps from orbi-curves to $\cC_{\Xrig,\Gamma} $ {\it relative to} 
$\sKbar_{g,r}(\Xrig, d)$,  of class $F$ (see \cite[\S8.3]{AV}).   

\begin{remark}\label{r:gerbesections} In \cite{AV}, the notation $\sKbar_{g,r} \left(\cC_{\Xrig,\Gamma} / \sKbar_{g,r}(\Xrig, d), F\right)$ is used to denote the space of relative twisted stable maps.  In this case, the marked point divisors on the universal curve are (nontrivial) gerbes over the moduli space.  In keeping with the notation of \cite[\S6.1.3]{AGV}, we use $\mathcal M$ to denote the space of relative stable maps \textit{with sections at the gerbe markings} (recall part \textit{(a)} of Definition~\ref{d:plg}).  $\sMbar_{g,r} \left(\cC_{\Xrig,\Gamma} / \sKbar_{g,r}(\Xrig, d), F\right)$ is a finite cover of $\sKbar_{g,r} \left(\cC_{\Xrig,\Gamma} / \sKbar_{g,r}(\Xrig, d), F\right)$, and is easily constructed as in \S\ref{moduli stacks} and \cite[\S6.1.3]{AGV} as
\begin{equation}\label{e:addsections}\sMbar_{g,r} \left(\cC_{\Xrig,\Gamma} / \sKbar_{g,r}(\Xrig, d), F\right) := \sG_1 \times_{\sKbar}  \sG_2 \times_{\sKbar} \cdots \times_{\sKbar} \sG_r,\end{equation}
where $\sKbar$ here denotes $\sKbar_{g,r} \left(\cC_{\Xrig,\Gamma} / \sKbar_{g,r}(\Xrig, d), F\right)$.   Note that by construction, there is a forgetful map $\sMbar_{g,r} \left(\cC_{\Xrig,\Gamma} / \sKbar_{g,r}(\Xrig, d), F\right) \to \f M^\orb_{g,r}(B\Gamma, d)$.
\end{remark}

Recall that a map \[S \to LG_{g,r}(\cX, d)\] is a family over $S$ of prestable orbicurves with gerbe markings, sections of these gerbes, a principal $\Gamma$-bundle $\cc P$, 
an isomorphism $\chi_*(\cc P) \simeq \log$, and sections of the associated vector bundle $\cc V_1$ satisfying some stability conditions.
Composing with the map $B\Gamma \to BG_1 \times B\CC^*$, the bundle $\cc P$ induces a principal $G_1$-bundle $\left[ \cc P \right]$ and a line bundle isomorphic to $\log$.
Moreover, the vector bundle $\cV_1$ is also the associated vector bundle to $\left[ \cc P \right]$.
Thus we get a family over $S$ of prestable orbicurves with gerbe markings, with the bundle $\left[\cc P\right]$, and with sections of the vector bundle $\cc V_1$ satisfying the same stability conditions as before.  After partially rigidifying the curve in the sense of \cite[Proposition 9.1.1]{AV} we obtain a map: 
$$S \to \sKbar_{g,r}(\Xrig, d).$$
As a consequence, we get a morphism
\begin{equation}\label{maptobase}
LG_{g,r}(\cX, d) \to \sKbar_{g,r}(\Xrig, d)
\end{equation}
Hence, the universal curve over $LG_{g,r}(\cX, d)$ maps to the universal curve $\cC_{\Xrig}$ over $\sKbar_{g,r}(\Xrig, d)$.
But it also maps to $B\Gamma$ via the universal bundle $\cc P$.  Both $B\Gamma$ and $\cC_{\Xrig}$ map to $BG_1 \times B\CC^*$.
By construction of the morphism \eqref{maptobase}, the corresponding maps from the universal curve over $LG_{g,r}(\cX, d)$ to $BG_1 \times B\CC^*$ are equal, so that we get a map from the universal curve over 
$LG_{g,r}(\cX, d)$ to $\cC_{\Xrig,\Gamma}$.
Such a map induces a morphism
\begin{equation}\label{maptomaps}
LG_{g,r}(\cX, d) \to  \sMbar_{g,r} \left(\cC_{\Xrig,\Gamma} / \sKbar_{g,r}(\Xrig, d), F\right).
\end{equation}
More precisely, this map is a closed immersion and the stack $LG_{g,r}(\cX, d)$ can be identified with the closed substack consisting of stable maps $f$ such that, for every marking $\sigma_i$, the point $f(\sigma_i)$ is sent to the corresponding marking $\sigma_i \in \cC_{\Xrig}$ via the map $\cC_{\Xrig,\Gamma} \to \cC_{\Xrig}$.

\begin{proposition}\label{p: projmod}
$(a)$ The morphism $LG_{g,r}(\cX, d) \to \sKbar_{g,r}(\Xrig, d)$ defined in \eqref{maptobase} is proper,  quasi-finite, and of Deligne--Mumford type. 

$(b)$ The stack $LG_{g,r}(\cX, d)$ is a global quotient stack with projective coarse moduli space.
\end{proposition}

\begin{proof}
(a) The properties listed follow from the description \eqref{maptomaps} of $LG_{g,r}(\cX, d)$ as a space of relative stable maps (see \cite[Theorem 1.4.1 and Section 8.3]{AV}).

(b) This is due to Lemma \ref{quotient target} below, \cite[Corollary 1.0.3]{AGOT} and \cite[Theorem 2.14]{EHKV}.
\end{proof}

\begin{lemma}\label{quotient target}  The stack $\cC_{\Xrig,\Gamma}$ is a global quotient by a linear algebraic group, and has projective coarse moduli space.
\end{lemma}

\begin{proof}
 First note that the coarse moduli of $\cC_{\Xrig,\Gamma}$ and $\cC_{\Xrig}$ coincide, and that
$\cC_{\Xrig}$ has projective coarse moduli by \cite{AV}.
 
Second, the stack $\cC_{\Xrig}$ is a quotient stack by \cite[Corollary 1.0.3]{AGOT} and \cite[Theorem 2.14]{EHKV}.  Hence there is a scheme $R$ and a finite flat surjective morphism
$R\ra \cC_{\Xrig} $, by \cite[Proposition 5.1]{KreschGeometry}. 
Let $P$ be the $G_1  \times \CC^*$-bundle on $R$ corresponding to the map $R\ra BG_1 \times B\CC ^*$.
There is a natural map from the algebraic space $P$ to the fiber product 
$ R':=R\times _{\cC_{\Xrig   } } 
\cC_{\Xrig,\Gamma}$, realizing $R'$ as a global quotient.   The scheme $R$ is projective, since the map from $R$ to the coarse moduli of $\cC_{\Xrig}$
is finite. Therefore $R'$ has projective coarse moduli $R$.
Using again \cite[Proposition 5.1]{KreschGeometry}, there is a scheme $R''$ and a finite flat surjective morphism $R''\ra R'$.
The composition $R'' \to \cC_{\Xrig,\Gamma}$ is finite, flat, and surjective. 

Applying \cite[Proposition 5.1]{KreschGeometry} a third time (now in reverse direction), we conclude that $\cC_{\Xrig,\Gamma}$ is a global quotient.
\end{proof}

Note that the above construction can be repeated over the space $\sKbar^{eq}_{g,r}(\cc Y, d)$.  
We have the following  fiber product diagram, defining $\cC_{\cY, \Gamma}$,
\begin{equation}\label{diagram2}\begin{tikzcd}[row sep = tiny]
 &\cC_{\cY, \Gamma}  \ar[rd] \ar[dd]& \\
\cC_{\Xrig, \Gamma}  \ar[ru, hookrightarrow] \ar[dd] \ar[rr, crossing over]& & B\Gamma \ar[dd]\\
&\cC_{ \cc Y} \ar[dd] \ar[rd] & \\
\cC_{\Xrig} \ar[ru, hookrightarrow] \ar[dd] \ar[rr, crossing over]& & BG_1 \times B\CC^* \\
&  \sKbar^{eq}_{g,r}(\cc Y, d')  & \\
\sKbar_{g,r}(\Xrig, d) \ar[ru, hookrightarrow] & &
\end{tikzcd}
\end{equation}
Define the space $LG_{g,r}^{eq}(\cc Y, d')$ as the closed substack of the moduli space of relative stable maps
$$\sMbar_{g,r} \left(\cC_{\cY, \Gamma} / \sKbar^{eq}_{g,r}(\cc Y, d'), F\right)$$
consisting of stable maps $f$ such that, for every marking $\sigma_i$, the point $f(\sigma_i)$ is sent to the corresponding marking $\sigma_i \in \cC_{\cc Y}$ via the map $\cC_{\cY,\Gamma} \to \cC_{\cc Y}$.
As a consequence, Diagram \eqref{diagram2} shows that $LG_{g,r}(\cX, d)$ embeds into the larger space $LG_{g,r}^{eq}(\cc Y, d')$ of Landau--Ginzburg maps.

Define an open substack $U$ of $LG_{g,r}^{eq}(\cc Y, d')$ as the preimage of $U_{\cc Y}$ 
under the map $LG_{g,r}^{eq}(\cc Y, d') \to \sKbar^{eq}_{g,r}(\cc Y, d')$. The diagram
\begin{equation}\label{d:pullback}
\begin{tikzpicture}[scale=1]
\node (A) at (0,1.5) {$LG_{g,r}(\cX, d)$};
\node (B) at (2,1.5) {$U$};
\node (C) at (4,1.5) {$LG_{g,r}^{eq}(\cc Y, d')$};
\node (D) at (0,0) {$\sKbar_{g,r}(\Xrig, d)$};
\node (E) at (2,0) {$U_\cY$};
\node (F) at (4,0) {$\sKbar_{g,r}^{eq}(\cc Y, d')$};
\draw[right hook->,] (A) -- (B);
\draw[right hook->,] (B) -- (C);
\draw[right hook->,] (D) -- (E);
\draw[right hook->,] (E) -- (F);
\draw[->,] (A) -- (D);
\draw[->,] (B) -- (E);
\draw[->,] (C) -- (F);
\end{tikzpicture}
\end{equation}
is cartesian, where the first pair of horizontal arrows are closed immersions and the second pair of horizontal arrows are open immersions.  
Note that $U$ is smooth over $\Spec (\CC )$.

There is a natural morphism $B\Gamma \ra BG_1$, induced by the quotient map $\Gamma \ra G_1$.
In turn, it induces the forgetful morphism $\Mfrak^\orb_{g,r}(B\Gamma, d)_{\log} \ra \Mfrak^\orb_{g,r}(BG_1, d)\ra \Mfrak^\tw_{g,r}(BG_1, d)$.
Let $\Mfrak^\orb_{g,r}(B\Gamma, d)_{\log}^\circledcirc$ be the inverse image of $\Mfrak^\tw_{g,r}(BG_1, d)^{\circledcirc}$
under the above forgetful morphism. The following theorem collects together the key points of the construction in this subsection.

\begin{theorem}\label{c:swn}
Over $\Mfrak^\orb_{g,r}(B\Gamma, d)_{\log}^\circledcirc$, there exists a two-term resolution 
\[ [A_1 \stackrel{d_1}{\to} B_1] \cong \RR\pi_*(\cc V_1)\]
by vector bundles and an open substack $ U  \subseteq \tot(A_1)$  over $ \Mfrak^\orb_{g,r}(B\Gamma, d)_{\log}^\circledcirc$ which is a smooth separated Deligne--Mumford stack such that:
\begin{enumerate}
\item $Z(\beta_1) = LG_{g,r}(\cX, d)$, where $\beta_1 \in H^0(U, p_1^*B_1|_U) $ is the section induced from $d_1$. Here
$p_1 : \tot A_1 \ra \Mfrak^\orb_{g,r}(B\Gamma, d)_{\log}^\circledcirc$ is the projection.
\item There exists an open immersion of $U$ into a global quotient Deligne--Mumford stack $LG_{g,r}^{eq}(\cc Y, d')$, whose coarse moduli space is projective.
\end{enumerate}
\end{theorem}

\begin{proof}
Let $\gamma:\cC_{B\Gamma} \to\cC_{BG_1}$ denote the natural map between universal curves over the stacks $\Mfrak^\orb_{g,r}(B\Gamma, d)_{\log}$ and $\Mfrak^\tw_{g,r}(BG_1, d)$, 
and consider the pullback of the sequence \eqref{e:tautV} via $\gamma$,
\[ 0 \to \cc V_1 \to \cc V_1 \otimes \widetilde{\cc N}^{\oplus N} \xrightarrow{\delta_1} \gamma^*\cc Q \to 0,\]
where $\widetilde{\cc N}$ denotes $\gamma^* \cN_{BG_1}$ and, by abuse of notation, we are using $\cc V_1$ to denote bundles on the universal curves over both $\Mfrak^\orb_{g,r}(B\Gamma, d)_{\log}$ and $\Mfrak^\tw_{g,r}(BG_1, d)$.  Note that both $\cc V_1 \otimes \widetilde{\cc N}^{\oplus N}$ and $\gamma^*(\cc Q)$ are $\pi$-acyclic over $\Mfrak^\orb_{g,r}(B\Gamma, d)_{\log}^\circ$.  Hence we have the locally free resolution 
\begin{equation}\label{e:beta1}
[A_1\xrightarrow{d_1} B_1] : =[ \pi_*(\cc V_1 \otimes \widetilde{\cc N}^{\oplus N}) \to \pi_*(\gamma^* \cc Q)]\cong \RR\pi_*\cV_1
\end{equation}
on the stack $\Mfrak^\orb_{g,r}(B\Gamma, d)_{\log}$. 

As in the proof of Lemma \ref{UY}, the stack $U$ from \eqref{d:pullback} is naturally an open substack of 
$\tot(A_1)\xrightarrow{p_1} \Mfrak^\orb_{g,r}(B\Gamma, d)_{\log}^\circledcirc$.  
To see that $LG_{g,r}(\cX, d)$ is exactly the vanishing locus of $\beta_1$ in $U$, we note that the open condition defining $U$ in 
$\tot(A_1)$, when restricted to the cone $\tot( \pi_*(\cc V_1))$, gives exactly the stability condition defining $LG_{g,r}(\cX, d)$.  This proves part (a).
 
 For part (b), observe that
 the same argument as in the proof of Lemma \ref{quotient target} shows that $\cC_{\cY,\Gamma}$
is a quotient stack with projective coarse moduli. By \cite[Corollary 1.0.3]{AGOT} and \cite[Theorem 2.14]{EHKV}, 
the space $LG_{g,r}^{eq}(\cY, d')$  is also a quotient stack with projective coarse moduli space.
\end{proof}

\section{The GLSM theory for convex hybrid models}\label{s:GLSMtheory}
\subsection{The Fundamental Factorization}\label{s:FF}

By Theorem~\ref{c:swn}, the stack $LG_{g,r}^{eq}(\cc Y, d')$ satisfies Assumption $(\star)$ of \S\ref{ss:sc}
 as well as the $\pi$-ampleness assumption in \S \ref{subsubsec: cochain-level}.
By Proposition \ref{prop: satisfy all}, there is an admissible resolution (i.e.~satisfying Conditions~\ref{assume ev chainlevel}, \ref{assume alpha chainlevel}, and \ref{compatible}) 
\[[\widetilde A_1 \oplus \widetilde A_2 \to \widetilde  B_1 \oplus \widetilde B_2 ] \cong \R\pi_* (\cV_1 \oplus \cV_2) \quad \textrm{on $\overline{U}  \subseteq LG_{g,r}^{eq}(\cc Y, d')$}\]
with evaluation map $\op{ev}^i_{\widetilde A}$ and cosection $\widetilde \alpha$.  In particular,  observe that we have two resolutions of $\R\pi_* \cV_1$ lying over the space $U$.

Apply Lemma~\ref{l:surjective} to
$[\widetilde A_1 \to \widetilde B_1] \cong [A_1 \to B_1]$ and let $\bar A_2 = \widetilde A_2$ and $\bar B_2 = \widetilde B_2$, to obtain the following diagram:
\begin{equation}\label{e:roofbar}
\begin{tikzcd}
& \bar A 
 \ar[ldd, twoheadrightarrow,  "f_{\widetilde A}", swap] \ar[r, "\bar d"] \ar[rdd, twoheadrightarrow, near end, "f_{A_1}"] & \bar B \ar[rdd, twoheadrightarrow, "f_{B_1}"] \ar[ldd, twoheadrightarrow, crossing over, "   f_{\widetilde B}", near end, swap] 
& \\
& & & \\
 \widetilde  A \ar[r] & \widetilde  B  & A_1 \ar[r, "d_1"]  
  & B_1
\end{tikzcd}
\end{equation}
where the left square is a quasi-isomorphism, the right square is a composition of the projection followed by a quasi-isomorphic cochain map
$[\bar{A}_1 \ra \bar B_1] \ \ra [A_1\ra B_1]$, and 
the diagonal maps are all surjective.  
With $\op{ev}^i_{\widetilde A}$ as in Condition 1 of Definition \ref{admissible}, we define $\op{ev}^i_{\bar A}$ as the composition
\[
\op{ev}^i_{\bar A} : \bar A \xrightarrow{f_{\widetilde A}} \widetilde A \xrightarrow{\op{ev}^i_{\widetilde A} }  \O_U.
\]
Similarly, we define $\overline{\alpha}^\vee$ as the composition
\[ \overline{\alpha}^\vee : \op{Sym} \bar A \otimes \bar B \xrightarrow{\op{Sym}(f_{\widetilde A}) \otimes f_{\widetilde B}} \op{Sym} \widetilde A \otimes \widetilde B \xrightarrow{ \widetilde{\alpha}^\vee} \O_U.
\]
By Lemma~\ref{l:surjadmi} $[ \bar A \to \bar B]$ is admissible.

Let $p_{\tot(\bar A)}$ denote the projection $\tot(\bar A) \to U$ and consider the section $\zeta \in \Gamma(\tot(\bar A), p_{\tot(\bar A)}^* A_1)$ given by
\[
 \zeta= p_{\tot(\bar A)}^* f_{A_1} \circ \text{taut}_{\bar A} - p_{\tot(\bar A)}^*(\text{taut}_{A_1}),
 \]
where $\text{taut}_{A_1}$ and $\text{taut}_{\bar A}$ are the tautological sections of the bundles ${A_1} \to U$ and $\bar A \to \tot(\bar A)$ respectively (note that the former uses the fact that $U$ is an open substack of $\tot A_1$).

Recalling Definition~\ref{eval on tot A} there exist evaluation maps 
\begin{equation}\label{evA}
\op{ev}^i \colon \tot(\bar A) \to 
I[V/G]  
\end{equation}
We define $\tot(\bar A)^\circ $ as the open substack of $\tot(\bar A)$ such that $\op{ev}^i$ maps to the semistable locus $I[\VmodtG]$
\[\tot(\bar A)^\circ := \bigcap_{i=1}^r (\op{ev}^i)^{-1}I[\VmodtG]. \]

\begin{definition}\label{d:square}
Define the space
\[ \square = \square_{g,r,d} := \{\zeta = 0\}  \subseteq \tot(\bar A)^\circ.\]
Since the map $\bar A \to A_1$ is surjective, the space $\square$ is a smooth Deligne--Mumford stack.  
\end{definition}

\begin{remark}
Note that the stack $\square$ is determined by two different resolutions.
First over $\Mfrak^\orb_{g,r}(B\Gamma, d)_{\log}^{\circ}$ we construct a resolution $[A_1 \stackrel{d_1}{\to} B_1]$ of $\mathbb R\pi_*(\cc V_1)$ by vector bundles and define an open set $U  \subseteq \tot(A_1)$ 
which is a smooth separated Deligne--Mumford stack with quasiprojective coarse moduli space (see Corollary~\ref{c:swn}).   
Second, over $U$ we construct a resolution $[\bar A \stackrel{\bar d}{\to} \bar B]$ of $\mathbb R\pi_*(\cc V)$ by vector bundles.  The space $\square$ is given as a subset of $\tot(\bar A)$ over $U$. 
We refer to this construction involving two resolutions as
the \newterm{two-step procedure}.
\end{remark}

\begin{remark}\label{alternative def by cutting}
Due to the fact that the stack $\tot(\bar A)$ lies over $U  \subseteq \tot(A_1)$, the relative dimension of $\tot(\bar A) \to \Mfrak^\orb_{g,r}(B\Gamma, d)_{\log}^\circ$ is given by $\rank(A_1) + \rank(\bar A)$.  
If we view the pullback of $\bar B$ to $\tot(\bar A)$ as an obstruction bundle, the virtual dimension relative to $\Mfrak^\orb_{g,r}(B\Gamma, d)_{\log}^\circ$ is 
$\rank(\mathbb R \pi_*(\cV)) + \rank(A_1)$, where $$\rank(\mathbb R \pi_*(\cV)) := \rank(\bar A) - \rank(\bar B).$$  This virtual dimension is dimension $\rank(A_1)$ more than it should be.  To correct for this overcounting, we must restrict to the locus of points $(a_1, \bar a) \in \tot(\bar A)$ such that $f_{\widetilde A}(\bar a) = a_1$.  This is exactly the locus $\square$ defined by  the vanishing of $\zeta.$ Restricting $\bar B$ to this locus will then yield a stack of the correct virtual dimension.  
\end{remark}

\begin{definition}\label{d:koszulcut}
Define the vector bundle $E \to \square$ as (the pullback of) $\bar B$. 
The vector bundle $E$ has a section $ \bar \beta$ and a cosection $ \overline{\alpha}^\vee$.  For notational simplicity we will drop the bars and denote these simply as $\beta$ and $\alpha$ when no confusion is likely to occur.
We define
$$K = K_{g,r,d} \in \dabsfact{[\square_{g,r,d}/\CC^*_R], - \op{ev}^* (\boxplus_{i=1}^r w)}$$ to  be the Koszul factorization $\{-\alpha, \beta\}$.  We refer to $\{-\alpha, \beta\}$ as the \newterm{fundamental factorization}.
\end{definition}
It is straightforward to check that the above data
satisfies all the properties promised in \S \ref{s:plan}.  In particular we record the following fact.
\begin{lemma}\label{l:finalsupp}
The support of $K_{g,r,d}$ is equal to $LG_{g,r}(\cZ, d)$.  In particular, it has proper support.  
\end{lemma}
\begin{proof}
Recall that by  Theorem~\ref{c:swn}~(a), the stack $U$ is an open substack of $\op{tot} A_1$ over $\Mfrak^\orb_{g,r}(B\Gamma, d)_{\log}^\circledcirc $ with $Z(\beta_1) = LG_{g,r}(\cX, d)$.  In particular, we may view $\op{tot} \bar A$ over $U$ as an open substack of $\op{tot} A_1 \times_{\Mfrak^\orb_{g,r}(B\Gamma, d)_{\log}^\circledcirc } \op{tot} \bar A$.  Now, by Lemma~\ref{explicit  Zbeta}, the zero locus $Z(\beta)$ in $\tot(\bar A)$ is isomorphic to $\tot \pi_* \cV$ over $U$.  Hence, we may view $Z(\beta)$ as an open subset of $\op{tot} A_1 \times_{\Mfrak^\orb_{g,r}(B\Gamma, d)_{\log}^\circledcirc } \tot \pi_* \cV$.

Recall that 
\[
 \zeta= p_{\tot(\bar A)}^* f_{A_1} \circ \text{taut}_{\bar A} - p_{\tot(\bar A)}^*(\text{taut}_{A_1}).
 \]
 Since  $[\bar A_1 \to \bar B_1] \to [A_1 \to B_1]$ is a quasi-isomorphism, we have a commutative diagram
 \[
 \begin{tikzcd}
  \pi_* \cV_1 \ar[r] \ar[d, "\op{Id}"] & \bar A \ar[d, "f_{A_1}"] \\
    \pi_* \cV_1 \ar[r] & A_1. \\
 \end{tikzcd}
 \]
  It follows that the restriction 
 \[
 \zeta|_{\op{tot} A_1 \times_{\Mfrak^\orb_{g,r}(B\Gamma, d)_{\log}^\circledcirc } \tot \pi_* \cV}
 \]
  defines an open subset of
 \[
 \Delta  \times_{\Mfrak^\orb_{g,r}(B\Gamma, d)_{\log}^\circledcirc } \tot \pi_* \cV_2
 \]
 where $\Delta$ is diagonal map of $\op{tot} \pi_* \cV_1$ over $\Mfrak^\orb_{g,r}(B\Gamma, d)_{\log}^\circledcirc$.  This is isomorphic to $\tot \pi_* \cV$.  By definition of $U$, the intersection 
$ (\Delta  \times_{\Mfrak^\orb_{g,r}(B\Gamma, d)_{\log}^\circledcirc } \tot \pi_* \cV_2) \cap U$ is contained in $\tot (\bar A)^\circ$.  
  It follows that $\square \cap Z(\beta) $, which by definition is $Z(\zeta) \cap Z(\beta) \cap \tot(\bar A)^\circ$, is isomorphic to the open substack of $\tot \pi_* \cV$ over $\Mfrak^\orb_{g,r}(B\Gamma, d)_{\log}^\circledcirc$ defined by $U$ i.e. it is equal to $\tot \pi_* \cV_2$ over $Z(\beta_1) = LG_{g,r}(\cX, d)$.  This is just $LG_{g,r}(\cT, d)$.

  To finish the proof, note that by Proposition~\ref{PV support}, intersecting $LG_{g,r}(\cT, d)$ with $Z(\alpha)$ gives the degeneracy locus of \eqref{cosection dw} which in this case is $LG_{g,r}(\cZ, d)$.  This is a closed substack of $LG_{g,r}(\cX, d)$ hence proper by Proposition~\ref{p: projmod}.
\end{proof}

\begin{lemma}
The space $[\square / \C^*_R]$ is a nice quotient stack. 
\end{lemma}
\begin{proof} Since $U$ is a smooth Deligne--Mumford quotient stack it is in particular a nice quotient stack i.e. $U = [T/H]$ for some noetherian scheme $T$ and reductive linear algebraic group $H$ admitting an ample family of line bundles.  Therefore $\bar A$ corresponds to an $H$-equivariant vector bundle $E$ on $T$.  The weight decomposition of $\bar A$ coming from the $\C^*_R$-action induces the same decomposition on $E$.  This gives a $\C^*_R$-action on $\tot E$ such that $\tot(\bar A) = [\tot E / H \times \C^*_R ]$.  The projection map from $\tot E$ to $U$ is equivariant with respect to the projection map $H \times \C^*_R \to H$.  Hence, we can pull back the $H$-equivariant ample family of line bundles on $T$ to $\tot E$ and twist by all characters of $\C^*_R$ to obtain a $H \times \C^*_R$-equivariant family of ample line bundles on $\tot E$.  
\end{proof}

\begin{remark}
The two lemmas above insure that the fundamental factorization will provide a well-defined Fourier--Mukai transform (see Diagram~\ref{e:DFM}).  We use this fact implicitly in the remainder of the article.
\end{remark}

\subsection{Independence of choices}\label{s:ind}
In this section, we show that the various choices made in this construction do not have a large affect on the fundamental factorization.
More precisely, two different choices yield two different factorizations related by pushforward (see Definition \ref{relatedbypush}).
In \S\ref{s:GLSMinvs}, we will see that GLSM invariants are thus independent of the choices made.

\subsubsection{Choice of $\alpha^\vee$}
In equation \eqref{eq: alphachainmap}, the map $\alpha^\vee$ is determined only up to a homotopy.  We must therefore check that the PV factorization $\{-\alpha, \beta\}$ depends only on the homotopy class of $\alpha^\vee$.  This is proven in \cite{PV}:
\begin{lemma}[{\cite[Section~4.3]{PV}}]
Given two homotopic maps $\alpha^\vee$ and ${\alpha}'^\vee$ realizing \eqref{eq: alphachainmap}, there is an induced isomorphism in $\dabsfact{[\tot \bar A/\CC^*_R],  -\op{ev}^* (\boxplus_{i=1}^r w)}$ between the corresponding PV factorizations $\{-\alpha, \beta\}$ and $\{-\alpha ', \beta\}$.
\end{lemma}
\begin{proof}
We summarize the explanation given at the beginning of \S4.3 of \cite{PV}.  Namely, a homotopy is a map
\[
h : \op{Sym}^{d-2} A \otimes \bigwedge\nolimits^2 B \to \O_S.
\]
The endomorphism $ \bullet \ \lrcorner  \exp(-h^\vee)$ of the exterior algebra on $B$ gives a $\Gamma$-equivariant isomorphism between $\{-\alpha, \beta\}$ and $\{-\alpha', \beta \}$ with inverse   $ \bullet \ \lrcorner  \exp(h^\vee)$ .
\end{proof}

\subsubsection{Related by pushforward}
The following definition will be used in proving that GLSM invariants do not depend on the various choices of resolutions.
\begin{definition}\label{relatedbypush}
We say two factorizations $K \in \dabsfact{[\square/\CC^*_R],  \op{ev}^* (\boxplus_{i=1}^r w)}$ and $K' \in \dabsfact{[\square '/\CC^*_R],  \op{ev}^* (\boxplus_{i=1}^r w)}$ are \newterm{related by pushforward} if there exists
\begin{enumerate}
\item a set of smooth spaces $\square_j \to U_j$  for $1 \leq j \leq n$ and $\square_{j, j+1} \to U_{j,j+1}$ for $1 \leq j \leq n-1$, each with $\CC^*_R$-equivariant evaluation maps $\op{ev}^i: \square_j \to I[\VmodtG]$ for each marked point $1 \leq i \leq r$;
\item a set of factorizations 
\[K_j \in \dabsfact{[\square_j/\CC^*_R], - \op{ev}^* (\boxplus_{i=1}^r w)} \text{ and } K_{j,j+1} \in \dabsfact{[\square_{j, j+1}/\CC^*_R], - \op{ev}^* (\boxplus_{i=1}^r w)};\] and
\item a commuting diagram of stacks over $\sMbar_{g,r}$ where all diagonal arrows are closed immersions
\[
\begin{tikzcd}
& \square_1 \ar[dd] \ar[dl, hook, swap,  "f = l_1"] \ar[dr, hook, "r_1"]  \ar[dd] & &  \ar[dl, hook, swap, "l_2"] \cdots \ar[dr, hook, "r_{n-1}"] & &
\ar[dd] \square_n \ar[dl, swap, hook, "l_n"] \ar[dr, hook, "f' = r_n"] & \\
\square \ar[dd] & &  \ar[dd]\square_{1,2} & \cdots & \ar[dd] \square_{n-1, n}& &\ar[dd] \square ' \\
& U_1 \ar[dl, hook, swap,  "g = \bar l_1"] \ar[dr, hook, "\bar r_1"] & & \ar[dl, hook, swap, "\bar l_2"] \cdots \ar[dr, hook, "\bar r_{n-1}"] & &
 U_n \ar[dl, swap, hook, "\bar l_n"] \ar[dr, hook, "g' = \bar r_n"] & \\
U & & U_{1,2} & \cdots & U_{n-1, n}& & U '
\end{tikzcd}
\]
\end{enumerate}
such that there are isomorphisms
\[\mathbb R f_* (K_1) \cong K, \text{  } \mathbb R f'_*( K_n) \cong K',\] and \[\mathbb R (r_j)_* (K_j) \cong K_{j, j+1} \cong \mathbb R (l_{j+1})_*(K_{j+1}) \text{ for }1 \leq j \leq n-1.\]
\end{definition}
We will see in \S\ref{s:GLSMinvs} that if factorizations are related by pushforward, they define the same GLSM invariants.
Thus the purpose of this section is to show that the different choices of resolutions made in the two step procedure yield factorizations which are related by pushforward. 
  
In the remainder of this section we will make heavy use of the following proposition, which is proven in \cite{PV}.
\begin{proposition}[{\cite[Proposition 4.3.1]{PV}}] \label{PV original}
Let $V$ be a vector bundle on a smooth global quotient stack $\cX$, let $w \in \op{H}^0(\cX, \cL)$ be a superpotential, and let $\{-\alpha, \beta\}$ be the Koszul factorization associated to sections $\alpha \in \op{H}^0(\cX, V^\vee \otimes \cL)$ and $\beta \in \op{H}^0(\cX, V)$ satisfying $\langle \beta,  \alpha \rangle = w$.  Let $V_1 \subseteq V$ be a subbundle such that $\beta \text{ mod } V_1$ is a regular section of $V / V_1$.  Assume that the zero locus $\cX' = Z(\beta \text{ mod }V_1)$ is smooth
 and consider the induced sections
\[
\beta' = \beta|_{\cX'} \in \op{H}^0(\cX', V / V_1) \text{ and } \alpha' = \alpha|_{\cX'} \in \op{H}^0(\cX', (V / V_1)^\vee).
\]
Assume also that either $w|_{\cX'}$ is a non-zero-divisor or $w = 0$ and the zero loci $Z(\alpha, \beta)$ and $Z(\alpha', \beta')$ are proper.  Then one has an isomorphism
\[
\{\alpha, \beta \} \cong \mathbb R i_* \{\alpha', \beta' \}
\]
in $\dabsfact{\cX, w}$ where $i : \cX' \to \cX$ is the inclusion.
\end{proposition}
  
\subsubsection{Different choices of evaluation map}
The Koszul factorization $K$ lies in
$$\dabsfact{[\square/\CC^*_R],  \op{ev}^* (\boxplus_{i=1}^r w)}$$
and thus implicitly depends on the choice of evaluation maps $\op{ev}_{\bar A}^i$ of \eqref{eq: evchainmap}.  In this section, we fix a resolution $[\bar A \to \bar B]$ and vary the evaluation map, ultimately showing that the result is related by pushforward. 

Let ${\op{ev}}_{\bar A}: \bar A \to \pi _* ( \mathcal V |_{\sG} )$ denote the direct sum of the evaluation maps ${\op{ev}_{\bar A}^i}$ as $i$ ranges from $1 $ to $r$.
As a map between cochain complexes, the evaluation map $\op{ev}_{\bar A}$ is determined only up to homotopy.  
Given two distinct cochain level realizations $\op{ev}_{\bar A}$ and $\op{ev}_{\bar A} '$ of the map $[\bar A \to \bar B] \to   \pi _* ( \mathcal V |_{\sG} )$ for which the resolution $[\bar A \to \bar B]$ is admissible, let $\op{ev}_{\bar A}$ and $\op{ev}_{\bar A} ' $ denote the corresponding induced maps $\tot(\bar A) \to I[V/G]$.
Let $\{-\alpha, \beta\}$ and $\{-\alpha ', \beta '\}$ denote the Koszul factorizations in $\dabsfact{[\tot \bar A/\CC^*], -\op{ev}^*(\boxplus_{i=1}^r w)}$ and $\dabsfact{[\tot \bar A/\CC^*], -\op{ev}^{\prime *}(\boxplus_{i=1}^r w)}$.

Let $Q$ denote the vector bundle $\pi _* ( \mathcal V |_{\sG} )$ over $U$.  For our two choices of cochain level evaluation maps $\op{ev}_{\bar A}$ and $\op{ev}_{\bar A} ' $, there exists a homotopy diagram:
\[
  \begin{tikzcd}
  \bar A \ar[r, "\bar d"] \ar[d, bend left, "\op{ev}_{\bar A}"] \ar[d, bend right, swap, "\op{ev}_{\bar A} '  "]
  & \bar B \ar[dl, "h"] \\
   Q &
   \end{tikzcd}
  \]
  such that $\op{ev}_{\bar A} - \op{ev}_{\bar A} '  = h \circ \bar d$.  
  We extend the above diagram as follows.  Let 
\[
 \widehat A := \bar A \oplus Q \text{ and }  \widehat B := \bar B \oplus Q.
\]
We then define
\begin{align*}
\hat d: \bar A & \to \bar B \\
(a,q) & \mapsto ( \bar d(a),q)
\end{align*}
i.e. $\hat d = (\bar d, \op{id}_{Q})$.
  Define further the maps ${\op{ev}}_{\widehat A} := (\op{ev}_{\bar A}, \op{id})$ and ${\op{ev}}_{\widehat A} ' := (\op{ev}_{\bar A} ' , 0)$ from $\widehat A$ to $Q$.  We observe that the map $\widehat h := (h, \op{id})$ yields a homotopy: ${\op{ev}}_{\widehat A} - {\op{ev}}_{\widehat A} ' =\hat h \circ  \hat d$.
  
\begin{lemma}
The complex $[ \widehat A \xrightarrow{\hat d} \widehat B]$ gives an admissible resolution of
 $\mathbb R \pi_* \cV$ with the evaluation map given by either ${\op{ev}}_{\widehat A}$ or ${\op{ev}}_{\widehat A} '$.
\end{lemma}
 \begin{proof}
 It is obvious that $[ \widehat A \xrightarrow{\hat d} \widehat B] \cong \mathbb R \pi_* \cV$.  We must check that Conditions~\ref{assume ev chainlevel}, \ref{assume alpha chainlevel}, and \ref{compatible} are satisfied for the resolution $[\widehat A \to \widehat B]$ using either evaluation map. Condition~\ref{assume ev chainlevel} is immediate.  To check Conditions \ref{assume alpha chainlevel} and \ref{compatible}, we define the following maps:
\[\widehat Z : \Sym \widehat A  \xrightarrow{\op{proj}} \Sym \bar A \xrightarrow{Z} \O_U^r\] and 
\[\widehat{\alpha}^\vee: \Sym \widehat A \otimes \widehat B \xrightarrow{\op{proj}} \Sym \bar A \otimes \bar B \xrightarrow{\alpha^\vee} \O_U.\]
It is clear that \eqref{eq: alphachainmaporiginal} commutes with $\widehat Z$ and $\widehat{\alpha}^\vee$ replacing $Z$ and $\alpha^\vee$ and that $\widehat{\alpha}^\vee|_{\op{Sym}^{d-1} A \otimes B}$ represents $\R(\pi_d)_*f_d$.  Thus Condition \ref{assume alpha chainlevel} holds.  Furthermore Condition \ref{compatible} holds after replacing $Z$ and $\op{ev}^i_A$ with $\widehat Z$ and ${\op{ev}}_{\widehat A}^i$.  This proves that $[ \widehat A \xrightarrow{\hat d} \widehat B]$ gives an admissible resolution with respect to ${\op{ev}}_{\widehat A}$.

Recall from the discussion surrounding \eqref{derived compatibility}, the map $\pi_*{\Sigma_i}_*{\Sigma_i}^* \varkappa_w \circ \op{natural}: \Sym Q \to \O^r_U$.  Label this map $Z_w$.  Define
\[\widehat Z ': \Sym \widehat A \xrightarrow{\op{proj}} \Sym \bar A \oplus \Sym Q \xrightarrow{ Z ' + Z_w} \O_U^r.\]
It is clear to see that Condition \ref{compatible} holds after replacing $Z$ and $\op{ev}^i_A$ with $\widehat Z '$ and ${\op{ev}}_{\widehat A}^{i \prime}$.  By quasi-homogeneity, one can easily check using Euler's homogeneous function theorem that there exists a map $\alpha_w^\vee$ such that 
\begin{equation}
\label{eq: alphaw}
\begin{CD}
\op{Sym} Q @>>>  \op{Sym} Q \otimes Q \\
@V Z_w VV @V \alpha_w^\vee VV\\
 \O_{U}^r@>\op{sum}>> \O_{U}
\end{CD}
\end{equation}
commutes.
Using this and  ${\alpha}'^\vee$, one can construct a map
\[(\widehat{\alpha}')^\vee: \Sym \widehat A \otimes \widehat B \to \O_U\]
 such that \eqref{eq: alphachainmaporiginal} commutes after replacing $Z, \alpha^\vee$ with $\widehat Z ', (\widehat{\alpha}')^\vee$ such that $(\widehat{\alpha}')^\vee|_{\op{Sym}^{d-1} A \otimes B}$ represents $\R(\pi_d)_*f_d$.  Thus Condition \ref{assume alpha chainlevel} holds as well.  This proves that $[ \widehat A \xrightarrow{\hat d} \widehat B]$ gives an admissible resolution with respect to ${\op{ev}}_{\widehat A} '$.
\end{proof}

We will  denote by $ \widehat{\op{\ev}}$ and $\widehat{\op{ev}}'$ the corresponding geometric evaluation maps from $\tot(\widehat A)$ to $I[V/G]$.
Let $\{-\widehat{\alpha}, \widehat{\beta}\}$  respectively $\{-\widehat{\alpha} ', \widehat{\beta} '\}$ denote the PV factorizations, in $\dabsfact{[\tot \widehat  A/\CC^*], -\widehat{\op{ev}}^*(\boxplus_{i=1}^r w)}$ respectively $\dabsfact{[\tot \widehat A/\CC^*], -\widehat{\op{ev}}^{\prime *}(\boxplus_{i=1}^r w)}$, defined by the admissible resolution $[\widehat A \to \widehat B]$ and the maps ${\op{ev}}_{\widehat A}$ and  $\widehat{\alpha}^\vee$ respectively
 ${\op{ev}}_{\widehat A} '$ and  $(\widehat{\alpha}')^\vee$.

The map from $\widehat A$ to $A_1$ is given by the projection to $\bar A$ composed with the map $\bar A \to A_1$.  
The locus $\tot \bar A^\circ$ (defined just before Definition~\ref{d:square}) depends on which evaluation map is used, thus so does the definition of $\square$.  Let $\square$ and $\square '$ denote the two loci in $\tot \bar A$ given by Definition~\ref{d:square}, using the evaluation maps $\op{ev}_{\bar A}$ and $\op{ev}_{\bar A} ' $ respectively.  Applying Definition~\ref{d:square} to $[\widehat A \to \widehat B]$, we see the the corresponding loci in $\tot \widehat A$ are equal to $\square \times \tot Q$ and $\square ' \times \tot Q$ respectively.
Let $i$ and $i ' $ denote the inclusions of $\square $ and $\square '$ into $\square \times \tot Q$ and $\square ' \times \tot Q$.
 
\begin{lemma}
We have the following relations: 
\[\mathbb R i_*\{- \alpha, \beta\} \cong \{-\widehat{\alpha}, \widehat{\beta}\} \text{ and }\mathbb R i_* '\{- \alpha ', \beta '\} \cong \{-\widehat{\alpha} ', \widehat{\beta} '\}.\]
\end{lemma}

\begin{proof}
This is a direct application of Proposition~\ref{PV original}.  In this case $B$ is viewed as a sub-bundle of $\widehat B$ via the natural inclusion.  Note that $\widehat{\beta} \mod B$ is simply the tautological section to $Q$ and so trivially a regular section.
\end{proof}
 
Define
\[L := \ker({\op{ev}}_{\widehat A} - {\op{ev}}_{\widehat A} ') \text{ and } M:= \ker(\hat h).
\]
Since both ${\op{ev}}_{\widehat A} - {\op{ev}}_{\widehat A} '$ and $\hat h$ are clearly surjective over all of $U$, the sheaves $L$ and $M$ are both vector bundles.  Consider the inclusion $\tot L \hookrightarrow \tot \widehat A$.  Since the evaluation maps $\widehat\ev$ and  $\widehat\ev '$ are equal on $\tot L$, the inclusion 
induces pushforwards to both $\dabsfact{[\tot \widehat  A/\CC^*], -\widehat{\op{ev}}^*(\boxplus_{i=1}^r w)}$ and $\dabsfact{[\tot \widehat A/\CC^*], -\widehat{\op{ev}}^{\prime *}(\boxplus_{i=1}^r w)}$.  

Let $\left(\tot L\right)^\circ$ denote the open locus where $\widehat{\op{ev}}|_L$ maps to $\left(I\cc T\right)^r$, where we recall $\cc T = \VmodtG \hookrightarrow [V/G]$.
Let $\square_L$ denote the locus $\{\zeta = 0\}  \subseteq \left(\tot L\right)^\circ$ of Definition \ref{d:square}.  It is easy to check that the map $L \to \widehat A \to A_1$ is still surjective, so $\square_L$ is smooth.  
 
Let $j: \square_L \to \square \times \tot Q$ and $j': \square_L \to \square ' \times \tot Q$ denote the inclusions.
\begin{lemma}
The complex $[L \to M]$ gives an admissible resolution of $\mathbb R \pi_* \cV$.  The corresponding PV factorization $\{ - \alpha '', \beta ''\}$ satisfies the following:
\begin{align*}
\mathbb R j_*\{ - \alpha '', \beta ''\} &\cong \{-\widehat{\alpha}, \widehat{\beta}\} \\
\mathbb R j ' _*\{ - \alpha '', \beta ''\} &\cong \{-\widehat{\alpha} ', \widehat{\beta} '\}.
\end{align*}
\end{lemma}
 
\begin{proof}
It is an easy exercise to check that $[L \to M]$ is quasi-isomorphic to $[\widehat A \to \widehat B]$ and therefore to $\mathbb R \pi_* \cV$.  We define the evaluation map $L \to Q$ to simply be the restriction of $\widehat \ev$ or $\widehat \ev '$ (they are equal on $L$ by construction).
To check surjectivity consider the following.  Given $p \in Q$, choose $a \in \bar A$ such that $\bar \ev '(a) = p$.
This is possible since $\bar \ev '$ is surjective.
Let $p' = \bar \ev(a) - p$.
Then $\widehat \ev(a, p) - \widehat \ev '(a, p) = \bar \ev(a) + p' - \bar \ev '(a) = 0$.
Therefore $(a, p) \in \widehat A$ actually lies in $L$, and $\widehat \ev '(a, p') = p$.
This shows that Condition~\ref{assume ev chainlevel} is satisfied.
To verify Condition \ref{assume alpha chainlevel}, we simply restrict the map $\widehat{\alpha}^\vee: \Sym \widehat A \otimes \widehat B \to \O_U$.
Condition \ref{compatible} is apparent by construction.
 
Let $\{ - \alpha '', \beta ''\}$ denote the corresponding PV factorization on $[\square_L/\CC^*_R]$.  Applying Proposition~\ref{PV original} to each of $\{-\widehat{\alpha}, \widehat{\beta}\}$ and $\{-\widehat{\alpha} ', \widehat{\beta} '\}$ yields the result.  In the notation of Proposition~\ref{PV original}, the vector bundle $V$ is given by $\widehat B$ and $V_1$ is $M$.  The map $
\bar \beta \mod M$ is equal to the section ${\op{ev}}_{\widehat A} - {\op{ev}}_{\widehat A} '$ defining $\tot L$ inside $\tot \widehat A$ (and $\square_L$ inside both $\square \times Q$ and $\square ' \times Q$), and therefore  regular.
\end{proof}

We arrive at the desired statement.
\begin{proposition}
The factorizations $\{- \alpha, \beta\}$ and $\{- \alpha ', \beta '\}$ are related by pushforward.
\end{proposition}

\begin{proof}
By the previous two lemmas we have $ \mathbb R i_*\{- \alpha, \beta\} \cong \{-\widehat{\alpha}, \widehat{\beta}\} \cong \mathbb R j_*\{ - \alpha '', \beta ''\}$  and  $\mathbb R i_*\{- \alpha ', -\beta '\} \cong \{-\widehat{\alpha} ', \widehat{\beta} '\} \cong \mathbb R j ' _*\{ - \alpha '', \beta ''\}$.
The conclusion follows.
\end{proof}
  
 \subsubsection{Different choices of resolution}
  
 We next consider the choice of resolution $[\bar A \to \bar B]$ over $U$.

\begin{lemma}\label{l:r2}
Given two different choices of admissible resolutions $[\bar A \to \bar B]$ and $[\bar A' \to \bar B']$ of $\mathbb R\pi_*(\cc V)$ over $U$, the factorizations $\{\alpha, \beta\} $ and  $\{\alpha ', \beta '\}$ are related by pushforward.
\end{lemma}

\begin{proof}
By Lemma~\ref{l:surjective}, there exists a roof diagram realizing the quasi-isomorphism between $[\bar A \to \bar B]$ and $[\bar A' \to \bar B']$ such that the roof is a two term complex $[\bar A '' \to \bar B '']$ of vector bundles and all maps are surjective.  By Lemma~\ref{l:surjadmi}, the roof provides another admissible resolution.  The evaluation map on $\bar A ''$ may be induced either by the map to $\bar A$ or the map to $\bar A '$.  These may not agree, however by the previous section the corresponding factorizations are related by pushforward.  

Thus we reduce to the situation that there exists a morphism of cochain complexes
$[\bar A \to \bar B]\to [\bar A' \to \bar B']$
realizing the quasi-isomorphism, and such that $[\bar A \to \bar B]$ is admissible and the evaluation map $\op{ev}_{\bar A}$ from $\bar A$ factors through $\bar A '$.  Up to homotopy, this map decomposes as 
\[
[\bar A \to \bar B]\to [\bar A'\oplus \bar B \to \bar B'\oplus \bar B] \to [\bar A' \to \bar B']. \]
The map $\bar A \to \bar A'\oplus \bar B$ is injective due to the fact that 
$[\bar A \to \bar B]\to [\bar A' \to \bar B']$ is a quasi-isomorphism.
The second map has (up to homotopy) a right inverse consisting of injections, thus we may assume without loss of generality that
the maps of vector bundles in $$[\bar A \to \bar B]\to [\bar A' \to \bar B']$$
are injective.  Because the map $\bar A \to A_1$ factors through $\bar A '$, as does the evaluation map $\op{ev}_{\bar A}$, the closed immersion $\tot(A) \to \tot(\bar A ')$ induces an immersion $\square \to \square '$.  The section $\beta ' \mod \bar B$ defines the locus $\square  \subseteq \square '$, therefore by Proposition~\ref{PV original},
the Koszul factorization  
 $\{-\alpha ', \beta '\}|_\square$ is isomorphic to the pushforward of $\{-\alpha, \beta\}|_{\square '}$ under the map $\square \to \square '$.
 \end{proof}

The resolution $[{A_1} \xrightarrow{d_1} {B_1}]$ of $\mathbb R\pi_*(\cc V_1)$ in Lemma~\ref{UY} and consequently the closed immersion of $LG_{g,r}(\cX, d) \to \sKbar_{g,r}(\Xrig, d)$ in Proposition~\ref{p: projmod} depend on a choice of closed immersion from the universal curve $\underline{\mathcal C} \subseteq \PP^{N-1}$ (recall \S~\ref{ssqp}).  We next show that two fundamental factorizations coming from different choices of this immersion are related by pushforward.  

Consider two different resolutions $i: \underline{\mathcal C} \to \PP^{N-1}$ and $i': \underline{\mathcal C} \to \PP^{M - 1}$.  Pulling back the first two terms of the Euler sequence on $\PP^{N-1}$ resp. $\PP^{M-1}$ to $\cC_{BG_1}$ yields two vector bundles with sections: 
\[\cO_{\cC_{BG_1}} \ra (\cN_{BG_1})^{\oplus N} \text{ resp. } \cO_{\cC_{BG_1}} \ra (\cc M_{BG_1} )^{\oplus M }.\]
Tensoring the two yields a section of $(\cN_{BG_1})^{\oplus N} \otimes (\cc M_{BG_1})^{\oplus M }$ which defines the map 
$\cC_{BG_1} \to \PP^{N + M - 1}$ coming from the Segre embedding $\PP^{N-1} \times \PP^{M - 1} \to \PP^{N + M -1}$.  Note the section factors as
\[\cO_{\cC_{BG_1}} \ra \cN_{BG_1}^{\oplus N} \to (\cN_{BG_1})^{\oplus N} \otimes {\cc M_{BG_1} }^{\oplus M }.\]
Let $N' := N +M$ and let $ \cN_{BG_1} ' := \cN_{BG_1} \otimes {\cc M}_{BG_1}$.  Tensoring the above inclusions by $\cV_1$ yields the following map of short exact sequences
\begin{equation}\label{e:2res}
\begin{tikzcd}
0 \ar[r] & \cV_1 \ar[r] \ar[d, "\op{id}"]& \cV_1 \otimes \cN_{BG_1}^{\oplus N} \ar[d, hookrightarrow] \ar[r] &\cc Q \ar[r] \ar[d] & 0 \\
0 \ar[r] & \cV_1 \ar[r] & \cV_1 \otimes (\cN_{BG_1} ' )^{\oplus N ' } \ar[r] & \cc Q ' \ar[r] & 0
\end{tikzcd}
\end{equation}
where $\cc Q$ and $\cc Q '$ are defined as the respective cokernels.  In the notation of Lemma~\ref{UY}, let $A_1 := \mathbb{R} \pi_*  \cV_1 \otimes \cN_{BG_1}^{\oplus N}$ resp. $A_1 ' := \cV_1 \otimes (\cN_{BG_1} ' )^{\oplus N ' }$ and $B_1 := \mathbb{R} \pi_* \cc Q$ resp. 
$B_1 ' := \mathbb{R} \pi_* \cc Q '$.  

Following the embedding construction of \S~\ref{ssqp} and \S~\ref{genhybmod} with each of the two resolutions of $\cV_1$ from \eqref{e:2res}, we obtain open substacks $U \subset \tot(A_1)$ resp. $U' \subset \tot(A_1 ')$ over $ \Mfrak^\orb_{g,r}(B\Gamma, d)_{\log}^\circ$ such that all the properties of Theorem~\ref{c:swn} hold.  Furthermore, from the inclusion $A_1 \hookrightarrow A_1 '$ obtained from \eqref{e:2res}, one checks that there is an induced morphism $U \hookrightarrow U'$.  
By replacing $U '$ with a smaller open subset if necessary, we may assume we are in the following situation.

\begin{reduction}\label{red} Given two resolutions  $[{A_1} \to {B_1}] $ and $ [{A_1}' \to {B_1}']$ of $\mathbb R\pi_*(\cc V_1)$ over $\Mfrak^\orb_{g,r}(B\Gamma, d)_{\log}^\circ$, constructed as in \S~\ref{ssqp},
we can assume without loss of generality that there exists a quasi-isomorphism
$$f \colon [{A_1} \to {B_1}] \to [{A_1}' \to {B_1}']$$
realized at the cochain level, such that the maps $f_{A_1} \colon {A_1} \to {A_1}'$ and $f_{B_1}: {B_1} \to {B_1}'$ are injective, and $U  \subseteq \tot({A_1})$ is the preimage of $U'  \subseteq \tot{{A_1}'}$ under $f$.   We can assume further that there exist admissible resolutions on both $U$ and $U'$.
\end{reduction}
\begin{lemma}\label{l:r1}  With the situation given as above, 
there exist admissible resolutions $[\bar A \to \bar B]$ of $\mathbb R\pi_*(\cc V)$ over $U$ and $[\bar A' \to \bar B']$ of $\mathbb R\pi_*(\cc V)$ over $U'$ together with a closed immersion $\bar f: \tot(\bar A) \to \tot(\bar A')$ such that the diagram
\[
\begin{tikzcd}
 \square \arrow{d} \arrow[r, hookrightarrow, "\bar f"] & \square ' \arrow{d}\\
 U \arrow[r, hookrightarrow, "f"] & U'
\end{tikzcd}
\]
commutes, and the pushforward $\mathbb R \overline f_*(\{-\alpha, \beta\})$ is isomorphic to $\{-\alpha ', \beta '\}$ in the derived category of factorizations.
\end{lemma}

\begin{proof}
By abuse of notation, we let ${A_1}, {A_1}', {B_1}, {B_1}'$ denote the pulled-back vector bundles over $U'$.  Let $\bar A ' \to \bar B'$ denote a resolution of $\mathbb R\pi_*(\cc V)$ over $U'$ which surjects onto $[{A_1}' \to {B_1}']$.  Let $C$ denote the cokernel of $f: {A_1} \to {A_1}'$.  We have the following commutative diagram:
\[
\begin{tikzcd}
 & \bar B \ar[dd, two heads] \ar[rr, hookrightarrow, "\bar g" near end] & &\bar B' \ar[dd, two heads]  \ar[rr] & & C \ar[dd, "\cong"] \\
 \bar A \ar[dd, crossing over, two heads, "\rho" near start] \ar[ur] \ar[rr, hookrightarrow, crossing over, "\bar f" near end] & & 
                                                    \bar A '   \ar[ur, "\bar d '"] \ar[rr, crossing over, "\bar \gamma" near end] 
                                                                              &    &C \ar[ur, "\cong"]  & \\
  &  {B_1}  \ar[rr, hookrightarrow, "g" near end] & & {B_1}'  \ar[rr] & & C \\
  {A_1} \ar[ur] \ar[rr, hookrightarrow, " f" near end] &&  A_1 ' \ar[uu, crossing over, twoheadleftarrow, "\rho ' " near start] \ar[ur, " d '"] \ar[rr, " \gamma" near end] &&C \ar[ur, "\cong"] \ar[uu, crossing over, twoheadleftarrow, "\cong" near start]    &
 \end{tikzcd}
\]
where $\bar \gamma$ is the composition $\bar A' \to {A_1}' \to C$, the vector bundle $\bar A = \text{ker}(\bar \gamma)$, and similarly $\bar B$ is the kernel of the composition from $\bar B'$ to $C$.  Note that $\bar A$ surjects onto ${A_1}$, and that $[\bar A \to \bar B]$ is quasi-isomorphic to $[\bar A ' \to \bar B']$.

Let $\tot(\bar A|_U)$ and $\tot(\bar A')$ denote the total spaces of $\bar A$ and $\bar A'$ over $U$ and $U'$ respectively. Let $p:T \to U$ and $p': T' \to U'$ denote the projections.  Recall that $p^*(A_1)$ and ${p '}^*(A_1 ')$ have sections
\[\zeta  = {p}^*(\rho  \circ \text{taut}_{\bar A }) - {p}^*(\text{taut}_{{A_1} })\] and \[\zeta ' = {p'}^*(\rho ' \circ \text{taut}_{\bar A '}) - {p'}^*(\text{taut}_{A_1 '})\] whose zero loci define $\square $ and $\square '$ respectively.  The above diagram shows the map $\tot(\bar A|_U) \to \tot(\bar A')$ induced by $\bar f$ sends $\square$ into $\square '$.  In fact $\square$ is the zero locus of the section ${p'}^* (\gamma \circ \text{taut}_{ A_1 '}) \in \Gamma(\square ' , {p'}^*( C))$.

On $\square '$ we have the obstruction bundle $E' = {p'}^*(\bar B')$ which contains $E = {p'}^*(\bar B)$ as a sub-bundle.  $E'$ has a section $\beta ' = \bar d ' \circ \text{taut}_{\bar A '}$.  Note that $E' / E ={p'}^*(\bar B'/\bar B)= {p'}^*(C)$ and $\beta '\text{mod}(E)$ is given by 
\[\beta' \text{mod}(E) = {p'}^* (\bar \gamma \circ \text{taut}_{\bar A '})= {p'}^* (\gamma \circ \text{taut}_{ A_1 '}),\]
where the last equality holds since we are on $\square '$.
By Proposition~\ref{PV original}, the claim follows.
\end{proof}

\begin{proposition}\label{p:ind}
Any two Koszul factorizations 
\[\{-\alpha, \beta\} \in \dabsfact{[\square/\CC^*_R], - \op{ev}^* (\boxplus_{i=1}^r w)} \text{ and } \{-\alpha ', \beta '\} \in \dabsfact{[\square '/\CC^*_R], - \op{ev}^* (\boxplus_{i=1}^r w)}\] 
constructed via the two-step procedure are related by pushforward.
\end{proposition}
\begin{proof}

Given two resolutions $[{A_1} \xrightarrow{d_1} {B_1}]$ and $[A_1 ' \xrightarrow{d_1 '} {B_1}']$ 
of $\mathbb R\pi_*(\cc V_1)$ over $\Mfrak^\orb_{g,r}(B\Gamma, d)_{\log}^\circ$ 
 and corresponding open sets $U  \subseteq \tot({A_1})$ and $U'  \subseteq \tot({A_1}')$, by Reduction~\ref{red} we can assume that these two complexes are related by a quasi-isomorphism
 \[f: [{A_1} \to {B_1}] \to [{A_1}' \to {B_1}']\]
 such that 
 \begin{itemize} 
 \item the maps $f_A:{A_1} \to {A_1}'$ and $f_B: {B_1} \to {B_1}'$ are injective; and
 \item  $U  \subseteq \tot({A_1})$ is the preimage of $U'  \subseteq \tot{{A_1}'}$ under $f$.
 \end{itemize}
 In this particular case, by Lemma~\ref{l:r1} 
 there exist resolutions $[\bar A \to \bar B]$ of $\mathbb R\pi_*(\cc V)$ over $U$ and $[\bar A' \to \bar B']$ of $\mathbb R\pi_*(\cc V)$ over $U'$
 together with a closed immersion $\bar f: \tot(\bar A) \to \tot(\bar A')$ such that the factorization 
 $\{-\alpha, \beta\}$ on $\square$ is equivalent to the factorization $\{-\alpha ', \beta '\}$ on $\square '$ by pushforward under $\bar f|_\square$.
 
Thus we reduce to the case that $U$ is fixed. Then by Lemma~\ref{l:r2} we conclude that the factorizations are related by pushforward.
\end{proof}

\subsection{Rigidified evaluation and the state space}\label{s:rigev2}
\subsubsection{Rigidified evaluation}\label{ss:rigev}
Recalling  Definition~\ref{d:square} there exist $\CC^*_R$-equivariant evaluation maps 
\begin{equation}\label{evsquare}
\op{ev}^i \colon\square \to I\cc T  =  \coprod_{(g)}[ (V^{ss}(\theta))^g) / C_G(g)]
\end{equation}
where $(g)$ runs over all conjugacy classes in $G$.  For each conjugacy class $(g)$, we choose a representative $g$; then $C_G(g)$ denotes the centralizer of $g$ in $G$, and $(V^{ss}(\theta))^g$ denotes the points of $V^{ss}(\theta)$ which are fixed by $g$.
These maps combine to form a map
\[
\op{ev} \colon  \square \to \left(I\cc T\right)^r = \coprod_{\bg} \left[ \big(\bigoplus_{i=1}^r (V^{ss}(\theta))^{g_i} / \prod_{i=1}^r C_G(g_i) \right].
\]
where $\bg$ runs over all ordered sets of $r$ conjugacy classes $(g_1), \ldots, (g_r)$.

As described in \S\ref{s:rigev}, the map $\op{ev}$ is $\CC^*_R$-equivariant with respect to the natural action on $\op{tot} A$ and the diagonal action of $\CC^*_R$ on $V^r$.  Thus it induces a map (which we will also denote by $\op{ev}$) from $[\op{tot} A /\CC^*_R ]$ to the quotient of $(I[\VmodtG])^r$ by $\CC^*_R$. 
Also, notice that the action of $G^r \times \CC^*_R$ has a generic stabilizer given by $\lan J\ran $. 
$\lan J\ran  = G \cap \CC^*_R \subseteq \op{Gl}(V)$.  
This gives an exact sequence,
\[ \begin{array}{ccccccccc}
1          &  \to & \lan J  \ran   &  \to &  G^r \times \CC^*_R & \to           &  \overbrace{\Gamma \times_{\CC^*} \cdots \times_{\CC^*} \Gamma}^{r-\text{times}} & \to &  1 \\
&    &                      &      & (g_1, ..., g_r, t)           & \mapsto  & (g_1t, g_2t , ..., g_rt)                   &  &
\end{array} \] 
where $\Gamma \times_{\CC^*} \cdots \times_{\CC^*} \Gamma $ denotes the fiber product over the character $\chi$, that is,
$\Gamma \times_{\CC^*} \cdots \times_{\CC^*} \Gamma $ is  the kernel of map 
$ \Gamma \times \cdots \times \Gamma \ra \CC  ^* \times \cdots \times \CC ^*$ sending  \[  (\gamma _1, ..., \gamma _r) \mapsto (\chi (\gamma _1)\chi (\gamma _2 )^{-1} , 
\chi (\gamma _2) \chi (\gamma _3)^{-1}, ..., \chi (\gamma _{r-1}) \chi (\gamma _{r})^{-1}). \]
Analogously,
\[
1 \to \lan J\ran  \to \prod_{i=1}^r C_G(g_i) \times \CC^*_R \to C_\Gamma(g_1) \times_{\CC^*} \cdots \times_{\CC^*} C_\Gamma(g_r) \to 1
\] 
where $C_\Gamma(g) = C_G(g) \cdot \CC^*_R$ is the centralizer of $g$ in $\Gamma$.  
Rigidification gives a map,
\[
\mathrm{rigidify} \colon [V^r / (G^r \times \CC^*_R)] \to [V^r / (\Gamma \times_{\CC^*} \cdots \times_{\CC^*} \Gamma)].
\]
Let $\bg$ denote an ordered set of $r$ conjugacy classes $(g_1), \ldots, (g_r)$.  The map $\mathrm{rigidify}$ induces the map
\[
I_\mathrm{rig} \colon 
\left[ \left( \bigoplus_{i=1}^r (V^{ss}(\theta))^{g_i} \right) / \left( \prod_{i=1}^r C_G(g_i) \times \CC^*_R \right) \right] \to 
\left[ \left( \bigoplus_{i=1}^r (V^{ss}(\theta))^{g_i} \right) / \left( C_\Gamma\bg \right) \right],
\]
where 
$C_\Gamma\bg := C_\Gamma(g_1) \times_{\CC^*} \cdots \times_{\CC^*} C_\Gamma(g_r)$.
Note that the source of $I_\mathrm{rig}$ is simply a component of the quotient of $(I[\VmodtG])^r$ by $\CC^*_R$.
Combined, these observations allow us to make the following definition.
\begin{definition}\label{d:evmap}  
Define the \newterm{evaluation map} as the composition of $I_\mathrm{rig}$ with $\op{ev}$.
\[
\op{ev}  = I_\mathrm{rig} \circ \op{ev} \colon [\square/\CC^*_R]  \to  \coprod_{\bg} \left[ \big(\bigoplus_{i=1}^r (V^{ss}(\theta))^{g_i} / C_\Gamma\bg \right],
\]
where the union is over all $r$-tuples $\bg$ of conjugacy classes in $G$.
\end{definition}

\subsubsection{The state space of a GLSM}

\begin{definition}\label{d:stsp}
	Define \[ \cc H^{\op{ext}}_{(g)} :=  \op{HH}_*\left([(V^{ss}(\theta))^{g} / C_\Gamma(g)], w\right).\] 
	The \newterm{extended GLSM state space} is defined to be $\cc H^{\op{ext}} := \bigoplus_{(g)} \cc H^{\op{ext}}_{(g)}$  where $(g)$ runs over all conjugacy classes in $G$.
\end{definition}

\begin{remark}\label{r:stsp}
	To better understand the target of the evaluation map $\op{ev}$, note that
	\[ 
	\begin{array}{cl} \op{HH}_* \left( \left[ \big(\bigoplus_{i=1}^r (V^{ss}(\theta))^{g_i} / C_\Gamma\bg \right], {\textstyle \boxplus_{i=1}^r} w \right) &= \bigotimes_{i=1}^r \op{HH}_*\left([(V^{ss}(\theta))^{g_i} / C_\Gamma(g_i)], w\right)\\
	& =\bigotimes_{i=1}^r  \cc H^{\op{ext}}_{(g_i)} \\
	&  \subseteq (\cc H^{\op{ext}})^{\otimes r}.
	\end{array}\]
\end{remark}

\subsection{The restricted state space} \label{s:theresstaspa}
In fact we will deal most often with a subspace of $\cc H^{\op{ext}}$ called the \newterm{restricted state space} and denoted by 
\[\cc H^{res}  \subseteq \cc H^{\op{ext}},\]
similarly to \cite[Equation (5.13)]{PV}.
The definition and properties of the restricted state space are not strictly necessary for the results of this paper, thus we will save the bulk of the discussion for the sequel \cite{CKFGS}.  We record here only a brief summary of the key facts for completeness of exposition.

\begin{proposition}[\cite{CKFGS}]\label{p:decomp}
The vector space $\cc H^{\op{ext}}_{(g)}$ decomposes as a direct sum
\[\cc H^{\op{ext}}_{(g)} = \bigoplus_{(h)} e_{(h)}(\cc H^{\op{ext}}_{(g)}),\]
where $(h)$ runs over all conjugacy classes of $G$.
\end{proposition} 

\begin{definition}
Define the restricted state space to be
\[ \cc H^{res} : = \bigoplus_{(g)} e_{(id)}(\cc H^{\op{ext}}_{(g)}),\]
where $(id)$ denotes the identity conjugacy class.  Thus on each twisted sector $[(V^{ss}(\theta))^{g} / C_\Gamma(g)]$, the subspace $\cc H^{res}$ contains the piece of $\cc H^{\op{ext}}_{(g)}$ corresponding to the identity under the decomposition of proposition~\ref{p:decomp}.
\end{definition}

\begin{remark}
In the following paper \cite{CKFGS}, we will define a pairing on the restricted state space, as is necessary for the definition of a cohomological field theory.
\end{remark}

\begin{remark}\label{r:resPV}
In the special case of a GLSM $(V, G, \CC^*_R, \theta, w)$ where $[\VmodtG]$ is affine, Proposition~\ref{p:decomp} was proven by Polishchuk--Vaintrob (\cite{PV}, Theorem 2.6.1).
\end{remark}

\begin{remark}
Note that when $[\VmodtG]$ is a smooth variety, the vector spaces $\cc H^{\op{ext}}$ and $\cc H^{res}$ are equal.  
\end{remark}

\subsection{GLSM invariants}\label{s:GLSMinvs}

Denote by $\overline{U}$ the closure of $U  \subseteq LG^{eq}_{g,r}(\cc Y, d')$ (see Theorem~\ref{c:swn}).
\begin{lemma} \label{lemma: proj resol}
There exists a resolution of singularities $\widetilde U \to \overline U$ such that
the stack $\widetilde{U}$ is a smooth proper Deligne--Mumford quotient stack with a projective coarse moduli space.
\end{lemma}

\begin{proof}
Since $\overline{U}$ is a separated Deligne--Mumford quotient stack with a projective coarse moduli space,
we can write $\overline{U}$ as a quotient $[Y/G]$ of some quasi-projective scheme $Y$ by a linear action of a 
reductive group $G$ (\cite[Proposition 5.1]{KreschGeometry}).  Since the map from $Y$ to the coarse moduli space is affine (\cite[Remark 4.3]{KreschGeometry}), 
we can now apply \cite[Converse 1.12]{MFK} to find an ample $G$-linearization $L$ on $Y$ such that $Y \subset (\PP ^m)^s$, i.e., every point of $Y$ is stable.
Now we take a resolution $\widetilde{Y}$ of $Y$ by a sequence of blowing-ups of regular centers to yield $\widetilde{U} :=[\widetilde{Y}/G]$.
Denote by $\widetilde{\PP ^m}$ the outcome of the same blowing-up procedure on $\PP ^m$.
By \cite[Section 3]{Kir}, there is an ample $G$-linearization $L'$ on $\widetilde{Y}$ 
such that $\widetilde{Y} \subset (\widetilde{\PP ^m}) ^s$.
Since the coarse moduli space of $[\widetilde{Y}/G]$ is a uniform categorical and uniform geometrical quotient, 
it coincides with the GIT quotient of $\widetilde{Y}$, which is quasi-projective.
\end{proof}

By Definition~\ref{d:evmap} there exists an evaluation map
\[
\op{ev}  \colon  [\square/\CC^*_R]  \to  \coprod_{\bg} \left[ \big(\bigoplus_{i=1}^r (V^{ss}(\theta))^{g_i} / C_\Gamma\bg \right],
\]
where the union is over all $r$-tuples $\bg$ of conjugacy classes in $G$.

Denote by $\widetilde{p} \colon \square \to \widetilde U$ the composition of the projection $p \colon \square \to U$ followed by the open immersion $U \to \widetilde U$.

We use $K = K_{g,r, d}$ to define an integral transform $\Phi_{K_{g,r, d}}$ by the following diagram.
\begin{equation}\label{e:DFM}
\begin{tikzcd}[column sep=small]
\dabsfact{[\square/\CC^*_R],  \op{ev}^* (\boxplus_{i=1}^r w)} \ar[r, "-\overset{\mathbb{L}}{\otimes} K_{g,r, d}"] & \dabsfact{[\square/\CC^*_R], 0}_{LG_{g,r}( \cc Z, d)}  \ar[d, "\R \widetilde p_*"]  \\
\dabsfact{ \coprod_{\bg} \left[ \big(\oplus_{i=1}^r (V^{ss}(\theta))^{g_i} / C_\Gamma\bg \right], {\textstyle \boxplus_{i=1}^r} w} \ar[u, "\mathbb L\op{ev}^*"]  
& \ \ \ \ \ \ \ \ \ \dabsfact{[\widetilde{U}/\CC^*_R], 0} \simeq \dabsfact{[\widetilde{U}/\mu_{\bdeg}]}
\end{tikzcd}
\end{equation}
  Note that by  Lemma~\ref{l:finalsupp},  tensoring with $K_{g,r, d}$ gives a factorization supported on $LG_{g,r}(\cc Z, d)$.  Since the action of $\CC^*_R$ on $LG_{g,r}(\cc Z, d)$ and on $\widetilde{U}$ is trivial, the stack $[LG_{g,r}(\cc Z, d) / \C^*_R]$ is proper over $[\widetilde{U} / \C^*_R]$.  Thus the pushforward $\mathbb R \widetilde{p}_*$ is well-defined.  

The equivalence 
\[\dabsfact{[\widetilde{U}/\CC^*_R], 0} \cong \dabsfact{[\widetilde{U}/\mu_{\bdeg}]}\] is \cite[Proposition 1.2.2]{PV}. Pushing forward to Hochschild homology yields a map
\[ (\Phi_{K_{g,r,d}})_*: (\cc H^{\op{ext}})^{\otimes r} \to \op{HH}_*( [\widetilde{U}/\mu_{\bdeg}]).\]
We further pull back by the quotient map 
\[ q: \widetilde{U} \to [\widetilde{U} /\mu_{\bdeg}]
\]
to obtain a map to $\op{HH}_*( \widetilde{U} )$.

Given an $r$-tuple of conjugacy classes $\bg = (g_1), \ldots, (g_r)$, let $\square_{\bg}$ denote the open and closed substack of $\square$ such that $\op{ev}|_{\square_{\bg}}$ maps to $\left[ \big(\bigoplus_{i=1}^r (V^{ss}(\theta))^{g_i} / C_\Gamma\bg \right]$, and let $\widetilde  U_{\bg}$ be the open and closed substack of $\widetilde  U$ which $\square_{\bg}$ maps into via $\widetilde{p}$.
Denote by $m_i$ the order of $g_i$.  Then $1/(m_1 \cdots m_r)$ is the degree of the map 
\[\sMbar_{g,r} \left(\cC_{\Xrig,\Gamma} / \sKbar_{g,r}(\Xrig, d), F\right) \to \sKbar_{g,r} \left(\cC_{\Xrig,\Gamma} / \sKbar_{g,r}(\Xrig, d), F\right)\] from Remark~\ref{r:gerbesections} after restricting to the open and closed subset indexed by $\bg$.  
Define the map $\ord_{\bg}: \op{HH}_*( \widetilde{U}_{\bg} ) \to \op{HH}_*( \widetilde{U}_{\bg} )$ to be multiplication by $(m_1 \cdots m_r)$, and let 
\[\ord: \op{HH}_*( \widetilde{U} ) \to \op{HH}_*( \widetilde{U} )\]
be the direct sum of $\ord_{\bg}$ over all choices of $\bg$.  Because our moduli spaces have been constructed to have sections at the marked point gerbes, this correction factor is necessary to obtain the correct invariants.  This phenomena arises in orbifold Gromov--Witten theory as well (see \cite[\S6.1.3]{AGV}).

Recall that  for $\cX$ a smooth 
separated Deligne--Mumford quotient stack with projective coarse moduli space, we defined a HKR {\it morphism} 
$$\bar{\phi}_{\op{HKR}} \colon \op{HH}_*(\cX) \to \op{H}^*(\cX)$$ in Definition~\ref{d:hkrmor}. 
\begin{definition}\label{d:invariants}
Let $(V, G,\CC^*_R, \theta, w)$ be a convex hybrid model.  For $g, r$ satisfying $2g-2 + r> 0$
and $d \in \text{Hom}_{\Z}(\widehat G, \QQ)$, 
define
 $\Lambda_{g,r,d}: (\cc H^{res})^{\otimes r} \to \op{H}^*(\sMbar_{g,r})$ as the following map
\[\Lambda_{g,r,d}(s_1, \ldots, s_r) := \op{proj}_* \left(\frac{\op{td}(T_{\widetilde{U}/\Mfrak^\orb_{g,r}(B\Gamma, d)_{\log}})}{\op{td}(\mathbb R\pi_* \mathcal V)}  \cup \left( \bar \phi_{\op{HKR}}\circ \ord \circ q^* \circ{(\Phi_{K_{g,r, d}})_*}(s_1, \ldots, s_r) \right) \right)\]
where $\op{proj}$ is the projection to $\sMbar_{g,r}$ and $s_1, \ldots, s_r$ are elements of $\cc H^{res}$.  

The set $\{\Lambda_{g,r,d}\}$ for all choices of $g, r$ and $d$ will be referred to as the GLSM invariants of $(V, G, \CC^*_R, \theta, w)$.
\end{definition}

\begin{remark}
The relative tangent bundle $T_{{U}/\Mfrak^\orb_{g,r}(B\Gamma, d)_{\log}}$  is the vector bundle $A_1$.
\end{remark}

\begin{remark}
In the second paper \cite{CKFGS}, we will show that the collection $\{\Lambda_{g,r,d}\}$ defines a cohomological field theory.
\end{remark}
\begin{theorem}\label{t:finind}
Given a hybrid model GLSM $(V, G, \CC^*_R, \theta, w)$, convex over $BG_1$, the invariants $\Lambda_{g,r,d}$ are independent of the choice of resolutions.
\end{theorem}

\begin{proof}
Let $\square ' \to U'$ and $\square \to U$ denote two instances of the two-step procedure of the previous section for different choices of resolutions. 
Let $K$ and $K'$ denote the factorizations in $\square $ and $\square '$ respectively, and let 
 $\Lambda_{g,r,d}$ and $\Lambda_{g,r,d} '$ denote the corresponding sets of invariants for $(V, G, \CC^*_R, \theta, w)$.
 
 By Proposition~\ref{p:ind}, we can assume  without loss of generality that there exists a commuting diagram 
  \[
\begin{tikzcd}
\square \ar[d] \ar[r, hook, "f"] &\square ' \ar[d] \\
U \ar[r, hook, "g"] & U'
   \end{tikzcd}
   \]
   such that $f$ and $g$ are closed immersions, and
  \begin{equation}\label{e:pushit}
  \mathbb{R} f_*(K) = K '.
  \end{equation}
    
    Let $\widetilde{U}$ and $\widetilde{U}'$ denote resolutions of the closures $\overline{U}$ and $\overline{U}'$.  
These can be constructed so that there exists a closed immersion
$\tilde g: \widetilde{U} \to \widetilde{U}'$ such that $\widetilde{U}'$ has projective coarse moduli space (Lemma \ref{lemma: proj resol}).
By \eqref{e:pushit} we see that
 \begin{equation}\label{e:pushit2}
 \tilde g_* \circ (\Phi_K)_* = (\Phi_{K '})_*.
  \end{equation}

 By Theorem 1 of \cite{KV} and a standard Bertini-type argument as in the proof of
  Theorem 1 of \cite{KV}, we can construct a fiber diagram
  \begin{equation}\label{e:sq}
\begin{tikzcd}
\wt V \ar[r, "\hat g"] \ar[d, swap, "\pi"]  & \wt V' \ar[d, "\pi '"] \\
\wt U \ar[r, "\tilde g"] & \wt U '. \\
\end{tikzcd}
\end{equation}
  such that $\pi$, $\pi '$ are finite flat surjective maps and $\wt V$, $\wt V'$ are smooth projective varieties.
 Then
\begin{align}\label{e:eq12}
\tilde g_* \left( \td(T_{\widetilde{U}/{\widetilde{U} '}}) \bar \phi_{\op{HKR}} ( -)\right) &= \tilde g_* \left( \frac{\td(T_{\widetilde{U}/{\widetilde{U} '}})}{\deg \pi} \pi_*  \bar \phi_{\op{HKR}}^{\wt V} \pi^* ( -)\right) \\
&= \tilde g_* \left( \frac{1}{\deg \pi} \pi_* \td(T_{\widetilde{V}/{\widetilde{V} '}}) \bar \phi_{\op{HKR}}^{\wt V} \pi^* ( -)\right) \nonumber
 \\
&=  \frac{1}{\deg \pi '}  \pi '_* \hat g_*\left( \td(T_{\widetilde{V}/{\widetilde{V} '}}) \bar \phi_{\op{HKR}}^{\wt V} \pi^* ( -)\right) \nonumber
 \\
&=  \frac{1}{\deg \pi '}  \pi '_*\left(\bar  \phi_{\op{HKR}}^{\wt V '}\hat g_* \pi^* ( -)\right) \nonumber
\\
&=  \frac{1}{\deg \pi '}  \pi '_* \bar \phi_{\op{HKR}}^{\wt V '}{\pi '}^*  \tilde g_* ( -) \nonumber
\\
&=  \bar  \phi_{\op{HKR}}  \tilde g_*  ( -) .  \nonumber
\end{align}
Here the second equality is because \eqref{e:sq} is a fiber square.  The third also follows for the same reason.  The fourth equality is the statement that
$$\hat g_*\left( \td(T_{\widetilde{V}/{\widetilde{V} '}})\bar \phi_{\op{HKR}}^{\wt V}(-) \right) = \bar \phi_{\op{HKR}}^{\wt V '}\hat g_*(-).$$
This follows from Theorem~\ref{t:ram} in the case that $X = \cZ = \wt V$, $\cY = \wt V '$, and $K = \cc O_{\wt V}$, after pre-composing with $\phi_{\op{HKR}}^{\wt V}$.
The fifth equality of \eqref{e:eq12} is by flat base change.

Combining  \eqref{e:pushit} and \eqref{e:eq12} implies the result.   
\end{proof}

\begin{remark}\label{r:reldef} 
Because the kernel $K_{g, r, d}$  is supported on $LG_{g,r}(\cZ, d)$, a closed substack of $U$, the integral transform
$\Phi _{K_{g, r, d}}$ can be factored as 
\[ D (\cX, w)^r  \xrightarrow{\Phi ^{\op{rel}}_{K_{g, r, d}}} D([U/\mathbb{C}^*_R], 0)_{LG_{g,r}(\cZ, d)} \to D ([\widetilde{U}/\mathbb{C}^*_R], 0) .  \]
Let $\pi: \wt V \to \wt U$ be a proper, generically finite, surjective morphism from a smooth projective scheme as in Definition~\ref{d:hkrmor}.  Let $V := U \times_{\wt U} \wt V$ and let $Z_V := LG_{g,r}(\cZ, d) \times_{\wt U} \wt V$.
We obtain a commuting diagram
         \begin{equation}\label{e:relFM} \xymatrix{ \op{HH}_* (\cX, w)^r \ar[d]_{q^* \circ (\Phi ^{\op{rel}}_{K_{g, r, d}})_*} \ar[dr]^{q^* \circ (\Phi _{K_{g, r, d}})_*}   &  \\
         \op{HH}_*(U)_{LG_{g,r}(\cZ, d)} =  \op{HH}_*(\widetilde U)_{LG_{g,r}(\cZ, d)}  \ar[r] \ar[d]_{\pi^* }  & \op{HH}_* (\widetilde{U} ) \ar@/^2pc/[ddd]^{\bar{\phi}^{\wt U} _{\op{HKR}} } \ar[d]_{\pi^*} \\  
                  \op{HH} (V)_{ Z_V}   = \op{HH}_* (\widetilde{V})_{Z_V}  \ar[r] \ar[d]_{{\phi}^{\op{rel}} _{\op{HKR}} }  & \op{HH}_* (\widetilde{V} ) \ar[d]_{{\bar \phi^{\wt V}} _{\op{HKR}}} \\  
         \op{Hod}^*(V)_{ Z_V}       = \op{Hod}^* (\widetilde{V})_{Z_V}       \ar[r]   \ar[rdd]_{\mathrm{\tfrac{1}{\op{deg}(\pi)}proj_*}}    & \op{H}^* (\widetilde{V}) \ar[d]_{\tfrac{1}{\op{deg}(\pi)}\pi_*} \\
                                                                                      & \op{H}^* (\widetilde{U}) \ar[d]_{\mathrm{proj}_*} \\
                                                                                                     & \op{H}^*(\overline{M}_{g, r}), } \end{equation}
with the following notation.   The vector space $\op{Hod}^*(V)_{Z_V} $ denotes the Hodge cohomology of $V$ supported on $Z_V$,
the map $\op{Hod}^* (\widetilde{V})_{Z_V} \to \op{H}^* (\widetilde{V})$ is the inclusion
followed by the twisted Hodge decomposition of Definition~\ref{d:hkrmor}, and the map $\phi^{\op{rel}} _{\op{HKR}}$ is the relative HKR isomorphism from Remark \ref{rem: relative HKR}.  The map $\mathrm{proj}_*$ along the bottom left diagonal is the composition of the analytification and the pushforward by integration, i.e.
$$\op{Hod}^*(V)_{ Z_V} \to  \op{H}^{*}_{\bar{\partial}} (V )_{Z_V} \xrightarrow{\mathrm{proj}^{an}_*}   
\op{H}^*_{\bar{\partial}}(\overline{M}_{g, r}) \cong \op{H}^*(\overline{M}_{g, r}) $$  defined by the requirement 
\begin{eqnarray*}\label{eqn: rel push}  \int _{\overline{M}_{g, r}} \alpha \cup \mathrm{proj}^{an}_*  \beta = \int _{V} (\mathrm{proj}^{an})^* \alpha \cup \beta \end{eqnarray*}
for $\alpha \in  \op{H}^*_{\bar{\partial}}(\overline{M}_{g, r})$, $\beta \in \op{H}^*_{\bar{\partial}}(V)_{ Z_V} $. 
Using  the Hodge decomposition for compact K\"ahler orbifolds, 
it is straightforward to check that $\mathrm{proj}^{an}_*$ is well-defined and  compatible with the map $\mathrm{proj}_*$ appearing as the bottom right vertical arrow.
In summary, Definition~\ref{d:invariants} does not depend on the choice of $\widetilde U$.
   \end{remark}

\begin{remark}
Using \eqref{e:relFM}, Theorem~\ref{t:finind} can be proven without reference to the fact that the resolution $\wt U$ has projective coarse moduli space, however the argument is more complicated.
                                                                                                     \end{remark}

\section{Comparisons with other constructions}\label{s:comwitothcon}
\subsection{Comparison with Gromov--Witten theory and cosection localization}
\label{sec: compare GW}
In this section we compare the above GLSM invariants for a \textit{geometric phase} with various Gromov--Witten type invariants defined via a virtual fundamental class.

\subsubsection{The virtual cycle for GLSM invariants in a geometric phase}\label{sss:6.1.1}
For all of \S\ref{sec: compare GW}, we specialize to the case where the hybrid model is a so-called \newterm{geometric phase} (recall Definition~\ref{d:phases}).  We will define a \textit{virtual fundamental class} using the $\ZZ_2$-localized Chern class of \cite{PVold}.
This class can be used to define enumerative invariants via a Chow or cohomology level Fourier--Mukai transform.  

The remainder of \S\ref{sss:6.1.1} is then devoted to showing that these invariants agree with the GLSM invariants of Definition~\ref{d:invariants} (See Proposition~\ref{prop :GWcomp}).  This will be used in \S\ref{sss:6.1.2} to compare our invariants to cosection localized invariants and certain Gromov--Witten invariants.

Specializing to a geometric phase implies $\bdeg =1 $, therefore $\lan J \ran = 1$ and
$\Gamma = G \times \CC^*_R$. The character $\chi\in \widehat{\Gamma}$ is just the projection onto $\CC^*_R$, and $\eta_{\bdeg}=\eta_1=\op{id}_{\CC^*_R}$. 

Consider the vector bundle $\cE$ (with fiber $V_2$) on $\cX$, whose total space is
$$\tot\cE=\cT :=[(V_1^{ss}(\theta) \times V_2) / G],$$ 
 with projection $q\colon \cT\ra \cX$. 
 Since $w$ is linear on $V_2$, it gives rise to a section
$f\in H^0(\cX,\cE^\vee)$ of the {\it dual} vector bundle $\cE ^\vee $. 
Recall that nondegeneracy of the hybrid model implies that $f$ is a regular section with smooth zero locus $\cZ :=Z(f)=Z(dw)$.  For the remainder of this subsection, we will assume that $\cX$ is a {\it smooth variety}.

The notations in \S \ref{s:conFF} simplify, since $G_1$ is simply $G$, the stack $\cX$ is equal to $\Xrig=[V_1^{ss}(\theta)/ G]$, and 
$$ LG_{g,r}(\cX, d) \cong \sMbar_{g,r}(\cX, d),\;\;\; LG_{g,r}(\cZ, d) \cong \sMbar_{g,r}(\cZ, d)).$$
For simplicity we will sometimes denote these spaces by $LG(\cX)$ and $LG(\cZ)$ respectively.

Let  $\cV_2 '$ denote $\cV_2 (-\sG )$.
By Proposition~\ref{prop: satisfy all geometric}, we may choose an admissible resolution $[\bar{A} \to \bar{B}]$ of $\mathbb R \pi_* \cV$ over 
$U $ which splits as 
\[[\bar A_1 \to \bar B_1] \oplus [\bar A_2 ' \oplus \cc V_2|_{\sG} \to \bar B_2]\]
where 
\[[\bar A_2 ' \oplus \cc V_2|_{\sG} \to \bar B_2] \cong \mathbb R \pi_* \cV_2 \text{  and  } [\bar A_2 '  \to \bar B_2] \cong \mathbb R \pi_* \cV_2 '.\]  
Let $\square '$ denote the intersection of $\square$ and   $\tot \bar A_1 \oplus \bar A_2 ' $.  Let $U'$ denote the intersection of $\square$ and   $\tot \bar A_1  $.

\begin{remark}
The space $\square '$ can be constructed from scratch by mimicking the construction of $\square$, but replacing every instance of $\log$ with $\omega_{\cC}$.
\end{remark}
Recall from \S\ref{s:FF}  that we also have a resolution
\[ [\widetilde A_1 \to \widetilde B_1 ] \cong \mathbb R\pi_* \cV_1\] lying 
over the closure $\overline U$
together with a surjection $\bar A_1|_U \twoheadrightarrow \widetilde A_1|_U$ and the evaluation map $U' \to \cX^r$ factors through 
\[U' \to \tot (\bar A_1) \to \tot (\widetilde A_1) \to [V_1/G]^r.\]

Let $\tilde d: \widetilde U \to \overline U$ denote the desingularization of $\overline U$ from \S\ref{s:GLSMinvs} and define
\[P := \op{Proj}(\tot(\tilde d^*(\widetilde A_1) \oplus \cc O_{\widetilde U})).\] 
The map $\widetilde A_1 \to \cV_1|_\sG$ defines a geometric evaluation map $\tot(\widetilde A_1) \to [V_1/G]^r$.  Let $\tot (\widetilde A_1)^\circ$ denote the preimage of the semi-stable locus, so we obtain
\[\tot (\widetilde A_1)^\circ \to \cX^r.\]
This defines a rational map to $\op{ev}_P: P \dashrightarrow \cX^r$.  Let $\widetilde P$ denote a smooth resolution of $\overline{\Gamma_{\op{ev}_P}} \in P \times \cX^r$.  

We have the following commuting diagram

\begin{equation}\label{d:complicated}
\begin{tikzcd}
 \square \ar[r, "z_\cT"]  \ar[d, " l"] 
 & \tot (\bar A)^\circ   \ar[rd, "f_\cT"] \ar[d, " \bar l"] & & & \\
\ar[u, bend left, "s"]  \square ' \ar[r, "z_\cT ' "]   \ar[d, "m "]  &  \tot (\bar A_1 \oplus \bar A_2 ')^\circ \ar[rd, "f_\cT ' "] \ar[d] & \op{Tot}^\circ \ar[d, " l^\circ"] \ar[r, "i_\cT"] & \widetilde{\op{Tot} }   \ar[r, "\op{ev}_\cT"]  \ar[d, "\tilde l"] & \cT ^r    \ar[d]  \\
   U'   \ar[r, "z_\cX"]               &         \tot (\bar A_1)^\circ       \ar[r, "f_\cX"]       & \tot (\widetilde A_1|_{\widetilde{U}})^\circ \ar[r, "i_\cX"]  &\widetilde P   \ar[r, "\op{ev}_\cX "]     \ar[dd, bend left = 30, leftarrow, near end, "\tilde k"]      &          \cX^r       \\
LG(\cZ) \ar[rr, "\iota"]     \ar[u]       &     &   {P_{\cZ}}^\circ  \ar[r, "j_\cZ"] \ar[dr, swap, "i_{\cZ}"] \ar[u, " k^\circ"]   &  P_{\cZ} \ar[r, crossing over, "\op{ev}_\cZ"]  \ar[u, "k"] &           \cZ ^r  \ar[u]      \\
 & & & \widetilde P_\cZ \ar[u, "d_{\cZ}"] 
 &
\end{tikzcd}
\end{equation}
where the remaining spaces are defined so that all small squares (and the parallelogram) are cartesian (although the bottom left rectangle is commutative), and $d_\cZ: \widetilde P_{\cZ} \to P_{\cZ}$ is a desingularization of $P_{\cZ}$.
For convenience we will  let $n$ denote the composition $n := i_\cT \circ f_\cT \circ z_\cT: \square \to \widetilde{\op{Tot}}$.
Note that the evaluation map $\op{ev}: \square \to \cT^r$ factors as $\op{ev}_\cT \circ n$.

Recall from Definition~\ref{d:phiplus}, there is an equivalence of categories
\[
\tilde \phi_+: \dbcoh{\cZ} \xrightarrow{\sim} \dabsfact{[\cT / \CC^*_R], w},
\] 
which sends the structure sheaf $\cO_{\cZ}$ to the Koszul factorization $S_1$ of \eqref{S1}.

The fundamental factorization in Definition \ref{d:koszulcut}
\[
K = K_{g,r,d} \in \dabsfact{[{\square}/\CC^*_R], -\op{ev}^*{\textstyle \boxplus_{i=1}^r} w}
\]
is defined as the Koszul factorization on the section and cosection 
\[
\O_{\square} \xrightarrow{\beta } \bar B \xrightarrow{-\alpha^\vee} \O_{\square }
\]
i.e.
\[
K = \{ -\alpha , \beta  \}.
\]
(Note: to declutter notation, in this section, we use $\bar A, \bar B$, etc. to denote the corresponding pullback of the bundle to whichever space we are working on.)

The tensor product of $K$ with $n^* \circ \op{ev}_\cT^* (S_1^{\boxtimes r})$ is the Koszul factorization
\[
K^{S_1}:= \{ \alpha^{S_1}, \beta^{S_1} \} \in \dabsfact{[\square/\CC^*_R], 0} 
\]
where
\[
\alpha^{S_1} := (-\alpha, \op{ev}^*(q^*f)^r) \ \ \tand \ \ \beta^{S_1} :=  (\beta , \op{ev}^*\taut ) 
\]
are a cosection and a section of 
$
\bar B^{S_1} := \bar B  
\oplus  \pi_* (\cV_2|_{\sG}).
$

Recall that $LG_{g,r}(\cT , d) = Z(\beta)  \subseteq \square$.  Similarly let \[LG(\cT, d) ' := Z(\beta|_{\square '})  \subseteq \square '.\]
Furthermore, as $LG_{g,r}(\cZ, d)$ lies in $\square'$, we have
\[
LG_{g,r}(\cZ, d) = Z(\beta) \cap Z(\alpha)  \subseteq \square' \subseteq \square
\]
by Lemma~\ref{l:finalsupp}.

Let $K '$ denote $s^*(K)$.  Note that this is a factorization of zero on $\square '$.
By \cite[Proposition 4.3.1]{PV}, we observe that $\mathbb R s_*(K ') = K^{S_1}$.  
\begin{definition}\label{def: loc ch vir cycle}
Define the \newterm{virtual fundamental class} to be the cycle class,
\begin{align*}
[LG_{g,r}(\cZ, d)]^{\op{vir}}        & :=   \op{td}(\bar B|_{\square '}) \ \chp_{LG_{g,r}(\cZ, d)}^{\square '} (K ') [\square '] \\
                                    & \in A_*(LG_{g,r}(\cZ, d))_{\QQ}.
\end{align*}
\end{definition}

Using \cite{Hirano}, we can translate between the Fourier--Mukai transform for a geometric phase using a fundamental factorization and a Fourier--Mukai transform in the traditional sense, using an object of the derived category.  Namely we have the following lemma.
\begin{lemma}
There exists an object $K_{\cZ}  \in \op{D}(\widetilde{P}_{\cZ})$ such that
\begin{equation}
(\Phi_{K_{g,r,d}})_* \circ  ((\tilde \phi_+)^r)_* = (\Phi_{K_\cZ} )_*,
\label{eq: same transforms}
\end{equation}
where $\Phi_{K_\cZ}: \cZ^r \to \widetilde U$ denotes the Fourier--Mukai transform with kernel $K_\cZ$.
Furthermore, the object $K_{\cZ}$ satisfies
\begin{equation}
\label{KSP}
 \mathbb R ( \tilde l \circ  n)_* K^{S_1}   = \mathbb R \tilde k_{ *} K_{\cZ}.
\end{equation}
\end{lemma}

\begin{proof}
Since $\op{ev}_\cX \circ i_\cX: \tot (\widetilde A_1|_{\widetilde U})^\circ  \to \cX^r$ is smooth, the pullback $(\op{ev}_\cX \circ i_\cX)^*f$ gives a regular section of $\cV_2 |_{\sG}$.  Hence, we can  apply Theorem~\ref{thm: Hirano} and
Proposition \ref{FM commutative diagram} to obtain an  object $K^{\op{pre}}_{\cZ} \in \op{D}( \tot (\widetilde A_1|_{\widetilde U})^\circ)_{LG(\cZ)}$  such that the following diagram commutes
\begin{equation}  \label{moduli diagram}
\begin{tikzcd}
  & \op{D}(  [\op{Tot}^\circ /\CC ^*_R], \op{ev}_\cT \circ i_\cT^* {\textstyle \boxplus_{i=1}^r} w)  \ar[rr, "-\overset{\mathbb{L}}{\otimes} \mathbb R (z_\cT \circ f_\cT)_*  K"]  
& & \op{D}([  \op{Tot}^\circ /\CC ^*_R] , 0)_{[LG(\cZ)/ \C^*_R]} \ar[d, " \mathbb R  l^\circ_*"]  
\\
 & \op{D}( P_{\cZ}^\circ) \ar[u, "\tilde \phi_{ P_{\cZ}^\circ, +}"] \ar[rr, "\mathbb R k ^\circ_*(-\overset{\mathbb{L}}{\otimes} K^{\op{pre}}_{\cZ})"] & & \op{D}( \tot (\widetilde A_1|_{\widetilde U})^\circ)_{LG(\cZ)} \end{tikzcd} 
\end{equation} 
where $\tilde \phi_{P_{\cZ}^\circ, +}$ is notation for $\tilde \phi_+$ in the particular case $Z =  P_{\cZ}^\circ$ and $T = \op{Tot}^\circ$.

We can then define
\[
K_{\cZ} := (i_\cZ)_* K_{\cZ}^{\op{pre}} \in \op{D}(\widetilde{P}_{\cZ}).
\]
Let $\tilde \phi_{P_{\cZ}, +}$ denote $\tilde \phi_+$ in the particular case $Z = P_{\cZ}$ and $T = \widetilde{\op{Tot}}$ (in this case, it may not be an equivalence but the functor exists nonetheless).

The following diagram, which implies \eqref{eq: same transforms},  commutes.
\begin{equation} \label{d:compkey}\begin{tikzcd}
\op{D}( [\cT ^r /\CC ^*_R] , {\textstyle \boxplus_{i=1}^r} w ) \ar[r, "\mathbb L \op{ev}_\cT^*"]   &  \op{D}(  [\widetilde{\op{Tot}} /\CC ^*_R], \op{ev}^*_\cT {\textstyle \boxplus_{i=1}^r} w)  \ar[r, "-\overset{\mathbb{L}}{\otimes} \mathbb R n_*  K"]  
&  \op{D}([  \widetilde{\op{Tot}} /\CC ^*_R] , 0)_{[LG(\cZ) / \C^*_R]} \ar[d, " \mathbb R  \tilde{l}_*"]  
\\
\op{D}(\cZ^r) \ar[dr, swap, "\mathbb L (\op{ev}_\cZ \circ d_\cZ)^*"] \ar[r, "\mathbb L (\op{ev}_\cZ)^*"] \ar[u, "(\tilde \phi_+)^r"] & \op{D}(P_{\cZ} ) \ar[u, "\tilde\phi_{P_{\cZ}, +}  "] \ar[r, "\mathbb R  k _*(-\overset{\mathbb{L}}{\otimes} \mathbb R (d_\cZ)_* K_{\cZ})"] & \op{D}(\widetilde{P}) \\
&  \op{D}(\widetilde{P}_{\cZ} ) \ar[ur, swap, "\mathbb R \tilde k _*(-\overset{\mathbb{L}}{\otimes} K_{\cZ})" ] &
 \end{tikzcd}
\end{equation}
The left square follows from flat base change. 
For the right square, we have
\begin{align*}
\RR \tilde l _* (\tilde\phi_{P_{\cZ}, +} (A) \overset{\mathbb{L}}{\otimes} \mathbb R n_*  K) & =
\RR \tilde l _* (\tilde\phi_{P_{\cZ}, +} (A) \overset{\mathbb{L}}{\otimes} \mathbb R n_* \RR (i_\cT)_* \LL i_\cT^* K)  \\
& =
\RR \tilde l _*  \RR (i_\cT)_* (\LL i_\cT^* \tilde\phi_{P_{\cZ}, +} (A) \overset{\mathbb{L}}{\otimes} \mathbb R n_* \LL i_\cT^* K)  \\
& =
\RR (i_\cX)_*  \RR (l^\circ)_*  (\phi_{P^\circ_{\cZ}, +} (\LL j_\cZ^*A) \overset{\mathbb{L}}{\otimes} \mathbb R (z_\cT \circ f_\cT)_* K)  \\
& =
\RR (i_\cX)_*  \RR k^\circ_* (\LL j_\cZ^*A \overset{\mathbb{L}}{\otimes} K^{\op{pre}}_\cZ)  \\
& =
  \RR k_* \RR (j_\cZ)_* (\LL j_\cZ^*A \overset{\mathbb{L}}{\otimes} K^{\op{pre}}_\cZ)  \\
  & =
  \RR k_*  (A \overset{\mathbb{L}}{\otimes} \RR (d_\cZ)_* K_\cZ).
\end{align*}
  Most of these equalities are simply the projection formula of Proposition~\ref{p:projform}, however the fourth equality  requires \eqref{moduli diagram}.
 Finally, for the bottom triangle we have
\begin{align*}
\mathbb R k_* (\mathbb L(\op{ev}_\cZ)^* A \overset{\mathbb{L}}{\otimes} \mathbb R (d_\cZ)_* K_{\cZ}))& = \mathbb R k_* (\mathbb L(\op{ev}_\cZ)^* A \overset{\mathbb{L}}{\otimes} \mathbb R (j_\cZ)_* K^{\op{pre}}_{\cZ}) \\
& = \mathbb R k_* \mathbb R (j_\cZ)_*  (\mathbb L j_\cZ^* \mathbb L(\op{ev}_\cZ)^* A \overset{\mathbb{L}}{\otimes} K^{\op{pre}}_{\cZ}) \\
& = \mathbb R \tilde k_* \RR (i_\cZ)_* (\LL (i_\cZ)^* \mathbb L (\op{ev}_\cZ \circ d_\cZ)^* A \overset{\mathbb{L}}{\otimes}  K^{\op{pre}}_{\cZ}) ) \\
& = \mathbb R \tilde k_* (\mathbb L (\op{ev}_\cZ \circ d_\cZ)^* A \overset{\mathbb{L}}{\otimes}  \RR (i_\cZ)_*   K^{\op{pre}}_{\cZ}) ) \\
& = \mathbb R \tilde k_* (\mathbb L (\op{ev}_\cZ \circ d_\cZ)^* A \overset{\mathbb{L}}{\otimes} K_{\cZ}). \end{align*}

Now to justify \eqref{KSP}, plug $\O_{\cZ^r}$ into diagram~\eqref{d:compkey} to get
\begin{align*}
 \mathbb R ( \tilde l \circ  n)_* K^{S_1}  & =  \mathbb R \tilde l_* \circ \mathbb R n_* (K \overset{\mathbb{L}}{\otimes} \mathbb{L} (\op{ev}_\cT \circ n)^* (\tilde \phi_+)^r \O_{\cZ^r})   \\
 & =  \mathbb R \tilde l_* ( \mathbb R n_* K \overset{\mathbb{L}}{\otimes} \mathbb{L} (\op{ev}_\cT)^* (\tilde \phi_+)^r \O_{\cZ^r})     \\
 & =  \mathbb R \tilde l_* ( \mathbb R n_* K \overset{\mathbb{L}}{\otimes}  \tilde \phi_{P_\cZ,+} \O_{P_\cZ})   \\
  &  = \mathbb R k_{ *} \mathbb R (d_\cZ)_{*}K_{\cZ} \\
 &  = \mathbb R \tilde k_{ *} K_{\cZ}
\end{align*}
where the first line is by the definition of $K^{S_1}$ and Remark~\ref{r:os1}, the second line is the projection formula, the third line is by the left square of 
diagram~\eqref{d:compkey}, the fourth line is by the right square of diagram~\eqref{d:compkey}, and the last line is by diagram~\ref{d:complicated}. \end{proof}

\begin{lemma}\label{lem: MXvir}
The following equality holds in $A_*(LG_{g,r}(\cZ, d))_{\Q}$
\[
    [LG_{g,r}(\cZ, d)]^{\op{vir}}  = \op{td}(T_{\widetilde{P}_{\cZ}/\Mfrak^\orb_{g,r}(B\Gamma, d)_{\log}}) \op{td}(\mathbb R\pi_* \mathcal V')^{-1}  \op{ch}^{\widetilde{P}_{\cZ}}_{LG(\cZ)} (  K_{\cZ} ) [\widetilde{P}_{\cZ}].
\]
\end{lemma}

\begin{proof}

Now, we have the following sequence of equalities
\begin{align}\label{eq: MXvir} 
    [LG_{g,r}(\cZ, d)]^{\op{vir}} 
& =   \op{td}(\bar{B}|_{\square '} ) \ \chp_{LG(\cZ)}^{\square '} ( K ') [\square '] \nonumber \\
& = \op{td}(\bar{B} ) \op{td}(\bar{A_2} ')^{-1}  \ch_{LG(\cZ)}^{U '} (\mathbb R (m\circ l)_*  K^{S_1}) [ U'  ] \nonumber \\
& =   \op{td}(\bar{B} \oplus \widetilde A_1 \oplus A_1) \op{td}(\bar{A_1} \oplus \bar{A_2} ')^{-1}  \ch_{LG(\cZ)}^{\tot \widetilde A_1|_{\widetilde{U}}} (\mathbb R (l^\circ \circ f_\cT \circ z_\cT )_*  K^{S_1}) [\tot \widetilde A_1|_{\widetilde{U}} ] \nonumber \\
& =  \op{td}(\bar{B} \oplus \widetilde A_1 \oplus A_1) \op{td}(\bar{A_1} \oplus \bar{A_2} ')^{-1}  \ch_{LG(\cZ)}^{\widetilde{P}} (\mathbb R (\tilde l\circ n )_*  K^{S_1}) [\widetilde{P} ] \nonumber \\
& =  \op{td}(T_{\widetilde{P}/\Mfrak^\orb_{g,r}(B\Gamma, d)_{\log}}) \op{td}(\mathbb R\pi_* \mathcal V')^{-1}  \ch_{LG(\cZ)}^{\widetilde{P}} (\mathbb R (\tilde l\circ n )_*  K^{S_1}) [\widetilde{P} ] \nonumber \\
& =  \op{td}(T_{\widetilde{P}/\Mfrak^\orb_{g,r}(B\Gamma, d)_{\log}}) \op{td}(\mathbb R\pi_* \mathcal V')^{-1}  \op{ch}^{\widetilde{P}}_{LG(\cZ)} ( \mathbb R \tilde k_* K_{\cZ} ) [\widetilde{P}]  \nonumber \\
& =  \op{td}(T_{\widetilde{P}_{\cZ}/\Mfrak^\orb_{g,r}(B\Gamma, d)_{\log}}) \op{td}(\mathbb R\pi_* \mathcal V')^{-1}  \op{ch}^{\widetilde{P}_{\cZ}}_{LG(\cZ)} (  K_{\cZ} ) [\widetilde{P}_{\cZ}]
\end{align}
where $\cV '$ denotes $\cV_1 \oplus \cV_2 (-\sG )$. 
The first equality is the definition.  The second comes from Proposition~\ref{prop: chiodo}, note that this applies since $K'$ is locally-contractible off $LG(\cZ)$ by Proposition~\ref{p:Ksupp} and equation \eqref{eq: support equals LG} hence strictly exact off $LG(\cZ)$ by \cite[Lemma 2.2 (1)]{KOcos}.  The third follows from \cite[Corollary 18.1.2]{Fulton}.  The fourth equality is obvious from the definition of localized Chern characters.  The fifth line comes from the equality
\[
\mathbb R\pi_ * \cV' = \bar{A_1} \oplus \bar A_2' \oplus \bar B[1],
\]
together with the fact that  after restricting to $LG(\cZ) \subset \tot(\widetilde A_1|_U) \subset \widetilde P$, we have
\[
\left(T_{\widetilde{P}/\Mfrak^\orb_{g,r}(B\Gamma, d)_{\log}}\right) \bigg |_{LG(\cZ)} =  \left(A_1 \oplus \widetilde A_1\right) \bigg |_{LG(\cZ)}\]
by Theorem~\ref{c:swn} (a).
  The sixth equality of \eqref{eq: MXvir} follows immediately from \eqref{KSP}.  The seventh follows from \cite[Theorem 18.2]{Fulton} (using the isomorphism $K_0(LG(\cZ)) \cong K_0( \op{D}({\widetilde{P}})_{LG(\cZ)})$).  This completes the proof.
\end{proof}

We now prove that our GLSM invariants are the same as those defined by $[LG_{g,r}(\cZ, d)]^{\op{vir}}$.  In the next subsection, we will argue that these agree with the cosection localized Gromov--Witten type invariants defined by Chang--Li  \cite{CL} and and generalized by Fan--Jarvis--Ruan \cite{FJR15}.

 By abuse of notation we will let $\iota_*[LG_{g,r}(\cZ, d)]^{\op{vir}}$ and $\ch(K_\cZ)$ now denote the corresponding classes in cohomology rather than the Chow ring.
\begin{proposition} \label{prop :GWcomp}
Let the setup be as in the beginning of this section.  Given $\vec \gamma = (\gamma_1, \ldots, \gamma_r) \in \op{H}^*(\cZ)^{\otimes r}$, the following are equal:
\begin{align*} 
 (\op{proj}_* \circ \tilde k)_* (\op{ev}_\cZ^*\vec \gamma \cup ( i_\cZ \circ \iota)_*[LG_{g,r}(\cZ, d)]^{\op{vir}}) 
 = \Lambda_{g,r,d}(\varphi^{\td}_* \gamma_1, \ldots, \varphi^{\td}_* \gamma_r). 
\end{align*}
\end{proposition}

\begin{proof}
Let $\tilde u$ denote the map $\tilde u: \widetilde P \to \widetilde U$.  We have,
\begin{align*}
 & \tilde u_* \circ \tilde k_* (\bar{\op{ev}}^*\vec \gamma \cup ( i_\cZ \circ \iota)_*[LG_{g,r}(\cZ, d)]^{\op{vir}})  \\
 = & \tilde u_* \circ  \tilde k_*( \op{td}(T_{\widetilde{P}_{\cZ}/\Mfrak^\orb_{g,r}(B\Gamma, d)_{\log}}) \op{td}(\mathbb R\pi_* \mathcal V')^{-1}  \op{ch} ( K_{\cZ} ) ev^*\vec \gamma) \\
 = & \op{td}(T_{\widetilde{U}/\Mfrak^\orb_{g,r}(B\Gamma, d)_{\log}}) \op{td}(\mathbb R\pi_* \mathcal V')^{-1}\tilde u_* \circ  \tilde k_*( \op{td}(T_{\widetilde{P}_{\cZ}/\widetilde{U}}) \op{ch} ( K_{\cZ} ) ev^*\vec \gamma) \\
 =& \op{td}(T_{\widetilde{U}/\Mfrak^\orb_{g,r}(B\Gamma, d)_{\log}}) \op{td}(\mathbb R\pi_* \mathcal V')^{-1}(\Phi^{\op{H}}_{\op{td}(T_{\widetilde{P}_{\cZ}/{\widetilde{U}}}) \op{ch} ( K_{\cZ})})_*(\vec \gamma) 
 \\ =& \op{td}(T_{\widetilde{U}/\Mfrak^\orb_{g,r}(B\Gamma, d)_{\log}}) \op{td}(\mathbb R\pi_* \mathcal V')^{-1}\bar \phi_{\op{HKR}}(\Phi_{K_{\cZ}})_*\bar \phi_{\op{HKR}}^{-1}(\vec \gamma) 
  \\ =& \op{td}(T_{\widetilde{U}/\Mfrak^\orb_{g,r}(B\Gamma, d)_{\log}}) \op{td}(\mathbb R\pi_* \mathcal V')^{-1}\bar \phi_{\op{HKR}}(\Phi_{K_{g,r,d}})_*( (\tilde \phi_+)_* \circ \bar \phi_{\op{HKR}}^{-1} \gamma_1, \ldots, (\tilde \phi_+)_* \circ \bar \phi_{\op{HKR}}^{-1} \gamma_r) 
  \\ =& \op{td}(T_{\widetilde{U}/\Mfrak^\orb_{g,r}(B\Gamma, d)_{\log}}) \op{td}(\mathbb R\pi_* \mathcal V)^{-1}\bar \phi_{\op{HKR}}(\Phi_{K_{g,r,d}})_*(\varphi_*^{\td} \gamma_1, \ldots, \varphi_*^{\td} \gamma_r).
\end{align*}
The first equality is by Lemma~\ref{lem: MXvir}.  
The second is the projection formula.   The third is just the definition of the cohomological Fourier--Mukai transform, and
the fourth is Theorem~\ref{t:ram}, which we may apply by Lemma~\ref{lemma: proj resol}. 
The fifth equality is by \eqref{eq: same transforms}. 
The last equality follows from $\mathbb R\pi_* \mathcal V = \mathbb R\pi_* \mathcal V' \oplus \pi_* \cV_2|_\sG $ and the definition of $\varphi_*^{\td}$ (Definition~\ref{d:phitodd}).  Note that the maps $\op{ord}$ and $q^*$ of Definition~\ref{d:invariants} are both the identity in this case, since $\cX$ is a smooth variety.
Applying $\op{proj}_*$ finishes the theorem.
\end{proof}

\subsubsection{Cosection Localized Gromov--Witten invariants}\label{sss:6.1.2}

In \cite{CL}, the cosection localization methods of Kiem--Li \cite{KiemLi} are used to construct a cosection localized virtual class for a particular geometric phase GLSM, namely the quintic 3-fold GLSM of Example~\ref{e:quintic}.  The cosection localized virtual class is used to define enumerative invariants for general GLSMs in \cite{FJR15}.  The construction can be summarized as follows (see Section~3 of \cite{CL} and Section~5 of \cite{FJR15} for details).

 The object 
\[ \mathbb E := \mathbb R \pi_* \cV\] gives a perfect obstruction theory over $LG(\cT)' $ relative to $\Mfrak^\orb_{g,r}(B\Gamma, d)_{\log}^\circ$.  Then given a superpotential $w: \cT \to \mathbb{A}^1$,
Equation~\eqref{dwS} defines a cosection $dw$ of $\cc Ob_{LG(\cT) '} = \mathbb R^1 \pi_* \cV(-\sG)$. 
Let $h^1/h^0 (\mathbb E)(dw)$ be the refined cone stack associated to $\mathbb E$ and $dw$, (see Equation (3.16) of \cite{CL}).
As shown in the proof of Proposition \ref{prop: CLL},  
the relative intrinsic normal cone $[c_{LG(\cT)' /\Mfrak^\orb_{g,r}(B\Gamma, d)_{\log}^\circ}]$ lies in $Z_*(h^1/h^0 (\mathbb E)(dw))_\QQ$. 
Hence, we may define the following class.
\begin{definition}
The \newterm{cosection localized virtual fundamental class} is the class
\[
[LG_{g,r}(\cT, d)']^{\op{vir}}_{dw} :=   s^!_{h^1/h^0 (\mathbb E), dw} [c_{LG(\cT)' /\Mfrak^\orb_{g,r}(B\Gamma, d)_{\log}^\circ}].
\]
where 
$$s^!_{h^1/h^0 (\mathbb E), dw}: A_*(h^1/h^0 (\mathbb E)(dw))_{\QQ}  \ra A_* ([LG_{g,r}(\cZ, d)])_{\QQ}$$
is the cosection localized Gysin map \cite[Proposition 1.3]{KiemLi}.
\end{definition}

The following result follows from results in \cite{CLL} and the compatibility of $\alpha$ and $dw$ (Lemma~\ref{alpha and dw}).
\begin{proposition}\label{prop: CLL}
The following equality holds in $A_*(LG_{g,r}(\cZ, d))_{\Q}$
\[
[LG_{g,r}(\cZ, d)]^{\op{vir}} = [LG_{g,r}(\cT, d)']^{\op{vir}}_{dw} 
\]
\end{proposition}
\begin{proof}
Note that $LG(\cT)'$ is realized as the zero locus of a section $\beta \in \Gamma( \square', \bar B)$ in the smooth Deligne--Mumford stack $\square '$, and that over $\square '$, the object $\mathbb R \pi_* \cV$ is resolved by the two term complex $[\bar A_1 \oplus \bar A_2 ' \to \bar B]$.  In this special setting, the cosection localized virtual class can be described more concretely.  

Let $C_{LG(\cT)' / {\square '}}$ be the normal cone to $LG(\cT)'$  in $\square '$.
Since $LG(\cT)'=Z(\beta)$, the cone $C_{LG(\cT)' / {\square '}}$ can be viewed as
a closed substack in $\bar B|_{LG(\cT)'}$. Since $\ka |_{LG(\cT)'} \in \Hom (\bar B|_{LG(\cT)'}, \cO _{LG(\cT)'})$ becomes zero when it is restricted to $C_{LG(\cT)' / {\square '}}$
by $\ka \circ \kb =0$ in $\Box '$, 
the cone cycle is contained in $$\bar B|_{LG(\cT)'} (\ka |_{LG(\cT)'} ) := \bar B|_{Z(\ka, \kb )} \cup \ker (\bar B |_{LG(\cT)' \setminus Z(\ka, \kb)} \ra \cO _{LG(\cT)' \setminus Z(\ka, \kb)}).$$
Again there exists a cosection localized Gysin map 
$$s^!_{\bar B|_{LG(\cT)'}, \ka |_{LG(\cT)'}} : A_* ( \bar B|_{LG(\cT)'} (\ka |_{LG(\cT)'})  )_{\QQ}  \ra A_* (Z(\ka, \kb ))_{\QQ}. $$

The relative instrinsic normal cone  $c_{LG(\cT)' /\Mfrak^\orb_{g,r}(B\Gamma, d)_{\log}^\circ}$  can be explicitly described as the quotient cone stack $[C_{LG(\cT)' / {\square '}} / \bar A]$ since $\square'$ is smooth with relative tangent bundle $\bar A$.   This yields the following fiber diagram
\begin{equation} \label{relative cone diagram}
\begin{tikzcd}
C_{LG(\cT)' / {\square '}}\ar[r] \ar[d] & \op{tot} \bar B \ar[d] \\
{c_{LG(\cT)' /\Mfrak^\orb_{g,r}(B\Gamma, d)_{\log}^\circ}} \ar[r]  &  {[\op{tot} \bar B / \bar A]}  \\
\end{tikzcd}
\end{equation}

 We now have the following chain of equalities
\begin{align*}
[LG_{g,r}(\cT, d)']^{\op{vir}}_{dw} = \ &   s^!_{h^1/h^0 (\mathbb E), dw} [c_{LG(\cT)' /\Mfrak^\orb_{g,r}(B\Gamma, d)_{\log}^\circ}] \\
 = \ & s^!_{\bar B|_{LG(\cT)'}, \ka |_{LG(\cT)'}} \gamma^* [c_{LG(\cT)' /\Mfrak^\orb_{g,r}(B\Gamma, d)_{\log}^\circ}] \\
 = \ & s^!_{\bar B|_{LG(\cT)'}, \ka |_{LG(\cT)'}} [C_{LG(\cT)' / {\square '}}] \\
 = \ & \op{td}(\bar B|_{\square '}) \ \chp_{LG_{g,r}(\cZ, d)}^{\square '} (K ') [\square ']  \\
 = \ & [LG_{g,r}(\cZ, d)]^{\op{vir}}.
\end{align*}

The first line is by definition.  The second line uses the definition of the Gysin map together with the fact that $\sigma$ agrees with $ \ka |_{LG(\cT)'}$ (Lemma~\ref{alpha and dw}).   The fourth line follows from \eqref{relative cone diagram}.  The fifth line is \cite[Proposition 5.10]{CLL}.  The final line is also by definition.
\end{proof}

With the above result and Proposition~\ref{prop :GWcomp}, we can identify our GLSM invariants in the geometric phase with invariants defined by the cosection localized virtual class.  
\begin{theorem} \label{t:mainGWcomp}
After identifying $\op{H}^*(\cZ)$ with $\op{HH}_*([\cT/\CC^*_R], w)$ via $\varphi^{\td}_*$,
the invariants defined by the cosection localized virtual fundamental class of \cite{CL} are equal to the invariants $\Lambda_{g,r,d}$ for the corresponding GLSM from Definition~\ref{d:invariants}.  
\end{theorem}
\begin{proof}
This follows immediately from Propositions~\ref{prop :GWcomp} and \ref{prop: CLL}.
\end{proof}

By results of \cite{CL, KOcos, CLcos2}, the cosection localized virtual class is equal to the Behrend--Fantechi \cite{BF} virtual class on the space of stable maps to $\cZ$ up to a sign.  This was proven at the level of numerical invariants for the quintic threefold by Chang and J. Li  in \cite{CL}.  The general case was proven by the third author and J.-Oh in \cite{KOcos}, and, using different methods, for $\cZ$ a hypersurface in projective space by Chang and M.-L. Li in \cite{CLcos2}.

\begin{theorem}[{\cite{CL, KOcos, CLcos2}}]\label{thm: CL}
Given a geometric phase GLSM, the cosection localized virtual class agrees with the Behrend--Fantechi virtual class up to a sign: 
\[[LG_{g,k}(\cT, d)']^{\op{vir}}_{dw}  = (-1)^{\rank(\mathbb R \pi_*(\cV_2 '))}\left[\sMbar_{g,k}(\cZ,d)\right]^{\op{vir}},\]
where $\rank(\mathbb R \pi_*(\cV_2 '))$ is defined to be $\rank( A_2) - \rank( B_2)$ for $[A_2 \to B_2]$ a resolution of $\mathbb R \pi_*(\cV_2 ')$ by vector bundles.  
\end{theorem}

\begin{corollary}
After identifying $\op{H}^*(\cZ)$ with $\op{HH}_*([\cT/\CC^*_R], w)$ via $\varphi^{\td}_*$,
the invariants defined by  $\Lambda_{g,r,d}$ are equal to the Gromov--Witten invariants of $\cZ$ up to a sign.
\end{corollary}

\begin{proof}
This follows immediately from Theorem~\ref{t:mainGWcomp} and Theorem~\ref{thm: CL}.  
\end{proof}

This holds, in particular, when $V_2 = 0$, although in this case the proof can be simplified considerably.  
\begin{corollary} \label{cor: w=0}
In the special case where $V_2 =0$ and $w=0$, the Gromov--Witten invariants defined by the Behrend--Fantechi virtual cycle for the GIT quotient $\cX = [V_1 \sslash_\theta G]$ agree with $\Lambda_{g,r,d}$. 
\end{corollary}
\begin{proof}
In this special case, we get the Behrend--Fantechi virtual cycle without having to apply the cosection localization method i.e.\ \cite[Proposition 5.10]{CLL} is used in the case where the cosection is 0.
\end{proof}

\subsection{Comparison with the Polishchuk--Vaintrob construction}
\label{sec: compare PV}
Let $G$ denote a finite abelian subgroup of $(\CC^*)^m$, and let $w: \A^m \to \A^1$ denote a $G$-invariant function with an isolated singularity at the origin. 
Suppose that a $\CC^*_R$-action on $\A ^m$ is given by an embedding into $(\CC^*)^m$ such that: $w$ has a homogeneous degree $\bdeg$ with respect to $\CC^*_R$-action
and $G\cap \CC^*_R$ is a finite cyclic group $\lan J \ran$ with order $\bdeg$. 
With this input data, a cohomological field theory called
Fan--Jarvis--Ruan--Witten (FJRW) theory has been developed by Fan--Jarvis--Ruan \cite{FJR13} in the analytic category 
and by Polishchuk--Vaintrob \cite{PV} in the algebraic category using factorizations.  These theories can be thought of as giving invariants of the singularity defined by 
\[  w: [\A^m/ G] \to \A^1.\]

A GLSM corresponding to such a singularity is known as an (abelian) \newterm{affine phase GLSM}.  We construct it as follows.
Take $\Gamma = G \cdot \CC^*_R  \subseteq (\CC^*)^m$, set $V_1=\Spec(\CC)$ and $V_2=\A^m$, and let $\theta : G \ra \CC^*$ be trivial.
 Note that $(V, G, \CC^*_R, \theta, w)$ is an abelian GLSM 
such that $\VG$ is the \'etale gerbe $BG$. Here $\chi$ is defined by $g \cdot \lambda \mapsto \lambda^\bdeg$ for $g\in G, \lambda\in \CC^*_R$. 
In this section we show the invariants $\Lambda_{g,r,d}$ defined for the GLSM $(V, G, \CC^*_R, \theta, w)$ agree with Polishchuk--Vaintrob's definition of the FJRW invariants of $ w: [\A^m/ G] \to \A^1$.

Note first that in the special case where $V_1$ is rank zero, 
the stack $LG_{g,r}(BG, 0)$ is a smooth and proper Deligne--Mumford stack over $\Spec (\CC)$, with projective coarse moduli.  This follows from Proposition~\ref{p: projmod} after observing that $\sMbar_{g,r}(\Xrig, 0) = \sMbar_{g,r}(BG_1, 0)$ is smooth.  
\subsubsection{The Polishchuk--Vaintrob construction}
In \cite{PV} the stack $LG_{g,r}(BG, 0)$ is denoted as $\cc S_{g,r} = \cc S_{g,r, \Gamma, \chi}$ and is referred to as the moduli of {\it $\Gamma$-spin structures}.  We will adopt this notation for the remainder of this section ot simplify the comparison.
In \cite{PV}, a cohomological field theory {\it with coefficients in $\CC[\widehat G]$} is defined as follows.  First, a rigidified stack $\cc S_{g, r}^{\op{rig}} \to \cc S_{g, r}$ is defined, which parametrizes $\Gamma$-spin structures together with trivializations
\[\cc P|_{p_i} \cong \Gamma/\langle h_i \rangle\]
for each marked point $p_1, \ldots, p_r$, where $h_i$ denotes the generator of the isotropy group at $p_i$, viewed as an element of $\Gamma$.   
There is a natural action of $G^r$ on $S_{g, r}^{\op{rig}}$ by scaling the respective trivializations.  Note that the induced action of the diagonal $G \leq G^r$ is trivial.

A $\CC^*_R$-equivariant resolution  \[[A \stackrel{d}{\to} B] \cong \mathbb R\pi_*( \cc V)\]  is constructed over $\cc S_{g, r}^{\op{rig}}$ satisfying Conditions~\ref{assume ev chainlevel}, \ref{assume alpha chainlevel}, and \ref{compatible}.  This is then used to define a 
Koszul factorization $K^{PV} = K^{PV}_{g,r} \in \dabsfact{[\op{tot}(A)/\CC^*_R],  - \op{ev}^* (\boxplus_{i=1}^r w)}$.\footnote{To be precise the resolution is constructed over a different rigidification $S_{g, r}^{\op{rig}, \circ}$.  The corresponding factorization is then pulled back to define $K^{PV}$ (see the end of \S4.2 of \cite{PV}).}  

The $G^r$ action on  $S_{g, r}^{\op{rig}}$ extends to an action on $\op{tot}(A)$.  The trivializations at the marked points allow one to define a $G^r$-equivariant evaluation map 
\[\op{ev}: \op{tot}(A) \to \coprod_{\bg} \bigoplus_{i=1}^r V^{g_i},\]
where the disjoint union is over all conjugacy classes (i.e. $r$-tuples) $\bg$ in $G^r$.  
This defines a functor 
\[ \begin{array}{rrcl}
\Phi_{K^{PV}_{g,r}}: &\dabsfact{\coprod_{\bg} [\bigoplus_{i=1}^r V^{g_i}/\Gamma],  {\textstyle \boxplus_{i=1}^r} w} &\to &\dabsfact{[S_{g, r}^{\op{rig}}/\Gamma], 0} \cong \op{D}_G(S_{g,r}^{\op{rig}})\\
&E &\mapsto & p_* ( K^{PV}_{g,r} \otimes \op{ev}^*(E)),
\end{array}
\]
where $p$ is the map $[\op{tot}(A)/\Gamma] \to [\cc S_{g, r}^{\op{rig}}/\Gamma]$.

It is shown in \cite[Corollary 2.6.2]{PV} that $(\cc H^{\op{ext}})^{\otimes r}$, the $r$-fold product of the state space, is isomorphic to $\op{HH}_*(\coprod_{\bg} \bigoplus_{i=1}^r V^{g_i},  {\textstyle \boxplus_{i=1}^r} w)^{G^r}$.  We obtain a map on $(\cc H^{\op{ext}})^{\otimes r}$ as follows:
\begin{align} \phi_g: (\cc H^{\op{ext}})^{\otimes r} \hookrightarrow &\op{HH}_*\left({ \textstyle \coprod_{\bg}  [\bigoplus_{i=1}^r } V^{g_i}/\Gamma],  {\textstyle \boxplus_{i=1}^r} w\right) \xrightarrow{ {\Phi_{K^{PV}_{g,r}}}_*} \op{HH}_*(\cc S_{g, r}^{\op{rig}} \times BG) \cong \nonumber \\ &\op{HH}_*(S_{g, r}^{\op{rig}}) \otimes \CC[\widehat{G}] \xrightarrow{\bar \phi_{\op{HKR}}} \op{H}^*( \cc S_{g,r}^{\op{rig}}) \otimes \CC[\widehat{G}]. \nonumber
\end{align}
Finally, if we let $\op{st}_g: \cc S_{g, r}^{\op{rig}} \to \sMbar_{g,r}$ denote the forgetful map, the Polishchuk--Vaintrob invariants are defined to be 
\[\Lambda_{g,r}^{PV} := \frac{1}{\deg(\op{st}_g)} \cdot {\op{st}_g}_* \circ \phi_g: (\cc H^{\op{ext}})^{r} \to \op{H}^*(\sMbar_{g,r}) \otimes \CC[\widehat{G}].\]
By \cite[Theorem 5.1.2]{PV}, these invariants define a cohomological field theory with coefficients in $\CC[\widehat G]$.

By \cite[Theorem 2.6.1]{PV} , the state space $\cc H^{\op{ext}}$ decomposes into a direct sum $\oplus_{h \in G}  e_h(\cc H^{\op{ext}})$ indexed by elements of $G$ (see \S\ref{s:theresstaspa} and Remark~\ref{r:resPV}).  More explicitly, the vector space $\CC[\widehat G]$ has an idempotent basis $\{e_h\}_{h \in G}$ where 
\[ e_h := \frac{1}{|G|}\sum_{\eta \in \widehat G} \eta^{-1}(h) [\eta].\]
Here the notation $e_h(\cc H^{\op{ext}})$ denotes the image of the map $e_h$ in $\cc H^{\op{ext}}$.  Recall that $\cc H^{\op{red}} := e_{id}(\cc H^{\op{ext}})$.

Define the {\it reduced} map 
by
\[ \phi_g^{\op{red}} =  \pi_1 \circ \phi_g |_{{\cc H^{\op{red}}}^{\otimes r}}\]
where $\pi_1$ is the map \[\op{H}^*( \cc S_{g,r}^{\op{rig}}) \otimes \CC[\widehat{G}] \xrightarrow{e_{id}} \op{H}^*( \cc S_{g,r}^{\op{rig}}) \otimes \CC[e_{id}] \cong \op{H}^*( \cc S_{g,r}^{\op{rig}}).\]

Let $\bg$ be an r-tuple of conjugacy classes $(g_1), \ldots, (g_r)$ and choose $s_i \in \cc H^{\op{red}}_{(g_i)}$.
Polishchuk--Vaintrob define the \newterm{reduced PV invariants} by 
\[\lambda_{g,r}^{PV}(s_1, \ldots, s_r) := \sigma_g \bg \cdot \frac{1}{\deg(\op{st}_g)} \cdot {\op{st}_g}_* ( \td( R \pi_* (\cc V_2))^{-1} \cdot \phi_g^{\op{red}}(s_1, \ldots, s_r)),\]
where $\sigma_g \bg$ is a certain root of unity depending on $\bg$.

\subsubsection{The comparison}
In this section we prove that the reduced PV invariants $\lambda_g^{PV}$ agree with those of the GLSM defined in \ref{d:invariants}.
\begin{lemma}\label{l:reducedinvts}
The map $\phi_g^{\op{red}}$ is equal to $s^* \circ \phi_g|_{{\cc H^{\op{red}}}^{\otimes r}}$  where $s$ is the map 
\[s: \cc S_{g, r}^{\op{rig}} \to \cc S_{g, r}^{\op{rig}} \times BG.\] 
\end{lemma}
\begin{proof}
The definition of $\phi_g$ implicitly uses the fact that the vector space $\CC[\widehat G]$ is isomorphic to $\op{HH}_*(BG) \cong \CC[G]$ (see, e.g., \cite[Example 6.4]{Cal03}).  
This isomorphism is given by the map 
\[\op{ch}: \CC[\widehat G]\cong K(BG) \otimes \CC \to \op{HH}_*(BG) \cong \CC[G]\]
which maps the character $[\eta] \in \widehat{G}$ to 
\[\op{ch}([\eta]) := \frac{1}{|G|}\sum_{h \in G} \eta(h^{-1}) h.\]  Note that $\op{ch}(e_h) = h/|G|$.  
Furthermore, under this isomorphism, the decomposition
\[
\cc H^{\op{ext}} = \bigoplus_{h \in G}  e_h(\cc H^{\op{ext}})
\]
is compatible with the isomorphism
\[
\op{HH}_*(S_{g, r}^{\op{rig}} \times BG) \cong \op{HH}_*(S_{g, r}^{\op{rig}}) \otimes_k \op{HH}_*(BG). 
\]
Hence, it's enough to consider the case where $S_{g, r}^{\op{rig}}$ is a point.

Consider the map $\bar s: pt \to BG$.  One observes that 
the pullback map
\[\bar s^*: \CC[G] \cong \op{HH}_*(BG) \to \op{HH}_*(pt) \cong \CC\] 
is given by
\[\bar s^* (\ch(e_h)) = \ch (\bar s^*(e_h)) =  \frac{1}{|G|} \sum_{\eta \in \widehat G} \eta^{-1}(h) \ch(\cO_{pt})=\delta_{h, id}.\]
Thus applying the operator $e_{id}$ to $\CC[\widehat G]$ is equivalent to pulling back by $\bar s$.
\end{proof}

Let  $(V, G, \CC^*_R, \theta, w)$ be the GLSM described at the beginning of this section.  Note in particular that $G$ is finite.
\begin{proposition}\label{p:factcomp}
Over $LG_{g,r}(BG, 0)$ there exists a resolution $[A_0 \stackrel{d_0}{\to} B_0]$ of $\pi_*( \cc V)$ satisfying Conditions~\ref{assume ev chainlevel}, \ref{assume alpha chainlevel}, and \ref{compatible}.  
The induced Koszul factorization $\{-\alpha_0, \beta_0\}$ on $\op{tot}(A_0)$ defines $K_{g,r,0}$.  

In particular, 
$\Lambda_{g,r,d}(s_1, \ldots, s_r)$ can be computed as
\[ \Lambda_{g,r,d}(s_1, \ldots, s_r) = \op{proj}_* \left({\op{td}}(-  \mathbb R\pi_* \mathcal V) \cup \left( \bar \phi_{\op{HKR}}\circ \ord \circ q^* \circ{(\Phi_{\{-\alpha_0, \beta_0\}})_*}(s_1, \ldots, s_r) \right) \right)\]
where
\[q:  \cc S_{g,r} \to [\cc S_{g,r}/\mu_{\bdeg}]\] denotes quotient by the trivial action of $\mu_{\bdeg}$ on $\cc S_{g,r}$ and 
$(\Phi_{\{-\alpha_0, \beta_0\}})_*: {\cc H^{\op{red}}}^{\otimes r} \to \cc S_{g,r}$ is the integral transform with kernel $\{-\alpha_0, \beta_0\}$.
\end{proposition}
\begin{proof}
Since $G$ is finite, the open subset $\Mfrak^\orb_{g,r}(B\Gamma, 0)_{\log}^\circ  \subseteq \Mfrak^\orb_{g,r}(B\Gamma, 0)_{\log}$ is in fact equal to $\cc S_{g,r} = LG_{g,r}(BG, 0)$.  In particular  it is a quotient stack with projective coarse moduli.  
This allows the construction of a resolution $[A_0 \stackrel{d_0}{\to} B_0]$ of $\pi_*( \cc V)$ satisfying Conditions~\ref{assume ev chainlevel}, \ref{assume alpha chainlevel}, and \ref{compatible} over $\Mfrak^\orb_{g,r}(B\Gamma, d)_{\log}^\circ$.  Thus our two-step procedure  for constructing the fundamental factorization $K$ is unnecessary; one can use $[A_0 \to B_0]$ to construct a Koszul resolution $\{-\alpha_0, \beta_0\}$ on $\op{tot}(A_0)$.  This may be viewed simply as a degenerate case of 
our general construction, where $V_1 = 0$.
\end{proof}

From \cite[Section 4]{PV}, it is apparent that the $\CC^*_R$-equivariant resolution $[A \xrightarrow{d} B]$ can be assumed to be the pullback of $[A_0 \xrightarrow{d_0} B_0]$ under the map $\cc S_{g,r}^{\op{rig}} \to \cc S_{g,r}$.  
Consequently the fundamental factorization $K^{PV}$ is the pullback of $\{-\alpha_0, \beta_0\}$ under the map $\tot(A) \to \tot (A_0)$.
Indeed this is why $\{-\alpha_0, \beta_0\}$ is referred to as a {\it PV factorization} (Definition~\ref{d:pvfact}).

\begin{theorem}\label{p:PVcomp1}
The GLSM invariants $\Lambda_{g,r,0}$ are related to the reduced invariants of \cite{PV} by an explicit factor.
Given  $s_i \in \cc H^{\op{red}}_{(g_i)}$ for $1 \leq i \leq r$, let $m_i$ be the order of $g_i$ as in \S\ref{s:GLSMinvs}, then
\[
\begin{array}{rcl}\Lambda_{g,r,0}(s_1, \ldots, s_r) &=& \frac{(m_1 \cdots m_r)\deg(\op{proj})}{\sigma_g \bg}\lambda_{g,r}^{PV}(s_1, \ldots, s_r) \\
&=&   \frac{(m_1 \cdots m_r)}{\deg(\op{rig})}   {\op{st}_g}_* ( \td( \RR \pi_* (\cc V_2))^{-1} \cdot \phi_g^{\op{red}}(s_1, \ldots, s_r)),
\end{array}
\] where $\op{proj}$ is the forgetful map $\cc S_{g,r} \to \sMbar_{g,r}$.
\end{theorem}

\begin{proof}
Consider the following commutative diagram.
\[
\begin{tikzcd}
 \dabsfact{[\tot( A)/\Gamma], 0} \arrow[d, "\mathbb R \bar \pi_*"] \arrow[r, swap, "\mathbb L \widetilde t^*"] & \dabsfact{[\tot( A)/\CC^*_R], 0} \arrow[d, "\mathbb R \pi_*"]  
 & \arrow[l, "\mathbb L \widetilde{\op{rig}}^*"] \dabsfact{[\tot( A_0)/\CC^*_R], 0} \arrow[d, "\mathbb R  {\pi_0}_*"] \\
 \dabsfact{[\cc S_{g,r}^{\op{rig}}/\Gamma], 0} \ar[d, "\cong"] \ar[r, swap, "\mathbb L  t^*"] & \dabsfact{[\cc S_{g,r}^{\op{rig}}/\CC^*_R], 0}  \ar[d, "\cong"]&  \arrow[l, "\mathbb L \overline{\op{rig}}^*"]  \dabsfact{[\cc S_{g,r}]/\CC^*_R], 0}  \ar[d, "\cong"] \\
  \dabsfact{[\cc S_{g,r}^{\op{rig}}/G]} \ar[r, swap, "\mathbb L t^*"]& \dabsfact{[\cc S_{g,r}^{\op{rig}}/\mu_{\bdeg}]}  &\arrow[l, "\mathbb L  \overline{\op{rig}}^*"] \dabsfact{[\cc S_{g,r}/\mu_{\bdeg}]}
\end{tikzcd}
\]

The top squares commute by flat pullback.  The bottom squares commute by definition of the equivalence occurring in each vertical arrow.
Let \[\tilde q: \cc S_{g,r}^{\op{rig}} \to [\cc S_{g,r}^{\op{rig}}/\mu_{\bdeg}]\] denote the quotient by the trivial action of $\mu_{\bdeg}$ on $\cc S_{g,r}^{\op{rig}}$.
The proof follows by observing that both of the above invariants are equal to
\[\op{proj}_* \circ \op{rig}_* \left(\op{\td}( - \mathbb R\pi_* \mathcal V)
\bar \phi_{\op{HKR}} \circ \tilde q^* \circ {\Phi_{ \mathbb L  \widetilde{\op{rig}}^*  \{-\alpha_0, \beta_0\}}}_* (s_1, \ldots, s_r)\right)\]
after scaling by either $(m_1 \cdots m_r)/\deg(\op{rig})$ in the case of $\Lambda_{g,r,0}$ or $\sigma_g \bg /\deg(\op{rig} \circ \op{proj})$ in the case of $\lambda_{g,r}^{PV}$.   For $\Lambda_{g,r,0}$ this follows from flat pullback of factorizations together with the definition of $\bar \phi_{\op{HKR}}$ and 
Proposition~\ref{p:factcomp}.  For $\lambda_{g,r}^{PV}$
one also uses Lemma~\ref{l:reducedinvts} and the fact that 
$\mathbb L  \widetilde t^* K^{PV}$ is equal to  $\mathbb L \widetilde{\op{rig}}^* \{-\alpha_0, \beta_0\}$.
\end{proof}

\subsection{Comparison with other affine phases}\label{s:enhanced PV comparison}
\begin{definition}\label{d:equivGLSM}
Define the GLSM $(V, G, \CC^*_R, \theta, w)$ to be \newterm{equivalent} to $(V', G ' , \CC^*_R,\theta ', w')$ if the associated $\CC^*_R$-equivariant LG spaces
$([V/\!\!/_{\theta } G] , \bar{w})$ and $([V'/\!\!/_{\theta '} G'] , \bar{w}')$ are isomorphic where
$\bar{w}$ and $\bar{w}'$ are the induced regular functions on $[V/\!\!/_{\theta } G]$,  $[V'/\!\!/_{\theta '} G']$, respectively. 
\end{definition}
Recall in \S \ref{sec: compare PV}, we had a finite group $G$ and a one-dimensional (not necessarily connected) commutative algebraic group $\Gamma$.
Let us rename $G$ by $G'$, $V$ by $V' = \A^m$,  etc.  This gives a GLSM $(\A^m, G', \CC^*_R, 0, w')$.  Let $(V, G, \CC^*_R, \theta, w)$ be an equivalent abelian \newterm{affine phase} i.e.\ assume that $\VG \cong B G'$ and that the LG space  obtained from $(V, G, \CC^*_R, \theta, w)$ agrees with that of $(\A^m, G', \CC^*_R, 0, w')$ in the sense of Definition~\ref{d:equivGLSM}.
In this section, we will show that the GLSM invariants associated to $(V, G, \CC^*_R, \theta, w)$ agree with those of $(\A^m, G', \CC^*_R, 0, w')$ and consequently, with 
the FJRW invariants constructed in \cite{PV}.

The moduli spaces $\cc S_{g,r} = LG_{g,r}(BG', d)$
and $LG_{g,r}(\VG, d)$ are  isomorphic, as can be seen, for instance, from the construction in \S~\ref{s:conFF}.  Thus we can restrict our attention to degree zero, where the moduli space is non-empty.
Let $[A_1 \xrightarrow{d_1} B_1]$ be the resolution of $\mathbb R\pi_*(\cV_1)$ and $U$ the open Deligne--Mumford substack of $\tot(A_1)$ defined in Theorem~\ref{c:swn} (see also Diagram \ref{d:pullback}).
Let $[\bar A \xrightarrow{\bar d} \bar B]$ denote the resolution of $\mathbb R\pi_*(\cV)$ over $U$ \eqref{e:roofbar}.  We may assume without loss of generality that this complex splits as 
$[\bar A_1 \oplus \bar A_2 \xrightarrow{\bar d_1, \bar d_2} \bar B_1 \oplus \bar B_2]$
where $[\bar A_i \xrightarrow{\bar d_i} \bar B_i]$ is a resolution of $\mathbb R\pi_*(\cV_i)$.
Recall that $\square $ lies in $ \tot(\bar A)$ over $U$, and $\{-\alpha, \beta\}$ is the Koszul factorization associated to the vector bundle $E$ equal to the pullback of $\bar B$ with section $\beta $ induced by  $(\bar d_1, \bar d_2)$ and cosection $\alpha$.

Let $Z$ denote the restriction of $\bar A_2$ to $LG_{g,r}(BG', 0)$.
Recall that $LG_{g,r}(BG', 0)$ is a closed substack of $U$.  
Note that ${\bar A_1}|_{LG_{g,r}(BG', 0)}$ 
has a tautological section induced by the map $\pi_*(\cV_1) \to \bar A_1$.  This gives a closed immersion of $LG_{g,r}(BG', 0)$ into the total space $\tot(\bar A_1)$ over $U$.  
Combining this with the identity map on $\bar A_2$ yields a closed immersion
\[j: Z \hookrightarrow \square  \subseteq \tot(\bar A).\]

Let $E'$ denote the pullback of the vector bundle $\bar B_2$ to $Z$.  The bundle $E'$ has a natural section $\beta '$ induced by $\bar d_2$, and a cosection $\alpha '$ defined by restricting $\alpha$.

\begin{proposition}
The factorizations $K' = \{-\alpha ', \beta '\}$ and $K = \{-\alpha, \beta\}$ are related by pushforward.
\end{proposition}

\begin{proof}
This follows from a slight variation of the arguments of \S\ref{s:ind}.
Consider the inclusion of vector bundles $\bar B_2  \subseteq E$ over $\square$.  The quotient $E/\bar B_2$ is equal to $ \bar B_1$, and 
\[\beta \mod( \bar B_2) = \bar \beta_1.\]

The zero locus of $\bar \beta_1$ in $\tot(\bar A)$ is the total space $\tot(\pi_*(\cV_1) \oplus \bar B_2)$ over $U$.  After restricting to $\square$, the zero locus of $\bar \beta_1$ is exactly the inclusion of $Z$ by $j$.
Note that on each connected component, 
\[
\begin{array}{ll} \dim Z &= \dim LG_{g,r}(BG', 0) + \rank \bar A_2 \\
& = \dim \fM_{g,r} + \rank \bar A_2  \\
& = \dim \Mfrak^\orb_{g,r}(B\Gamma, 0)_{\log}^\circ  + \rank(\mathbb R \pi_*(\cV_1))+ \rank \bar A_2 \\
 &= \dim \square - \rank E/\bar B_2.
 \end{array}
\]
The first equality is from the definition of $Z$ and the second equality is from \cite[Proposition 3.2.6]{PV} or \cite[Theorem 2.2.6]{FJR13}. 
The third equality is due to the fact that $\Mfrak^\orb_{g,r}(B\Gamma, 0)_{\log}^\circ$ is a finite cover of $\Mfrak^\orb_{g,r}(BG , 0)^\circ$.  Therefore $\dim  \Mfrak^\orb_{g,r}(B\Gamma, 0)_{\log}^\circ  = \dim \Mfrak^\orb_{g,r}(BG , 0)^\circ = \dim \fM_{g,r} - \chi (\mathbb R \pi_*  \cP \times_{G } \mathfrak g )
= \dim \fM_{g,r} - \chi (\mathbb R \pi_*\cV_1)$, by orbifold Riemann--Roch \cite[Theorem~7.2.1]{AGV}.
The fourth equality is from the definitions of $\square$ and $E$.

We conclude that $\beta \mod( \bar B_2)$ is a regular section of $E/\bar B_2$.  We observe further that the section $\beta|_{Z} \in \Gamma(Z, E')$ is equal to $\beta '$, 
and that $\alpha '  = \alpha \mod((E')^\perp|_Z)$ by construction.  It follows from Proposition~\ref{PV original} that  
 \begin{eqnarray}\label{push_PV_Fact} \mathbb R j_*(\{-\alpha ', \beta '\}) = \{-\alpha, \beta\}, \end{eqnarray}  
 which completes the proof.
\end{proof}

\begin{theorem}\label{t:equivGLSM}
The invariants $\Lambda_{g,r,0}$ for the equivalent GLSMs $(V, G, \CC^*_R, \theta, w)$ and $(\A^m, G', \CC^*_R, \bar \theta' = 0, w')$ are equal.
\end{theorem}
\begin{proof}
The proof follows a similar argument to Theorem~\ref{t:finind}.
Observe that by construction, the factorization $\{-\alpha_0, \beta_0\}$ of Proposition~\ref{p:factcomp} can be assumed without loss of generality to be exactly $\{-\alpha ', \beta '\}$ from the previous proposition.  Here $\tot(A_0)$ is identified with $Z$.
 Recall $\widetilde{U}$ denotes a desingularization of the closure of $U$, and let 
 \[\tilde i: \cc S_{g,r}  \to \widetilde{U}\]
 denote the inclusion.
By the previous proposition, 
\[\tilde i_* \circ (\Phi_{K '})_* =  (\Phi_K )_*.\]
The GLSM invariants for $(\A^m, G', \CC^*_R, 0, w')$ are by definition equal to 
 \[ \op{proj}_* \circ \tilde i_* \left({\op{td}}( \ominus  \mathbb R\pi_* \mathcal V_2) \cup \left( \bar \phi_{\op{HKR}} \circ({\Phi_{K '}})_* ( - ) \right) \right)\]
up to a scaling, where $\op{proj}_*$ is the projection to $\sMbar_{g,r}$.   The invariants for $(V, G, \CC^*_R, \theta, w)$ on the other hand are given by 
\[ \op{proj}_* \left({\op{td}}(T_{\widetilde{U}/\Mfrak^\orb_{g,r}(B\Gamma, d)_{\log}} \ominus  \mathbb R\pi_* \mathcal V) \cup \left( \bar \phi_{\op{HKR}} \circ{(\Phi_{K})_*}( - ) \right) \right)
\] after the same scaling.

The relative tangent bundle of $\tilde i$ is equal to 
 \begin{align*}
 &T_{\cc S_{g,r}/Z} + T_{Z/\square} + T_{\square/\tilde U} \\
 \cong& (- \bar A_2) +  (- \bar B_1) + (\bar A_1 + \bar A_2 -  A_1) \\
 \cong& \mathbb R\pi_* \cc V_1 - A_1,\end{align*}
 where the isomorphism $T_{Z/\square} \cong - \bar B_1$ follows from the fact proven in the previous proposition that $\bar \beta_1$ is a regular section of $\square$.
 Note further that the pullback of $T_{\widetilde{U}/\Mfrak^\orb_{g,r}(B\Gamma, d)_{\log}}$ via $\tilde i$ is equal to $A_1$.  In particular
 \[  -\mathbb R\pi_* \cV_2  = -\mathbb R\pi_*(\cV) + \mathbb R\pi_* \cc V_1 =  -\mathbb R\pi_*(\cV) + \tilde i^*(T_{\widetilde{U}/\Mfrak^\orb_{g,r}(B\Gamma, d)_{\log}}) + T_{\cc S_{g,r} / \widetilde{U}}.\] 
Applying the projection formula and \eqref{e:eq12}, we conclude that 
   \begin{align}
  &\tilde i_* \left( \td\left(\ominus \mathbb R\pi_*\cV_2\right) 
  \bar \phi_{\op{HKR}} ( -)\right) \label{e:ram3}
  \\  \nonumber
  = &\td\left(T_{\widetilde{U}/\Mfrak^\orb_{g,r}(B\Gamma, d)_{\log} } \ominus \mathbb R \pi_* \cV\right)\tilde i_* \left( \td( T_{\cc S_{g,r} / \widetilde{U}}) \bar \phi_{\op{HKR}} ( -)\right)\\
=&\td\left(T_{\widetilde{U}/\Mfrak^\orb_{g,r}(B\Gamma, d)_{\log}} \ominus \mathbb R \pi_* \cV\right)\bar \phi_{\op{HKR}} \tilde i_*  ( -). \nonumber
  \end{align}
 The result follows.
 \end{proof}
 
 \begin{corollary}
 Given an affine phase GLSM $(V, G, \CC^*_R, \theta, w)$ constructed as above, let 
 \[\bar w: [\VmodtG] = [\A^m/G'] \to \A^1\]
 be the corresponding singularity.
 Then the GLSM invariant $\Lambda_{g,r,0}(s_1, \ldots, s_r)$ is equal  to the  reduced invariant of \cite{PV} of the singularity after scaling by the factor $(m_1 \cdots m_r) \frac{\deg(\op{proj})}{\sigma_g \bg}$.
 \end{corollary}
 
  \begin{proof}  This follows immediately from the previous theorem and Theorem~\ref{p:PVcomp1}.
   \end{proof}

\bibliographystyle{alpha} 

\bibliography{Part1bib}{}

\begin{bibdiv}
\begin{biblist}

\bib{ACC}{article}{
      author={Arinkin, Dima},
      author={C{\u{a}}ld{\u{a}}raru, Andrei},
      author={Hablicsek, M{\'a}rton},
       title={Formality of derived intersections and the orbifold {H}{K}{R}
  isomorphism},
        date={2014},
     journal={preprint arXiv:1412.5233},
}

\bib{AGOT}{article}{
      author={Abramovich, Dan},
      author={Graber, Tom},
      author={Olsson, Martin},
      author={Tseng, Hsian-Hua},
       title={On the global quotient structure of the space of twisted stable
  maps to a quotient stack},
        date={2007},
        ISSN={1056-3911},
     journal={J. Algebraic Geom.},
      volume={16},
      number={4},
       pages={731\ndash 751},
         url={https://doi.org/10.1090/S1056-3911-07-00443-2},
      review={\MR{2357688}},
}

\bib{AGV}{article}{
      author={Abramovich, Dan},
      author={Graber, Tom},
      author={Vistoli, Angelo},
       title={{G}romov-{W}itten theory of {D}eligne-{M}umford stacks},
        date={2008},
     journal={American Journal of Mathematics},
       pages={1337\ndash 1398},
}

\bib{AJ}{article}{
      author={Abramovich, Dan},
      author={Jarvis, Tyler},
       title={Moduli of twisted spin curves},
        date={2003},
     journal={Proceedings of the American Mathematical Society},
      volume={131},
      number={3},
       pages={685\ndash 699},
}

\bib{AS1}{article}{
      author={Acosta, Pedro},
      author={Shoemaker, Mark},
       title={Quantum cohomology of toric blowups and {L}andau-{G}inzburg
  correspondences},
        date={2018},
        ISSN={2313-1691},
     journal={Algebr. Geom.},
      volume={5},
      number={2},
       pages={239\ndash 263},
         url={https://doi.org/10.14231/AG-2018-008},
      review={\MR{3769893}},
}

\bib{AV}{article}{
      author={Abramovich, Dan},
      author={Vistoli, Angelo},
       title={Compactifying the space of stable maps},
        date={2002},
     journal={Journal of the American Mathematical Society},
      volume={15},
      number={1},
       pages={27\ndash 75},
}

\bib{BDFIK}{article}{
      author={Ballard, Matthew},
      author={Deliu, Dragos},
      author={Favero, David},
      author={Isik, M~Umut},
      author={Katzarkov, Ludmil},
       title={Resolutions in factorization categories},
        date={2016},
     journal={Advances in Mathematics},
      volume={295},
       pages={195\ndash 249},
}

\bib{BF}{article}{
      author={Behrend, Kai},
      author={Fantechi, Barbara},
       title={The intrinsic normal cone},
        date={1997},
     journal={Inventiones Mathematicae},
      volume={128},
      number={1},
       pages={45\ndash 88},
}

\bib{BFK14}{article}{
      author={Ballard, Matthew},
      author={Favero, David},
      author={Katzarkov, Ludmil},
       title={A category of kernels for equivariant factorizations and its
  implications for hodge theory},
        date={2014},
     journal={Publications math{\'e}matiques de l'IH{\'E}S},
      volume={120},
      number={1},
       pages={1\ndash 111},
}

\bib{Cal03}{article}{
      author={C{\u{a}}ld{\u{a}}raru, Andrei},
       title={The {M}ukai pairing, {I}: the {H}ochschild structure},
        date={2003},
     journal={preprint arXiv:0308079},
}

\bib{CCK}{article}{
      author={Cheong, Daewoong},
      author={Ciocan{-}Fontanine, Ionu{\c{t}}},
      author={Kim, Bumsig},
       title={Orbifold quasimap theory},
        date={2015},
     journal={Mathematische Annalen},
      volume={363},
      number={3-4},
       pages={777\ndash 816},
}

\bib{CdlOGP}{article}{
      author={Candelas, Philip},
      author={de~la Ossa, Xenia},
      author={Green, Paul~S},
      author={Parkes, Linda},
       title={A pair of {C}alabi-{Y}au manifolds as an exactly soluble
  superconformal theory},
        date={1991},
     journal={Nuclear Physics B},
      volume={359},
      number={1},
       pages={21\ndash 74},
}

\bib{CKFGS}{unpublished}{
      author={Ciocan{-}Fontanine, Ionut},
      author={Favero, David},
      author={Gu\'er\'e, J\'er\'emy},
      author={Kim, Bumsig},
      author={Shoemaker, Mark},
       title={Fundamental factorization of a {G}{L}{S}{M} part {I}{I}:
  Coh{F}{T} axioms},
        note={preprint},
}

\bib{Chiodo}{article}{
      author={Chiodo, Alessandro},
       title={The {W}itten top {C}hern class via {K}-theory},
        date={2006},
     journal={Journal of Algebraic Geometry},
      volume={15},
      number={4},
       pages={681\ndash 707},
}

\bib{CJR}{article}{
      author={Clader, Emily},
      author={Janda, Felix},
      author={Ruan, Yongbin},
       title={Higher-genus wall-crossing in the gauged linear sigma model},
        date={2017},
     journal={preprint arXiv:1706.05038},
}

\bib{CKGWall}{article}{
      author={Ciocan{-}Fontanine, Ionut},
      author={Kim, Bumsig},
       title={Quasimap wall-crossings and mirror symmetry},
        date={2016},
     journal={preprint arXiv:1611.05023},
}

\bib{CKM}{article}{
      author={Ciocan{-}Fontanine, Ionu{\c{t}}},
      author={Kim, Bumsig},
      author={Maulik, Davesh},
       title={Stable quasimaps to {G}{I}{T} quotients},
        date={2014},
     journal={Journal of Geometry and Physics},
      volume={75},
       pages={17\ndash 47},
}

\bib{CL}{article}{
      author={Chang, Huai-Liang},
      author={Li, Jun},
       title={{G}romov--{W}itten invariants of stable maps with fields},
        date={2011},
     journal={International mathematics research notices},
      volume={2012},
      number={18},
       pages={4163\ndash 4217},
}

\bib{CLcos2}{article}{
      author={Chang, Huai-Liang},
      author={Li, Mu-Lin},
       title={Invariants of stable quasimaps with fields},
        date={2020},
        ISSN={0002-9947},
     journal={Trans. Amer. Math. Soc.},
      volume={373},
      number={5},
       pages={3669\ndash 3691},
         url={https://doi.org/10.1090/tran/8011},
      review={\MR{4082252}},
}

\bib{CLL}{article}{
      author={Chang, Huai-Liang},
      author={Li, Jun},
      author={Li, Wei-Ping},
       title={{W}itten's top {C}hern class via cosection localization},
        date={2015},
     journal={Inventiones mathematicae},
      volume={200},
      number={3},
       pages={1015\ndash 1063},
}

\bib{ChenR2}{incollection}{
      author={Chen, Weimin},
      author={Ruan, Yongbin},
       title={Orbifold {G}romov-{W}itten theory},
        date={2002},
   booktitle={Orbifolds in mathematics and physics ({M}adison, {WI}, 2001)},
      series={Contemp. Math.},
      volume={310},
   publisher={Amer. Math. Soc.},
     address={Providence, RI},
       pages={25\ndash 85},
         url={http://dx.doi.org/10.1090/conm/310/05398},
}

\bib{ChenR1}{article}{
      author={Chen, Weimin},
      author={Ruan, Yongbin},
       title={A new cohomology theory of orbifold},
        date={2004},
     journal={Comm. Math. Phys.},
      volume={248},
      number={1},
       pages={1\ndash 31},
         url={http://dx.doi.org/10.1007/s00220-004-1089-4},
}

\bib{ChR}{article}{
      author={Chiodo, Alessandro},
      author={Ruan, Yongbin},
       title={{L}andau-{G}inzburg/{C}alabi-{Y}au correspondence for quintic
  three-folds via symplectic transformations},
        date={2010},
     journal={Invent. Math.},
      volume={182},
      number={1},
       pages={117\ndash 165},
         url={http://dx.doi.org/10.1007/s00222-010-0260-0},
}

\bib{Edidin}{article}{
      author={Edidin, Dan},
       title={{R}iemann-{R}och for {D}eligne-{M}umford stacks},
        date={2012},
     journal={A celebration of algebraic geometry},
      volume={18},
       pages={241\ndash 266},
}

\bib{EHKV}{article}{
      author={Edidin, Dan},
      author={Hassett, Brendan},
      author={Kresch, Andrew},
      author={Vistoli, Angelo},
       title={Brauer groups and quotient stacks},
        date={2001},
        ISSN={0002-9327},
     journal={Amer. J. Math.},
      volume={123},
      number={4},
       pages={761\ndash 777},
  url={http://muse.jhu.edu/journals/american_journal_of_mathematics/v123/123.4edidin.pdf},
      review={\MR{1844577}},
}

\bib{EisMF}{article}{
      author={Eisenbud, David},
       title={Homological algebra on a complete intersection, with an
  application to group representations},
        date={1980},
     journal={Transactions of the American Mathematical Society},
      volume={260},
      number={1},
       pages={35\ndash 64},
}

\bib{EP}{article}{
      author={Efimov, Alexander~I},
      author={Positselski, Leonid},
       title={Coherent analogues of matrix factorizations and relative
  singularity categories},
        date={2015},
     journal={Algebra \& Number Theory},
      volume={9},
      number={5},
       pages={1159\ndash 1292},
}

\bib{FJR07}{unpublished}{
      author={Fan, Huijun},
      author={Jarvis, Tyler},
      author={Ruan, Yongbin},
       title={{T}he {W}itten equation and its virtual fundamental cycle},
        date={2007},
        note={preprint arXiv:0712.4025},
}

\bib{FJR13}{article}{
      author={Fan, Huijun},
      author={Jarvis, Tyler},
      author={Ruan, Yongbin},
       title={The {W}itten equation, mirror symmetry, and quantum singularity
  theory},
        date={2013},
     journal={Annals of Mathematics},
      volume={178},
      number={1},
       pages={1\ndash 106},
}

\bib{FJR15}{article}{
      author={Fan, Huijun},
      author={Jarvis, Tyler},
      author={Ruan, Yongbin},
       title={A mathematical theory of the gauged linear sigma model},
        date={2017},
     journal={Geometry \& Topology},
      volume={22},
      number={1},
       pages={235\ndash 303},
}

\bib{Fulton}{book}{
      author={Fulton, William},
       title={Intersection theory},
   publisher={Springer Science \& Business Media},
        date={2013},
      volume={2},
}

\bib{G1}{article}{
      author={Givental, Alexander~B.},
       title={Equivariant {G}romov-{W}itten invariants},
        date={1996},
     journal={Internat. Math. Res. Notices},
      number={13},
       pages={613\ndash 663},
         url={http://dx.doi.org/10.1155/S1073792896000414},
}

\bib{GKZ}{book}{
      author={Gelfand, Israel~M},
      author={Kapranov, Mikhail},
      author={Zelevinsky, Andrei},
       title={Discriminants, resultants, and multidimensional determinants},
   publisher={Springer Science \& Business Media},
        date={2008},
}

\bib{GR}{article}{
      author={Guo, Shuai},
      author={Ross, Dustin},
       title={The genus-one global mirror theorem for the quintic threefold},
        date={2017},
     journal={preprint arXiv:1703.06955},
}

\bib{Hirano}{article}{
      author={Hirano, Yuki},
       title={Derived {K}n{\"o}rrer periodicity and {O}rlov's theorem for
  gauged {L}andau--{G}inzburg models},
        date={2017},
     journal={Compositio Mathematica},
      volume={153},
      number={5},
       pages={973\ndash 1007},
}

\bib{Hirano17}{article}{
      author={Hirano, Yuki},
       title={Equivalences of derived factorization categories of gauged
  {L}andau--{G}inzburg models},
        date={2017},
     journal={Advances in Mathematics},
      volume={306},
       pages={200\ndash 278},
}

\bib{HKR}{incollection}{
      author={Hochschild, Gerhard},
      author={Kostant, Bertram},
      author={Rosenberg, Alex},
       title={Differential forms on regular affine algebras},
        date={2009},
   booktitle={Collected papers},
   publisher={Springer},
       pages={265\ndash 290},
}

\bib{IIM}{article}{
      author={Inoue, Daisuke},
      author={Ito, Atsushi},
      author={Miura, Makoto},
       title={{I}-functions of {C}alabi--{Y}au 3-folds in {G}rassmannians},
        date={2016},
     journal={preprint arXiv:1607.08137},
}

\bib{Isik}{article}{
      author={Isik, Mehmet~Umut},
       title={Equivalence of the derived category of a variety with a
  singularity category},
        date={2012},
     journal={International Mathematics Research Notices},
      volume={2013},
      number={12},
       pages={2787\ndash 2808},
}

\bib{Kel}{article}{
      author={Keller, Bernhard},
       title={Invariance and localization for cyclic homology of {DG}
  algebras},
        date={1998},
        ISSN={0022-4049},
     journal={J. Pure Appl. Algebra},
      volume={123},
      number={1-3},
       pages={223\ndash 273},
         url={https://doi.org/10.1016/S0022-4049(96)00085-0},
      review={\MR{1492902}},
}

\bib{keller1998cyclic}{article}{
      author={Keller, Bernhard},
       title={On the cyclic homology of ringed spaces and schemes},
        date={1998},
     journal={arXiv preprint math},
        ISSN={9802007/},
}

\bib{Kir}{article}{
      author={Kirwan, Frances~Clare},
       title={Partial desingularisations of quotients of nonsingular varieties
  and their {B}etti numbers},
        date={1985},
        ISSN={0003-486X},
     journal={Ann. of Math. (2)},
      volume={122},
      number={1},
       pages={41\ndash 85},
         url={https://doi.org/10.2307/1971369},
      review={\MR{799252}},
}

\bib{KiemLi}{article}{
      author={Kiem, Young-Hoon},
      author={Li, Jun},
       title={Localizing virtual cycles by cosections},
        date={2013},
     journal={Journal of the American Mathematical Society},
      volume={26},
      number={4},
       pages={1025\ndash 1050},
}

\bib{KM}{article}{
      author={Kontsevich, Maxim},
      author={Manin, Yu},
       title={{G}romov-{W}itten classes, quantum cohomology, and enumerative
  geometry},
        date={1994},
     journal={Communications in Mathematical Physics},
      volume={164},
      number={3},
       pages={525\ndash 562},
}

\bib{KOcos}{unpublished}{
      author={Kim, Bumsig},
      author={Oh, Jengseok},
       title={{L}ocalized {C}hern characters for 2-periodic complexes},
        date={2018},
        note={preprint},
}

\bib{Kon}{article}{
      author={Kontsevich, Maxim},
       title={Deformation quantization of {P}oisson manifolds},
        date={2003},
     journal={Letters in Mathematical Physics},
      volume={66},
      number={3},
       pages={157\ndash 216},
}

\bib{KreschGeometry}{inproceedings}{
      author={Kresch, Andrew},
       title={On the geometry of {D}eligne-{M}umford stacks},
        date={2009},
   booktitle={Algebraic geometry seattle 2005. part 1},
       pages={259\ndash 271},
}

\bib{KV}{article}{
      author={Kresch, Andrew},
      author={Vistoli, Angelo},
       title={On coverings of {D}eligne-{M}umford stacks and surjectivity of
  the {B}rauer map},
        date={2004},
        ISSN={0024-6093},
     journal={Bull. London Math. Soc.},
      volume={36},
      number={2},
       pages={188\ndash 192},
         url={https://doi.org/10.1112/S0024609303002728},
      review={\MR{2026412}},
}

\bib{LP}{article}{
      author={Lin, Kevin~H.},
      author={Pomerleano, Daniel},
       title={Global matrix factorizations},
        date={2013},
        ISSN={1073-2780},
     journal={Math. Res. Lett.},
      volume={20},
      number={1},
       pages={91\ndash 106},
         url={https://doi.org/10.4310/MRL.2013.v20.n1.a9},
      review={\MR{3126725}},
}

\bib{LPS}{article}{
      author={Lee, Yuan-Pin},
      author={Priddis, Nathan},
      author={Shoemaker, Mark},
       title={A proof of the {L}andau-{G}inzburg/{C}alabi-{Y}au correspondence
  via the crepant transformation conjecture},
        date={2016},
     journal={Ann. Sci. \'Ec. Norm. Sup\'er.},
      volume={49},
      number={6},
}

\bib{MFK}{book}{
      author={Mumford, D.},
      author={Fogarty, J.},
      author={Kirwan, F.},
       title={Geometric invariant theory},
     edition={Third},
      series={Ergebnisse der Mathematik und ihrer Grenzgebiete (2) [Results in
  Mathematics and Related Areas (2)]},
   publisher={Springer-Verlag, Berlin},
        date={1994},
      volume={34},
        ISBN={3-540-56963-4},
         url={https://doi.org/10.1007/978-3-642-57916-5},
      review={\MR{1304906}},
}

\bib{Olsson07}{article}{
      author={Olsson, Martin~C},
       title={({L}og) twisted curves},
        date={2007},
     journal={Compositio Mathematica},
      volume={143},
      number={2},
       pages={476\ndash 494},
}

\bib{Pol16}{article}{
      author={Polishchuk, Alexander},
       title={Homogeneity of cohomology classes associated with {K}oszul matrix
  factorizations},
        date={2016},
     journal={Compositio Mathematica},
      volume={152},
      number={10},
       pages={2071\ndash 2112},
}

\bib{PVold}{book}{
      author={Polishchuk, Alexander},
      author={Vaintrob, Arkady},
       title={Algebraic construction of {W}itten's top {C}hern class},
      series={Contemp. Math.},
   publisher={Amer. Math. Soc., Providence, RI},
        date={2001},
      volume={276},
         url={https://doi.org/10.1090/conm/276/04523},
      review={\MR{1837120}},
}

\bib{PV11}{inproceedings}{
      author={Polishchuk, Alexander},
      author={Vaintrob, Arkady},
       title={Matrix factorizations and singularity categories for stacks
  [factorisations matricielles et cat{\'e}gories des singularit{\'e}s pour les
  champs alg{\'e}briques]},
        date={2011},
   booktitle={Annales de l'institut fourier},
      volume={61},
       pages={2609\ndash 2642},
}

\bib{PV}{article}{
      author={Polishchuk, Alexander},
      author={Vaintrob, Arkady},
       title={Matrix factorizations and cohomological field theories},
        date={2016},
        ISSN={0075-4102},
     journal={J. Reine Angew. Math.},
      volume={714},
       pages={1\ndash 122},
         url={https://doi.org/10.1515/crelle-2014-0024},
      review={\MR{3491884}},
}

\bib{ram}{article}{
      author={Ramadoss, Ajay~C},
       title={The {M}ukai pairing and integral transforms in {H}ochschild
  homology},
        date={2010},
     journal={Moscow Mathematical Journal},
      volume={10},
      number={3},
       pages={629\ndash 645},
}

\bib{Shipman}{article}{
      author={Shipman, Ian},
       title={A geometric approach to {O}rlov's theorem},
        date={2012},
     journal={Compositio Mathematica},
      volume={148},
      number={5},
       pages={1365\ndash 1389},
}

\bib{Swan2}{article}{
      author={Swan, Richard~G},
       title={{H}ochschild cohomology of quasiprojective schemes},
        date={1996},
     journal={Journal of Pure and Applied Algebra},
      volume={110},
      number={1},
       pages={57\ndash 80},
}

\bib{Witten1}{article}{
      author={Witten, Edward},
       title={Algebraic geometry associated with matrix models of
  two-dimensional gravity},
        date={1993},
     journal={Topological methods in modern mathematics (Stony Brook, NY,
  1991)},
       pages={235\ndash 269},
}

\bib{Wi2}{incollection}{
      author={Witten, Edward},
       title={Phases of {$N=2$} theories in two dimensions},
        date={1997},
   booktitle={Mirror symmetry, {II}},
      series={AMS/IP Stud. Adv. Math.},
      volume={1},
   publisher={Amer. Math. Soc.},
     address={Providence, RI},
       pages={143\ndash 211},
}

\bib{Yek}{article}{
      author={Yekutieli, Amnon},
       title={The continuous {H}ochschild cochain complex of a scheme
  (survey)},
        date={2003},
      volume={324},
       pages={227\ndash 233},
         url={https://doi.org/10.1090/conm/324/05744},
      review={\MR{1986127}},
}

\end{biblist}
\end{bibdiv}

\end{document}